\documentclass[11pt]{article}

\usepackage[a4paper,margin=2.4cm]{geometry}
\usepackage{amsmath,amssymb,amsthm,mathtools,mathrsfs}
\usepackage[T1]{fontenc}
\usepackage{lmodern}
\usepackage{booktabs,tabularx,array,microtype}
\usepackage[hidelinks]{hyperref}
\newtheorem{theorem}{Theorem}[section]
\newtheorem{lemma}[theorem]{Lemma}
\newtheorem{proposition}[theorem]{Proposition}
\newtheorem{corollary}[theorem]{Corollary}
\theoremstyle{remark}
\newtheorem*{remark}{Remark}
\theoremstyle{plain}
\numberwithin{equation}{section}
\newcommand{\R}{\mathbb R}
\newcommand{\C}{\mathbb C}

\newcommand{\Div}{\operatorname{div}}

\newcommand{\norm}[1]{\left\lVert#1\right\rVert}

\allowdisplaybreaks
\title{Logarithmic Stability for the Inverse Boundary Problem
in Three-Dimensional Isotropic Elasticity}

\author{
Huaian Diao\textsuperscript{1}
\qquad
Umidjon Durdiev\textsuperscript{1,2,3}\\[0.8em]
\small
\textsuperscript{1}School of Mathematics, Jilin University,\\[-0.1em]
\small Changchun, P.\ R.\ of China\\[0.2em]
\small
\textsuperscript{2}Bukhara State University, Uzbekistan\\[0.2em]
\small
\textsuperscript{3}V.~I.~Romanovskiy Institute of Mathematics,\\[-0.1em]
\small Uzbekistan Academy of Science\\[0.6em]
\small
\textit{E-mail addresses:}\\[0.15em]
\small
\texttt{diao@jlu.edu.cn} (H.~Diao)\\[0.1em]
\small
\texttt{umidjan93@list.ru} (U.~Durdiev)
}

\date{}
\hypersetup{
  pdfauthor={Huaian Diao and Umidjon Durdiev},
  pdftitle={Logarithmic Stability for the Inverse Boundary Problem
  in Three-Dimensional Isotropic Elasticity}
}
\begin{document}
\maketitle

\begin{abstract}
We establish logarithmic stability for the simultaneous
recovery of both Lam\'e parameters in three-dimensional
static isotropic elasticity from the full
Dirichlet-to-Neumann map.
The $L^\infty(\Omega)$ estimate holds uniformly on
every fixed bounded $C^{22}$ class of uniformly
strongly convex coefficients, without prescribed
exterior coefficients, agreement of boundary values
or derivatives, or smallness assumptions on material
gradients.
For every integer $s\ge2$, we also obtain boundary
stability estimates with exponents $(s-q)/s$ for
all Cartesian derivatives of orders $0\le q<s$,
where the measurement error $\delta$ is the
operator-norm discrepancy between the
Dirichlet-to-Neumann maps.
These exponents are sharp on admissible classes
containing a constant material satisfying the
a priori inequalities strictly.
On such classes, no uniform modulus of continuity
exists for boundary derivatives of order $s$.

The proof combines an exact four-component reduction
with finite-regularity complex geometrical optics
constructions, while deriving complete matrix-valued
integral identities from the Betti identity for
elastic displacements.
An amplitude annihilated by the leading physical
adjoint lift produces, at the first nonvanishing order,
the third polarization in a rank-three characteristic
quotient.
Holomorphically weighted matrix moments,
Cauchy--Plemelj factorization, and a fixed holomorphic
projection then yield quantitative control of
the relative transport and recovery of both coefficients.
Boundary stability estimates and paired extensions
transfer the analysis from an auxiliary common-exterior
setting to the original measurement domain.
The exact-data argument also establishes global
uniqueness for $C^{22}$ Lam\'e parameters.
\end{abstract}

\clearpage
\setcounter{tocdepth}{2}
\tableofcontents
\clearpage

\section{Introduction}\label{sec:w0-introduction}

The inverse boundary value problem for isotropic elasticity
concerns the recovery of the two spatially varying Lam\'e
parameters from boundary displacement and traction measurements.
We investigate the stability of this recovery in three dimensions,
establishing logarithmic stability for the coefficients in the
interior and power-type stability estimates for their boundary
derivatives. Both estimates are expressed in terms of the
discrepancy between the full Dirichlet-to-Neumann maps
on the original boundary.

\subsection{The elasticity system and the inverse boundary problem}
\label{sec:r2-model}

Let $\Omega\subset\R^3$ be a bounded domain with smooth boundary,
and let $\lambda$ and $\mu$ be real-valued Lam\'e parameters.
For a displacement field $u:\Omega\to\C^3$, the strain tensor,
stress tensor, and elasticity operator are given by
\begin{equation}\label{eq:w0-physical-operator}
\begin{gathered}
\varepsilon(u)=\frac12(\nabla u+\nabla u^T),\\
\sigma_{\lambda,\mu}(u)
=\lambda(\Div u)I_3+2\mu\varepsilon(u),\\
L_{\lambda,\mu}u=\Div\sigma_{\lambda,\mu}(u),
\end{gathered}
\end{equation}
where $I_3$ denotes the $3\times3$ identity matrix.
We consider the static isotropic elasticity system
\begin{equation}\label{eq:w0-elastic-dirichlet}
\begin{cases}
\displaystyle
\Div\bigl(
\lambda(\Div u)I_3+
\mu(\nabla u+\nabla u^T)
\bigr)=0
&\text{in }\Omega,\\[0.4em]
u=f
&\text{on }\partial\Omega.
\end{cases}
\end{equation}
Throughout the paper, we use the component conventions
\begin{equation}\label{eq:r2-tensor-conventions}
(\nabla u)_{ij}=\partial_j u_i,\qquad
(a\otimes b)_{ij}=a_i b_j,\qquad
A:B=\sum_{i,j}A_{ij}B_{ij}.
\end{equation}

For an integer $s\ge2$ and fixed constants $M,c_0>0$,
define the admissible class
\begin{equation}\label{eq:w0-interior-class}
\begin{aligned}
\mathscr C^s(\Omega;M,c_0)
=\bigl\{(\lambda,\mu)\in C^s(\overline\Omega;\R)^2:\ &
\mu\ge c_0,\quad 3\lambda+2\mu\ge c_0,\\
&
\|\lambda\|_{C^s(\overline\Omega)}
+\|\mu\|_{C^s(\overline\Omega)}
\le M
\bigr\}.
\end{aligned}
\end{equation}
The inequalities in \eqref{eq:w0-interior-class} hold pointwise
on $\overline\Omega$.  The strong-convexity conditions $\mu\ge c_0$ and $3\lambda+2\mu\ge c_0$
in \eqref{eq:w0-interior-class}, together with Korn's
inequality \cite{Ciarlet1988}, ensure that \eqref{eq:w0-elastic-dirichlet}
has a unique weak solution
$u_f\in H^1(\Omega;\C^3)$
for each $f\in H^{1/2}(\partial\Omega;\C^3)$ (cf. \cite{Ciarlet1988},).

Let $\nu$ denote the outward unit normal to $\partial\Omega$.
For each $f\in H^{1/2}(\partial\Omega;\C^3)$, let $u_f$
be the unique weak solution of \eqref{eq:w0-elastic-dirichlet}.
The Dirichlet-to-Neumann (DN) map associates the prescribed
boundary displacement with the corresponding elastic traction:
\begin{equation}\label{eq:w0-dn-map}
\begin{aligned}
\Lambda_{\lambda,\mu}^{\Omega}&:
H^{1/2}(\partial\Omega;\C^3)
\longrightarrow H^{-1/2}(\partial\Omega;\C^3),\\
\Lambda_{\lambda,\mu}^{\Omega}f
&=\sigma_{\lambda,\mu}(u_f)\nu.
\end{aligned}
\end{equation}
The traction is understood in the weak sense defined
by \eqref{eq:w5-physical-dn}.

For two admissible coefficient pairs
$(\lambda_j,\mu_j)\in\mathscr C^s(\Omega;M,c_0)$,
$j=1,2$, we define their boundary measurement
discrepancy by
\begin{equation}\label{eq:w0-measured-error}
\delta:=
\bigl\|
\Lambda_{\lambda_1,\mu_1}^{\Omega}
-\Lambda_{\lambda_2,\mu_2}^{\Omega}
\bigr\|_{H^{1/2}\to H^{-1/2}}.
\end{equation}
The operator norm in this definition is taken from
$H^{1/2}(\partial\Omega;\C^3)$ to
$H^{-1/2}(\partial\Omega;\C^3)$.
In particular, $\delta=0$ if and only if the two
full DN maps coincide.

For any smooth domain $D$ on which the coefficient
pairs are defined, we write
\begin{equation}\label{eq:r2-material-map-abbreviation}
\Lambda_j^D=\Lambda_{\lambda_j,\mu_j}^{D},
\qquad j=1,2.
\end{equation}

The inverse boundary value problem is to determine
the two Lam\'e parameters $\lambda$ and $\mu$ from the full DN map,
\[
\Lambda_{\lambda,\mu}^{\Omega}
\longmapsto
(\lambda,\mu)\big|_{\Omega},
\]
and to quantify their recovery in terms of the
measurement discrepancy $\delta$.
In particular, we seek stability estimates for
both coefficients in $L^\infty(\Omega)$ and for
their boundary derivatives.
The measurements are made on the original boundary
$\partial\Omega$; no exterior coefficients are prescribed,
and the candidate materials need not have matching
boundary values or derivatives.

To formulate the boundary stability estimates,
we use the multi-index notation
\begin{equation}\label{eq:r2-multi-index}
a\in\mathbb N_0^3,\qquad
|a|=a_1+a_2+a_3,\qquad
\partial^a=
\partial_1^{a_1}\partial_2^{a_2}\partial_3^{a_3}.
\end{equation}
For each integer $0\le q\le s$, define
\begin{equation}\label{eq:w0-boundary-jets}
J_q=
\sum_{|a|\le q}
\left(
\|\partial^a(\lambda_1-\lambda_2)\|_{C^0(\partial\Omega)}
+
\|\partial^a(\mu_1-\mu_2)\|_{C^0(\partial\Omega)}
\right).
\end{equation}
Here the derivatives are the boundary traces
of the corresponding Cartesian derivatives
of the coefficients in $C^s(\overline\Omega)$.

\subsection{Previous results and contributions}

The inverse boundary value problem for isotropic elasticity
is a system analogue of Calder\'on's conductivity problem
\cite{Calderon1980}.
Sylvester and Uhlmann \cite{SylvesterUhlmann1987}
established global uniqueness for smooth positive
conductivities in dimensions $n\ge3$ using complex
geometrical optics (CGO) solutions. In two dimensions, Nachman \cite{Nachman1996} established global uniqueness and reconstruction for smooth conductivities, while Astala and P\"aiv\"arinta  \cite{AstalaPaivarinta2006} studied uniqueness for bounded measurable conductivities.  On the quantitative side, Alessandrini
\cite{Alessandrini1988} established logarithmic
stability under suitable a priori assumptions,
whereas Mandache \cite{Mandache2001}
proved exponential instability for the inverse
Schr\"odinger problem, with corresponding
consequences for conductivity.
The elasticity system introduces additional
difficulties through its coupled principal part
and the simultaneous variation of two
Lam\'e parameters.

Boundary identification of the Lam\'e parameters,
including stability estimates for their boundary values,
was established in \cite{AkamatsuNakamuraSteinberg1991}.
Higher-order boundary determination and reconstruction
from localized measurements were subsequently developed
in \cite{NakamuraUhlmann1995,LinNakamura2017}.
For smooth Lam\'e parameters on a smooth Riemannian manifold,
on a smooth Riemannian manifold, the full symbol of the elastic DN map was computed, and the determination of all boundary derivatives was established in \cite{TanLiu2023}.  
The boundary estimates in the present paper quantify
derivative recovery under finite regularity assumptions.
Quantitative boundary reconstruction in the presence of
random measurement errors was studied in
\cite{CaroLaiLinZhou2021}.
The latter result concerns stochastic perturbations of
boundary measurements, rather than the deterministic
operator-norm discrepancy considered here.

In two dimensions, Imanuvilov and Yamamoto
\cite{ImanuvilovYamamoto2015} established full-data global
uniqueness for the isotropic Lam\'e system without a
smallness assumption on the coefficients.  For three-dimensional interior recovery, Nakamura and
Uhlmann developed a CGO approach to the elasticity system
\cite{NakamuraUhlmann1994,NakamuraUhlmannErratum2003}.
Their corrected uniqueness theorem requires sufficiently
small gradients of both candidate shear moduli in a
sufficiently high regularity norm, without an analogous
smallness assumption on the first Lam\'e parameter.
Related CGO constructions and perturbative uniqueness
results were obtained in
\cite{EskinRalston2002,EskinRalston2004,ER}.
More recent uniqueness results for three-dimensional
isotropic elasticity impose axisymmetry, variable
separation, or directional quasianalyticity assumptions
on the shear modulus
\cite{ChenJiangLiuTao2026}.
For partial Cauchy data on an arbitrary open boundary
subset, Imanuvilov, Uhlmann, and Yamamoto
\cite[Theorem~0.1]{ImanuvilovUhlmannYamamoto2012}
proved uniqueness when the shear moduli are constant
and the two first Lam\'e parameters agree on the
inaccessible part of the boundary.

Quantitative recovery has also been studied under
more restrictive a priori assumptions.
In \cite{BerettaFranciniVessella2014}, uniqueness and Lipschitz stability were established for piecewise constant Lam\'e  parameters on a prescribed partition, under certain geometric assumptions. 
The analysis was extended to non-flat interfaces in
\cite{BerettaFranciniMorassiRossetVessella2014}.
For time-harmonic elastic waves, uniqueness and Lipschitz
stability from local DN data were established for
piecewise constant Lam\'e parameters and density on a
known partition, under geometric and frequency assumptions
\cite{BerettaDeHoopFranciniVessellaZhai2017}.
Lipschitz stability has also been obtained from
Neumann-to-Dirichlet measurements using monotonicity
and localized potentials, for finite-dimensional
coefficient classes satisfying additional regularity
and ordering assumptions
\cite{EberleHarrachMeftahiRezgui2021}.
These results rely on restricted admissible classes,
in contrast to the full bounded regularity classes
considered in the present paper.

A related but distinct problem is the determination
of an elastic inclusion.
Logarithmic stability for recovering an unknown
inclusion in a piecewise homogeneous isotropic body
was established in
\cite{AlessandriniDiCristoMorassiRosset2014}.
In that setting the unknown is the inclusion geometry,
whereas the present problem concerns the simultaneous
recovery of two spatially varying Lam\'e parameters.

In our preceding work \cite{DDsource}, developed
independently of the OpenAI preprint
\cite{OpenAIElasticity2026}, we established full-data
global uniqueness for smooth, strongly convex Lam\'e
parameters in three dimensions without smallness or
structural assumptions.
A central feature of our approach is the derivation
of complete, holomorphically weighted $3\times3$
matrix identities from the Betti identity for
physical elastic solutions.
The missing third physical adjoint polarization is
obtained from the first nonvanishing correction to
a reduced amplitude annihilated by the leading lift.
These matrix identities then provide the compatibility
needed to construct an invertible, exterior-normalized
relative transport through Cauchy--Plemelj factorization.
The present paper develops this physical-solution
and matrix-moment approach into a quantitative
stability theory.

The main result of this paper is logarithmic stability for
the simultaneous recovery of both Lam\'e parameters from
the full Dirichlet-to-Neumann map defined by \eqref{eq:w0-dn-map}.
Theorem~\ref{thm:w0-target-A} establishes the
$L^\infty(\Omega)$ estimate \eqref{eq:w0-target-stability}
uniformly on every fixed admissible class
$\mathscr C^{22}(\Omega;M,c_0)$, without smallness
assumptions on the material gradients, prescribed
boundary agreement, or known exterior coefficients.
A central ingredient is the quantitative recovery
of complete matrix moments from physical elastic
solutions (Theorem~\ref{thm:w1-all-moments}).
The exact reduction alone does not identify the unrestricted
boundary map of the auxiliary system with the measured
elastic Dirichlet-to-Neumann map.
Quantitative control of the relative transport,
together with boundary determination and paired
extensions, then yields stability from measurements
on the original boundary.

Theorem~\ref{thm:w0-target-B} gives a separate family
of boundary stability estimates.
For every integer $s\ge2$ and $0\le q<s$, the
Cartesian boundary derivatives satisfy
\eqref{eq:w0-target-boundary}, with the power
$\delta^{(s-q)/s}$.
These estimates hold uniformly on each fixed
admissible $C^s$ class, without additional
prior-slack assumptions.
Proposition~\ref{prop:w0-boundary-sharpness}
shows that the powers are optimal on classes
containing a constant material satisfying the
a priori inequalities strictly.
Under the same hypothesis, no uniform vanishing
modulus exists for derivatives of order $s$.
The sharpness assertion concerns the boundary
estimates, not the interior logarithmic exponent.

Finally, the zero-error assertion of
Proposition~\ref{prop:w0-interior-assembly},
with $m=4$, yields global uniqueness for
$C^{22}(\overline\Omega)$ Lam\'e parameters
(Corollary~\ref{cor:w0-target-C}).
This relaxes the $C^\infty$ regularity assumptions
in the independent uniqueness results
\cite{DDsource,OpenAIElasticity2026}.
The conclusion follows from the present
quantitative argument without invoking either
of the preceding uniqueness theorems.

\subsection{Main results}\label{sec:m4-main-results}

\begin{theorem}\label{thm:w0-target-A}
Fix the bounded smooth domain $\Omega\subset\R^3$ and constants $M,c_0>0$.
Let $\mathscr C^{22}(\Omega;M,c_0)$ be the class defined in
\eqref{eq:w0-interior-class} with $s=22$. There exists $\gamma_{*,\mathrm{ext}}\in(0,1)$,
depending only on $\Omega$, $M$, and $c_0$, with
the following property: for every $\alpha$ satisfying
\begin{equation}\label{eq:w0-target-exponent}
0<\alpha<\frac{\gamma_{*,\mathrm{ext}}}{112},
\end{equation}
there exist constants $C_\alpha,\delta_{A,0}>0$ such that
\begin{equation}\label{eq:w0-target-stability}
\|\lambda_1-\lambda_2\|_{L^\infty(\Omega)}
+
\|\mu_1-\mu_2\|_{L^\infty(\Omega)}
\le C_\alpha |\log\delta|^{-\alpha}
\end{equation}
for every two pairs
$(\lambda_j,\mu_j)\in
\mathscr C^{22}(\Omega;M,c_0)$, $j=1,2$,
with $0<\delta<\delta_{A,0}$.
Here $\delta$ is defined in \eqref{eq:w0-measured-error},
and the constants $C_\alpha,\delta_{A,0}$ depend only
on $\Omega$, $M$, $c_0$, and $\alpha$.
\end{theorem}

\begin{remark}
The estimate \eqref{eq:w0-target-stability}  holds uniformly over the full admissible class
$\mathscr C^{22}(\Omega;M,c_0)$.
No exterior coefficients are prescribed, the boundary traces and
derivatives of the two materials need not agree, and no smallness
condition is imposed on their gradients.

The $C^{22}$ regularity assumption is sufficient for the requirements
of the present quantitative construction. Indeed, the $C_x^3$ projection estimate \eqref{eq:w3-C3-output}, obtained
by interpolation, yields a positive power of the error when $m>3$.
The mixed-derivative count in \eqref{eq:w3-order-budget}, together
with the regularity requirements in \eqref{eq:r2-regularity-orders},
gives the sufficient material regularity order $s=2m+14$.
The smallest admissible integer $m=4$ therefore yields $s=22$.
We explain this derivative count in
Subsection~\ref{sec:m4-regularity-scope}.

Theorem~\ref{thm:w0-target-A} follows from the positive-error
assertion of Proposition~\ref{prop:w0-interior-assembly} with $m=4$.
The auxiliary enlarged coefficient class is constructed
from $\mathscr C^{22}(\Omega;M,c_0)$ in
Subsection~\ref{sec:r2-original-assembly}.
Its bounds, geometry, and exponent parameter
$\gamma_{*,\mathrm{ext}}$ depend only on the original
a priori data and are independent of the measurement error.
In particular, the auxiliary construction imposes
no additional assumptions on the original materials.
\end{remark}

Boundary derivatives satisfy the following order-dependent estimates.

\begin{theorem}\label{thm:w0-target-B}
Fix an integer $s\ge2$ and the class
$\mathscr C^s(\Omega;M,c_0)$ in \eqref{eq:w0-interior-class}.
There exist constants $C,\delta_{B,0}>0$, depending only on
$\Omega,s,M,c_0$, such that for every two members of this
class, simultaneously for all integers $0\le q<s$,
\begin{equation}\label{eq:w0-target-boundary}
 J_q\le C\delta^{(s-q)/s},\qquad q\in\{0,\ldots,s-1\},\quad
 0<\delta<\delta_{B,0}.
\end{equation}
Here $J_q$ and $\delta$ are defined in
\eqref{eq:w0-boundary-jets} and \eqref{eq:w0-measured-error}, respectively.
If $\delta=0$, then $J_q=0$ for every $0\le q<s$.
\end{theorem}

\begin{remark}
The boundary values satisfy a Lipschitz estimate, while the
powers for higher-order derivatives arise from balancing the
measurement error against the Taylor remainder in
\eqref{eq:v63-coefficient-noise}.
In particular, the estimate
$J_{s-1}\le C\delta^{1/s}$ provides the boundary control
required for the paired extension in
Section~\ref{sec:w5-extensions}.
This construction uses the uniform $C^s$ a priori bound
but requires no stability estimate for $J_s$.
\end{remark}

The exponents in Theorem~\ref{thm:w0-target-B} are optimal
on admissible classes $\mathscr C^s(\Omega;M,c_0)$ containing a constant material
satisfying the a priori inequalities strictly.

\begin{proposition}
\label{prop:w0-boundary-sharpness}
Let $s\ge2$ be an integer, and suppose that
$\mathscr C^s(\Omega;M,c_0)$ contains a constant pair
$(\lambda_b,\mu_b)$ satisfying
\begin{equation}\label{eq:w0-sharpness-slack}
\mu_b>c_0,\qquad
3\lambda_b+2\mu_b>c_0,\qquad
\|\lambda_b\|_{C^s(\overline\Omega)}
+\|\mu_b\|_{C^s(\overline\Omega)}<M.
\end{equation}
For each $q\in\{0,\ldots,s\}$, there exist constants
$c_q,C_q>0$ and a family of smooth coefficient pairs
$(\lambda_{a,q},\mu_b)\in\mathscr C^s(\Omega;M,c_0)$,
defined for all sufficiently small $a>0$, such that
\begin{equation}\label{eq:r2-calculation-001}
0<\delta_{a,q}\le C_q a^s,
\qquad
J_{q,a}\ge c_q a^{s-q}.
\end{equation}
Here $\delta_{a,q}$ and $J_{q,a}$ denote the discrepancies
defined in \eqref{eq:w0-measured-error} and
\eqref{eq:w0-boundary-jets}, respectively, between
$(\lambda_{a,q},\mu_b)$ and $(\lambda_b,\mu_b)$.

Consequently, for $0\le q<s$, no estimate
\[
J_q\le\omega(\delta),
\qquad
\omega(t)=o\bigl(t^{(s-q)/s}\bigr)
\quad\text{as }t\downarrow0,
\]
holds uniformly on this class.
For $q=s$, no uniform estimate
$J_s\le\omega(\delta)$ is possible with
$\omega(t)\to0$ as $t\downarrow0$.
\end{proposition}

The zero-error conclusion of Proposition~\ref{prop:w0-interior-assembly}
also yields uniqueness under finite regularity.

\begin{corollary}
\label{cor:w0-target-C}
Let
$(\lambda_j,\mu_j)\in C^{22}(\overline\Omega;\mathbb R)^2$,
$j=1,2$, satisfy
\[
\mu_j>0,\qquad 3\lambda_j+2\mu_j>0
\quad\text{on }\overline\Omega.
\]
If
\[
\Lambda_{\lambda_1,\mu_1}^{\Omega}
=
\Lambda_{\lambda_2,\mu_2}^{\Omega},
\]
then
\[
\lambda_1=\lambda_2,\qquad
\mu_1=\mu_2
\quad\text{in }\Omega.
\]
\end{corollary}

The boundary statements and the uniqueness corollary are
proved in Subsection~\ref{sec:r2-main-boundary-proofs}.

\subsection{Regularity requirements}
\label{sec:m4-regularity-scope}

The $C^{22}$ regularity assumption in
Theorem~\ref{thm:w0-target-A} arises from the
quantitative transport construction.
The coefficient-rigidity theorem
(Theorem~\ref{thm:w4-rigidity}) requires the
first-order transport residual to be small in $C_x^2$.
By \eqref{eq:w3-residual-Leibniz}, this requires
$C_x^3$ control of the projection error.
Proposition~\ref{prop:w3-defect} yields the
interpolation estimate \eqref{eq:w3-C3-output},
\[
\|H-H_0\|_{C_x^3C_\zeta^0}
\le C\omega_{\rm F}^{\beta_m},
\qquad
\beta_m=\frac{m-3}{2m}.
\]
Thus, a positive power of the error is obtained
precisely when $m>3$.

The remaining regularity requirements arise from
mixed spatial and complex-directional derivatives
of the transport frames.
With $N=m+2$, the derivative count
\eqref{eq:w3-order-budget} gives
\[
N+(N+3)=2m+7<r_{\rm fr}=2m+8.
\]
Applying Theorem~\ref{thm:finite-frames} with
$r=r_{\rm fr}+2$ requires $C^{r_{\rm fr}+4}$
bounds for the transport coefficients.
Proposition~\ref{prop:w1-exact-reduction} shows
that the elastic reduction involves derivatives
of the Lam\'e parameters through order two;
see \eqref{eq:r2-reduced-coefficient-bound}.
Consequently, Corollary~\ref{cor:w2-physical-budgets}
gives the sufficient material regularity
\[
s=r_{\rm fr}+6=2m+14,
\]
as specified in \eqref{eq:r2-regularity-orders}.
The smallest admissible choice $m=4$ therefore
gives $r_{\rm fr}=16$ and $s=22$.

This regularity threshold is sufficient for the
present construction.
It also suffices for global uniqueness, as stated in
Corollary~\ref{cor:w0-target-C}.
The higher-regularity result is given in
Proposition~\ref{prop:w0-interior-assembly},
with the admissible logarithmic exponents specified
in \eqref{eq:m4-general-exponent}.

\subsection{Proof strategy}
\label{sec:r2-proof-strategy}

The argument first treats a fixed common-exterior class and then transfers
the resulting estimate to the original domain by boundary recovery and
paired extensions.

For the auxiliary problem, fix an integer $m\ge4$
and set $s=2m+14$. Let $D\subset\R^3$ be a bounded
smooth domain, and let
$(\lambda_{\rm bg},\mu_{\rm bg})$ be a fixed,
uniformly strongly convex $C^s$ pair on $\R^3$,
constant outside a fixed ball. We consider global
coefficient pairs with uniform $C^s$ bounds satisfying
\begin{equation}\label{eq:m4-common-exterior}
(\lambda_j,\mu_j)=(\lambda_{\rm bg},\mu_{\rm bg})
\quad\text{on }\R^3\setminus D,\qquad j=1,2.
\end{equation}
For a fixed measurement domain $D$, write
\begin{equation}\label{eq:w0-domain-error}
\begin{gathered}
\delta_D=
\norm{\Lambda_1^D-\Lambda_2^D}_{
H^{1/2}(\partial D;\C^3)
\to H^{-1/2}(\partial D;\C^3)},\\
\delta_\Omega=\delta,\qquad
\delta_{B_{R_e}}=\Delta_e.
\end{gathered}
\end{equation}
Here $B_{R_e}$ is the enlarged measurement domain
introduced in the second part of the argument.
The common-exterior geometry is fixed in
Subsection~\ref{sec:r2-common-exterior-setup}.
In particular, the coefficients agree in a
neighborhood of the larger working sphere
$\partial B_R$.
Lemma~\ref{lem:w1-domain-comparison} gives
$\delta_{B_R}\le C\delta_D$.
Thus the Betti identity for the elastic system
(Lemma~\ref{lem:w1-betti}) may be applied on $B_R$
with the boundary pairing bounded in terms of the
discrepancy measured on $\partial D$.

The exact four-component reduction
(Proposition~\ref{prop:w1-exact-reduction}),
together with the finite-order transport frames
and right inverse of
Theorems~\ref{thm:finite-frames} and
\ref{thm:finite-right-inverse}, provides the
required CGO solutions with uniform remainder
estimates. A central difficulty is that the leading
physical adjoint lift has a two-dimensional image
in the rank-three characteristic quotient.
Proposition~\ref{prop:w1-high-frames} supplies the
missing third polarization through the first
nonvanishing correction to an amplitude annihilated
by the leading lift.
Applying the Betti identity to these solutions relates
the polynomially weighted matrix moments in the fixed
primal and dual quotient frames to DN pairings of
elastic solutions.
Theorem~\ref{thm:w1-all-moments} bounds all nine
entries in terms of the DN discrepancy $\delta_D$
and the CGO remainders.
The dependence on the test degree and localization
width is given explicitly by $K_{n,\varsigma}$ in
\eqref{eq:w1-test-family}, while the constants and
the threshold $h_*$ are independent of these test
parameters.

The moment estimates control the negative Fourier
modes of the boundary traces of the frame products
$W_a=Y_{2,a}^{T}X_{1,a}$; see \eqref{eq:w3-WF}
and Lemma~\ref{lem:w3-negative}.
Lemma~\ref{lem:w3-cauchy} shows that the corresponding
Cauchy extensions $R_{C,a}$ are pointwise invertible
on the closed leaf disks.
The primal and dual quotient frames then define
$G_a=X_{1,a}R_{C,a}^{-1}Y_{2,a}^{T}$, which satisfies
\[
2\partial_{\bar z}G_a+P_1G_a-G_aP_2=0;
\]
see \eqref{eq:w3-local-G}.
Quantitative comparison on overlaps yields the
global invertible quotient endomorphisms $G$ and
$H=G^{-T}$, with uniform regularity bounds and
controlled deviation from the identity on
$\partial B_R$; see Proposition~\ref{prop:w3-gluing}.
At each fixed physical point, the projection
$H_0=\Pi H$ is holomorphic in the complex null
direction. Proposition~\ref{prop:w3-defect}
controls the projection error.
The first-order transport equation then yields
the $C_x^2$ residual estimate
\eqref{eq:w3-rigidity-output}.
The trace-free Hessian estimate and the
coefficient-rigidity theorem
(Theorem~\ref{thm:w4-rigidity})
convert this residual into estimates for both
Lam\'e parameters.

More precisely, the estimates assembled in
\eqref{eq:w0-CE1} give
\[
\mathfrak d_D
\le C\left(
h+\delta_Dh^{-2}e^{C_0/h}
\right)^{\gamma'\beta_m/14},
\qquad
\beta_m=\frac{m-3}{2m},
\quad 0<\gamma'<\gamma_*,
\]
whenever the auxiliary parameters satisfy the
fixed smallness conditions.
Here $h>0$ is the CGO parameter,
$\mathfrak d_D$ is the $L^\infty$ coefficient
discrepancy defined in
\eqref{eq:r2-common-error-definitions},
and $\gamma_*$ is given by \eqref{eq:w3-gamma}.
Logarithmic optimization yields the
common-exterior stability estimate in
Proposition~\ref{prop:w0-common-assembly}.

To remove the common-exterior assumption, we first
recover the boundary derivatives directly from
the physical DN map.
The uniform boundary energy expansion
(Theorem~\ref{thm:v63-energy-expansion})
and the triangular recovery of successive normal
derivatives yield, after tangential interpolation,
the estimates
\[
J_q\le C\delta^{(s-q)/s},
\qquad 0\le q<s,
\]
of Theorem~\ref{thm:v63-boundary-jet-stability}.
The perturbation construction in
Proposition~\ref{prop:w5-boundary-sharpness}
establishes the sharpness of these powers
under the stated strict-slack assumptions.

The highest-order estimate needed for the
extension argument is
\[
J_{s-1}\le C\delta^{1/s}.
\]
Theorem~\ref{thm:v62-compensated-extension}
uses this boundary control and an arbitrary
positive tolerance $\epsilon_{\rm tol}$ to
construct paired $C^s$ extensions with fixed
strong-convexity bounds and a common exterior
background. The extensions agree with the
original coefficients on $\overline\Omega$,
but need not agree with each other in
$B_{R_e}\setminus\overline\Omega$.
Writing $\epsilon_{\rm ext}$ for their exterior
$L^\infty$ discrepancy, the extension theorem
and Proposition~\ref{prop:v61-enlargement-error}
give
\[
\begin{aligned}
\epsilon_{\rm ext}
&\le C(J_{s-1}+\epsilon_{\rm tol}),\\
\Delta_e
&\le C(\delta+\epsilon_{\rm ext})
 \le C(\delta+J_{s-1}+\epsilon_{\rm tol}).
\end{aligned}
\]
These are the estimates
\eqref{eq:v62-compensated-error}
and \eqref{eq:w0-enlargement}, respectively.
In particular, $\Delta_e$ is controlled by
the original measurements and is not
an additional datum.

For $0<\delta$ sufficiently small, choosing
$\epsilon_{\rm tol}=\delta+J_{s-1}$ gives
\begin{equation*}
\Delta_e\le C\delta^{1/s},
\end{equation*}
as established in
Corollary~\ref{cor:w5-noise-transfer}.
All extensions belong to one fixed enlarged
coefficient class, whose bounds and geometry
are independent of $\delta$ and
$\epsilon_{\rm tol}$.
The common-exterior stability estimate can
therefore be applied with uniform constants.
The power $1/s$ changes only the multiplicative
constant in the logarithmic comparison, not
the admissible logarithmic exponent.

Finally, when $\delta=0$, the boundary estimates
give $J_{s-1}=0$.
For every sufficiently small $t>0$, the
extension construction remains valid with
$\epsilon_{\rm tol}=t$, and
\eqref{eq:w0-exact-original-noise} gives
\[
\Delta_e(t)\le Ct.
\]
If $\Delta_e(t)=0$, the common-exterior
exact-data result applies directly.
Otherwise, its positive-error estimate
provides a numerical upper bound for the
unchanged original coefficient discrepancy
that tends to zero as $t\downarrow0$.
This proves the exact-data conclusion without
taking a limit of the extensions themselves.
Proposition~\ref{prop:w0-interior-assembly}
completes both parts of the argument.

\subsection{Notation and conventions}
\label{sec:r2-notation}

The tensor conventions are fixed in \eqref{eq:r2-tensor-conventions}. A
superscript $T$ denotes ordinary matrix transpose or formal transpose for
the complex-bilinear pairing, without complex conjugation. The
three-component physical displacement $u$ is distinguished from the
four-component reduced unknown $\boldsymbol z$ in
\eqref{eq:r2-reduced-unknown}. The elastic traction uses the outward normal,
whereas boundary-layer coordinates are directed inward.

For a normalized null direction $\theta=e_1+ie_2$,
the complex leaf coordinate $z$ and transverse
coordinate $t_\perp$ are defined in
\eqref{eq:w1-leaf-coordinates}.
The affine spinor coordinates $\zeta$ and $\upsilon$
are defined in \eqref{eq:r2-spinor-charts}.
The Fourier projections $P_+$ and $P_-$ act on
the boundaries of the leaf disks, as in
\eqref{eq:r2-circle-projections}.
The fixed projection $\Pi$ in
\eqref{eq:w3-fixed-Pi} acts on endomorphism sections
over the projective space of complex null directions
at each fixed physical point $x$.
Thus the projected field $H_0=\Pi H$, defined in
\eqref{eq:r2-transport-projection-definitions},
is holomorphic in the null direction, whereas
the Cauchy extensions $R_{C,a}$ are holomorphic
in the leaf coordinate $z$.

The CGO parameter is $h>0$, with $\tau=h^{-1}$.
We use $\varsigma$ for the transverse localization
width, $\epsilon$ for the boundary-layer scale,
and $\epsilon_{\rm tol}$ for the positive tolerance
in the paired extension construction.
Their respective roles are specified in
\eqref{eq:w1-test-family},
\eqref{eq:v63-physical-tests}, and
\eqref{eq:w5-extension-input}.
The integer $m\ge4$ indexes the regularity in
the interior estimates.

For the common-exterior problem, set
\begin{equation}\label{eq:r2-common-error-definitions}
\begin{gathered}
\eta_D=h+\delta_Dh^{-2}e^{C_0/h},
\qquad
\beta_m=\frac{m-3}{2m},\\
\mathfrak d_D=
\|\lambda_1-\lambda_2\|_{L^\infty(D)}
+
\|\mu_1-\mu_2\|_{L^\infty(D)}.
\end{gathered}
\end{equation}
Here $\delta_D$ is defined in
\eqref{eq:w0-domain-error}, and $C_0$ is the
fixed constant from
Theorem~\ref{thm:w1-all-moments}.
The Fourier error $\omega_{\rm F}$ and rigidity
error $\varepsilon_{\rm rig}$ are specified in
\eqref{eq:r2-fourier-error} and
\eqref{eq:w4-rigidity-input}, respectively.
Unless otherwise indicated, $C$ denotes a positive
constant depending on the fixed geometry, a priori
bounds, and explicitly fixed auxiliary parameters,
and may change from line to line.

\section{Physical identities and the common-exterior framework}
\label{sec:w1-reduction}

This section fixes the physical variational pairing, compares the elastic DN
maps defined by \eqref{eq:w0-dn-map} on nested domains, and establishes the
exact reduction. Under the common-collar hypothesis, the Betti identity
expresses a reduced volume pairing in terms of displacement and traction
data.

\subsection{Physical variational identities and uniform energy bounds}
\label{sec:physical-preliminaries}

For any fixed smooth domain $D$, let $\varepsilon(u)$ be the strain in
\eqref{eq:w0-physical-operator}. Define the bilinear and Hermitian energies by
\begin{equation}\label{eq:w5-physical-form}
\begin{split}
 a_D(u,v)&=\int_D\{\lambda\operatorname{div}u\operatorname{div}v
                     +2\mu\varepsilon(u):\varepsilon(v)\}\,\mathrm dx,
 \qquad \mathfrak a_D(u,v)=a_D(u,\overline v).
\end{split}
\end{equation}
The form $a_D$ in \eqref{eq:w5-physical-form} is complex bilinear,
and $\mathfrak a_D$ is Hermitian. The elastic Betti identities use $a_D(u,v)$,
whereas coercive energy estimates use
$\mathfrak a_D(u,u)=a_D(u,\overline u)$.
Both forms act on $H^1(D;\mathbb C^3)$. The equation is the static
elasticity equation $\operatorname{div}\sigma_{\lambda,\mu}(u)=0$.
For the solution $u_f$ with trace $f$, the DN map defined
as in \eqref{eq:w0-dn-map} on $D$ satisfies
\begin{equation}\label{eq:w5-physical-dn}
 \langle\Lambda^D f,g\rangle=a_D(u_f,v_g),\qquad
 f,g\in H^{1/2}(\partial D;\mathbb C^3).
\end{equation}
Here $\langle\cdot,\cdot\rangle$ denotes the complex-bilinear duality
between $H^{-1/2}(\partial D;\mathbb C^3)$ and
$H^{1/2}(\partial D;\mathbb C^3)$, and $v_g$ is any $H^1$ lifting of $g$.
The weak equation makes it independent of the lifting and agrees with
traction for the outward unit normal. The original-domain discrepancy remains $\delta$ from
\eqref{eq:w0-measured-error}; the domain-dependent discrepancy is
$\delta_D$ from \eqref{eq:w0-domain-error}.

For a symmetric complex matrix $E$, define its trace-free part by
\begin{equation}\label{eq:r3-energy-deviator}
E_0=E-\tfrac13\operatorname{tr}(E)I_3.
\end{equation} The priors
\eqref{eq:w0-interior-class} give
\begin{equation}\label{eq:r2-convex-energy}
 \begin{aligned}
 \lambda|\operatorname{tr}E|^2+2\mu|E|^2
 &=2\mu|E_0|^2+\frac{3\lambda+2\mu}{3}|\operatorname{tr}E|^2
 \ge c_0|E|^2.
 \end{aligned}
\end{equation}
Zero extension, Plancherel's identity and Poincar\'e's inequality yield
\begin{equation}\label{eq:r2-energy-coercivity}
 \begin{gathered}
 2\|\varepsilon(w)\|_{L^2(D)}^2
   =\|\nabla w\|_{L^2(D)}^2+\|\operatorname{div}w\|_{L^2(D)}^2,\\
 \mathfrak a_D(w,w)\ge c_D\|w\|_{H^1(D)}^2,
 \qquad w\in H_0^1(D;\mathbb C^3).
 \end{gathered}
\end{equation}
Applying \eqref{eq:r2-energy-coercivity} to the zero-trace difference between
the solution and a bounded lifting of its trace, we obtain
\begin{equation}\label{eq:w5-dirichlet-bound}
 \|u_f\|_{H^1(D)}\le C_D\|f\|_{H^{1/2}(\partial D)}.
\end{equation}
The trace lifting can equivalently be defined using the quotient norm
of the trace of $H^1(D)$; on the fixed smooth boundary that norm is
equivalent to the usual $H^{1/2}$ norm. This observation fixes the norm
used in the physical estimates of Sections~\ref{sec:w1-reduction},
\ref{sec:v63-boundary-jets}, and \ref{sec:w5-extensions}. On $\Omega$ we abbreviate
$\mathfrak a_\Omega$ by
\begin{equation}\label{eq:r2-energy-abbreviation}
 \mathfrak a=\mathfrak a_\Omega.
\end{equation}
A single material is understood when this abbreviation is used.

\subsection{Comparison of DN discrepancies on nested domains}\label{sec:r2-domain-comparison}
To transfer measurements to an enlarged domain, we compare the DN
discrepancies on $\partial\Omega$ and $\partial B_{R_e}$, defined in
\eqref{eq:w0-domain-error}, while retaining the exterior material mismatch.
For the scalar conductivity equation, a formula relating the original
and enlarged-domain DN maps is given in
\cite[Section~2]{GarciaZhang2013}.

Fix a ball $B_{R_e}$ with $\overline\Omega\Subset B_{R_e}$. The comparison
uses only bounded strongly convex material values and allows an exterior
mismatch at the inner boundary.

\begin{proposition}\label{prop:v61-enlargement-error}
Let two material pairs on the fixed $B_{R_e}$ be uniformly bounded and
strongly convex, and let $\delta$ be their DN discrepancy on
$\Omega$. Put
\begin{equation}\label{eq:v61-exterior-error}
 \epsilon_{\rm ext}^{B}=
 \|\lambda_1-\lambda_2\|_{L^\infty(B_{R_e}\setminus\overline\Omega)}+
 \|\mu_1-\mu_2\|_{L^\infty(B_{R_e}\setminus\overline\Omega)}.
\end{equation}
Then the enlarged-domain discrepancy satisfies
\begin{equation}\label{eq:v61-enlargement-bound}
 \Delta_e\le
 C(\delta+\epsilon_{\rm ext}^{B}).
\end{equation}
Here $\delta$ and $\Delta_e$ are defined in
\eqref{eq:w0-measured-error} and \eqref{eq:w0-domain-error}, respectively. The constant depends only on the two fixed domains, the upper material
bounds and the strong-convexity lower bound.
\end{proposition}
\begin{proof}
Let $u_1,u_2$ solve the two problems on $B_{R_e}$ with outer traces $F,G$,
respectively. By the definition in \eqref{eq:w5-physical-dn}, symmetry of
the bilinear energy and the weak equations,
\begin{equation}\label{eq:w5-physical-alessandrini}
\begin{aligned}
 &\langle(\Lambda_1^{B_{R_e}}-\Lambda_2^{B_{R_e}})F,G\rangle\\
 &\quad=\int_{B_{R_e}}(\lambda_1-\lambda_2)
       \operatorname{div}u_1\operatorname{div}u_2\,\mathrm dx\\
 &\qquad+2\int_{B_{R_e}}(\mu_1-\mu_2)
       \varepsilon(u_1):\varepsilon(u_2)\,\mathrm dx.
\end{aligned}
\end{equation}
More explicitly, $a_1(u_1,u_2)=\langle\Lambda_1F,G\rangle$.
In $a_2(u_1,u_2)$ the first argument may be replaced by the material-two
solution with trace $F$, because its difference from $u_1$ has zero trace
and $u_2$ solves the material-two equation. This gives the second DN
pairing and proves the identity without a boundary-matching assumption.

The restrictions $u_j|_\Omega$ solve their respective equations. Define
\begin{equation}\label{eq:r3-inner-displacement-traces}
f=\operatorname{Tr}_{\partial\Omega}u_1,\qquad g=\operatorname{Tr}_{\partial\Omega}u_2.
\end{equation}
Then the integral over $\Omega$
in \eqref{eq:w5-physical-alessandrini} is
$\langle(\Lambda_1^\Omega-\Lambda_2^\Omega)f,g\rangle$.
We therefore have the exact split identity
\begin{equation}\label{eq:v61-enlargement-identity}
\begin{aligned}
 \langle(\Lambda_1^{B_{R_e}}-\Lambda_2^{B_{R_e}})F,G\rangle
 ={}&\langle(\Lambda_1^\Omega-\Lambda_2^\Omega)f,g\rangle\\
 &+\int_{B_{R_e}\setminus\overline\Omega}
       (\lambda_1-\lambda_2)\operatorname{div}u_1\operatorname{div}u_2\,\mathrm dx\\
 &+2\int_{B_{R_e}\setminus\overline\Omega}
       (\mu_1-\mu_2)\varepsilon(u_1):\varepsilon(u_2)\,\mathrm dx.
\end{aligned}
\end{equation}
The uniform Dirichlet estimate \eqref{eq:w5-dirichlet-bound} and the
trace bound on the fixed inner domain give
$\|u_1\|_{H^1(B_{R_e})}+\|f\|_{H^{1/2}(\partial\Omega)}
\le C\|F\|_{H^{1/2}(\partial B_{R_e})}$, and the analogous estimate
for $u_2,g,G$. Cauchy--Schwarz bounds the two exterior terms by
$C\epsilon_{\rm ext}^{B}\|F\|_{H^{1/2}}\|G\|_{H^{1/2}}$.
The inner term is bounded by the same expression with $\delta$.
All pairings in this proof are bilinear; Hermitian energy was used only for
the solution norm estimates. Taking the supremum over $F,G\in H^{1/2}(\partial
B_{R_e};\mathbb C^3)$ with $\|F\|_{H^{1/2}}=\|G\|_{H^{1/2}}=1$, we obtain
\eqref{eq:v61-enlargement-bound} by the definition of $\Delta_e$ in
\eqref{eq:w0-domain-error}. This completes the proof.
\end{proof}

\begin{remark}
This comparison separates two sources of enlarged-domain error:
the measured discrepancy inside $\Omega$ and the material mismatch
outside it. Thus the paired extension needs to make only the exterior
coefficient values close; the comparison itself does not require their
boundary derivatives to agree.
\end{remark}

The identity \eqref{eq:w5-physical-alessandrini}, applied on a single
fixed domain $D$, gives the forward bound
used by the sharpness examples:
\begin{equation}\label{eq:w5-forward-lipschitz}
 \|\Lambda_1^D-\Lambda_2^D\|_{H^{1/2}\to H^{-1/2}}
 \le C_D(\|\lambda_1-\lambda_2\|_\infty+\|\mu_1-\mu_2\|_\infty).
\end{equation}
Indeed, \eqref{eq:w5-dirichlet-bound} bounds each solution, and
Cauchy--Schwarz bounds the two terms in
\eqref{eq:w5-physical-alessandrini}.

\subsection{Conventions, coefficient bounds, and operator types}\label{sec:r2-reduction-definitions}

The four-component reduction has Laplacian principal part. A related
reduction obtained by adjoining the divergence of the displacement appears
in \cite{AngIkehataTrongYamamoto1998}. The reduced unknown is defined in
\eqref{eq:r2-reduced-unknown}; the material variables and operator mapping
properties are specified below before the factorization is stated.

For an open set $O\subset\mathbb R^3$, use the physical operator,
stress and strain from \eqref{eq:w0-physical-operator}, with
$u:O\to\mathbb C^3$ and the component conventions
\eqref{eq:r2-tensor-conventions}. Superscript $T$ denotes the ordinary
matrix or formal differential transpose for the complex-bilinear pairing.
The Hermitian energy is $\mathfrak a_D$ from \eqref{eq:w5-physical-form}.

For an integer $r\ge2$, assume real coefficients with
$\|\lambda\|_{C^r(O)}+\|\mu\|_{C^r(O)}\le M$, $M\ge1$, and
$\mu\ge c_0>0$, $3\lambda+2\mu\ge c_0$. Define
\begin{equation}\label{eq:w1-material-variables}
 \kappa=\frac{\lambda+\mu}{\lambda+2\mu},\quad
 c=\mu^{-1/2},\quad \ell_\mu=\log\mu,\quad g_\mu=\nabla\ell_\mu,
 \quad d=\kappa\sqrt\mu,\quad e=\tfrac12(3\kappa-2)\sqrt\mu.
\end{equation}
All these are scalar fields except $g_\mu:O\to\mathbb R^3$.
The following bounds explicitly control the denominators:
\begin{equation}\label{eq:w1-denominators}
 \begin{gathered}
 5c_0/3\le\lambda+2\mu\le3M,\qquad
 1-\kappa\ge c_0/(3M),\qquad
 \kappa-1/4\ge c_0/(12M),\\
 M^{-1/2}\le c\le c_0^{-1/2},\qquad
 \lambda=\mu\frac{2\kappa-1}{1-\kappa}.
 \end{gathered}
\end{equation}
Indeed, $\lambda+2\mu=((3\lambda+2\mu)+4\mu)/3$,
$1-\kappa=\mu/(\lambda+2\mu)$, and
$4\kappa-1=(3\lambda+2\mu)/(\lambda+2\mu)$.

Set
\begin{equation}\label{eq:w1-reduced-coefficients}
 \begin{aligned}
 F_\mu&=2c[\nabla^2\ell_\mu-g_\mu\otimes g_\mu
                    +( |g_\mu|^2-\Delta\ell_\mu)I_3],\\
 V_\mu&=-\nabla^2\ell_\mu-\tfrac12g_\mu\otimes g_\mu
               -(\tfrac12\Delta\ell_\mu+\tfrac14|g_\mu|^2)I_3,\\
 q_\mu&=\Delta\ell_\mu-|g_\mu|^2.
 \end{aligned}
\end{equation}
$F_\mu,V_\mu$ are symmetric $3\times3$ matrix fields, and $q_\mu$ is
scalar. We use the physical transformation in \cite[Section~3, p.~6]{ER}
in the following normalization:

\begin{equation}\label{eq:w1-T-S-M}
 \begin{aligned}
 T_\mu(v,f)&=cv+\mu^{-1}\nabla f-f\nabla(\mu^{-1}),\\
 S_{\mu,\kappa}(p,q)&=\sqrt\mu\,p+\nabla\big(q/(1-\kappa)\big),\\
 \mathcal M_{\mu,\kappa}\binom vf
 &=\binom{\Delta v-g_\mu\operatorname{div}v+F_\mu\nabla f
                      +V_\mu v+F_\mu g_\mu f}
 {\Delta f+d\operatorname{div}v-e\,g_\mu\cdot v+q_\mu f}.
 \end{aligned}
\end{equation}
Here $v:O\to\mathbb C^3$, $f:O\to\mathbb C$, and
\begin{equation}\label{eq:r2-reduced-unknown}
 \boldsymbol z=\binom vf:O\longrightarrow\mathbb C^4,\qquad
 T=T_\mu,\quad S=S_{\mu,\kappa},\quad
 \mathcal M=\mathcal M_{\mu,\kappa},\quad L=L_{\lambda,\mu}.
\end{equation}
The operators $T$ and $S$ in \eqref{eq:w1-T-S-M} are $3\times4$
first-order differential operators. On any fixed bounded smooth
$O'\Subset O$, the coefficient bounds \eqref{eq:w1-denominators}
and the product rule give bounded maps
\begin{equation}\label{eq:r2-physical-mappings}
 \begin{gathered}
 T,S:H^2(O';\mathbb C^4)\longrightarrow H^1(O';\mathbb C^3),\\
 \mathcal M:H^2(O';\mathbb C^4)\longrightarrow L^2(O';\mathbb C^4),\qquad
 S:L^2(O';\mathbb C^4)\longrightarrow H^{-1}(O';\mathbb C^3).
 \end{gathered}
\end{equation}

Let $e_1,e_2,e_3$ denote the coordinate vectors in $\R^3$.
For $h>0$ and a nonzero constant null vector $\theta\in\C^3$, set
\begin{equation}\label{eq:r2-phase-parameters}
\begin{gathered}
\theta\cdot\theta=0,\qquad \varphi=\theta\cdot x,\qquad
D_\theta=\theta\cdot\nabla_x,\qquad e_\theta=(\theta,0)^T,\qquad
\sigma_{\rm ph}\in\{+1,-1\}.
\end{gathered}
\end{equation}
We use the exponential $e^{\sigma_{\rm ph}\varphi/h}$,
with no additional factor of $i$.

\begin{proposition}
\label{prop:w1-exact-reduction}
Under the coefficient assumptions of this subsection, let
$L,T,S,\mathcal M$ be defined by \eqref{eq:w1-T-S-M} and
\eqref{eq:r2-reduced-unknown}, and fix a bounded smooth
$O'\Subset O$. The following identities hold in the sense
of distributions.
\begin{enumerate}
\item[(i)] The elastic operator admits the factorization
\begin{equation}\label{eq:w1-factorization}
\begin{gathered}
LT=S\mathcal M,\qquad
\mathcal M=\Delta I_4+\sum_{k=1}^3A_k\partial_k+B,\\
A_k=\begin{pmatrix}
-g_\mu e_k^T&F_\mu e_k\\
d e_k^T&0
\end{pmatrix},\qquad
B=\begin{pmatrix}
V_\mu&F_\mu g_\mu\\
-e\,g_\mu^T&q_\mu
\end{pmatrix}.
\end{gathered}
\end{equation}
For $C^2$ Lam\'e parameters, this factorization holds in
$H^{-1}_{\rm loc}(O;\C^3)$ for every
$\boldsymbol z\in H^2_{\rm loc}(O;\C^4)$.
For every integer $r\ge2$ under the stated $C^r$ bounds,
\begin{equation}\label{eq:r2-reduced-coefficient-bound}
\sum_{k=1}^3\|A_k\|_{C^{r-2}(O')}
+\|B\|_{C^{r-2}(O')}
\le C(r,M,c_0,O').
\end{equation}

\item[(ii)] The formal transpose $S^T$ and the composite operator
\begin{equation}\label{eq:r2-composite-K}
\mathcal K=S^TT:
H^2(O';\C^4)\longrightarrow L^2(O';\C^4)
\end{equation}
are given by
\begin{equation}\label{eq:w1-ST-K}
\begin{gathered}
S^Tu=\binom{\sqrt\mu\,u}{-\operatorname{div}u/(1-\kappa)},\\
\mathcal K(v,f)=
\binom{v+c\nabla f+cf g_\mu}
{-\big[c\operatorname{div}v-\tfrac12c\,g_\mu\cdot v
+\mu^{-1}\Delta f+\mu^{-1}q_\mu f\big]/(1-\kappa)}.
\end{gathered}
\end{equation}
They satisfy the formal transpose identity
\begin{equation}\label{eq:w1-adjoint-intertwining}
\mathcal M^TS^T=T^TL.
\end{equation}

\item[(iii)] With the phase conventions in
\eqref{eq:r2-phase-parameters}, one has
\begin{equation}\label{eq:w1-phase-identities}
\begin{gathered}
e^{-\sigma_{\rm ph}\varphi/h}\mathcal M
e^{\sigma_{\rm ph}\varphi/h}
=\mathcal M+\sigma_{\rm ph}h^{-1}(2D_\theta+A_\theta),
\qquad A_\theta=\sum_{k=1}^3\theta_kA_k,\\
e^{-\sigma_{\rm ph}\varphi/h}\mathcal K
e^{\sigma_{\rm ph}\varphi/h}
=\mathcal K+\sigma_{\rm ph}h^{-1}J_\theta,\\
J_\theta(v,f)=\binom{cf\theta}
{-(c\theta\cdot v+2\mu^{-1}D_\theta f)/(1-\kappa)}.
\end{gathered}
\end{equation}
The first-order operators $S^T$ and $J_\theta$ define bounded maps
\begin{equation}\label{eq:r2-lift-mappings}
\begin{gathered}
S^T:H^1(O';\C^3)\longrightarrow L^2(O';\C^4),\\
J_\theta:H^1(O';\C^4)\longrightarrow L^2(O';\C^4).
\end{gathered}
\end{equation}
\end{enumerate}
\end{proposition}

\begin{proof}
The scalar and vector identities \eqref{eq:w1-app-scalar} and
\eqref{eq:w1-app-vector} give $LT=S\mathcal M$ with the
coefficients in \eqref{eq:w1-factorization}.
For compactly supported smooth $(p,q)$ and $u$,
\begin{equation}\label{eq:r2-transpose-pairing}
\begin{aligned}
\int_O S(p,q)\cdot u\,dx
&=\int_O\left(\sqrt\mu\,p\cdot u-
q\frac{\operatorname{div}u}{1-\kappa}\right)dx\\
&=\int_O(p,q)\cdot S^Tu\,dx.
\end{aligned}
\end{equation}
This proves the formula for $S^T$ in \eqref{eq:w1-ST-K}.
Substitution of $T$ from \eqref{eq:w1-T-S-M} gives the
formula for $\mathcal K$. Transposing $LT=S\mathcal M$
and using $L^T=L$ yields \eqref{eq:w1-adjoint-intertwining}.

The substitution
\begin{equation}\label{eq:r2-phase-substitution}
e^{-\sigma_{\rm ph}\varphi/h}\partial_k
e^{\sigma_{\rm ph}\varphi/h}
=\partial_k+\sigma_{\rm ph}h^{-1}\theta_k,
\qquad \theta\cdot\theta=0,
\end{equation}
applied to \eqref{eq:w1-factorization} and \eqref{eq:w1-ST-K}
gives \eqref{eq:w1-phase-identities}.
The product rule gives the boundedness in
\eqref{eq:r2-lift-mappings}.

The coefficient formulas \eqref{eq:w1-reduced-coefficients}
involve at most two derivatives of the Lam\'e parameters.
The denominator bounds \eqref{eq:w1-denominators} and
Leibniz' rule therefore give
\eqref{eq:r2-reduced-coefficient-bound}.
For $C^2$ Lam\'e parameters and
$\boldsymbol z\in H^2_{\rm loc}(O;\C^4)$,
\eqref{eq:r2-physical-mappings} and the boundedness of
$L:H^1(O';\C^3)\to H^{-1}(O';\C^3)$ imply
\begin{equation}\label{eq:r2-factorization-weak-spaces}
\begin{gathered}
T\boldsymbol z\in H^1_{\rm loc},\qquad
\mathcal M\boldsymbol z\in L^2_{\rm loc},\\
LT\boldsymbol z-S\mathcal M\boldsymbol z\in H^{-1}_{\rm loc}.
\end{gathered}
\end{equation}
Approximation of the Lam\'e parameters in $C^2$ and of
$\boldsymbol z$ in $H^2_{\rm loc}$ extends the identities
to the stated distributional setting.
\end{proof}

\begin{remark}
The factorization maps a reduced solution
$\boldsymbol z\in H^2_{\rm loc}(O;\C^4)$ to an elastic
displacement $u=T\boldsymbol z$ satisfying $Lu=0$.
For a weak elastic solution $u\in H^1_{\rm loc}(O;\C^3)$,
the formal transpose identity gives the lift
$S^Tu\in L^2_{\rm loc}(O;\C^4)$ satisfying
$\mathcal M^T(S^Tu)=0$ in $\mathcal D'(O;\C^4)$.
These identities relate solutions of the two systems;
they do not, by themselves, identify the unrestricted
boundary map of the reduced system with the measured
elastic Dirichlet-to-Neumann map.
\end{remark}

\subsection{A common background and the paired quotient coordinates}\label{sec:r2-quotient-definitions}

The constant-background first-order term is removed before constructing
paired coordinates on the characteristic quotient and its dual.

Suppose the materials equal a fixed constant pair outside a fixed ball,
and use the material variables in \eqref{eq:w1-material-variables}.
The common background multiplier and transformed operators are
\begin{equation}\label{eq:w1-background}
 d_*=\kappa_*\sqrt{\mu_*},\qquad
 \mathsf R_*(x)=\begin{pmatrix}I_3&0\\-d_*x^T/2&1\end{pmatrix},\qquad
 \widetilde{\mathcal M}=\mathsf R_*^{-1}\mathcal M\mathsf R_*,\qquad
 \widetilde T=T\mathsf R_*,\quad\widetilde S=S\mathsf R_*.
\end{equation}
$\mathsf R_*$ is a $4\times4$ multiplication operator. Since
$\Delta(x\cdot v)=x\cdot\Delta v+2\operatorname{div}v$, the constant
background satisfies $\mathcal M_*\mathsf R_*=\mathsf R_*\Delta$.
The coefficients of $\widetilde{\mathcal M}-\Delta I_4$ thus have fixed compact
support, and their formulas involve at most two derivatives of the
Lam\'e parameters. Explicitly,
\begin{equation}\label{eq:w1-background-coefficients}
 \begin{aligned}
 \widetilde A_k&=\mathsf R_*^{-1}(2\partial_k\mathsf R_*+A_k\mathsf R_*),\\
 \widetilde B&=\mathsf R_*^{-1}
 (\Delta\mathsf R_*+\sum_kA_k\partial_k\mathsf R_*+B\mathsf R_*),\\
 \widetilde S^T&=\mathsf R_*^TS^T,\qquad
 \widetilde{\mathcal K}=\mathsf R_*^T\mathcal K\mathsf R_*,\qquad
 \widetilde J_\theta=\mathsf R_*^TJ_\theta\mathsf R_*.
 \end{aligned}
\end{equation}
For a material pair, define
\begin{equation}\label{eq:r2-contracted-differences}
 A_{\Delta,\theta}=A_{1,\theta}-A_{2,\theta},\qquad
 \widetilde A_{\Delta,\theta}=\widetilde A_{1,\theta}-\widetilde A_{2,\theta},
\end{equation}
where $A_{j,\theta}$ is the contraction in \eqref{eq:w1-phase-identities}.
These are differential compositions; in the last identity $J_\theta$
still differentiates $\mathsf R_*$. For the transformed primal and dual fields
\begin{equation}\label{eq:r2-transformed-fields}
 \widetilde{\boldsymbol z}=\mathsf R_*^{-1}\boldsymbol z,\qquad
 \widetilde y=\mathsf R_*^Ty,
\end{equation}
with $\mathsf R_*$ from \eqref{eq:w1-background}, one has
\begin{equation}\label{eq:w1-exact-covariance}
 \widetilde y^T(\widetilde{\mathcal M}_1-\widetilde{\mathcal M}_2)\widetilde{\boldsymbol z}
   =y^T(\mathcal M_1-\mathcal M_2)\boldsymbol z,\qquad
 \widetilde A_{\Delta,\theta}=\mathsf R_*^{-1}A_{\Delta,\theta}\mathsf R_*.
\end{equation}
The first equality includes all derivatives of $\mathsf R_*$.
The leading transport satisfies the exact operator identity
\begin{equation}\label{eq:w1-background-transport}
 2D_\theta+\widetilde A_\theta
 =\mathsf R_*^{-1}(2D_\theta+A_\theta)\mathsf R_*,\qquad
 \widetilde e_\theta=\mathsf R_*^{-1}e_\theta
  =\binom{\theta}{d_*x\cdot\theta/2}.
\end{equation}
Indeed the first coefficient formula contributes
$2\mathsf R_*^{-1}(D_\theta\mathsf R_*)$ as well as
$\mathsf R_*^{-1}A_\theta\mathsf R_*$. Thus the transformed
characteristic line is preserved. If the paired lifts are
$\widetilde{\mathsf U}=\mathsf R_*^T\mathsf U$ and
$\widetilde{\mathsf V}=\mathsf R_*^{-1}\mathsf V$, their pairing
remains the identity. The quotient connection is unchanged when the spatial derivatives of the
transported lifts $\widetilde{\mathsf U}$ and $\widetilde{\mathsf V}$ are
included, by \eqref{eq:w1-background-transport}. In particular
$\widetilde{\mathcal Y}_2^T\widetilde A_{\Delta,\theta}
\widetilde Z_1=\mathcal Y_2^TA_{\Delta,\theta} Z_1$.

For a normalized null vector $\theta=e_1+i e_2$, where $e_1,e_2$ are
real orthonormal vectors, let
\begin{equation}\label{eq:w1-leaf-coordinates}
 \begin{gathered}
 \nu_{\rm leaf}=e_1\times e_2,\qquad
 z=e_1\cdot x+i e_2\cdot x,\qquad t_\perp=\nu_{\rm leaf}\cdot x,\\
 \partial_{\bar z}=\tfrac12(\partial_{\Re z}+i\partial_{\Im z}),\qquad
 D_\theta=2\partial_{\bar z}.
 \end{gathered}
\end{equation}
The coordinate vectors in this display are local to the normalized
frame, not the coordinate vectors used in \eqref{eq:w1-factorization}.
For $e_\theta$ from \eqref{eq:r2-phase-parameters}, one has
$A_\theta e_\theta=0$. Define the primal quotient and its bilinear dual by
\begin{equation}\label{eq:r2-characteristic-quotient}
 E_\theta=\mathbb C^4/\mathbb Ce_\theta,\qquad
 E_\theta^*=\{(p,\eta)\in\mathbb C^3\times\mathbb C:\theta^Tp=0\}.
\end{equation} Introduce homogeneous spinor
coordinates $[s_0:s_1]\in\mathbb{CP}^1$ and
\begin{equation}\label{eq:w1-spinor}
 \begin{gathered}
 \theta=(s_0^2-s_1^2,i(s_0^2+s_1^2),2s_0s_1)^T,\quad
 Q=\tfrac12\begin{pmatrix}-s_1&-s_0\\i s_1&-i s_0\\s_0&-s_1\end{pmatrix},
 \quad n_\theta=(-s_1,s_0)^T,\\
 \theta^TQ=0,\qquad Qn_\theta=-\theta/2,\qquad
 \mathfrak r_a=-2a^TQ\quad(a\in\mathbb C^3),\qquad
 \mathsf U_\theta=\begin{pmatrix}Q&0\\0&1\end{pmatrix}.
 \end{gathered}
\end{equation}
The matrix $Q$ has size $3\times2$, $\mathfrak r_a$ has size
$1\times2$, and $\mathsf U_\theta$ is a $4\times3$ embedding of
$E_\theta^*$ from \eqref{eq:r2-characteristic-quotient}.
On the two affine spinor charts, set
\begin{equation}\label{eq:r2-spinor-charts}
 \zeta=s_1/s_0,\qquad \upsilon=s_0/s_1=1/\zeta,\qquad
 D_{\rm tr}(\zeta)=\operatorname{diag}(\zeta I_2,1).
\end{equation}
Here $n_\theta$ is a two-vector, never the polynomial degree.
$Q$ has rank two: $Q(a,b)^T=0$ implies
$s_1a=s_0b=s_0a-s_1b=0$, including the chart axes.
On a finite refinement of compact spinor charts select three rows of
$\mathsf U_\theta$ with a nonzero minor. We invert the selected nonzero minor of $\mathsf U_\theta$ and obtain a
holomorphic $4\times3$ matrix $\mathsf V_\theta$ satisfying
\begin{equation}\label{eq:w1-paired-lifts}
 \mathsf U_\theta^T\mathsf V_\theta=I_3,\qquad
 \ker\mathsf U_\theta^T=\mathbb Ce_\theta.
\end{equation}
Each selected determinant is bounded away from zero after shrinking the
compact subchart; the minimum over the finite cover is positive.
This is the additional geometric denominator, independent of materials.
More explicitly, let $q_{12},q_{13},q_{23}$ be the two-row minors of $Q$.
Direct substitution in \eqref{eq:w1-spinor} gives
\begin{equation}\label{eq:w1-minor-bound}
 |q_{12}|^2+|q_{13}|^2+|q_{23}|^2
       =\frac{(|s_0|^2+|s_1|^2)^2}{8}.
\end{equation}
At a unit spinor at least one has modulus at least $1/\sqrt{24}$;
the corresponding three-row minor of $\mathsf U_\theta$ has that
same modulus. Fixed smaller chart neighborhoods therefore give an
explicit geometric lower bound before any material is chosen.
On a holomorphic chart overlap the paired local choices obey
\begin{equation}\label{eq:w1-lift-overlap}
 \mathsf V_0=\mathsf V_\infty D_{\rm tr}^{-T}
                   +e_{\theta_0}\mathsf L,
\end{equation}
where $D_{\rm tr}$ is defined in \eqref{eq:r2-spinor-charts} and
$\mathsf L$ is a holomorphic $1\times3$ row. In fact
$\mathsf U_0=\mathsf U_\infty D_{\rm tr}$, so applying
$\mathsf U_0^T$ to the difference gives zero. Its kernel is the
characteristic line. A fixed nonzero component of $e_{\theta_0}$
on a smaller chart extracts the holomorphic row $\mathsf L$.

Substitution of $y=\mathsf U_\theta(\xi,\eta)^T$ into
$2D_\theta y-A_\theta^Ty=0$ gives
\begin{equation}\label{eq:w1-quotient-connection}
 \Gamma=\begin{pmatrix}n_\theta\mathfrak r_{g_\mu}&2d n_\theta\\
                   -\theta^TF_\mu Q&0\end{pmatrix},\qquad
 \mathcal A=\mathsf U_\theta^TA_\theta\mathsf V_\theta=-\Gamma^T,
 \qquad P=\mathcal A/2.
\end{equation}
Here $\Gamma,\mathcal A,P$ are $3\times3$ connection matrices;
$Q,n_\theta,\mathfrak r_{g_\mu}$ are defined in \eqref{eq:w1-spinor},
and $d,F_\mu$ in \eqref{eq:w1-material-variables} and
\eqref{eq:w1-reduced-coefficients}.
The substitution gives
$A_\theta^T\mathsf U_\theta=-\mathsf U_\theta\Gamma$;
transpose this equality and multiply by $\mathsf V_\theta$.
Thus the primal and dual frame equations are
\begin{equation}\label{eq:w1-quotient-equations}
 (2D_\theta+\mathcal A)X=0,\qquad
 (2D_\theta+\Gamma)Y=0,
 \quad\text{or}\quad
 2\partial_{\bar z}X+PX=0,\quad
 2\partial_{\bar z}Y-P^TY=0.
\end{equation}
The matrices $X,Y$ are $3\times3$ primal and dual quotient frames,
respectively; $D_\theta$ is defined in \eqref{eq:r2-phase-parameters}.
For any given primal solution $X$, the defect of $\mathsf V_\theta X$
lies in $\mathbb Ce_\theta$. Choose a spatially constant row on the chart satisfying
\begin{equation}\label{eq:r2-characteristic-complement-row}
 \ell_\theta\in\mathbb C^{1\times4},\qquad \ell_\theta e_\theta=1.
\end{equation}
Solve
\begin{equation}\label{eq:w1-primal-lift}
 \begin{gathered}
 2D_\theta\alpha_{\rm lift}=-\ell_\theta(2D_\theta+A_\theta)\mathsf V_\theta X,
 \qquad Z=\mathsf V_\theta X+e_\theta\alpha_{\rm lift},\\
 (2D_\theta+A_\theta)Z=0,\qquad \mathsf U_\theta^TZ=X,\qquad
 \mathsf U_\theta^TA_{\Delta,\theta} Z_1=2(P_1-P_2)X_1.
 \end{gathered}
\end{equation}
Here $Z$ is a $4\times3$ ambient amplitude, $\alpha_{\rm lift}$ is a
$1\times3$ row, and $A_{\Delta,\theta}$ is defined in
\eqref{eq:r2-contracted-differences}.
The construction of $X$ with uniform finite-order bounds uses
Theorem~\ref{thm:finite-frames}. The row $\alpha_{\rm lift}$ has size $1\times3$. Its three scalar
equations are solved on fixed leaf disks as described in
Section~\ref{sec:w1-physical-moments}.

With $\upsilon$ and $D_{\rm tr}$ as in \eqref{eq:r2-spinor-charts}, the representatives
$(1,\zeta)$ and $(\upsilon,1)$ give
\begin{equation}\label{eq:w1-chart-law}
 \begin{gathered}
 Q_0=\zeta Q_\infty,\quad \theta_0=\zeta^2\theta_\infty,\quad
 \Gamma_0=\zeta^2D_{\rm tr}^{-1}\Gamma_\infty D_{\rm tr},\\
 H^{(0)}=D_{\rm tr}^{-1}H^{(\infty)}D_{\rm tr},\qquad
 G^{(0)}=D_{\rm tr}^{T}G^{(\infty)}D_{\rm tr}^{-T},\qquad H=G^{-T}.
 \end{gathered}
\end{equation}
These are coordinate laws for a given endomorphism, not an assertion
that separately constructed transports agree. When
$\theta_{\rm n}=a\theta_{\rm h}$ for a positive, spatially constant
normalization factor, define
\begin{equation}\label{eq:r2-positive-rescaling}
 S_a=\operatorname{diag}(\sqrt a I_2,1),\qquad a>0.
\end{equation}
Then
\begin{equation}\label{eq:w1-normalization}
 \mathsf U_{\rm n}=\mathsf U_{\rm h}S_a,\qquad
 \Gamma_{\rm n}=aS_a^{-1}\Gamma_{\rm h}S_a,\qquad
 H_{\rm h}=S_aH_{\rm n}S_a^{-1}.
\end{equation}
The factor $a$ need not be holomorphic. Antiholomorphic derivatives must
be taken after returning to holomorphic coordinates and at fixed physical
$x$. In particular,
\begin{equation}\label{eq:w1-chart-dbar}
 \partial_{\bar\zeta}H^{(0)}
 =-\bar\zeta^{-2}D_{\rm tr}^{-1}
    (\partial_{\bar\upsilon}H^{(\infty)})D_{\rm tr}.
\end{equation}
For a spatially constant $b\in\mathbb C^*$, the characteristic
line and affine complex leaves of $\theta$ and $b\theta$ agree.
Choose a local square root of $b$ and define
\begin{equation}\label{eq:r2-complex-rescaling}
 S_b=\operatorname{diag}(\sqrt b I_2,1),\qquad b\in\mathbb C^*.
\end{equation}
Then
\begin{equation}\label{eq:w1-representative-law}
 \Gamma_b=bS_b^{-1}\Gamma S_b,\qquad H_b=S_b^{-1}HS_b,
 \quad
 2D_{b\theta}H_b+\Gamma_{1,b}H_b-H_b\Gamma_{2,b}
  =bS_b^{-1}(2D_\theta H+\Gamma_1H-H\Gamma_2)S_b.
\end{equation}
This proves representative invariance of the equation and its identity
boundary normalization. Lemma~\ref{lem:w3-comparison} compares independently constructed
local transports after these coordinate changes.

\subsection{The common-collar Betti identity and its exact measurement domain}\label{sec:r2-betti-domain}

Under agreement near the working sphere, the Betti identity relates the
reduced volume pairing to the elastic displacement--traction pairing. When
the materials also agree in the region between the measurement boundary and
the working sphere, their DN discrepancies can be compared without an
exterior mismatch term.

We use the elastic DN maps $\Lambda_j^D$ from
\eqref{eq:r2-material-map-abbreviation}, with the weak pairing
\eqref{eq:w5-physical-dn} and discrepancy \eqref{eq:w0-domain-error}.
In particular, $\delta_{B_R}$ is obtained by taking $D=B_R$.

Under agreement near the working sphere, the factorization transfers
the reduced volume identity to a pairing of genuine elastic solutions.
The next lemma identifies the corresponding DN pairing on
$\partial B_R$.

\begin{lemma}\label{lem:w1-betti}
Let the real $C^4$ materials be uniformly strongly convex on a
neighborhood $O$ of $\overline B_R$. Suppose that there is an
open set $N\subset O$, with $\partial B_R\subset N$, on which both
$\lambda_1=\lambda_2$ and $\mu_1=\mu_2$. Let $\boldsymbol z_1\in H^4(O;\mathbb C^4)$
solve $\mathcal M_1\boldsymbol z_1=0$ on $O$, and let
$u_2\in H^1(B_R;\mathbb C^3)$ solve $L_2u_2=0$. Set
\begin{equation}\label{eq:r2-Betti-physical-fields}
 u_1=T_1\boldsymbol z_1,\qquad y_2=S_2^Tu_2\in L^2(B_R;\mathbb C^4),
\end{equation}
where $T_j,S_j^T$ are defined in \eqref{eq:w1-T-S-M} and \eqref{eq:w1-ST-K}.
Then
\begin{equation}\label{eq:w1-betti}
 \int_{B_R}y_2^T(\mathcal M_1-\mathcal M_2)\boldsymbol z_1\,dx
 =\langle(\Lambda_2^{B_R}-\Lambda_1^{B_R})\operatorname{Tr}u_1,
                                       \operatorname{Tr}u_2\rangle.
\end{equation}
The boundary duality is complex bilinear and has traction in its first
slot.
\end{lemma}
\begin{proof}
Use the bilinear energy $a_{B_R}$ from \eqref{eq:w5-physical-form}.
For $u\in H^1(B_R;\mathbb C^3)$ with $Lu\in L^2$, define its weak traction by
\begin{equation}\label{eq:r2-weak-traction}
 \langle t_u,g\rangle
  =a_{B_R}(u,\mathcal E_{\rm tr}g)
        +\int_{B_R}(Lu)^T\mathcal E_{\rm tr}g\,dx,
\end{equation}
where $\mathcal E_{\rm tr}:H^{1/2}(\partial B_R;\mathbb C^3)\to
H^1(B_R;\mathbb C^3)$ is a bounded right inverse of the trace. Independence of $\mathcal E_{\rm tr}$ follows by testing the
 distributional equation against the zero-trace difference. Hence
$\|t_u\|_{H^{-1/2}}\le C(\|u\|_{H^1}+\|Lu\|_2)$, and subtraction gives
\begin{equation}\label{eq:w1-inhomogeneous-betti}
 \int(u^TLv-v^TLu)=\langle t_v,\operatorname{Tr}u\rangle
                         -\langle t_u,\operatorname{Tr}v\rangle.
\end{equation}
For $Lu=Lv=0$, identity \eqref{eq:w1-inhomogeneous-betti} gives
$\langle\Lambda\operatorname{Tr}u,\operatorname{Tr}v\rangle=\langle\Lambda\operatorname{Tr}v,\operatorname{Tr}u\rangle$.
Thus $\Lambda$ is symmetric for the bilinear pairing.
Uniform solvability follows from \eqref{eq:r2-energy-coercivity}
and \eqref{eq:w5-dirichlet-bound}.

Define
\begin{equation}\label{eq:r2-Betti-comparison-fields}
 w=\mathcal M_2\boldsymbol z_1,\qquad v=T_2\boldsymbol z_1.
\end{equation}
On the common collar $N$ of Lemma~\ref{lem:w1-betti},
\begin{equation}\label{eq:r2-Betti-collar-equalities}
 w=0,\qquad v=u_1,\qquad t_v^{(2)}=t_{u_1}^{(1)}.
\end{equation}
The complete transpose transfer is
\begin{equation}\label{eq:w1-transfer-boundary}
 \int(S_2^Tu_2)^Tw=\int u_2^TS_2w
       -\int_{\partial B_R}(u_2\cdot\nu)\frac{w_4}{1-\kappa_2}\,dS.
\end{equation}
The last term is zero, precisely because $w$ vanishes on $N$.
Since $\mathcal M_1\boldsymbol z_1=0$ and $S_2\mathcal M_2\boldsymbol z_1=L_2v$, the left side of
\eqref{eq:w1-betti} is $-\int u_2^TL_2v$.
By \eqref{eq:w1-inhomogeneous-betti}, equality of the collar tractions, and
the bilinear symmetry of $\Lambda_2$, we obtain
$-\langle\Lambda_1\operatorname{Tr}u_1,\operatorname{Tr}u_2\rangle+
\langle\Lambda_2\operatorname{Tr}u_1,\operatorname{Tr}u_2\rangle$.
There is no assumption that the comparison field $v$ is homogeneous.
\end{proof}

When the materials agree between an inner measurement boundary and the
working sphere, their DN discrepancies satisfy the following comparison.

\begin{lemma}
\label{lem:w1-domain-comparison}
Let $\overline D\Subset B_R$ be fixed smooth domains and let the two
uniformly strongly convex bounded materials coincide on $B_R\setminus
\overline D$. Then
\begin{equation}\label{eq:w1-domain-comparison}
 \delta_{B_R}\le C_{D,R}\delta_D.
\end{equation}
Here both discrepancies are defined in \eqref{eq:w0-domain-error}.
The constant depends only on the fixed domains and the
uniform upper and strong-convexity bounds for the coefficients.
\end{lemma}
\begin{proof}
Let $L_1u=L_2v=0$ on $B_R$, with traces $f,g$. Applying \eqref{eq:w5-physical-alessandrini} on $B_R$ and $D$, and using
equality of the coefficients outside $\overline D$, we obtain
\begin{equation}\label{eq:r2-domain-pairing-identity}
 \begin{aligned}
 \langle(\Lambda_1^{B_R}-\Lambda_2^{B_R})f,g\rangle
 &=\int_D\big[(\lambda_1-\lambda_2)\operatorname{div}u\operatorname{div}v
       +2(\mu_1-\mu_2)\varepsilon(u):\varepsilon(v)\big]dx\\
 &=\langle(\Lambda_1^D-\Lambda_2^D)\operatorname{Tr}_D u,
                    \operatorname{Tr}_D v\rangle .
 \end{aligned}
\end{equation}
The restrictions are solutions of their respective equations on $D$.
The fixed-domain trace theorem and \eqref{eq:w5-dirichlet-bound} imply
\begin{equation}\label{eq:r2-inner-trace-bounds}
 \|\operatorname{Tr}_D u\|_{H^{1/2}}\le C\|f\|_{H^{1/2}(\partial B_R)},
 \qquad
 \|\operatorname{Tr}_D v\|_{H^{1/2}}\le C\|g\|_{H^{1/2}(\partial B_R)}.
\end{equation}
Taking the supremum in \eqref{eq:r2-domain-pairing-identity} over $f,g\in
H^{1/2}(\partial B_R;\mathbb C^3)$ with $\|f\|_{H^{1/2}}=\|g\|_{H^{1/2}}=1$,
and using \eqref{eq:r2-inner-trace-bounds}, we obtain
\eqref{eq:w1-domain-comparison}. Both discrepancies are defined in
\eqref{eq:w0-domain-error}.
\end{proof}

When the measurements are made directly on $\partial B_R$,
no domain comparison is needed. Otherwise,
Lemma~\ref{lem:w1-domain-comparison} transfers the discrepancy
from $\partial D$ to $\partial B_R$ under its stated exterior
agreement assumption. In the prescribed-exterior problem,
$D=\Omega$ and $\delta_D=\delta$.
For the paired extensions constructed in
Section~\ref{sec:w5-extensions}, we take $D=B_{R_e}$ and
$\delta_D=\Delta_e$; the comparison with the original error
$\delta$ retains the exterior mismatch through
Proposition~\ref{prop:v61-enlargement-error}.
The collar assumption of Lemma~\ref{lem:w1-betti} is imposed
at the working sphere, not on the original materials at
$\partial\Omega$.

\section{Uniform finite-regularity frames and a local right inverse}
\label{sec:right-inverse}

Two analytic results underlie the CGO construction: finite-order leafwise
transport frames and a local right inverse for the conjugated elliptic
system. Their regularity requirements are distinct.

The matrix-transport construction for conjugated systems is described
in \cite[Sections~1--2]{ER}. The finite-order uniform estimates used
here are proved in Subsections~\ref{sec:r2-proof-frames} and
\ref{sec:r2-proof-inverse}. The frame existence argument uses
\cite[Theorem~1.3]{FP}; the quantitative bounds are obtained in the proof.

\subsection{Coefficient class and fixed geometry}\label{sec:r2-analytic-geometry}

The coefficient support and nested domains below are fixed independently of
the source and phase direction.

In this section $d_{\rm sys}\geq1$ is the system dimension, not the scalar
material coefficient denoted by $d$ elsewhere. The index $k\geq0$ is an
integer Sobolev order. Fix balls $\Omega_1\Subset\Omega_0\subset\mathbb R^3$,
a compact source ball $K_{\rm src}\Subset\Omega_1$, and a compact set
$K_{\rm coef}\Subset\operatorname{int}K_{\rm src}$. All containments mean
positive separation of the corresponding closed sets. Set
\begin{equation}\label{eq:finite-system-class}
 \begin{aligned}
 &\widetilde{\mathcal M}=\Delta I_{d_{\rm sys}}
       +\sum_{j=1}^3\widetilde A_j(x)\partial_j+\widetilde B(x),\\
 &\widetilde A_j,\widetilde B\in
 C_c^{k+14}(\mathbb R^3;\mathbb C^{d_{\rm sys}\times d_{\rm sys}}),\\
 &\operatorname{supp}(\widetilde A_j),\operatorname{supp}(\widetilde B)
        \subset K_{\rm coef},\qquad j=1,2,3,\\
 &\sum_{j=1}^3\lVert\widetilde A_j\rVert_{C^{k+14}}
        +\lVert\widetilde B\rVert_{C^{k+14}}\leq M_k.
 \end{aligned}
\end{equation}
The class in \eqref{eq:finite-system-class} includes all complex matrix
coefficients with this prior; ellipticity here comes from the fixed
Laplacian principal part. Define the closed source space
\begin{equation}\label{eq:w2-source-space}
 E_k=\{f\in H^k(\Omega_0;\mathbb C^{d_{\rm sys}}):
                  \operatorname{supp}f\subset K_{\rm src}\}.
\end{equation}
Support is distributional. Extension by zero identifies this space with
a closed subspace of $H^k(\mathbb R^3;\mathbb C^{d_{\rm sys}})$; it is
bounded because the support stays a fixed distance from $\partial\Omega_0$.

The normalized direction set and the conjugated operator are
\begin{equation}\label{eq:w2-phase-space}
 \begin{split}
 \Theta&=\{e_1+i e_2:e_1,e_2\in\mathbb R^3,
                  |e_1|=|e_2|=1,\ e_1\cdot e_2=0\},\\
 L_{h,\theta}&=e^{\theta\cdot x/h}\widetilde{\mathcal M} e^{-\theta\cdot x/h}
   =\widetilde{\mathcal M}-h^{-1}
       \left(2\theta\cdot\nabla+\sum_j\theta_j\widetilde A_j\right).
 \end{split}
\end{equation}
The dot product in the null condition $\theta\cdot\theta=0$ is bilinear.
No data discrepancy is used in this section; in particular $\delta$ is
not reused as a Fourier parameter.

\subsection{Finite-order leaf frames}\label{sec:r2-leaf-frames}

Uniform leafwise frames are obtained on complete disks from a finite atlas
over the compact coefficient--parameter class.

Let $D\Subset U\subset\mathbb C$ be fixed disks, $r\ge0$ an integer,
and $0<\alpha_{\rm H}<1$. The operator $\partial_{\bar z}$ is defined
in \eqref{eq:w1-leaf-coordinates}. For the uniform assertion, fix a compact
finite-dimensional smooth real parameter manifold $\mathfrak P$ and
larger spatial and parameter neighborhoods. Compact charts with boundary
are allowed. A jointly bounded $C^{r+2}$ family
$\{\mathsf B^{\mathfrak a}(z,p)\}$ has the weaker compact closure
\begin{equation}\label{eq:r2-frame-coefficient-closure}
 \mathscr B_r=\overline{\{\mathsf B^{\mathfrak a}\}}^{\,C^{r+1,\alpha_{\rm H}}},
 \qquad (\mathfrak a,p)\in\mathscr B_r\times\mathfrak P,
\end{equation}
where each $\mathsf B^{\mathfrak a}$ takes values in
$\mathbb C^{d_{\rm sys}\times d_{\rm sys}}$. The coefficient label has
the topology in \eqref{eq:r2-frame-coefficient-closure}; it is not a
differentiation variable.

\begin{theorem}\label{thm:finite-frames}
Let $D\Subset U\subset\C$ be the fixed disks,
$r\ge0$ an integer, and $0<\alpha_{\rm H}<1$.
For every
$\mathsf B\in C^{r,\alpha_{\rm H}}
(U;\C^{d_{\rm sys}\times d_{\rm sys}})$,
there exist a neighborhood $V\subset U$ of $\overline D$
and an invertible frame satisfying
\begin{equation}\label{eq:finite-frame-equation}
\mathsf F\in C^{r+1,\alpha_{\rm H}}
(V;GL(d_{\rm sys},\C)),
\qquad
\partial_{\bar z}\mathsf F=\mathsf B\mathsf F
\quad\text{in }V.
\end{equation}

For the compact coefficient class $\mathscr B_r$
in \eqref{eq:r2-frame-coefficient-closure} and the
fixed parameter space $\mathfrak P$, there exists
a finite cover by coefficient--parameter
neighborhoods carrying invertible frames
$\mathsf F_\ell^{\mathfrak a}(z,p)$, defined near
$\overline D$, satisfying
\[
\partial_{\bar z}\mathsf F_\ell^{\mathfrak a}
=\mathsf B^{\mathfrak a}\mathsf F_\ell^{\mathfrak a}.
\]
On smaller neighborhoods still covering
$\mathscr B_r\times\mathfrak P$, these frames
and their inverses satisfy
\begin{equation}\label{eq:finite-frame-uniform}
\max_{|\gamma|+|\beta|\le r}
\left(
\|\partial_z^\gamma\partial_p^\beta
\mathsf F_\ell^{\mathfrak a}\|_\infty
+
\|\partial_z^\gamma\partial_p^\beta
(\mathsf F_\ell^{\mathfrak a})^{-1}\|_\infty
\right)
\le C_r,
\end{equation}
where $\gamma$ denotes differentiation in the
two real leaf variables.
Both the frames and their inverses depend
continuously on $\mathfrak a$ in the mixed norms
of total order at most $r$.

The cover admits subordinate weights
$w_\ell(\mathfrak a,p)$, continuous in
$\mathfrak a$ and smooth in $p$, such that
\begin{equation}\label{eq:r2-frame-weight-properties}
w_\ell\ge0,\qquad
\sum_\ell w_\ell^2=1,\qquad
\max_{|\beta|\le r}
\|\partial_p^\beta w_\ell\|_\infty\le C_r.
\end{equation}
When $p$ consists of an oriented direction
and a transverse coordinate, the weights are
constant along each complex leaf.

The finite cover and the constants depend only
on the fixed coefficient class, disks, neighborhood
margins, and parameter geometry.
Finitely many systems admit a common finite
product refinement with the same properties.
\end{theorem}

\begin{remark}
The uniform estimates are obtained from a finite
atlas of local frames; no globally defined frame
over the parameter space is required.
The coefficient label $\mathfrak a$ enters only
continuously and is never differentiated.
This distinction allows the construction to be
uniform over the compact coefficient class.
\end{remark}

The proof is given in
Subsection~\ref{sec:r2-proof-frames}.

\subsection{The local right inverse and residual estimates}
\label{sec:r2-right-inverse-results}

\begin{theorem}
\label{thm:finite-right-inverse}
For the full coefficient class \eqref{eq:finite-system-class} there exist
$h_*>0$ and $C_k<\infty$, depending only on
$k,d_{\rm sys},M_k,K_{\rm coef},K_{\rm src},\Omega_1,\Omega_0$ and the
fixed cutoff geometry, with the following property. For every member of
the class, every $\theta\in\Theta$, and every $0<h<h_*$, there is a linear
operator
\begin{equation}\label{eq:finite-right-inverse}
 \begin{gathered}
 \mathcal R_{h,\theta}:E_k\longrightarrow
                 H^{k+1}(\Omega_0;\mathbb C^{d_{\rm sys}}),\\
 L_{h,\theta}\mathcal R_{h,\theta}f=f
                         \quad\text{in }\mathcal D'(\Omega_1),\\
 h^{-1}\lVert\mathcal R_{h,\theta}f\rVert_{H^k(\Omega_0)}
   +\lVert\mathcal R_{h,\theta}f\rVert_{H^{k+1}(\Omega_0)}
       \leq C_k\lVert f\rVert_{H^k(\Omega_0)}.
 \end{gathered}
\end{equation}
Here $E_k$, $\Theta$, and $L_{h,\theta}$ are defined in
\eqref{eq:w2-source-space} and \eqref{eq:w2-phase-space}.
The identity and both estimates hold for every $f\in E_k$.
The constants and threshold are uniform over the coefficient
class and all $\theta\in\Theta$, including both phase signs.
\end{theorem}

\begin{remark}
The same correction must satisfy both estimates: the
$H^k$ bound controls terms multiplied by $h^{-1}$,
whereas the $H^{k+1}$ bound controls the second-order
composite operator in the physical lift.
The right-inverse identity is asserted only on $\Omega_1$;
no support preservation or differentiability of the exact
inverse with respect to the direction is asserted.
The operator estimate imposes no smallness condition on
the source norm.
\end{remark}

The proof in Subsection~\ref{sec:r2-proof-inverse} specifies
\begin{equation}\label{eq:uniform-neumann-threshold}
 \begin{gathered}
 N=k+6,\qquad r_{\rm par}=N+4=k+10,\qquad
 C_\Sigma=\sum_{j=1}^5 C_j,\\
 \tau_0=\max\{16,(4C_\Sigma)^2\},\qquad h_* =\tau_0^{-1}.
 \end{gathered}
\end{equation}
where $C_1,\ldots,C_5$ are the uniform constants of the five separately
proved operator errors in \eqref{eq:uniform-five-term-bounds}.
Any enlargement of the fixed geometric lower threshold can be included
in the maximum. These constants are fixed by the displayed coefficient bounds and geometry.

Multiplying the right inverse by the residual scale gives the
correction used in the CGO ansatz.

\begin{corollary}\label{cor:w2-residual}
For the inverse of Theorem~\ref{thm:finite-right-inverse}, define
\begin{equation}\label{eq:r2-residual-correction}
 r_h=-h\mathcal R_{h,\theta}F,\qquad F\in E_k,
\end{equation}
where $E_k$ is the source space \eqref{eq:w2-source-space}. Then
\begin{equation}\label{eq:single-inverse-two-remainders}
 \lVert r_h\rVert_{H^k(\Omega_0)}\leq C_k h^2\lVert F\rVert_{H^k},
 \qquad
 \lVert r_h\rVert_{H^{k+1}(\Omega_0)}\leq C_k h\lVert F\rVert_{H^k}.
\end{equation}
Define the leading transport operator by
\begin{equation}\label{eq:r2-leading-transport}
 \mathsf T_\theta=2D_\theta+\sum_j\theta_j\widetilde A_j.
\end{equation}
If
$\mathsf T_\theta a_0=0$ and $\mathsf T_\theta a_1=\widetilde{\mathcal M} a_0$
on a fixed neighborhood, and if
$F=\chi_0\widetilde{\mathcal M} a_1\in E_k$ with $\chi_0=1$ on a fixed working
subdomain $\Omega_{\rm CGO}\Subset\operatorname{int}K_{\rm src}$, then
\[
\widetilde{\mathcal M}
\bigl(e^{-\theta\cdot x/h}(a_0+ha_1+r_h)\bigr)=0
\quad\text{in }\Omega_{\rm CGO}.
\]
The threshold is independent of the size and auxiliary test parameters
of $F$.
\end{corollary}
\begin{proof}
Since $r_h=-h\mathcal R_{h,\theta}F$, the two bounds in
\eqref{eq:finite-right-inverse} give
\[
 h^{-2}\|r_h\|_{H^k(\Omega_0)}
 +h^{-1}\|r_h\|_{H^{k+1}(\Omega_0)}
 \le C_k\|F\|_{H^k}.
\]
This proves \eqref{eq:single-inverse-two-remainders}.

To verify the equation, use the exact conjugation formula
\eqref{eq:w2-phase-space} and the two transport equations to obtain
\begin{equation}\label{eq:r2-calculation-002}
L_{h,\theta}(a_0+ha_1)=h\widetilde{\mathcal M} a_1.
\end{equation}
On the working subdomain $F=\widetilde{\mathcal M} a_1$, and
$L_{h,\theta}r_h=-hF$ by \eqref{eq:finite-right-inverse}. Therefore
\[
 L_{h,\theta}(a_0+ha_1+r_h)=h\widetilde{\mathcal M}a_1-hF=0
 \quad\text{on }\Omega_{\rm CGO}.
\]
Conjugating back gives the asserted solution. The threshold $h_*$
is fixed by Theorem~\ref{thm:finite-right-inverse} on the whole source
space $E_k$ and is independent of the norm or auxiliary parameters of $F$.
\end{proof}

The following corollary gives the regularity requirements for the frame
construction and the $k=1$ inverse in elasticity.

\begin{corollary}
\label{cor:w2-physical-budgets}
Fix the regularity orders
\begin{equation}\label{eq:r2-regularity-orders}
 m\ge4,\qquad r_{\rm fr}=2m+8,\qquad s=r_{\rm fr}+6=2m+14.
\end{equation}
Suppose that the
fixed-support reduced coefficients and the pulled-back transport
coefficients of sizes $2$, $3$ and $4$ are bounded in $C^{s-2}$ on
the fixed neighborhoods of \eqref{eq:finite-system-class} and
\eqref{eq:r2-frame-coefficient-closure}. Then one may choose their finite frame
atlas once, with total leaf--transverse--direction derivatives through
order $r_{\rm fr}+2$, before choosing a source or $h$.
For the four-component right inverse, take $k=1$ in
Theorem~\ref{thm:finite-right-inverse}.
For any family $F_\nu\in E_1$ with
$\|F_\nu\|_{H^1}\leq C_{\rm src} K_\nu$, $K_\nu>0$, its corrections satisfy
simultaneously, for either phase and every $0<h<h_*$,
\begin{equation}\label{eq:w2-arbitrary-amplitude-budget}
 \|r_{h,\nu}\|_{H^1(\Omega_0)}\leq CC_{\rm src}h^2K_\nu,
 \qquad
 \|r_{h,\nu}\|_{H^2(\Omega_0)}\leq CC_{\rm src}hK_\nu,
 \qquad r_{h,\nu}=-h\mathcal R_{h,\pm\theta}F_\nu.
\end{equation}
The atlas, $C$ and $h_*$ are independent of $\nu$ and $K_\nu$.
No differentiability of $F_\nu$ in its index or of the exact inverse
in the direction is required.
\end{corollary}
\begin{proof}
First apply Theorem~\ref{thm:finite-frames} with
$r=r_{\rm fr}+2$. By \eqref{eq:r2-regularity-orders},
\[
 r+2=r_{\rm fr}+4=s-2,\qquad r+1=s-3.
\]
The assumed joint $C^{s-2}$ bounds therefore give the $C^{r+2}$
family required by that theorem, with compact coefficient closure in
$C^{s-3,\alpha_{\rm H}}$ as in
\eqref{eq:r2-frame-coefficient-closure}. Its finite-product assertion,
applied to the systems of sizes $2$, $3$ and $4$, gives one finite
atlas with the bounds \eqref{eq:finite-frame-uniform} through total
order $r_{\rm fr}+2$. The atlas is fixed by the coefficient prior and
the neighborhoods, before any source or $h$ is chosen.

For the four-component right inverse, take $k=1$ in
Theorem~\ref{thm:finite-right-inverse}. The coefficient requirement in
\eqref{eq:finite-system-class} is satisfied because
\[
 k+14=15\le s-2=2m+12,\qquad m\ge4.
\]
Consequently \eqref{eq:finite-right-inverse} holds for both phases
with one constant $C$ and one threshold $h_*$. For each $F_\nu\in E_1$,
set $r_{h,\nu}=-h\mathcal R_{h,\pm\theta}F_\nu$ as in
Corollary~\ref{cor:w2-residual}. The operator estimate gives
\[
 h^{-2}\|r_{h,\nu}\|_{H^1(\Omega_0)}
 +h^{-1}\|r_{h,\nu}\|_{H^2(\Omega_0)}
 \le C\|F_\nu\|_{H^1}
 \le CC_{\rm src}K_\nu,
\]
which implies \eqref{eq:w2-arbitrary-amplitude-budget}.
The constant and threshold belong to an operator bound on all of
$E_1$; hence they are independent of $\nu$ and $K_\nu$.
The argument applies separately to each source and phase and takes
no derivative in the source index or in the direction of the exact
inverse. In the physical application, the assumed coefficient bounds
follow from Proposition~\ref{prop:w1-exact-reduction}, specifically
\eqref{eq:r2-reduced-coefficient-bound}, with the stated loss of two
material derivatives.
\end{proof}

\section{Elastic CGO solutions and matrix-valued moment estimates}
\label{sec:w1-physical-moments}

Three physical dual polarizations and a fixed primal frame provide the nine
matrix moments. The frames are chosen independently of the polynomial
degree, localization width, and CGO parameter.

\subsection{Fixed domains and the analytic inputs}\label{sec:r2-physical-domains}

Use the regularity orders in \eqref{eq:r2-regularity-orders}.
The two real materials have a uniform $C^s$ bound and the strong-convexity
bounds of Section~\ref{sec:w1-reduction}. They satisfy \eqref{eq:m4-common-exterior} for the input domain $D$
and the fixed background of Subsection~\ref{sec:r2-common-exterior-setup}. Choose, before all tests,
\begin{equation}\label{eq:w1-fixed-domains}
 \overline B_R\cup\overline B_{R_b}\Subset\Omega_{\rm CGO}
 \Subset\operatorname{int}K_{\rm src}\Subset\Omega_1\Subset\Omega_0,
 \qquad K_{\rm coef}\Subset\operatorname{int}K_{\rm src}.
\end{equation}
Here $K_{\rm coef}$ contains the conjugated coefficient supports;
$K_{\rm src}$ is a closed ball, and the other working sets can be
fixed balls. Fix $\chi_0\in C_c^\infty(\operatorname{int}K_{\rm src})$
with $\chi_0=1$ near $\overline\Omega_{\rm CGO}$.
All amplitudes and their first corrections are defined on a fixed
neighborhood $U_{\chi_0}$ of $\operatorname{supp}\chi_0$.
Use a fixed disk of radius $L$ larger than the working-ball radii and
slightly larger source disks for every leafwise construction.
All transport frames and their equations are defined on a neighborhood
of the support of every leaf cutoff used with them. Choose the target,
cutoff-support, and frame-existence disks nested with fixed positive
gaps; their radii never change with a test.

\subsubsection{Uniform frames for the physical systems}\label{sec:r2-physical-frame-input}

For the $2$, $3$, and $4$ dimensional systems in
\eqref{eq:w1-fast-system}, \eqref{eq:w1-quotient-equations}, and
\eqref{eq:w1-leaf-inverse}, respectively, a uniform joint $C^{r_{\rm fr}+4}$ coefficient bound on the fixed
larger cylinders supplies local invertible frames and their inverses
with derivatives of total order at most $r_{\rm fr}+2$ bounded uniformly
on finitely many coefficient--direction--transverse neighborhoods.
The coefficient label is not differentiated. Apply
Theorem~\ref{thm:finite-frames} with frame order $r=r_{\rm fr}+2$
and input order $r+2=r_{\rm fr}+4$, at
\eqref{eq:w1-leaf-inverse}, \eqref{eq:w1-fast-system}, and
\eqref{eq:w1-fixed-frame-bound}. The construction \eqref{eq:finite-frame-neumann} also supplies
continuity in the material label, with values in the
stated total-order frame topology on fixed smaller cylinders. Its
coefficient closure and finite-cover argument are proved in
Appendix~\ref{app:w2-parametrix}; no differentiability in the material
label is used.

\subsubsection{The physical right inverse}\label{sec:r2-physical-inverse-input}

For $j=1,2$ and each normalized null direction $\omega$, a linear map
$\mathcal R^{(j)}_{h,\omega}:E_1\to H^2(\Omega_0;\mathbb C^4)$,
with $E_1$ defined by \eqref{eq:w2-source-space} at $k=1$ and
$d_{\rm sys}=4$, satisfies, for $0<h<h_*$,
\begin{equation}\label{eq:w1-I2-inverse}
 \begin{gathered}
 e^{\omega\cdot x/h}\widetilde{\mathcal M}_j e^{-\omega\cdot x/h}
           \mathcal R^{(j)}_{h,\omega}F=F\quad\hbox{on }\Omega_1,\\
 h^{-1}\|\mathcal R^{(j)}_{h,\omega}F\|_{H^1(\Omega_0)}
       +\|\mathcal R^{(j)}_{h,\omega}F\|_{H^2(\Omega_0)}
 \le C\|F\|_{H^1(\Omega_0)}.
 \end{gathered}
\end{equation}
The map in \eqref{eq:w1-I2-inverse} is the $k=1$, $d_{\rm sys}=4$
specialization of \eqref{eq:finite-right-inverse}. The same map gives both
estimates; $h_*,C$ are uniform over the class
and the directions. No support preservation or direction differentiability
of this exact inverse is assumed. The material-to-coefficient bound gives
$C^{s-2}$ coefficients, and $s-2=2m+12\ge20>15=1+14$, so the
$k=1$ specialization of Theorem~\ref{thm:finite-right-inverse}
applies. The high-order frame construction and this low-order inverse
have different regularity requirements,
summarized by Corollary~\ref{cor:w2-physical-budgets}. In the present
notation its source index is $(n,t_0,\varsigma)$ on the positive side
and the fixed physical column label on the negative side.

\subsection{Finite-order leaf inverses and the physical dual columns}\label{sec:r2-physical-polarizations}

The two principal physical dual columns span only a two-dimensional subspace
of the rank-three quotient. A supplementary column supplies the remaining
direction.

For compactly supported leaf data on the larger disk use
\begin{equation}\label{eq:w1-leaf-inverse}
 \mathcal C_\theta F(z,t_\perp)=\frac1{2\pi}
            \int\frac{F(w,t_\perp)}{z-w}\,dA(w),\qquad
 \mathcal V_\theta F=\tfrac12 \mathsf F_\theta\mathcal C_\theta
                 (\chi_{\rm leaf}\mathsf F_\theta^{-1}F),
\end{equation}
where $(2D_\theta+A_\theta)\mathsf F_\theta=0$ is the four-component frame from
Theorem~\ref{thm:finite-frames}, and $\chi_{\rm leaf}=1$ on the required target
disk. Then $D_\theta\mathcal C_\theta F=F$ and
$(2D_\theta+A_\theta)\mathcal V_\theta F=F$ there.
The factor $1/(2\pi)$ is forced by $D_\theta=2\partial_{\bar z}$.
For data extended by zero from the larger disk, spatial derivatives
commute with the convolution; transverse derivatives commute because
the integral is at fixed $t_\perp$. In fixed pulled-back coordinates
its kernel is independent of direction parameters. The integral of
$|z-w|^{-1}$ on the fixed disk is uniformly finite. Leibniz' rule,
including derivatives of the fixed cutoff, therefore bounds the
joint $C^a$ norm of this inverse by finitely many joint source and
frame norms. There is no claimed gain of a transverse derivative.
For the scalar primal lift of \eqref{eq:w1-primal-lift} one may take
\begin{equation}\label{eq:w1-scalar-lift-solution}
 \alpha_{\rm lift}=-\tfrac12\mathcal C_\theta
   \bigl(\chi_{\rm leaf}\ell_\theta
           (2D_\theta+A_\theta)\mathsf V_\theta X\bigr).
\end{equation}
The cutoff is one on a neighborhood of the target disk. Applying $2D_\theta$ to \eqref{eq:w1-scalar-lift-solution}, and using
$D_\theta\mathcal C_\theta F=F$ with $\chi_{\rm leaf}=1$ near the target disk,
we obtain the scalar equation in \eqref{eq:w1-primal-lift}. The same
fixed-support convolution bounds give the finite total-order parameter
estimates.
This scalar correction uses the finite-frame theorem for $X$ but no elliptic right
inverse and no new fourth-dimensional frame.

For a fixed normalized spinor representative, define
\begin{equation}\label{eq:r2-principal-polarization-data}
 a_\theta=\bar\theta/2\in\mathbb C^3,\qquad \theta\cdot a_\theta=1,
 \qquad \Phi=\begin{pmatrix}f_1&f_2\\R_1&R_2\end{pmatrix}.
\end{equation}
Choose the $2\times2$ matrix $\Phi$ so that each column solves
\begin{equation}\label{eq:w1-fast-system}
 D_\theta\binom fR=
 \begin{pmatrix}0&-\mu\kappa/2\\\alpha_\mu&0\end{pmatrix}\binom fR,
 \qquad \alpha_\mu=D_\theta^2(\mu^{-1})
                     =-\tfrac c2\theta^TF_\mu\theta.
\end{equation}
Existence and finite-order bounds for $\Phi^{\pm1}$ use
Theorem~\ref{thm:finite-frames}. The determinant has $D_\theta\det\Phi=0$;
multiplication of its first column by $(\det\Phi)^{-1}$ gives
$\det\Phi=1$ without changing the equation or the finite-order bounds.
For either column define
\begin{equation}\label{eq:w1-fast-amplitude}
 \begin{gathered}
 H_F=2D_\theta(c^{-1}Ra_\theta)-g_\mu c^{-1}R+F_\mu\theta f,
 \quad w_F=-\tfrac12\mathcal C_\theta(\chi_{\rm leaf}H_F),\\
 \boldsymbol z_0^F=(c^{-1}Ra_\theta+w_F,f)^T,\qquad
 \mathcal Y_F=(cf\theta,-R)^T.
 \end{gathered}
\end{equation}
The source satisfies $\theta\cdot H_F=0$: its first two terms contract
to $2c^{-1}D_\theta R$, which is canceled by the last term using
\eqref{eq:w1-fast-system}. Hence $\theta\cdot w_F=0$ and
$2D_\theta w_F=-H_F$. Both blocks of
$(2D_\theta+A_\theta)\boldsymbol z_0^F=0$ follow, and
\begin{equation}\label{eq:w1-fast-defects}
 J_\theta \boldsymbol z_0^F=\mathcal Y_F,\qquad
 (2D_\theta-A_\theta^T)\mathcal Y_F
 =\binom{2c\theta(D_\theta f+\mu\kappa R/2)}
              {-2(D_\theta R-\alpha_\mu f)}=0.
\end{equation}

On each compact spinor chart, choose spatially constant vectors satisfying
\begin{equation}\label{eq:r2-supplementary-polarization-data}
 \xi_q\in\mathbb C^2,\qquad q_\theta=Q\xi_q\in\mathbb C^3,\qquad
 \det(n_\theta,\xi_q)=-1,
\end{equation}
where $Q,n_\theta$ are defined in \eqref{eq:w1-spinor}.
For example, take $\xi_q=(1/s_0,0)^T$ where $s_0$ is bounded away from
zero, and $\xi_q=(0,1/s_1)^T$ on the other chart. All resulting
normalization denominators have fixed positive lower bounds on the
chosen subcharts. The supplementary amplitudes are
\begin{equation}\label{eq:w1-slow-amplitude}
 \begin{gathered}
 \boldsymbol z_0^{\rm sup}=(q_\theta,0)^T,\qquad J_\theta \boldsymbol z_0^{\rm sup}=0,\qquad
 \mathcal M\boldsymbol z_0^{\rm sup}=(V_\mu q_\theta,-e\,g_\mu\cdot q_\theta)^T,\\
 \boldsymbol z_1^{\rm sup}=\mathcal V_\theta(\mathcal M\boldsymbol z_0^{\rm sup})=(v_{\rm sup},f_{\rm sup})^T,\qquad
 R_{\rm col}=c\theta\cdot v_{\rm sup},\qquad
 (2D_\theta+A_\theta)\boldsymbol z_1^{\rm sup}=\mathcal M\boldsymbol z_0^{\rm sup}.
 \end{gathered}
\end{equation}
Here $\mathcal V_\theta$ is the fixed leaf inverse in
\eqref{eq:w1-leaf-inverse}.
The scalar equation and the contracted vector equation give separately
\begin{equation}\label{eq:w1-slow-two-equations}
 D_\theta f_{\rm sup}=-\mu\kappa R_{\rm col}/2-e\,g_\mu\cdot q_\theta/2,
 \qquad
 D_\theta R_{\rm col}=\alpha_\mu f_{\rm sup}+(c/2)\theta^TV_\mu q_\theta.
\end{equation}
The first nonzero physical lift defines the ambient dual column:
\begin{equation}\label{eq:w1-slow-canonical}
 \mathcal Y_{\rm sup}=\mathcal K \boldsymbol z_0^{\rm sup}-J_\theta \boldsymbol z_1^{\rm sup}
 =\binom{q_\theta-cf_{\rm sup}\theta}{R_{\rm col}+\tfrac32c\,g_\mu\cdot q_\theta}.
\end{equation}
The contraction identity \eqref{eq:w1-slow-contracted-coefficient}
is proved in Appendix~\ref{app:w1-algebra}.
Consequently the two blocks of its dual-transport defect are
\begin{equation}\label{eq:w1-slow-defects}
 \begin{aligned}
 [(2D_\theta-A_\theta^T)\mathcal Y_{\rm sup}]_{1:3}
 &=-2c\theta[D_\theta f_{\rm sup}+\mu\kappa R_{\rm col}/2
                            +e\,g_\mu\cdot q_\theta/2]=0,\\
 [(2D_\theta-A_\theta^T)\mathcal Y_{\rm sup}]_4
 &=2[D_\theta R_{\rm col}-\alpha_\mu f_{\rm sup}
                            -(c/2)\theta^TV_\mu q_\theta]=0.
 \end{aligned}
\end{equation}

Define the $4\times3$ ambient dual matrix and its quotient frame by
\begin{equation}\label{eq:r2-ambient-dual-matrix}
 \mathcal Y=(\mathcal Y_F^1,\mathcal Y_F^2,\mathcal Y_{\rm sup})
             =\mathsf U_\theta Y.
\end{equation}
Here the columns are defined in \eqref{eq:w1-fast-amplitude} and
\eqref{eq:w1-slow-canonical}, and $\mathsf U_\theta$ in
\eqref{eq:w1-spinor}. The $3\times3$ quotient frame is
\begin{equation}\label{eq:w1-physical-Y}
 Y=\begin{pmatrix}
 -2cf_1n_\theta&-2cf_2n_\theta&\xi_q+2cf_{\rm sup}n_\theta\\
 -R_1&-R_2&R_{\rm col}+\tfrac32c\,g_\mu\cdot q_\theta
 \end{pmatrix},\qquad
 \det Y=-2c\det\Phi\det(n_\theta,\xi_q)=2c.
\end{equation}
The top block has two rows. Expansion of the determinant in the basis
$(n_\theta,\xi_q)$ proves the formula; the third column's multiple of
$n_\theta$ and its bottom entry do not affect it. Thus the ambient
matrix is $4\times3$, whereas $Y$ and its determinant are $3\times3$.

\subsubsection{Transverse localization}\label{sec:r2-transverse-localization}

Choose $\chi(t_\perp)$ supported in the frame chart and equal to one on
an open interval containing every permitted test support on its smaller
subchart. For principal columns use $a_0^{F,i}=\chi \boldsymbol z_0^{F,i}$ and
$a_1^{F,i}=\mathcal V_\theta(\mathcal M a_0^{F,i})$. For the supplementary column,
\begin{equation}\label{eq:w1-transverse-correction}
 \begin{gathered}
 a_0^{\rm sup}=\chi \boldsymbol z_0^{\rm sup},\qquad a_1^{\rm sup}=\chi \boldsymbol z_1^{\rm sup}+w_\chi,\qquad
 w_\chi=\mathcal V_\theta([\mathcal M,\chi]\boldsymbol z_0^{\rm sup}),\\
 [\mathcal M,\chi]\boldsymbol z_0^{\rm sup}=
 \binom{\chi''q_\theta-\chi'g_\mu(\nu_{\rm leaf}\cdot q_\theta)}
       {d\chi'(\nu_{\rm leaf}\cdot q_\theta)}.
 \end{gathered}
\end{equation}
Here $\nu_{\rm leaf}$ is defined in \eqref{eq:w1-leaf-coordinates},
and $\mathcal V_\theta$ in \eqref{eq:w1-leaf-inverse}.
Because this inverse integrates on each leaf separately, $w_\chi$ and
its available derivatives are zero on the open interval where $\chi=1$.
Also $J_\theta a_0^{\rm sup}=0$ and its scalar component is zero throughout the
working domain, not merely on the test support.
The equations $(2D_\theta+A_\theta)a_1^\ell=\mathcal M a_0^\ell$ hold on all of
$U_{\chi_0}$. Extension by zero transversely is legitimate because the
cutoff support is compact in its chart. None of these assertions applies
to the elliptic remainder, whose transverse support is unrestricted.

The constructions of the primal and dual quotient frames,
their lifts, and the localized supplementary amplitudes
give the following uniform estimates.

\begin{proposition}\label{prop:w1-high-frames}
For each fixed $0<R_0<R$, the frames
$\mathsf F_\theta,\Phi,X_j,Y_j,Z_j$ and the
supplementary amplitudes constructed in
\eqref{eq:w1-leaf-inverse},
\eqref{eq:r2-principal-polarization-data},
\eqref{eq:w1-quotient-equations},
\eqref{eq:w1-physical-Y},
\eqref{eq:w1-primal-lift}, and
\eqref{eq:w1-slow-amplitude}--\eqref{eq:w1-transverse-correction}
can be chosen, together with their leaf and transverse
cutoffs, on a fixed finite product cover of the compact
material-pair parameter space and
$\mathbb{CP}^1\times[-R_0,R_0]$.
Each smaller chart is compactly contained in its
frame-existence chart.
These choices are independent of $h$, the polynomial
degree, and the transverse localization center and width.

On the corresponding fixed cylinders, containing the
entire closed observation disks and all residual supports,
\begin{equation}\label{eq:w1-fixed-frame-bound}
\max_{|\beta|+|\gamma|\le r_{\rm fr}}
\left\|
\partial_p^\beta\partial_x^\gamma
\bigl(
X_j,X_j^{-1},Y_j,Y_j^{-1},Z_j,
\boldsymbol z_0^{\rm sup},
\boldsymbol z_1^{\rm sup},w_\chi
\bigr)
\right\|_{C^0}
\le C_m.
\end{equation}
Here $p$ denotes local direction parameters, and
transverse differentiation is counted among the
spatial derivatives.
Equivalently, in the fixed pulled-back coordinates,
the estimate controls mixed derivatives of total
leaf--transverse--direction order at most $r_{\rm fr}$.
The displayed fields depend continuously on the
material label in this topology.
The cover and $C_m$ depend only on the fixed
a priori class and geometry.

For the two materials, the matrix
$W=Y_2^TX_1$ is invertible on the full closed
observation disks, with
$W^{-1}=X_1^{-1}Y_2^{-T}$, and satisfies
\begin{equation}\label{eq:w1-W-equation}
2\partial_{\bar z}W
=-Y_2^T(P_1-P_2)X_1.
\end{equation}
In particular, $W$ is holomorphic on every
common-exterior leaf annulus.
\end{proposition}
\begin{proof}
Put $r=r_{\rm fr}+2$. Proposition~\ref{prop:w1-exact-reduction}
and \eqref{eq:r2-reduced-coefficient-bound} give the derivative count
\[
 C^{r_{\rm fr}+6}\text{ materials}
 \longmapsto C^{r_{\rm fr}+4}=C^{r+2}
 \text{ transport coefficients}.
\]
Theorem~\ref{thm:finite-frames} and
Corollary~\ref{cor:w2-physical-budgets} therefore supply a common finite
atlas with full-disk frame and inverse bounds through total order
$r_{\rm fr}+2$.

The amplitudes in \eqref{eq:w1-primal-lift},
\eqref{eq:w1-fast-amplitude}, \eqref{eq:w1-slow-amplitude}, and
\eqref{eq:w1-transverse-correction} are obtained by products,
differentiation, and the fixed-support leaf inverses
\eqref{eq:w1-leaf-inverse}. In the fixed pulled-back coordinates,
the Cauchy kernel is parameter independent and has uniformly bounded
$L^1$ norm on the integration disks. Differentiation and Leibniz' rule
thus bound each integral by the corresponding source and frame norms.
The explicit formulas use at most two of the reserved frame derivatives,
and their material coefficients use at most two material derivatives.
All the displayed fields are therefore bounded through total order
$r_{\rm fr}$. For the inverse of the physical quotient frame, use
\eqref{eq:w1-physical-Y}:
\[
 Y^{-1}=\frac{\operatorname{adj}Y}{2c},\qquad
 \partial_jY^{-1}=-Y^{-1}(\partial_jY)Y^{-1}.
\]
The lower bounds in \eqref{eq:w1-denominators} first bound $c^{-1}$
and $Y^{-1}$. Differentiating $YY^{-1}=I$ then bounds each higher
spatial or parameter derivative in terms of already controlled
derivatives. This proves \eqref{eq:w1-fixed-frame-bound}.

To check the material-label assertion, a bounded $C^s$ material
family has compact closure in $C^{s-1,\alpha_{\rm H}}$ on the fixed
smaller neighborhoods. The coefficient formulas with two derivatives
define continuous maps into
\[
 C^{s-3,\alpha_{\rm H}}
 =C^{r_{\rm fr}+3,\alpha_{\rm H}}
 =C^{r+1,\alpha_{\rm H}}.
\]
This is the topology of \eqref{eq:r2-frame-coefficient-closure}.
The coefficient-label continuity in Theorem~\ref{thm:finite-frames}
therefore passes through the fixed integrals and algebraic operations
above. The normalization denominators stay uniformly separated from
zero. The theorem's finite-product refinement gives the asserted
common cover with fixed margins for both materials and all systems.

Finally, \eqref{eq:w1-quotient-equations} and the product rule yield
\[
 2\partial_{\bar z}(Y_2^TX_1)
 =(P_2^TY_2)^TX_1-Y_2^TP_1X_1
 =-Y_2^T(P_1-P_2)X_1.
\]
This is \eqref{eq:w1-W-equation}. Wherever the materials agree,
$P_1-P_2=0$, so $W$ is holomorphic there.
\end{proof}

\begin{remark}
The supplementary physical polarization makes the
dual quotient frame invertible, as shown by the
determinant identity \eqref{eq:w1-physical-Y}.
Since the frames are fixed independently of the
moment tests, all nine matrix moments in
Theorem~\ref{thm:w1-all-moments} are obtained,
on each chart, from the same frame product $W$,
with uniform bounds throughout the subsequent
Fourier and transport arguments.

The construction uses local normalized representatives
of the projective null directions and does not
assert a global trivialization over $\mathbb{CP}^1$.
Only continuous dependence on the material label
is required.
\end{remark}

\subsection{Exact physical solutions and Sobolev remainder estimates}\label{sec:r2-physical-solutions}

The finite-order amplitudes are corrected to exact reduced solutions and
then transformed into physical elastic solutions. Both Sobolev remainder
estimates enter the physical trace and dual-lift bounds.

For a positive-phase leading amplitude $a_0^+$, use
\begin{equation}\label{eq:w1-correction-signs}
 a_1^+=-\mathcal V_{1,\theta}(\mathcal M_1a_0^+),\qquad
 a_1^{-,\ell}=+\mathcal V_{2,\theta}(\mathcal M_2a_0^{-,\ell}),
 \qquad \ell=F1,F2,\mathrm{sup}.
\end{equation}
For the supplementary negative column take exactly the localized
$a_1^{\rm sup}$ in \eqref{eq:w1-transverse-correction}, which is the same linear
construction by expansion of $\mathcal M(\chi \boldsymbol z_0^{\rm sup})$.
Assume $\|a_0^+\|_{C^{12}}\le CK$; the negative amplitudes are fixed
and have the corresponding bound with $K=1$. The fixed leaf inverse
and Leibniz' rule give
\begin{equation}\label{eq:w1-first-correction-budget}
 \|a_1^+\|_{H^3(U_{\chi_0})}
 +\|\chi_0\mathsf R_*^{-1}\mathcal M_1a_1^+\|_{H^1(\Omega_0)}\le CK,
 \qquad
 \|a_1^{-,\ell}\|_{H^3}+
 \|\chi_0\mathsf R_*^{-1}\mathcal M_2a_1^{-,\ell}\|_{H^1}\le C.
\end{equation}
For clarity, derivatives of the source $\mathcal M a_0$ through order four cost
at most six derivatives of $a_0$, well within the $C^{12}$ bound.
The $H^1$ residual uses at most three derivatives of $a_1$.

Write
\begin{equation}\label{eq:r3-background-amplitude}
\widetilde a=\mathsf R_*^{-1}a,
\end{equation}
where $\mathsf R_*$ is defined in \eqref{eq:w1-background}. Apply the source operators of
Theorem~\ref{thm:finite-right-inverse} at $k=1$, and define
\begin{equation}\label{eq:w1-exact-remainders}
 \begin{aligned}
 r_{1,h}^+&=-\mathsf R_*\mathcal R^{(1)}_{h,-\theta}
            (h\chi_0\widetilde{\mathcal M}_1\widetilde a_1^+),\\
 r_{2,h}^{-,\ell}&=-\mathsf R_*\mathcal R^{(2)}_{h,\theta}
            (h\chi_0\widetilde{\mathcal M}_2\widetilde a_1^{-,\ell}).
 \end{aligned}
\end{equation}
Each remainder in \eqref{eq:w1-exact-remainders} satisfies both
Sobolev estimates:
\begin{equation}\label{eq:w1-H1-H2}
 \begin{aligned}
 \|r_{1,h}^+\|_{H^1(\Omega_{\rm CGO})}&\le Ch^2K,&
 \|r_{1,h}^+\|_{H^2(\Omega_{\rm CGO})}&\le ChK,\\
 \|r_{2,h}^{-,\ell}\|_{H^1(\Omega_{\rm CGO})}&\le Ch^2,&
 \|r_{2,h}^{-,\ell}\|_{H^2(\Omega_{\rm CGO})}&\le Ch.
 \end{aligned}
\end{equation}
Here the remainders are defined in \eqref{eq:w1-exact-remainders}.
The reduced fields
\begin{equation}\label{eq:w1-exact-fields}
 \boldsymbol z_{1,h}^+=e^{\varphi/h}(a_0^++ha_1^++r_{1,h}^+),\qquad
 \boldsymbol z_{2,h}^{-,\ell}=e^{-\varphi/h}
               (a_0^{-,\ell}+ha_1^{-,\ell}+r_{2,h}^{-,\ell})
\end{equation}
solve the corresponding reduced equations on $\Omega_{\rm CGO}$ by \eqref{eq:w1-phase-identities}: after the first
transport cancellation, the residual is $h\mathcal M_j a_1$, and the remainder
cancels it where $\chi_0=1$. It is not necessary that the equation hold
on all of $\Omega_0$.

The genuine physical displacements and their canonical lifts are
\begin{equation}\label{eq:w1-physical-fields}
 u_{1,h}=T_1\boldsymbol z_{1,h}^+,\qquad
 u_{2,h}^{F,i}=-hT_2\boldsymbol z_{2,h}^{-,Fi},\qquad
 u_{2,h}^{\rm sup}=T_2\boldsymbol z_{2,h}^{-,{\rm sup}},\qquad
 y_{2,h}^\ell=S_2^Tu_{2,h}^\ell.
\end{equation}
Each $u$ solves the homogeneous physical equation by the exact
factorization, not by prescribing reduced boundary observations.
With the zeroth-order $3\times4$ phase symbol
\begin{equation}\label{eq:w1-physical-phase-symbol}
 \mathsf t_{\mu,\theta}(v,f)=\mu^{-1}f\theta,\qquad
 e^{-\sigma_{\rm ph}\varphi/h}T_\mu e^{\sigma_{\rm ph}\varphi/h}
       =T_\mu+\sigma_{\rm ph} h^{-1}\mathsf t_{\mu,\theta},
\end{equation}
the supplementary displacement has the exact expansion
\begin{equation}\label{eq:w1-slow-displacement}
 e^{\varphi/h}u_{2,h}^{\rm sup}
 =T_2a_0^{\rm sup}-\mathsf t_{2,\theta}a_1^{\rm sup}+hT_2a_1^{\rm sup}
                  +T_2r_h^{\rm sup}-h^{-1}\mathsf t_{2,\theta}r_h^{\rm sup}.
\end{equation}
The absent $h^{-1}\mathsf t a_0^{\rm sup}$ term is zero on the whole working
domain. Since $T:H^2\to H^1$ and $\mathsf t:H^1\to H^1$, the last
three terms are $O(h)$ in phase-removed $H^1$ by
\eqref{eq:w1-H1-H2}. For a principal column the corresponding expansion
is $\mathsf t(a_0+ha_1+r)-hT(a_0+ha_1+r)$, bounded in $H^1$.
For the positive column it is $(T+h^{-1}\mathsf t)(a_0+ha_1+r)$,
with $H^1$ bound $Ch^{-1}K$. By the product rule for the exponential factors, the amplitude bounds in
\eqref{eq:w1-K-budget}, and the remainder estimates in \eqref{eq:w1-H1-H2}, we
obtain the physical trace bounds
\begin{equation}\label{eq:w1-physical-traces}
 \|u_{1,h}\|_{H^1(B_R)}\le Ch^{-2}e^{R/h}K,\qquad
 \|u_{2,h}^{F,i}\|_{H^1(B_R)}+\|u_{2,h}^{\rm sup}\|_{H^1(B_R)}
       \le Ch^{-1}e^{R/h}.
\end{equation}

With $\mathcal K$ from \eqref{eq:r2-composite-K} for material two,
the canonical lifts \eqref{eq:w1-physical-fields} satisfy
\begin{equation}\label{eq:w1-canonical-expansions}
 \begin{aligned}
 e^{\varphi/h}y_{2,h}^{F,i}
 &=J_\theta a_0+h(J_\theta a_1-\mathcal K a_0)+J_\theta r
                            -h^2\mathcal K a_1-h\mathcal K r,\\
 e^{\varphi/h}y_{2,h}^{\rm sup}
 &=\mathcal K a_0^{\rm sup}-J_\theta a_1^{\rm sup}+h\mathcal K a_1^{\rm sup}
                            -h^{-1}J_\theta r^{\rm sup}+\mathcal K r^{\rm sup}.
 \end{aligned}
\end{equation}
The mapping properties \eqref{eq:r2-composite-K} and
\eqref{eq:r2-lift-mappings}, together with \eqref{eq:w1-H1-H2}, give
\begin{equation}\label{eq:r2-canonical-error-estimates}
 \begin{aligned}
 \|h(J_\theta a_1-\mathcal K a_0)\|_2&\le Ch,&
 \|J_\theta r\|_2+\|h^2\mathcal K a_1\|_2+\|h\mathcal K r\|_2&\le Ch^2,\\
 \|h\mathcal K a_1^{\rm sup}\|_2&\le Ch,&
 \|h^{-1}J_\theta r^{\rm sup}\|_2+\|\mathcal K r^{\rm sup}\|_2&\le Ch.
 \end{aligned}
\end{equation}
The first row applies to either principal column and the second to the
supplementary column; all norms are in $L^2(B_R)$. Thus every term in
\eqref{eq:w1-canonical-expansions} is controlled by the same remainder.
The leading matrix on the whole integration ball is
\begin{equation}\label{eq:w1-localized-leading-matrix}
 \begin{gathered}
 \mathcal Y_\chi=(\chi\mathcal Y_F^1,\chi\mathcal Y_F^2,
                          \mathcal Y_{\chi,{\rm sup}}),\\
 \mathcal Y_{\chi,{\rm sup}}=\chi\mathcal Y_{\rm sup}-J_\theta w_\chi
       -\binom{0_{\mathbb C^3}}
                   {c\chi'(\nu_{\rm leaf}\cdot q_\theta)/(1-\kappa)}.
 \end{gathered}
\end{equation}
The last term is the commutator of $\mathcal K$ with the transverse
cutoff on $(q_\theta,0)^T$; its calculation is given in
\eqref{eq:w1-app-localized-lift}. All terms are bounded independently
of the positive test, and both correction terms vanish wherever
$\chi=1$. In particular,
\begin{equation}\label{eq:w1-canonical-L2}
 e^{\varphi/h}y_{2,h}=\mathcal Y_\chi+q_h,\qquad
 \|q_h\|_{L^2(B_R)}\le Ch,\qquad
 \mathcal Y_\chi=\mathsf U_\theta Y_2\ \hbox{where }\chi=1.
\end{equation}
The factor $h^{-1}$ is controlled by the $H^1$ estimate $O(h^2)$;
the other second-order term requires the simultaneous $H^2$ estimate
$O(h)$. No uniform $H^5$ remainder and no pointwise rank of the exact
canonical matrix is used. At fixed $h$, interior Laplacian regularity
upgrades the exact reduced fields from $H^2$ to $H^4$ near
$\overline B_R$, so Lemma~\ref{lem:w1-betti} applies. The upgrade in
Appendix~\ref{app:w1-algebra} is only an identity-justification step;
its constants may depend on $h$.

\subsection{All nine moments in fixed quotient frames}\label{sec:r2-physical-moments}

The matrix moments are estimated using $\delta_D$ in
\eqref{eq:w0-domain-error} and the previously fixed frames. The polynomial
degree and transverse localization parameters may vary after the frames have
been chosen.

On a fixed smaller transverse subchart let $t_0$ be any permitted
center and choose $0<\varsigma\le\varsigma_*$ such that the support
below lies in the open interval $\chi=1$. Fix
$\psi_\perp\in C_c^\infty((-1,1))$, $\psi_\perp\ge0$, $\int\psi_\perp=1$, and put
\begin{equation}\label{eq:w1-test-family}
 \begin{gathered}
 r(t)=\sqrt{R^2-t^2},\qquad
 b_{\varsigma,t_0}(t)=\varsigma^{-1}\psi_\perp((t-t_0)/\varsigma),\qquad
 g_n(z,t)=(z/r(t))^n,\quad n\in\mathbb N_0,\\
 a_0^+=b_{\varsigma,t_0} g_nZ_1,\qquad
 K_{n,\varsigma}=A^n(n+1)^{12}\varsigma^{-13}.
 \end{gathered}
\end{equation}
The admissible transverse range is compact in $(-R,R)$, so $r$ has a
positive lower bound there. Choose $A>1$ larger than the ratio of the
large source-disk radius to this lower bound, increasing it for fixed
geometric factors. Define $r$ on a fixed neighborhood of the compact
test supports, and extend $b_{\varsigma,t_0} g_n$ by zero transversely.
Because $D_\theta z=D_\theta t=0$, the leading amplitude solves
transport. Leibniz' rule through order $12$ gives, on the entire large
working cylinder,
\begin{equation}\label{eq:w1-K-budget}
 \|a_0^+\|_{C^{12}}\le CK_{n,\varsigma}.
\end{equation}
Each differentiation of a power costs at most a factor $n+1$; the
largest derivative of $b_{\varsigma,t_0}$ is of order $\varsigma^{-13}$.
Since $r_{\rm fr}\ge16$, the one chosen high-order frame supplies all
these bounds.

The weighted matrix integrals in \eqref{eq:w1-noisy-moment} are obtained
from elastic solutions through the DN pairing \eqref{eq:w5-physical-dn}.

\begin{theorem}
\label{thm:w1-all-moments}
Under the fixed $C^{2m+14}$ prior and geometry of this section, select the
$r_{\rm fr}=2m+8$ local family of
Proposition~\ref{prop:w1-high-frames} for each pair before any tests.
Assume the common-collar condition of Lemma~\ref{lem:w1-betti}.
The measurement domain is either $D=B_R$, or a fixed smooth
domain with $\overline D\Subset B_R$ such that the materials
agree on $B_R\setminus\overline D$.
There exist constants $A>1,C,C_0>0$ and $0<h_*\le1$, uniform over all pairs and
active charts, such that for every $n\ge0$, every permitted center and
width $0<\varsigma\le\varsigma_*$, and every $0<h<h_*$,
\begin{equation}\label{eq:w1-noisy-moment}
 \left\|\int_{B_R}b_{\varsigma,t_0} g_nY_2^T(P_1-P_2)X_1\,dx\right\|
 \le CK_{n,\varsigma}
             \left(h+\delta_Dh^{-2}e^{C_0/h}\right).
\end{equation}
Here $b_{\varsigma,t_0},g_n,K_{n,\varsigma}$ are defined in
\eqref{eq:w1-test-family}, $P_j$ in \eqref{eq:w1-quotient-connection},
and $X_1,Y_2$ are the frames of Proposition~\ref{prop:w1-high-frames}.
The discrepancy $\delta_D$ is defined in \eqref{eq:w0-domain-error}.
The inverse threshold is independent of $n,\varsigma$
and of the linear amplitude size.
All nine entries use the same $X_1,Y_2$ and physical displacements
\eqref{eq:w1-physical-fields}. Constants may depend on $m$, the full
$C^{2m+14}$ prior, the fixed geometry, cutoffs, and compact atlas, but
not on the chosen materials, $n,t_0,\varsigma,h$, or $\delta_D$.
\end{theorem}
\begin{proof}
Construct the positive columns with \eqref{eq:w1-K-budget}, and the
three fixed negative columns with the uniform amplitude bounds
corresponding to $K=1$. Apply
Theorems~\ref{thm:finite-frames} and \ref{thm:finite-right-inverse} as in
\eqref{eq:w1-correction-signs}--\eqref{eq:w1-exact-remainders}, so the
bounds \eqref{eq:w1-H1-H2} and \eqref{eq:w1-canonical-L2} hold
simultaneously with $K=K_{n,\varsigma}$.

Define the first-order $4\times4$ difference operator by
\begin{equation}\label{eq:r2-reduced-difference}
 \mathcal D=\mathcal M_1-\mathcal M_2,
\end{equation}
and use $A_{\Delta,\theta}$ from \eqref{eq:r2-contracted-differences}. The exact phase expansion is
\begin{equation}\label{eq:w1-volume-expansion}
 \begin{aligned}
 h e^{-\varphi/h}\mathcal D \boldsymbol z_{1,h}^+
 ={}&A_{\Delta,\theta} a_0^++hA_{\Delta,\theta} a_1^+
                 +h\mathcal D a_0^++A_{\Delta,\theta} r_{1,h}^+\\
   &+h^2\mathcal D a_1^++h\mathcal D r_{1,h}^+.
 \end{aligned}
\end{equation}
In $L^2(B_R)$, the five nonleading terms have bounds respectively
$ChK,ChK,Ch^2K,Ch^2K,Ch^3K$.
Pair the expansion with $\mathcal Y_\chi+q_h$ from
\eqref{eq:w1-canonical-L2}; the exponential factors cancel exactly.
If $E_h$ denotes the sum of the five nonleading terms in
\eqref{eq:w1-volume-expansion}, then $\|E_h\|_2\le ChK$ and
$\|A_{\Delta,\theta} a_0^+\|_2\le CK$. Hence the entire integrated
error is bounded by
\begin{equation}\label{eq:w1-bilinear-error-budget}
 \|q_h\|_2\|A_{\Delta,\theta} a_0^+\|_2
 +(\|\mathcal Y_\chi\|_2+\|q_h\|_2)\|E_h\|_2
 \le ChK.
\end{equation}
Only the positive side has size $K$. Thus the full error is linear in
$K$, including the product of the two remainder terms. No product
estimate in $L^\infty$ is needed. On the support of $b_{\varsigma,t_0}$ the exact quotient
contraction in \eqref{eq:w1-primal-lift} gives
\begin{equation}\label{eq:w1-nine-volume}
 \left\|h\int_{B_R}y_{2,h}^T\mathcal D \boldsymbol z_{1,h}^+\,dx
       -2\int_{B_R}b_{\varsigma,t_0} g_nY_2^T(P_1-P_2)X_1\,dx\right\|
       \le ChK.
\end{equation}
For $a,b\in\{1,2,3\}$ the $(a,b)$ entry uses the physical pair
$(u_{2,h}^{\ell(a)},u_{1,h}^{(b)})$, with the polarization indexing
\begin{equation}\label{eq:r3-polarization-indexing}
\ell(1)=F1,\qquad\ell(2)=F2,\qquad\ell(3)=\mathrm{sup}.
\end{equation}
This specifies the three-by-three array of nine pairings, without
replacing any row by an auxiliary solution.
The estimates are on the whole ball, including where the leading
positive amplitude vanishes and the exact remainder need not vanish.

The materials coincide on an open neighborhood of $\partial B_R$;
apply Lemma~\ref{lem:w1-betti} to each of these nine pairs:
\begin{equation}\label{eq:w1-nine-betti}
 \int_{B_R}(y_{2,h}^{\ell(a)})^T\mathcal D \boldsymbol z_{1,h}^{+,(b)}\,dx
 =\langle(\Lambda_2^{B_R}-\Lambda_1^{B_R})
                 \operatorname{Tr}u_{1,h}^{(b)},
                 \operatorname{Tr}u_{2,h}^{\ell(a)}\rangle.
\end{equation}
By \eqref{eq:w1-physical-traces} and the trace theorem, multiplication
by $h$ bounds each entry by
\begin{equation}\label{eq:w1-boundary-budget}
 C\delta_{B_R}h\,(h^{-2}K)(h^{-1})e^{2R/h}
       \le C\delta_Dh^{-2}e^{C_0/h}K,
\end{equation}
using Lemma~\ref{lem:w1-domain-comparison} only on its stated domain
and exterior-equality hypotheses. All norms of a $3\times3$ matrix
are equivalent with fixed constants. Combining \eqref{eq:w1-nine-volume} with \eqref{eq:w1-boundary-budget} and
dividing by two, we obtain \eqref{eq:w1-noisy-moment}.
The inverses in \eqref{eq:w1-leaf-inverse} and
\eqref{eq:finite-right-inverse} are linear on arbitrary sources; the only
smallness requirement is its operator threshold $h<h_*$. No argument
uses smallness of $hK$ or invertibility of the exact canonical field.
\end{proof}

\begin{remark}
For $\delta_D=0$, fix an admissible $(n,t_0,\varsigma)$ and let $h\to0$
in \eqref{eq:w1-noisy-moment} to obtain the exact zero moment.
For positive error the
inequality holds with simultaneous quantifiers; tests may later depend
on the error without changing the family or imposing $hK$ smallness.
For the nonmatching extensions in the original-domain problem, we apply
Theorem~\ref{thm:w1-all-moments} with $D=B_{R_e}$ and $\delta_D=\Delta_e$. The comparison to original
measurements then uses Proposition~\ref{prop:v61-enlargement-error}
with its exterior-mismatch term, as developed in
Section~\ref{sec:w5-extensions}.
\end{remark}

\section{Global solutions of the relative transport equations}
\label{sec:w3-transport}

The matrix moments control the negative boundary Fourier modes of
$W_a=Y_{2,a}^TX_{1,a}$. Invertible Cauchy extensions on the closed leaf
disks and quantitative comparison on overlaps then yield global
relative-transport endomorphisms $G$ and $H=G^{-T}$.

\subsection{The fixed input and the measurement domain}\label{sec:r2-transport-input}
\label{subsec:w3-input}
Use the quotient coordinates of Section~\ref{sec:w1-reduction}, the frames
of Proposition~\ref{prop:w1-high-frames}, and the moment estimates of
Theorem~\ref{thm:w1-all-moments}, with the regularity orders in
\eqref{eq:r2-regularity-orders} and the corresponding common-exterior
material prior. Corollary~\ref{cor:w2-physical-budgets} supplies the finite
atlas and residual estimates.

The measurement domain $D$ need not equal $\Omega$. The discrepancies
$\delta_D,\Delta_e$ are defined in \eqref{eq:w0-domain-error}, and $\eta_D$
in \eqref{eq:r2-common-error-definitions}. When the materials agree outside
$D$, Lemma~\ref{lem:w1-domain-comparison} bounds the working-sphere
discrepancy by $\delta_D$; its constant is absorbed in the moment estimate.
The enlarged-domain application uses $\Delta_e$, and the enlarged-domain map
requires no additional observation.

Choose fixed radii
\begin{equation}\label{eq:w3-radii}
 \overline D\subset B_{R_{\rm sup}},\qquad
 R_{\rm sup}<R_0<R_s<R,\qquad r_*=(R^2-R_0^2)^{1/2}>0.
\end{equation}
The two connections agree outside $\overline B_{R_{\rm sup}}$.
The common-exterior material input also says that the material pairs
themselves agree there. Define the logarithmic material ratio by
\begin{equation}\label{eq:r2-material-log-ratio}
 \rho=\log(\mu_1/\mu_2):B_R\longrightarrow\mathbb R.
\end{equation}
The common-exterior assumption gives $\rho=0$ on the fixed shell
$E$ in \eqref{eq:w4-shell-space}. Use the normalized coordinates
\eqref{eq:w1-leaf-coordinates} and radius \eqref{eq:w1-test-family},
writing $t=t_\perp$. Define the complete leaf disks by
\begin{equation}\label{eq:r2-leaf-disks}
 B_{\theta,t}=B_R\cap\{x:\nu_{\rm leaf}\cdot x=t\},\qquad |t|\le R_0.
\end{equation}
The frames on the disks \eqref{eq:r2-leaf-disks} are given on neighborhoods of the complete closed
disks and the annuli used in the proof of Lemma~\ref{lem:w3-negative}.

For each material pair, fix one finite family of $3\times3$ frames
$X_{j,a},Y_{j,a}$, $j=1,2$, on direction--transverse patches. The family
is held fixed for every degree $n$, width $\varsigma$, and parameter $h$.
After pullback to fixed cylinders, frames and square-frame inverses have
uniform derivatives of total order at most $r_{\rm fr}$. The frames satisfy \eqref{eq:w1-quotient-equations}, with
$\mathcal A_j,\Gamma_j,P_j$ defined in \eqref{eq:w1-quotient-connection}.
The transpose here is the ordinary bilinear transpose. Define the
frame products $W_a$ and the matrices $\mathfrak F_a$ by
\begin{equation}\label{eq:w3-WF}
 W_a=Y_{2,a}^TX_{1,a},\qquad
 \mathfrak F_a=Y_{2,a}^T(P_1-P_2)X_{1,a},\qquad
 2\partial_{\bar z}W_a=-\mathfrak F_a.
\end{equation}
In particular $\mathfrak F_a$ is supported in $\overline B_{R_{\rm sup}}$
and $W_a^{-1}=X_{1,a}^{-1}Y_{2,a}^{-T}$ exists on the whole closed disk.

Theorem~\ref{thm:w1-all-moments} supplies the following estimate for this
fixed frame family. For every $n\ge0$, every admissible center $t_0$
and width $0<\varsigma\le\varsigma_*$, and every $0<h<h_*$,
\begin{equation}\label{eq:w3-I4-input}
 \left\|\int b_{\varsigma,t_0}(t)
       \left(\frac{z}{r(t)}\right)^n\mathfrak F_a(x,\theta)\,dx\right\|
 \le C A^n(n+1)^{12}\varsigma^{-13}\eta_D,
 \qquad A>1.
\end{equation}
Here $b_{\varsigma,t_0}$ and the normalized profile are defined in
\eqref{eq:w1-test-family}; $\mathfrak F_a$ is defined in \eqref{eq:w3-WF}.
The constants $C,A$ and the upper bound $h_*$ are independent of
$n,t_0,\varsigma$ and the amplitude size; the dependence of the bound
on degree and width is explicit in \eqref{eq:w3-I4-input}.
The cutoffs have fixed parameter margins, and the frame family is held
fixed for every test.

For the global sum we also use a finite product cover of a compact
material-pair parameter space and
$\mathbb{CP}^1\times[-R_0,R_0]$, with the same local frame family available
on larger patches. Uniform bounds, a fixed finite cover with margins,
and continuity in the material label are input properties. No derivative
with respect to that label, and no parameter derivative of the exact
CGO right inverse, is required.
Here the required continuity means local availability of the same frame
construction on open coefficient neighborhoods, so that the fixed cover
and its material cutoffs can be chosen. It is supplied by
Theorem~\ref{thm:finite-frames} and the physical construction in
Proposition~\ref{prop:w1-high-frames};
no differentiability of the selected frame with respect to material is
used below. The matrices used here are $3\times3$ quotient
frames, not the $4\times3$ ambient lifts in the proof of Theorem~\ref{thm:w1-all-moments}.

\subsection{Moments, radii, and negative Fourier modes}\label{sec:r2-negative-modes}

Localized moment estimates and annular decay yield bounds for the negative
Fourier coefficients of the boundary traces of $W_a$; see \eqref{eq:w3-WF}.

For a sufficiently regular matrix-valued function $V$ on the unit circle, set
\begin{equation}\label{eq:r2-circle-projections}
 \begin{gathered}
 \widehat V(k)=\frac1{2\pi}\int_0^{2\pi}V(e^{i\phi})e^{-ik\phi}\,d\phi,\\
 P_-V=\sum_{k<0}\widehat V(k)e^{ik\phi},\qquad
 P_+V=\sum_{k\ge0}\widehat V(k)e^{ik\phi},\qquad V_-=P_-V.
 \end{gathered}
\end{equation}
The Cauchy integral in \eqref{eq:r2-cauchy-extension}, applied to $V$,
is holomorphic in the disk and has boundary trace $P_+V$.

The next lemma removes the transverse averaging from the moment
estimates and bounds the negative Fourier coefficients of the boundary
traces of $W_a$. It combines the estimates for low-degree moments with
annular decay for the high modes. The resulting boundary error is used
to construct the invertible Cauchy extensions $R_{C,a}$ and the local
solutions $G_a$.

\begin{lemma}\label{lem:w3-negative}
Under the hypotheses of Subsection~\ref{subsec:w3-input}, take
the negative Fourier part after pulling each leaf boundary back to
the unit circle, as in \eqref{eq:r2-circle-projections}. Define
\begin{equation}\label{eq:r2-leaf-moments}
 I_n(\theta,t)=\int_{B_{\theta,t}}g_n(z,t)\mathfrak F_a\,dA,
 \qquad \varrho=R_s/R<1.
\end{equation}
Here $B_{\theta,t},g_n,\mathfrak F_a$ are defined in
\eqref{eq:r2-leaf-disks}, \eqref{eq:w1-test-family}, and \eqref{eq:w3-WF}.
Choose a fixed $A_1>A$ and put
\begin{equation}\label{eq:w3-gamma}
 \gamma_*=
 \frac{\log(1/\varrho)}{\log A_1+\log(1/\varrho)}.
\end{equation}
There is a uniform $\eta_{\rm mom}>0$ such that, for
$0<\eta_D<\eta_{\rm mom}$,
\begin{equation}\label{eq:w3-negative-output}
 \|I_n(\theta,t)\|\le C A^n(n+1)^{12}\eta_D^{1/14},\qquad
 \|(W_{a,\partial})_-\|_{C^0}\le C_{\gamma'}\eta_D^{\gamma'/14}
 \quad(0<\gamma'<\gamma_*).
\end{equation}
Here $\eta_D$ is defined in \eqref{eq:r2-common-error-definitions}.
The constants depend only on the fixed frame bounds, atlas, cutoffs,
radii, and $\gamma'$.
\end{lemma}
\begin{proof}
Use $z=r(t)w$ to write
\begin{equation}\label{eq:r2-calculation-003}
I_n(\theta,t)=\int_{|w|<1}w^n r(t)^2
 \mathfrak F_a(t\nu_{\rm leaf}+r(t)(\Re w\,e_1+\Im w\,e_2),\theta)\,dA(w).
\end{equation}
The positive lower bound $r_*$ and the full-disk frame priors justify
one transverse derivative under the integral and give
$\|\partial_t I_n\|\le C A^n(n+1)^{12}$ on larger parameter patches.
The fixed source support also prevents a boundary support issue.
By the transverse derivative bound, the normalization $\int
b_{\varsigma,t_0}=1$, and \eqref{eq:w3-I4-input}, we obtain
\begin{equation}\label{eq:r2-calculation-004}
\|I_n(\theta,t_0)\|
 \le C A^n(n+1)^{12}(\varsigma+\eta_D\varsigma^{-13}).
\end{equation}
Take $\varsigma=\eta_D^{1/14}$. For a fixed legal width
$\varsigma_0>0$ smaller than all atlas margins, one may take
$\eta_{\rm mom}\le\min\{1,\varsigma_0^{14}\}$. This proves the first bound.

For a negative-mode index $k\ge1$, define
\begin{equation}\label{eq:r2-negative-fourier-coefficient}
c_{-k}=\frac1{2\pi}\int_0^{2\pi}W_a(re^{i\phi})e^{ik\phi}\,d\phi.
\end{equation}
Here $W_a$ and $r(t)$ are defined in \eqref{eq:w3-WF}
and \eqref{eq:w1-test-family}, respectively.
Counterclockwise orientation and Cauchy--Green give the exact identity
\begin{equation}\label{eq:w3-moment-normalization}
 \begin{split}
 2\pi i r^k c_{-k}
 &=\int_{|z|=r}z^{k-1}W_a(z)\,dz
   =2i\int_{|z|<r}z^{k-1}\partial_{\bar z}W_a\,dA\\
 &=-i r^{k-1}I_{k-1},\qquad
 c_{-k}=-\frac{I_{k-1}}{2\pi r}.
 \end{split}
\end{equation}
Thus degree $n$ controls mode $-(n+1)$; the sign and $r^{-1}$ factor
are not absorbed before this equality is established.

Set
\begin{equation}\label{eq:r3-intermediate-leaf-radius}
r_s(t)=(R_s^2-t^2)^{1/2},
\end{equation}
where $R_s$ is the fixed radius in \eqref{eq:w3-radii}. The fixed annulus
$r_s(t)\le |z|\le r(t)$ lies outside the source and has a
neighborhood of holomorphy for $W_a$. Laurent contour deformation
and the uniform inner-circle bound imply
\begin{equation}\label{eq:r2-calculation-006}
\|c_{-k}\|\le C(r_s(t)/r(t))^k\le C\varrho^k.
\end{equation}
The last inequality follows by cross-multiplying
$(R_s^2-t^2)/(R^2-t^2)\le R_s^2/R^2$.
Both inner and outer radii are uniformly nondegenerate.
The moment estimate and
$\sup_{k\ge1}k^{12}(A/A_1)^k<\infty$ now give
\begin{equation}\label{eq:r2-calculation-007}
\|c_{-k}\|\le C\min\{A_1^k\varepsilon,\varrho^k\},
 \qquad\varepsilon=\eta_D^{1/14}.
\end{equation}
Choose the Fourier splitting index by
\begin{equation}\label{eq:r3-fourier-splitting-index}
L=\log(A_1/\varrho),\qquad K=\lfloor\log(1/\varepsilon)/L\rfloor.
\end{equation} The sums for $k\le K$
and $k>K$ are bounded respectively by
$C\varepsilon A_1^K$ and $C\varrho^{K+1}$; both are at most
$C\varepsilon^{\gamma_*}$. The case $K=0$ satisfies the same bound.
Absolute convergence bounds the negative Fourier series in $C^0$.
For $\eta_D\le1$ this implies the stated weaker exponent $\gamma'$.
No small high-order norm of that series is inferred. All later
high-order bounds come from the full-boundary frame input.
\end{proof}

\subsection{An invertible Cauchy extension on the closed disk}\label{sec:r2-disk-inverse}

The Cauchy integral of the boundary trace of $W_a$ gives a holomorphic
matrix on the leaf disk. The small negative Fourier part and the continuous
invertible extension of the boundary trace ensure invertibility up to the
boundary, with the mixed derivative bounds needed for the local relative
transport. Related constructions for connection Laplacians appear in
\cite[Section~3.3]{Cekic2025}.

\begin{lemma}\label{lem:w3-cauchy}
Let $p$ vary in a fixed compact parameter chart and $N\ge2$.
Suppose $W_p^\sharp\in C^0(\overline{\mathbb D};GL(3,\mathbb C))$,
$\|(W_p^\sharp)^{-1}\|_{C^0}\le B$, and define its boundary trace by
\begin{equation}\label{eq:r3-circle-boundary-trace}
W_{\partial,p}=W_p^\sharp|_{\partial\mathbb D}.
\end{equation}
Assume
\begin{equation}\label{eq:w3-cauchy-input}
 \sup_{|\alpha|\le N}\|\partial_p^\alpha W_{\partial,p}\|_{C_\phi^{N+3}}
 \le M_N,\qquad
 \|(W_{\partial,p})_-\|_{C^0}\le\omega_{\rm F}.
\end{equation}
If $B\omega_{\rm F}<1/2$, then
\begin{equation}\label{eq:r2-cauchy-extension}
 R_{C,p}^\sharp(w)=\frac1{2\pi i}\int_{|\xi|=1}
                 \frac{W_{\partial,p}(\xi)}{\xi-w}\,d\xi
\end{equation}
extends continuously to the closed disk, is pointwise invertible there, and
\begin{equation}\label{eq:w3-cauchy-output}
 \|(R_C^\sharp)^{-1}\|_{C^0}\le2B,\qquad
 \|R_C^\sharp\|_{C_p^NC_w^N}
 +\|(R_C^\sharp)^{-1}\|_{C_p^NC_w^N}\le C(N,M_N,B).
\end{equation}
Here the mixed norm permits parameter and leaf derivatives separately
through order $N$; it is stronger than a total-order-$N$ bound.
\end{lemma}
\begin{proof}
For $k\ge0$, write the Fourier coefficients as
\[
 c_k(p)=\frac1{2\pi}\int_0^{2\pi}
              W_{\partial,p}(e^{i\phi})e^{-ik\phi}\,d\phi.
\]
Integration by parts $N+3$ times, using
\eqref{eq:w3-cauchy-input}, gives
\[
 \|\partial_p^\alpha c_k\|
 \le CM_N(1+k)^{-N-3},\qquad |\alpha|\le N,
\]
where $k=0$ is bounded directly. The Cauchy integral
\eqref{eq:r2-cauchy-extension} equals
$R_{C,p}^\sharp(w)=\sum_{k\ge0}c_k(p)w^k$ for $|w|<1$.
After $b\le N$ leaf derivatives, the series is bounded by
$CM_N\sum_{k\ge0}(1+k)^{b-N-3}<\infty$.
It therefore extends with the stated mixed derivatives to the
closed disk, and
$R_C^\sharp=W_\partial-(W_\partial)_-$ on its boundary.

Fix $p$ and set $E=(W_\partial)_-$. The boundary homotopy satisfies
\[
 \begin{aligned}
 M_t&=W_\partial-tE
     =W_\partial(I-tW_\partial^{-1}E),\\
 \|M_t^{-1}\|&\le\frac{B}{1-B\omega_{\rm F}}\le2B,
 \qquad 0\le t\le1.
 \end{aligned}
\]
Indeed $\|tW_\partial^{-1}E\|\le B\omega_{\rm F}<1/2$, so
the Neumann series applies uniformly in $t$.
The loop $\det W_\partial$ has winding number zero because it
extends to the nonvanishing function $\det W^\sharp$ on the disk.
Since $M_1=R_C^\sharp$ on the circle, homotopy invariance gives
\[
 \operatorname{wind}(\det R_C^\sharp|_{\partial\mathbb D},0)
 =\operatorname{wind}(\det W_\partial,0)=0.
\]
The argument principle now excludes zeros of $\det R_C^\sharp$ in
the disk; see \cite[Chapter~4, \S5.2, Theorem~18, pp.~152--153]{Ahlfors1979}.
It may be applied on circles $|w|=\rho<1$ sufficiently close to one:
continuity and boundary nonvanishing give a zero-free outer annulus,
and radial deformation preserves the winding number.

Thus $(R_C^\sharp)^{-1}$ is holomorphic inside and continuous on
the closed disk. Apply the scalar maximum principle to
$v^*(R_C^\sharp)^{-1}u$ for constant unit vectors $u,v$, and take
their supremum. The boundary bound above yields
\[
 \sup_{w\in\overline{\mathbb D}}\|(R_C^\sharp(w))^{-1}\|
 \le\sup_{|w|=1}\|(R_C^\sharp(w))^{-1}\|\le2B.
\]
The star is used here only for the matrix operator norm.
For any real leaf or parameter derivative, differentiation gives
\[
 \partial_j(R_C^\sharp)^{-1}
 =-(R_C^\sharp)^{-1}(\partial_jR_C^\sharp)(R_C^\sharp)^{-1}.
\]
Induction on the mixed derivative order, using the direct derivative
bounds and this zero-order inverse bound, proves
\eqref{eq:w3-cauchy-output}.
\end{proof}

\begin{remark}
Boundary invertibility alone would not exclude zeros of the
holomorphic determinant inside the disk. The existing nonvanishing
extension of $W_\partial$ fixes its winding number, which supplies that
interior conclusion. The pointwise invertibility of $R_{C,a}$ on the
closed disk permits $G_a$ in \eqref{eq:w3-local-G} to be defined up to
the leaf boundary.
\end{remark}

\subsection{Common coordinates and quantitative comparison}\label{sec:r2-chart-comparison}

Local transports constructed in different charts must first be expressed in
the same quotient coordinates before their boundary differences can be
compared.

We use the normalization law \eqref{eq:w1-normalization}, with
$S_a$ defined in \eqref{eq:r2-positive-rescaling}, and the spinor
transition law \eqref{eq:w1-chart-law}, with $D_{\rm tr}$ from
\eqref{eq:r2-spinor-charts}. For the primal transport, inversion and
transposition give the additional normalization formula
\begin{equation}\label{eq:w3-normalized-conversion}
 G_{\rm h}=S_a^{-T}G_{\rm n}S_a^T.
\end{equation}
The coordinate laws apply to a given endomorphism. The local solutions
of the relative transport equation constructed from different Cauchy
extensions are compared by Lemma~\ref{lem:w3-comparison}.
On compact overlaps all these changes and their required derivatives
are uniformly bounded. After removing the direction-scaling factor and applying the coordinate
transformations \eqref{eq:w3-normalized-conversion} and
\eqref{eq:w1-chart-law}, we express both local equations in the same
coordinates at the same physical point.

Once both equations are in these common coordinates, the matrix
$X_1^{-1}(G_a-G_b)X_2$ is holomorphic in the leaf coordinate.
The following maximum-principle estimate controls $G_a-G_b$
throughout the disk.

\begin{lemma}\label{lem:w3-comparison}
On a complete closed leaf disk let $G_a,G_b$ solve
$2\partial_{\bar z}G+P_1G-GP_2=0$ in the same coordinates.
Suppose invertible frames $X_1,X_2$ solve
$2\partial_{\bar z}X_j+P_jX_j=0$, with each frame and inverse
bounded by $M_0$. Then
\begin{equation}\label{eq:w3-comparison}
 \|G_a-G_b\|_{C^0(\overline B)}
 \le M_0^4\|G_a-G_b\|_{C^0(\partial B)}.
\end{equation}
\end{lemma}
\begin{proof}
Set
\begin{equation}\label{eq:r3-local-transport-comparison}
V=X_1^{-1}(G_a-G_b)X_2,
\end{equation}
where $X_j$ are the quotient frames in \eqref{eq:w1-quotient-equations}. The equations
\eqref{eq:w1-quotient-equations} imply
\begin{equation}\label{eq:r2-comparison-holomorphic}
 \partial_{\bar z}V=0,\qquad
 \|V\|_{C^0(\overline B)}\le\|V\|_{C^0(\partial B)}.
\end{equation}
The second inequality follows by applying the maximum principle to
constant-vector matrix elements. Multiplying by the common frame factors on the boundary and in the interior,
we obtain
\begin{equation}\label{eq:r2-comparison-norm-chain}
 \|G_a-G_b\|_{C^0(\overline B)}
 \le M_0^2\|V\|_{C^0(\overline B)}
 \le M_0^4\|G_a-G_b\|_{C^0(\partial B)},
\end{equation}
which is \eqref{eq:w3-comparison}.
\end{proof}

Define the local solution of the relative transport equation in its
initial quotient coordinates by
\begin{equation}\label{eq:w3-local-G}
 G_a=X_{1,a}R_{C,a}^{-1}Y_{2,a}^T,\qquad
 G_a^{-1}=Y_{2,a}^{-T}R_{C,a}X_{1,a}^{-1}.
\end{equation}
Here $R_{C,a}$ is the Cauchy extension \eqref{eq:r2-cauchy-extension}
pulled back from the unit disk to $B_{\theta,t}$, and $X_{j,a},Y_{j,a}$
are the frames of Proposition~\ref{prop:w1-high-frames}.
The frame equations \eqref{eq:w1-quotient-equations} and
$\partial_{\bar z}R_{C,a}=0$ give
$2\partial_{\bar z}G_a+P_1G_a-G_aP_2=0$.
On the leaf boundary $\partial B_{\theta,t}\subset\partial B_R$,
\begin{equation}\label{eq:w3-local-boundary}
 \begin{split}
 G_a-I&=X_{1,a}(R_{C,a}^{-1}-W_a^{-1})Y_{2,a}^T,\\
 R_{C,a}^{-1}-W_a^{-1}&=R_{C,a}^{-1}(W_a-R_{C,a})W_a^{-1}.
 \end{split}
\end{equation}
Consequently, after conversion to a common chart,
$\|G_a-I\|_{\partial B}\le C_b\omega_{\rm F}$ and
$\|G_a-G_b\|_{\overline B}\le C_{\rm ov}\omega_{\rm F}$.
The constants are fixed by the atlas, frame bounds and
Lemma~\ref{lem:w3-cauchy}. In particular, $G_a$ and $G_a^{-1}$
have a common zero-order bound $M_{\rm loc}$ independent of the
actual size of $\omega_{\rm F}$ below its Cauchy threshold.

\subsection{Projective weights, an invertible sum, and one threshold}\label{sec:r2-global-transport}

Leaf-constant weights combine the local transports without introducing
additional $\partial_{\bar z}$ terms. The overlap estimate supplies a
uniform invertibility threshold.

For a nonzero complex null vector, define
\begin{equation}\label{eq:r2-projective-leaf-projection}
\nu_{\rm leaf}([\theta])=\frac{\Re\theta\times\Im\theta}{|\Re\theta|^2},\qquad
 \pi(x,[\theta])=([\theta],\nu_{\rm leaf}([\theta])\cdot x).
\end{equation}
Nullity makes the two real vectors orthogonal with equal nonzero
length. Multiplication by $b\in\mathbb C^*$ multiplies both the
numerator and denominator by $|b|^2$; hence $\nu_{\rm leaf}$ is projectively
well defined and $D_\theta(\nu_{\rm leaf}\cdot x)=0$.

Refine the finite product cover to smaller sets with closures in the
original patches. To obtain this cover on projective directions, choose
a smooth normalized null representative on each larger compact spinor
chart and pull back the coefficient--oriented-direction frame atlas of
Theorem~\ref{thm:finite-frames} along that local section. Its continuity
and finite mixed bounds persist under this fixed smooth pullback.
Refine the resulting cover on the compact material-pair space times
$\mathbb{CP}^1\times[-R_0,R_0]$ once. Local constructions obtained
from different sections need not coincide: they are compared only
after conversion by \eqref{eq:w3-normalized-conversion} and
\eqref{eq:w1-chart-law}. Thus the following weights are
functions on projective directions themselves.
On the material factor choose continuous cutoffs
$\eta_a(\mathfrak a)$, and on its direction--transverse factor choose smooth
cutoffs $\psi_a(p)$, nonnegative and positive on the smaller sets.
They may be bounded by one. Compactness gives a positive lower
bound $s_*$ for $\sum_a\eta_a(\mathfrak a)^2\psi_a(p)^2$. Thus
\begin{equation}\label{eq:w3-weights}
 \chi_a(\mathfrak a,p)=\frac{\eta_a(\mathfrak a)^2\psi_a(p)^2}
                  {\sum_b\eta_b(\mathfrak a)^2\psi_b(p)^2},\quad
 \sum_a\chi_a=1,\quad \chi_a\ge0,\quad D_\theta(\chi_a\circ\pi)=0.
\end{equation}
Every required parameter derivative is bounded using the fixed
cutoffs and $s_*$. The material label is held fixed. This does not
reuse weights depending on a choice of oriented frame.

Fix $B\ge1$ bounding the full-disk inverse of $W_a$ from
\eqref{eq:w3-WF}, and set
\begin{equation}\label{eq:r2-cauchy-threshold}
 \omega_{\rm C}=(4B)^{-1}.
\end{equation}
For $\omega_{\rm F}<\omega_{\rm C}$, Lemmas~\ref{lem:w3-cauchy} and
\ref{lem:w3-comparison}, together with \eqref{eq:w3-local-boundary}, fix $M_{\rm loc}\ge1$, $C_b\ge1$, and $C_{\rm ov}\ge1$ in the
same compact quotient charts. Include all chart-norm comparison
constants in them. Put
\begin{equation}\label{eq:w3-common-threshold}
 \omega_{\rm G}=(2M_{\rm loc}C_{\rm ov})^{-1},\qquad
 \omega_{\rm tr}=\min\{1,\omega_{\rm C},\omega_{\rm G}\}.
\end{equation}
Let $C_H\ge1$ be the boundary constant constructed in the proof of
Proposition~\ref{prop:w3-gluing}. For transport and projection, set
\begin{equation}\label{eq:r2-transport-projection-threshold}
 \omega_* =\min\{\omega_{\rm tr},C_H^{-1}\}.
\end{equation} This last fixed restriction only ensures
that the boundary error parameter used in the projection is at most one.

On the fixed compact holomorphic spinor charts $K_a$, use the mixed norm
\begin{equation}\label{eq:r2-mixed-chart-norm}
 \|U\|_{C_x^qC_\zeta^r}
 =\max_a\max_{|\alpha|\le q,\,|\beta|\le r}
   \sup_{x\in\overline B_R,\,\zeta\in K_a}
       |\partial_x^\alpha\partial_{(\Re\zeta,\Im\zeta)}^\beta U_a(x,\zeta)|_{\rm F},
\end{equation}
where $U_a$ is the local matrix and $|\cdot|_{\rm F}$ is the Frobenius norm.
The spatial derivatives in \eqref{eq:r2-mixed-chart-norm} hold the
spinor representative fixed.

The overlap estimate makes the leaf-constant weighted sum invertible under a
fixed threshold, with mixed bounds for the subsequent direction projection.

\begin{proposition}\label{prop:w3-gluing}
Under the fixed input of Subsection~\ref{subsec:w3-input}, suppose
$\|(W_{a,\partial})_-\|_{C^0}\le\omega_{\rm F}<\omega_{\rm tr}$
on the same family. Then there are global invertible quotient
endomorphisms $G$ and $H=G^{-T}$ over
$\overline B_R\times\mathbb{CP}^1$ satisfying
\begin{equation}\label{eq:w3-global-output}
 \begin{split}
 &2D_\theta G+\mathcal A_1G-G\mathcal A_2=0,\qquad
 2D_\theta H+\Gamma_1H-H\Gamma_2=0,\\
 &\|G-I\|_{C^0(\partial B_R\times\mathbb{CP}^1)}
 +\|H-I\|_{C^0(\partial B_R\times\mathbb{CP}^1)}
 \le C_H\omega_{\rm F},\\
 &\|G\|_{C_x^mC_\zeta^2}+\|G^{-1}\|_{C_x^mC_\zeta^2}
 +\|H\|_{C_x^mC_\zeta^2}+\|H^{-1}\|_{C_x^mC_\zeta^2}\le C_m.
 \end{split}
\end{equation}
Here the connections are defined in \eqref{eq:w1-quotient-connection},
the threshold in \eqref{eq:w3-common-threshold}, and the norms in
\eqref{eq:r2-mixed-chart-norm}.
The norms use fixed compact holomorphic spinor charts and physical
$x$ derivatives. Every constant is uniform over the material class
and independent of $n,\varsigma,h,\delta_D$.
\end{proposition}
\begin{proof}
Express each $G_a$ in the same quotient chart before forming
\begin{equation}\label{eq:w3-global-sum}
 G(x,[\theta])=\sum_a\chi_a(\mathfrak a,\pi(x,[\theta]))G_a(x,[\theta]).
\end{equation}
Cutoff supports have margins, so each summand extends by zero off
its patch with all required derivatives. Scalar weights commute with
coordinate transitions; hence the sum is an endomorphism. The exact
relative equation is preserved by \eqref{eq:w3-weights} and linearity.
At any parameter choose an active index $a_0$. No differentiation of
this choice is made. The full-disk comparison yields
\begin{equation}\label{eq:w3-average-inverse}
 \begin{split}
 &\|G-G_{a_0}\|\le C_{\rm ov}\omega_{\rm F},\qquad
 \|G_{a_0}^{-1}(G-G_{a_0})\|<1/2,\\
 &G^{-1}=[I+G_{a_0}^{-1}(G-G_{a_0})]^{-1}G_{a_0}^{-1},
 \qquad \|G^{-1}\|\le2M_{\rm loc}.
 \end{split}
\end{equation}
It is this estimate, not convexity of invertible matrices, that proves
invertibility. The boundary bound for $G-I$ follows by averaging.
The identity $H-I=G^{-T}(I-G^T)$ gives the claimed boundary constant
$C_H$, which can be fixed using $2M_{\rm loc}$ before the actual
error is chosen.

If $R_{\rm sup}<|t|<R_0$, the whole leaf has $\mathfrak F_a=0$;
$W_a$ is holomorphic on the disk, $R_{C,a}=W_a$, and $G_a=I$.
Thus the sum joins exactly to $G=I$ for $|t|\ge R_0$ on an open
band. Near tangent leaves it is already the identity. No uniform
estimate on a shrinking disk is asserted or needed. This joining
band is a band of leaves; at nonzero noise it is not a claim that
$G=I$ at every point of the exterior spatial shell.

For the derivative estimates, set $N=m+2$ and use
\begin{equation}\label{eq:r2-leaf-pullback}
x_p(w)=t\nu_{\rm leaf}+r(t)(\Re w\,e_1+\Im w\,e_2).
\end{equation}
The boundary product has parameter derivatives through $N$ and
circle derivatives through $N+3$. Each mixed input therefore costs
at most
\begin{equation}\label{eq:w3-order-budget}
 N+(N+3)=2m+7<r_{\rm fr}=2m+8,\qquad r_{\rm fr}+6=2m+14.
\end{equation}
Leibniz distributes these orders among the two frames; it does not
require this many derivatives of each factor in addition. Lemma~\ref{lem:w3-cauchy} supplies the separate orders $N,N$ for $R_C$ and its
inverse. To obtain the same rectangular bound for $G_a^\sharp$,
the frame factors need at most $2N=2m+4$ derivatives, also available.

At a fixed physical point,
\begin{equation}\label{eq:r2-spatial-leaf-parameters}
p(x,\zeta)=([\theta(\zeta)],\nu_{\rm leaf}(\zeta)\cdot x),\qquad
 w(x,\zeta)=\frac{e_1(\zeta)\cdot x+i e_2(\zeta)\cdot x}
                  {r(\nu_{\rm leaf}(\zeta)\cdot x)}.
\end{equation}
The denominator is at least $r_*$. For $|\alpha|\le m$ and
$\beta_1+\beta_2\le2$, the real multivariable chain rule bounds
\begin{equation}\label{eq:w3-fixed-x-chain}
 \begin{split}
 &\|\partial_x^\alpha\partial_\zeta^{\beta_1}
       \partial_{\bar\zeta}^{\beta_2}
       [\chi_a(\mathfrak a,p(x,\zeta))G_a^\sharp(p(x,\zeta),w(x,\zeta))]\|_{C^0}\\
 &\quad\le C_{m,\rm geom}
 \max_{|\gamma|+j\le |\alpha|+\beta_1+\beta_2}
 \|\partial_p^\gamma\nabla_w^j(\chi_aG_a^\sharp)\|_{C^0}.
 \end{split}
\end{equation}
The order on the right is at most $m+2=N$. Coordinate conversions
and weights consume at most that order and have fixed smooth bounds.
The finite sum therefore has the asserted mixed regularity.
Repeatedly differentiating $GG^{-1}=I$ uses only derivatives within
the same rectangle and the zero-order bound
\eqref{eq:w3-average-inverse}; it controls $G^{-1}$, then $H,H^{-1}$.
Finally, differentiating $H=G^{-T}$ in the equation for $G$ and using
$\mathcal A_j=-\Gamma_j^T$ in the original matrix order, we obtain the
equation for $H$ in \eqref{eq:w3-global-output}.
\end{proof}

\begin{remark}
The global sum preserves the relative equation because its scalar
weights are constant along the leaves. Its invertibility follows
from the quantitative agreement of the local solutions.
\end{remark}

For a fixed $0<\gamma'<\gamma_*$, choose the Fourier constant
$C_{\rm F}\ge1$ in Lemma~\ref{lem:w3-negative}. One common smallness
choice for the physical input, transport, and projection is
\begin{equation}\label{eq:w3-eta-threshold}
 0<\eta_D<\eta_*:=\min\left\{\eta_{\rm mom},1,
              \left(\frac{\omega_*}{2C_{\rm F}}\right)^{14/\gamma'}\right\},
 \qquad 0<h<h_*.
\end{equation}
Define the Fourier error bound by
\begin{equation}\label{eq:r2-fourier-error}
 \omega_{\rm F}=C_{\rm F}\eta_D^{\gamma'/14}<\omega_*.
\end{equation}
Here $\eta_D$ and $\omega_*$ are defined in
\eqref{eq:r2-common-error-definitions} and
\eqref{eq:r2-transport-projection-threshold}.
The number $\gamma'$ and then all constants and thresholds are fixed
before varying the measured error or $h$. An additional rigidity
threshold is imposed only after the transport and projection estimates
have been proved; see \eqref{eq:w3-joint-threshold}.
At $\delta_D=0$ the same fixed family is used for all small $h$;
$\eta_D=h\downarrow0$. This observation supplies numerical error
bounds tending to zero and does not use a uniqueness theorem or
$\log0$.

\section{Holomorphic projection in the null direction}
\label{sec:w3-projection}

At each fixed physical point $x$, we project the global endomorphism
$H=G^{-T}$ onto endomorphisms holomorphic in the complex null direction.
The projected field $H_0=\Pi H$ has a finite-dimensional block form.
Direction differentiation and spatial interpolation bound $H-H_0$;
the first-order transport equation then gives the residual required
for material recovery. No invertibility of $H_0$ is required.

\subsection{A fixed nine-dimensional endomorphism space}\label{sec:r2-projection-space}

The fixed quotient bundle has a finite-dimensional space of holomorphic
endomorphisms. Its explicit block structure determines the projection used
for coefficient recovery.

Use homogeneous spinor coordinates $[s_0:s_1]$, with representatives
$(1,\zeta)$ and $(\upsilon,1)$ in the two affine charts. The dual quotient coordinates have the transition $D_{\rm tr}$
defined in \eqref{eq:r2-spinor-charts}. With the convention that
the tautological frames satisfy $e_0=\zeta e_\infty$, this gives
\begin{equation}\label{eq:w3-fixed-bundle}
 \begin{split}
 &E^*=\mathcal O(-1)^{\oplus2}\oplus\mathcal O,\qquad
 \mathcal E=\operatorname{End}(E^*),\\
 &\mathcal E=\mathcal O^{\oplus5}\oplus
             \mathcal O(-1)^{\oplus2}\oplus\mathcal O(1)^{\oplus2}.
 \end{split}
\end{equation}
This fixed bundle is independent of the material pair, $x$, and noise.

The transition law restricts the degrees of the holomorphic blocks.

\begin{lemma}\label{lem:w3-holomorphic-form}
For the bundle $\mathcal E$ in \eqref{eq:w3-fixed-bundle}, define
\begin{equation}\label{eq:r2-holomorphic-section-space}
 \mathcal H=H^0(\mathbb{CP}^1,\mathcal E).
\end{equation}
This space is exactly
\begin{equation}\label{eq:w3-nine-dimensional}
 \left\{\begin{pmatrix}A_H&0\\s_0\mathsf C_{H,0}+s_1\mathsf C_{H,1}&d_H\end{pmatrix}:
 A_H\in \mathbb C^{2\times2},\ d_H\in\mathbb C,\\
 \mathsf C_{H,0},\mathsf C_{H,1}\in\mathbb C^{1\times2}\right\},\qquad
 \dim_{\mathbb C}\mathcal H=9.
\end{equation}
Here the matrix is evaluated in the local spinor representative;
it is not a single trivialization over projective space. Neither
$A_H$ nor $d_H$ is assumed invertible.
\end{lemma}
\begin{proof}
For an arbitrary holomorphic endomorphism write its chart blocks
as $A_0,B_0,C_0^{\rm aff},d_0$ and
$A_\infty,B_\infty,C_\infty^{\rm aff},d_\infty$. The transition law
$H^{(0)}=D_{\rm tr}^{-1}H^{(\infty)}D_{\rm tr}$ reads
\begin{equation}\label{eq:r2-affine-block-transition}
\begin{aligned}
 A_0(\zeta)&=A_\infty(1/\zeta),&d_0(\zeta)&=d_\infty(1/\zeta),\\
 B_0(\zeta)&=\zeta^{-1}B_\infty(1/\zeta),&
 C_0^{\rm aff}(\zeta)&=\zeta C_\infty^{\rm aff}(1/\zeta).
 \end{aligned}
\end{equation}
The first and last diagonal blocks are entire and bounded at infinity,
hence constant. The entire upper-right column is $O(|\zeta|^{-1})$
and vanishes. The lower-left row is entire with at most linear growth.
Cauchy's estimates on expanding circles give
$C_0^{\rm aff}(\zeta)=\mathsf C_{H,0}+\zeta \mathsf C_{H,1}$ and
$C_\infty^{\rm aff}(\upsilon)=\upsilon \mathsf C_{H,0}+\mathsf C_{H,1}$.
These formulas are precisely the homogeneous row in
\eqref{eq:w3-nine-dimensional}. Conversely these matrices satisfy
the transitions and are holomorphic on both charts. Counting their
free coefficients gives $4+1+2+2=9$.
\end{proof}

Fix smooth Hermitian metrics on $\mathcal E$ and $\mathbb{CP}^1$
and let $\Pi$ be the $L^2$ orthogonal projection onto $\mathcal H$.
For a fixed orthonormal basis $\mathsf e_1,\ldots,\mathsf e_9$,
\begin{equation}\label{eq:w3-fixed-Pi}
 \Pi u=\sum_{a=1}^9
   \left(\int_{\mathbb{CP}^1}\langle u,\mathsf e_a\rangle\,dV\right)\mathsf e_a,
 \qquad \Pi I=I,
\end{equation}
with inner product linear in its first entry. This is a fixed linear
operator, not a projection depending on $x$ or on the material pair.

\begin{lemma}\label{lem:w3-projection-estimate}
Let $\Sigma=\mathbb{CP}^1$ and let $\mathcal E$ be the bundle
in \eqref{eq:w3-fixed-bundle}. For every integer $N\ge0$,
every $p>2$, and every $u\in W^{1,p}(\Sigma;\mathcal E)$,
one has
\begin{equation}\label{eq:w3-projection-bounds}
 \|\Pi u\|_{C^N}\le C_N\|u\|_{C^0},\qquad
 \|u-\Pi u\|_{W^{1,p}}\le C_p\|\bar\partial u\|_{L^p}.
\end{equation}
Here $\Pi$ is the projection in \eqref{eq:w3-fixed-Pi}.
For a family $u=u(x,\cdot)$, the projection commutes with
physical-space derivatives whenever these derivatives
exist in the indicated spaces.
\end{lemma}

\begin{proof}
Set $\Sigma=\mathbb{CP}^1$ and
$\mathcal H=H^0(\Sigma,\mathcal E)$.
By the finite-rank representation
\eqref{eq:w3-fixed-Pi}, the smoothness of
the fixed basis $\{\mathsf e_a\}_{a=1}^9$,
and H\"older's inequality,
\[
\begin{aligned}
\|\Pi u\|_{C^N}
&\le C_N\sum_{a=1}^9
\left|
\int_\Sigma
\langle u,\mathsf e_a\rangle\,dV
\right|
\le C_N\|u\|_{C^0},\\
\|\Pi v\|_{L^p}
&\le C_p\|v\|_{L^p},
\qquad v\in L^p(\Sigma;\mathcal E).
\end{aligned}
\]
Since $\Pi$ is independent of the physical
variable $x$, it also satisfies
\[
\partial_x^\gamma(\Pi u)
=\Pi(\partial_x^\gamma u)
\]
whenever the indicated derivatives exist.

The elliptic estimate for the Dolbeault operator
on the fixed holomorphic bundle $\mathcal E$
over $\Sigma$
\cite[Lemma~C.2.1]{McDuffSalamon2004}
gives
\begin{equation}\label{eq:w3-CR-elliptic}
\|v\|_{W^{1,p}}
\le C_p\bigl(
\|\bar\partial v\|_{L^p}
+\|v\|_{L^p}
\bigr).
\end{equation}
Initially valid for smooth sections, this estimate
extends to $W^{1,p}(\Sigma;\mathcal E)$ by density
and the continuity of $\bar\partial:W^{1,p}\to L^p$.

We claim that
\[
\|v\|_{W^{1,p}}
\le C_p\|\bar\partial v\|_{L^p},
\qquad
v\in W^{1,p}(\Sigma;\mathcal E),
\quad \Pi v=0.
\]
Otherwise, there exists a sequence
$v_k\in W^{1,p}(\Sigma;\mathcal E)$ such that
\[
\Pi v_k=0,\qquad
\|v_k\|_{W^{1,p}}=1,\qquad
\|\bar\partial v_k\|_{L^p}\longrightarrow0.
\]
By the compact embedding
$W^{1,p}(\Sigma;\mathcal E)\hookrightarrow L^p(\Sigma;\mathcal E)$,
after passing to a subsequence,
\[
v_k\rightharpoonup v
\quad\text{in }W^{1,p},
\qquad
v_k\longrightarrow v
\quad\text{in }L^p.
\]
Consequently,
\[
\bar\partial v=0,\qquad \Pi v=0.
\]
The first identity holds distributionally.
Elliptic regularity for $\bar\partial$ implies
$v\in\mathcal H$, and hence
\[
v=\Pi v=0.
\]
It follows from \eqref{eq:w3-CR-elliptic} that
\[
1=\|v_k\|_{W^{1,p}}
\le C_p\bigl(
\|\bar\partial v_k\|_{L^p}
+\|v_k\|_{L^p}
\bigr)\longrightarrow0,
\]
a contradiction.

Finally, for $w=u-\Pi u$ we have
\[
\Pi w=0,\qquad
\bar\partial w=\bar\partial u,
\]
because $\Pi u\in\mathcal H$.
Therefore,
\[
\|u-\Pi u\|_{W^{1,p}}
=\|w\|_{W^{1,p}}
\le C_p\|\bar\partial w\|_{L^p}
=C_p\|\bar\partial u\|_{L^p}.
\]
This proves \eqref{eq:w3-projection-bounds}.
\end{proof}

\subsection{Direction differentiation at a fixed physical point}\label{sec:r2-direction-defect}

Direction differentiation is taken with the physical point $x$ fixed. The
boundary deviation of $H$ from the identity controls its difference from the
holomorphic projection $H_0=\Pi H$.

The formulas \eqref{eq:w1-spinor} and
\eqref{eq:w1-quotient-connection} are polynomial in the homogeneous
spinor representative. Their material coefficients are independent
of direction. Therefore, in the holomorphic charts
\eqref{eq:r2-spinor-charts},
\begin{equation}\label{eq:w3-holomorphic-coefficients}
 \partial_{\bar\zeta}\theta=0,\qquad
 \partial_{\bar\zeta}\Gamma_j=0.
\end{equation}
These derivatives hold physical $x$ fixed. Define the relative transport
operator and the projected field by
\begin{equation}\label{eq:r2-transport-projection-definitions}
 \mathcal T(U)=2D_\theta U+\Gamma_1U-U\Gamma_2,\qquad H_0=\Pi H,
\end{equation}
where $\Pi$ is defined in \eqref{eq:w3-fixed-Pi}.

We propagate the boundary direction defect through the leaf disks
and use spatial interpolation to estimate $H-H_0$.

\begin{proposition}\label{prop:w3-defect}
Let $m>3$ be an integer. Suppose a jointly $C^2$ endomorphism $H$
on $\overline B_R\times\mathbb{CP}^1$ satisfies
\begin{equation}\label{eq:w3-H-exact}
 \mathcal T(H)=0
\end{equation}
with $\mathcal T$ defined in \eqref{eq:r2-transport-projection-definitions},
and the fixed-physical-point bounds
\begin{equation}\label{eq:w3-separated-bounds}
 \|H\|_{C_x^mC_\zeta^0}+\|H\|_{C_x^0C_\zeta^2}\le M_H,\qquad
 \|H-I\|_{C^0(\partial B_R\times\mathbb{CP}^1)}\le\omega\le1.
\end{equation}
On each leaf used to propagate a nonzero defect suppose the dual
frames $Y_1,Y_2$ and their inverses exist and have uniform bounds
on its whole closed disk.
Then $H_0$ from \eqref{eq:r2-transport-projection-definitions} is a
holomorphic endomorphism in direction and
\begin{equation}\label{eq:w3-defect-output}
 \begin{split}
 &\|\bar\partial H\|_{C_x^0C_\zeta^0}
       +\|H-H_0\|_{C_x^0C_\zeta^0}\le C\omega^{1/2},\\
 &\|H-H_0\|_{C_x^\ell C_\zeta^0}
       \le C_{m,\ell}\omega^{(m-\ell)/(2m)}\quad(0\le\ell<m),\\
 &\|H_0-I\|_{C^0(\partial B_R\times\mathbb{CP}^1)}\le C\omega,
       \qquad \|H_0\|_{C_x^mC_\zeta^0}\le C M_H.
 \end{split}
\end{equation}
The constants depend only on the fixed atlas, bundle, ball, frame
bounds and $M_H$. No invertibility of $H_0$ is assumed or concluded.
\end{proposition}
\begin{proof}
Take compact direction charts with fixed larger charts around them.
For fixed $x\in\partial B_R$, a real first difference with step
$a$ and its second-order Taylor remainder yield
\begin{equation}\label{eq:r2-direction-difference-quotient}
\|\partial_{\bar\zeta}(H-I)\|_{C^0}
 \le C(\omega/a+M_Ha).
\end{equation}
For small positive $\omega$, use $a$ comparable to
$\sqrt{\omega/M_H}$ within the fixed chart margin. For the remaining
$\omega\le1$, enlarge a constant depending on that margin and $M_H$.
If $\omega=0$, the boundary section is identically $I$ in direction
and its direction derivative is zero. Thus
define the direction defect by
\begin{equation}\label{eq:r2-direction-defect-field}
 V=\partial_{\bar\zeta}H.
\end{equation}
Its boundary norm is at most $C\omega^{1/2}$.

Differentiate \eqref{eq:w3-H-exact} at fixed physical $x$.
By \eqref{eq:w3-holomorphic-coefficients}, no derivative of
$\theta$ or $\Gamma_j$ appears, and
\begin{equation}\label{eq:r2-direction-defect-transport}
2D_\theta V+\Gamma_1V-V\Gamma_2=0.
\end{equation}
This is not differentiation at fixed moving leaf coordinates.
The transition formula \eqref{eq:w1-chart-dbar} applies to this derivative.
Consequently these coefficients are the components of a well-defined
bundle-valued $(0,1)$-form. By the dual-frame equation \eqref{eq:w1-quotient-equations},
direct multiplication yields
\begin{equation}\label{eq:r2-direction-defect-holomorphy}
D_\theta(Y_1^{-1}VY_2)=0.
\end{equation}
The holomorphic representative of $\theta$ is a nonzero complex
multiple of the normalized direction of the same leaf. Hence this is
a leaf Cauchy--Riemann equation. The scalar maximum principle applied
to constant-vector matrix elements, and the four frame factors, give
$\|V\|_{\overline B_{\theta,t}}\le C\|V\|_{\partial B_{\theta,t}}$.
The boundary circle lies on $\partial B_R$, so the first defect
bound follows. Where the construction has $H=I$ on an open
neighborhood, the defect is already zero.

For each $x$, Lemma~\ref{lem:w3-projection-estimate} and the fixed
embedding $W^{1,p}(\mathbb{CP}^1)\subset C^0$, $p>2$, imply
$\|H-H_0\|_{C_x^0C_\zeta^0}\le C\omega^{1/2}$.
Since $\Pi$ commutes with $\partial_x$, the independent spatial
prior gives $\|H-H_0\|_{C_x^mC_\zeta^0}\le C M_H$.
On the fixed ball, a simultaneous bounded extension in $C^0$ and
$C^m$, smoothing, and optimization of the smoothing scale give
\begin{equation}\label{eq:r2-fixed-domain-interpolation}
\|f\|_{C^\ell}\le C\left(
     \|f\|_{C^0}^{1-\ell/m}\|f\|_{C^m}^{\ell/m}
     +\|f\|_{C^0}\right),\qquad0\le\ell<m.
\end{equation}
One may obtain the extension in boundary charts by a fixed finite
linear combination of reflected values, matching derivatives up to
$m$, followed by fixed cutoffs. Its coefficients and charts are
independent of $f$; no extra derivative of $f$ is used. Apply this
inequality to each local matrix entry uniformly in direction. This
proves the spatial exponent in \eqref{eq:w3-defect-output} without
ever estimating $\partial_x^m\bar\partial_\zeta H$.
Finally $H_0-I=\Pi(H-I)$ on the sphere gives the boundary estimate
without a square-root loss.
\end{proof}

\begin{remark}
For the construction in Proposition~\ref{prop:w3-gluing},
the full-disk frame bounds are required only for $|t|\le R_0$;
on the remaining open region $H=I$.
Taking three spatial derivatives of $H-H_0$ gives a positive
error power when $m>3$ and a small $C_x^2$ transport residual,
as quantified in Subsection~\ref{sec:r2-projection-rigidity}.
\end{remark}

\subsection{The projection-to-rigidity estimate}\label{sec:r2-projection-rigidity}

The projection estimate is converted into a transport residual in the
$C_x^2$ norm required by material rigidity.

Use $\beta_m$ from \eqref{eq:r2-common-error-definitions} and apply
Proposition~\ref{prop:w3-defect} to the endomorphism $H=G^{-T}$ constructed in
Proposition~\ref{prop:w3-gluing}, with
$\omega=C_H\omega_{\rm F}\le1$ and a fixed high-order bound $M_H$.
For all spatial matrix norms use the maximum supremum
Frobenius norm over the same compact spinor charts and the displayed
spatial multi-indices, as in Theorem~\ref{thm:w4-rigidity}.
Fix $C_\Pi\ge1$ bounding $\Pi:C_\zeta^0\to C_\zeta^0$ in these
norms. If $C_{m,3}$ denotes the corresponding constant in
\eqref{eq:w3-defect-output}, one may choose
$C_P=C_{m,3}C_H^{\beta_m}$. Then
\begin{equation}\label{eq:w3-C3-output}
 \|H-H_0\|_{C_x^3C_\zeta^0}\le C_P\omega_{\rm F}^{\beta_m},\qquad
 \|H_0-I\|_{C^0(\partial B_R\times\mathbb{CP}^1)}
                  \le C_\Pi C_H\omega_{\rm F}.
\end{equation}
The projected object is a holomorphic endomorphism with the block form
of Lemma~\ref{lem:w3-holomorphic-form}. Projection alone is not used
to claim that it is an automorphism.

Define the projection error by
\begin{equation}\label{eq:r2-projection-error}
 E_0=H_0-H,
\end{equation}
where $H_0$ is defined in \eqref{eq:r2-transport-projection-definitions}.
By \eqref{eq:w3-H-exact},
\begin{equation}\label{eq:r2-projected-transport-residual}
\mathcal T(H_0)=2\theta\cdot\nabla_x E_0+\Gamma_1E_0-E_0\Gamma_2.
\end{equation}
For every $|\alpha|\le2$, differentiation gives
\begin{equation}\label{eq:w3-residual-Leibniz}
 \begin{split}
 \partial_x^\alpha\mathcal T(H_0)
 &=2\theta\cdot\nabla_x\partial_x^\alpha E_0\\
 &\quad+\sum_{\sigma\le\alpha}\binom\alpha\sigma
   [ (\partial_x^\sigma\Gamma_1)
                  (\partial_x^{\alpha-\sigma}E_0)
     -(\partial_x^\sigma E_0)
                  (\partial_x^{\alpha-\sigma}\Gamma_2)].
 \end{split}
\end{equation}
Thus only three spatial derivatives of $E_0$ and two of the connections
are needed. Set
\begin{equation}\label{eq:r2-transport-norm-constants}
\Theta_1=\max_{\text{charts},\zeta}\sum_{i=1}^3|\theta_i(\zeta)|,
 \qquad B_{\Gamma,j}\ge\|\Gamma_j\|_{C_x^2C_\zeta^0},\qquad
 C_T=2\Theta_1+4(B_{\Gamma,1}+B_{\Gamma,2}).
\end{equation}
Choose $B_{\Gamma,j}$ uniformly over the full material class. The factor $4$
is $\max_{|\alpha|\le2}\sum_{\sigma\le\alpha}\binom\alpha\sigma$;
submultiplicativity of the Frobenius norm in
\eqref{eq:w3-residual-Leibniz} therefore gives
$\|\mathcal T(E_0)\|_{C_x^2C_\zeta^0}
 \le C_T\|E_0\|_{C_x^3C_\zeta^0}$.
In particular all constants can be fixed before the error by
\begin{equation}\label{eq:w3-link-constants}
 \begin{split}
 B_H&=C_\Pi M_H,\qquad
 C_{\rm link}=\max\{1,C_TC_P+C_\Pi C_H\}.
 \end{split}
\end{equation}
Here $\|H_0\|_{C_x^3C_\zeta^0}\le B_H$, since $m\ge4$ and
$\Pi$ commutes with the spatial derivatives. These explicit choices give
\begin{equation}\label{eq:w3-rigidity-output}
 \|\mathcal T(H_0)\|_{C_x^2C_\zeta^0}
 +\|H_0-I\|_{C^0(\partial B_R\times\mathbb{CP}^1)}
 \le C_{\rm link}\omega_{\rm F}^{\beta_m}
 \le C_{\rm link} C_{\rm F}^{\beta_m}\eta_D^{\gamma'(m-3)/(28m)}.
\end{equation}
Here $\omega_{\rm F}\le1$ was used to absorb its linear contribution.
With the rigidity threshold for this fixed $B_H$, one sufficient
simultaneous choice is
\begin{equation}\label{eq:w3-joint-threshold}
 \begin{split}
 \omega_{\rm all}&=\min\left\{1,\omega_{\rm C},\omega_{\rm G},C_H^{-1},
       \left(\frac{\varepsilon_{\rm rig,0}}{2C_{\rm link}}\right)^{1/\beta_m}
                         \right\},\\
 \eta_{\rm all}&=\min\left\{\eta_{\rm mom},1,
                \left(\frac{\omega_{\rm all}}{2C_{\rm F}}\right)^{14/\gamma'}
                        \right\},\qquad 0<\eta_D<\eta_{\rm all},\quad0<h<h_*.
 \end{split}
\end{equation}
Then $\omega_{\rm F}<\omega_{\rm all}$ and the sum of the residual
and sphere errors is less than $\varepsilon_{\rm rig,0}/2$.
Theorem~\ref{thm:w4-rigidity} therefore applies. The shell condition is
imposed on the materials; the projected endomorphism $H_0$ is normalized
by the sphere estimate in \eqref{eq:w3-C3-output}.

\section{Quantitative material rigidity and common-exterior stability}
\label{sec:material-rigidity}
The holomorphic projection reduces coefficient recovery to a finite
system of block equations. We first establish the scalar coercive
estimate used in those equations, then prove material rigidity and
complete the common-exterior stability argument.

\subsection{Uniform coercivity for the trace-free Hessian}
\label{sec:w4-hessian}

A fixed shell normalization gives a coercive estimate for the trace-free
Hessian with bounded first-order drift. This estimate controls the
logarithmic shear-modulus ratio in the block recovery argument.

Fix $0<R_0<R$ and set
\begin{equation}\label{eq:w4-shell-space}
 \begin{gathered}
 E=\{x\in\mathbb R^3:R_0<|x|<R\},\\
 \mathcal V_{R_0}=\{v\in H^2(B_R;\mathbb R):v=0\text{ a.e. on }E\}.
 \end{gathered}
\end{equation}
The shell and both radii are fixed before varying coefficients or errors.
For a real vector field $\beta\in W^{1,\infty}(B_R;\mathbb R^3)$ define
\begin{equation}\label{eq:w4-Pbeta}
 \mathcal P_\beta v
 =\operatorname{tf}(\nabla^2v)
  -\operatorname{tf}(\beta\otimes\nabla v+\nabla v\otimes\beta),
 \qquad \operatorname{tf}K=K-\tfrac13(\operatorname{tr}K)I_3.
\end{equation}
The tensor convention is \eqref{eq:r2-tensor-conventions}; matrix
norms are Frobenius norms. Theorem~\ref{thm:w4-hessian} is independent
of the transport construction.

The shell normalization removes the homogeneous kernel, giving coercivity
uniformly over the bounded drift class.

\begin{theorem}\label{thm:w4-hessian}
For every finite $B_\beta$ there is a constant
$C=C(R,R_0,B_\beta)$ such that, for every real
$\beta\in W^{1,\infty}(B_R;\mathbb R^3)$ with
$\|\beta\|_{W^{1,\infty}}\le B_\beta$ and every
$v\in\mathcal V_{R_0}$,
\begin{equation}\label{eq:w4-hessian-estimate}
 \|v\|_{H^2(B_R)}+\|v\|_{L^\infty(B_R)}
 \le C\|\mathcal P_\beta v\|_{L^2(B_R)}.
\end{equation}
Here $\mathcal V_{R_0}$ and $\mathcal P_\beta$ are defined in
\eqref{eq:w4-shell-space} and \eqref{eq:w4-Pbeta}.
\end{theorem}

\begin{proof}
\emph{The principal estimate.}
A function in $\mathcal V_{R_0}$ vanishes in an open collar of
$\partial B_R$. Its zero extension $\widetilde v$ therefore belongs to
$H^2(\mathbb R^3)$, has support in $\overline B_{R_0}$, and has no
interface distribution at $\partial B_R$. For example, multiply by a
smooth cutoff equal to one on $\overline B_{R_0}$ and compactly
supported in $B_R$ before extending. For real $\xi$,
\begin{equation}\label{eq:r2-calculation-018}
\big|\xi\otimes\xi-\tfrac13|\xi|^2I_3\big|^2
 =\tfrac23|\xi|^4.
\end{equation}
Plancherel consequently gives the exact identity
\begin{equation}\label{eq:w4-principal-identity}
 \|\operatorname{tf}(\nabla^2\widetilde v)\|_2^2
 =\tfrac23\|\nabla^2\widetilde v\|_2^2
 =\tfrac23\|\Delta\widetilde v\|_2^2.
\end{equation}
The tensor term in \eqref{eq:w4-Pbeta} has norm at most
$2B_\beta|\nabla v|$. Also, for every $\eta>0$, the Fourier
inequality $|\xi|^2\le\eta^2|\xi|^4+(4\eta^2)^{-1}$ yields
\begin{equation}\label{eq:r2-calculation-019}
\|\nabla v\|_2\le\eta\|\nabla^2v\|_2+(2\eta)^{-1}\|v\|_2.
\end{equation}
Choose $\eta$ using the fixed number $B_\beta$, not by restricting
$B_\beta$. Absorption in \eqref{eq:w4-principal-identity} gives
\begin{equation}\label{eq:w4-garding}
 \|v\|_{H^2(B_R)}
 \le C(R,B_\beta)
       \big(\|\mathcal P_\beta v\|_{L^2(B_R)}+\|v\|_{L^2(B_R)}\big).
\end{equation}

\emph{The homogeneous weak kernel.}
Suppose $v\in\mathcal V_{R_0}$ and $\mathcal P_\beta v=0$, with
$\mathcal V_{R_0}$ and $\mathcal P_\beta$ defined in
\eqref{eq:w4-shell-space} and \eqref{eq:w4-Pbeta}. Set
\begin{equation}\label{eq:r2-hessian-kernel-fields}
 r=\nabla v\in H^1(B_R;\mathbb R^3),\qquad
 w=\Delta v\in L^2(B_R).
\end{equation}
Define
\begin{equation}\label{eq:w4-T}
 T_{ij}^{\ k}=\beta_i\delta_{jk}+\beta_j\delta_{ik}
                         -\tfrac23\delta_{ij}\beta_k.
\end{equation}
Repeated component indices run from one to three. Restoring the trace in the weak equation, we obtain, in the sense of
distributions,
\begin{equation}\label{eq:w4-r-system}
 \partial_jr_i=\tfrac13\delta_{ij}w+T_{ij}^{\ k}r_k.
\end{equation}
Apply $\partial_i$ and sum in $i$. Distributional derivatives of $v$
commute, and the product rule is valid for $T\in W^{1,\infty}$ and
$r\in H^1$. Thus
\begin{equation}\label{eq:w4-first-w}
 \tfrac23\partial_jw
 =\sum_{i,k}\big[(\partial_iT_{ij}^{\ k})r_k
                          +T_{ij}^{\ k}\partial_i r_k\big].
\end{equation}
The right side already lies in $L^2$; this proves $w\in H^1$
\emph{before} substituting \eqref{eq:w4-r-system} again. The substitution
then gives
\begin{equation}\label{eq:w4-w-system}
 \partial_jw=\tfrac53\beta_jw+B_{j\ell}r_\ell,
 \quad
 B_{j\ell}=\tfrac32\left(
       \sum_i\partial_iT_{ij}^{\ \ell}
       +\sum_{i,k}T_{ij}^{\ k}T_{ki}^{\ \ell}\right).
\end{equation}
For clarity, all contractions used here are
\begin{equation}\label{eq:w4-contractions}
 \begin{aligned}
 \sum_iT_{ij}^{\ i}&=\tfrac{10}{3}\beta_j,\\
 \sum_i\partial_iT_{ij}^{\ \ell}
   &=(\operatorname{div}\beta)\delta_{j\ell}
       +\partial_\ell\beta_j-\tfrac23\partial_j\beta_\ell,\\
 \sum_{i,k}T_{ij}^{\ k}T_{ki}^{\ \ell}
   &=\tfrac13|\beta|^2\delta_{j\ell}+\tfrac19\beta_j\beta_\ell.
 \end{aligned}
\end{equation}
The last identity follows by inserting \eqref{eq:w4-T}: the trace
$\sum_iT_{ii}^{\ \ell}$ vanishes and the remaining two sums combine
into $\frac13\sum_i\beta_iT_{ij}^{\ \ell}$.
Consequently,
\begin{equation}\label{eq:r2-hessian-first-order-field}
 U=(r_1,r_2,r_3,w)^T\in H^1(B_R;\mathbb R^4),
\end{equation}
and it satisfies
\begin{equation}\label{eq:w4-closed-system}
 \begin{gathered}
 \partial_jU=\mathcal C_jU\quad\text{a.e.},\qquad
 \|\mathcal C_j\|_\infty\le C(1+B_\beta+B_\beta^2),\\
 \mathcal C_j=
 \begin{pmatrix}
 (T_{ij}^{\ k})_{i,k=1}^3 &(\delta_{ij}/3)_{i=1}^3\\
 (B_{jk})_{k=1}^3 &5\beta_j/3
 \end{pmatrix}.
 \end{gathered}
\end{equation}
Equivalently, the only lower-row coefficient not already explicit is
\begin{equation}\label{eq:w4-B-expanded}
 B_{j\ell}=\frac32\left[
 (\operatorname{div}\beta)\delta_{j\ell}
 +\partial_\ell\beta_j-\frac23\partial_j\beta_\ell
 +\frac13|\beta|^2\delta_{j\ell}+\frac19\beta_j\beta_\ell\right].
\end{equation}
Only first weak derivatives of $\beta$ occur; they need not be continuous.

On the three-dimensional ball, $H^1\hookrightarrow L^6$.
Equation \eqref{eq:w4-closed-system} first implies
$U\in W^{1,6}$ and then, by $W^{1,6}\hookrightarrow C^0\cap L^\infty$,
$U\in W^{1,\infty}$. The Euclidean vector norm and the bounded coefficient matrices give
$|\nabla U|\le C_*|U|$ a.e., where
\begin{equation}\label{eq:r2-kernel-log-transform}
 C_*=(\textstyle\sum_j\|\mathcal C_j\|_\infty^2)^{1/2},\qquad
 q_a=\log(a+|U|),\quad a>0.
\end{equation}
The constant is fixed by the drift class. Apply the Sobolev chain rule to $q_a$. Its Lipschitz representative satisfies
\begin{equation}\label{eq:r2-calculation-020}
|\nabla q_a|\le C(1+B_\beta+B_\beta^2)\quad\text{a.e.},
\end{equation}
with the same constant for all $a$. Convexity of $B_R$ therefore gives
\begin{equation}\label{eq:w4-log-propagation}
 a+|U(x)|\le
 \exp\!\big(C(1+B_\beta+B_\beta^2)|x-x_0|\big)(a+|U(x_0)|).
\end{equation}
Since $v=0$ on the open shell, its weak derivatives $r,w$ vanish
there a.e., and the continuous representative of $U$ vanishes at every
point of $E$. Choose $x_0\in E$ in \eqref{eq:w4-log-propagation} and
let $a\downarrow0$. Then $U=0$, and $v=0$ follows from
$\nabla v=0$ and the shell normalization. This argument does not
restrict an $L^\infty$ coefficient matrix to an arbitrary path.

\emph{Uniformity over the Lipschitz closure.}
If the $H^2$ estimate in \eqref{eq:w4-hessian-estimate} failed uniformly,
there would exist $\beta_n$ and $v_n\in\mathcal V_{R_0}$ such that
\begin{equation}\label{eq:r2-calculation-021}
\|\beta_n\|_{W^{1,\infty}}\le B_\beta,\qquad
 \|v_n\|_{H^2}=1,\qquad
 \|\mathcal P_{\beta_n}v_n\|_2\longrightarrow0.
\end{equation}
The Lipschitz representatives on $\overline B_R$ are equibounded and
equicontinuous. Arzel\`a--Ascoli, weak-star compactness, and Rellich
compactness give, after passage to a subsequence,
\begin{equation}\label{eq:r2-calculation-022}
\begin{gathered}
 \beta_n\to\beta\text{ uniformly},\qquad
 \beta\in W^{1,\infty},\quad\|\beta\|_{W^{1,\infty}}\le B_\beta,\\
 v_n\rightharpoonup v\text{ in }H^2,\qquad v_n\to v\text{ in }H^1.
 \end{gathered}
\end{equation}
In particular $\beta_n\otimes\nabla v_n\to\beta\otimes\nabla v$
in $L^2$. More explicitly, the operator difference obeys
\begin{equation}\label{eq:w4-closure-residual}
 \| (\mathcal P_{\beta_n}-\mathcal P_\beta)v_n\|_2
 \le 2\|\beta_n-\beta\|_\infty\|\nabla v_n\|_2
 \longrightarrow0.
\end{equation}
Thus $\mathcal P_\beta v_n\to0$ strongly in $L^2$, while its
principal part converges weakly to that of $\mathcal P_\beta v$.
Passing to the limit against smooth compactly supported matrix fields, we
obtain $\mathcal P_\beta v=0$ in distributions. Strong $L^2$ convergence also gives
$\|v\|_{L^2(E)}=0$. The weak-kernel argument therefore applies to
this same Lipschitz limit and gives $v=0$. Notice that weak-star
convergence of $\nabla\beta_n$ suffices to retain the drift bound;
no convergence of the derived matrices $B_{j\ell}$ is needed to
pass to the second-order equation.
The uniform estimate \eqref{eq:w4-garding}, applied to $v_n$, now
contradicts $\|v_n\|_{H^2}=1$. Finally,
$H^2(B_R)\hookrightarrow L^\infty(B_R)$ proves the stated estimate.
\end{proof}

\begin{remark}
The estimate holds for every finite bound $B_\beta$; no smallness
condition is imposed on the drift.
The finite-type system \eqref{eq:w4-closed-system} is used only for a
\emph{homogeneous} weak kernel. For an inhomogeneous equation
$\mathcal P_\beta v=F\in L^2$, the proof of
\eqref{eq:w4-hessian-estimate} never differentiates $F$.
It does not assert $L^\infty$ estimates for $\nabla v$ or $\Delta v$.
\end{remark}

\subsection{Quantitative recovery from the block equations}\label{sec:r2-block-recovery}
\label{sec:w4-blocks}

The transport residual of the holomorphic endomorphism $H_0$ controls the
Lam\'e parameters through the connection blocks. The scalar and matrix
blocks are estimated in order before applying the trace-free Hessian
estimate.

Use the tensor convention \eqref{eq:r2-tensor-conventions} and the
mixed norms \eqref{eq:r2-mixed-chart-norm}, on the same compact
holomorphic charts as Proposition~\ref{prop:w3-gluing}. All spatial
derivatives hold the spinor representative fixed.

\subsubsection{Material priors and denominator bounds}\label{sec:r2-rigidity-priors}

Fix the ball and shell in \eqref{eq:w4-shell-space}. Let both real
material pairs satisfy, for $j=1,2$,
\begin{equation}\label{eq:w4-material-priors}
 \|\lambda_j\|_{C^4(\overline B_R)}+\|\mu_j\|_{C^4(\overline B_R)}
 \le M_4,\qquad \mu_j\ge c_0>0,\qquad 3\lambda_j+2\mu_j\ge c_0,
\end{equation}
and suppose the pairs agree on $E$ from \eqref{eq:w4-shell-space}.
For each material use $\kappa_j,c_j,d_j,g_{\mu,j}$ from
\eqref{eq:w1-material-variables}, $F_{\mu_j}$ from
\eqref{eq:w1-reduced-coefficients}, and $\rho$ from
\eqref{eq:r2-material-log-ratio}.
Thus $F_{\mu_j}^T=F_{\mu_j}$, $\rho$ is real, and $\rho=0$ on $E$.
Continuity supplies the same material equalities on $\partial B_R$;
no identity for the transport throughout $E$ is assumed.
Define fixed positive numbers
\begin{equation}\label{eq:r2-rigidity-denominator-constants}
\kappa_-:=\tfrac14+\frac{c_0}{12M_4},\qquad
 \nu_-:=\frac{c_0}{3M_4}.
\end{equation}
The following nonoptimal bounds make every material division explicit:
\begin{equation}\label{eq:w4-denominators}
 \begin{gathered}
 \lambda_j+2\mu_j\ge 5c_0/3,\qquad
 \lambda_j+\mu_j\ge 2c_0/3,\qquad
 \lambda_j+2\mu_j\le 3M_4,\\
 \kappa_j-\tfrac14
    =\frac{3\lambda_j+2\mu_j}{4(\lambda_j+2\mu_j)}
       \ge\frac{c_0}{12M_4},\qquad
 1-\kappa_j=\frac{\mu_j}{\lambda_j+2\mu_j}\ge\nu_-,\\
 d_j\ge\kappa_-\sqrt{c_0},\qquad
 M_4^{-1/2}\le c_j\le c_0^{-1/2},\qquad
 2c_1(\kappa_1^{-1}-1)\ge 2\nu_-M_4^{-1/2}.
 \end{gathered}
\end{equation}
In particular $\kappa_-\le\kappa_j\le1-\nu_-$.
The chain and product rules give fixed finite $C^4$ bounds for the
scalar material functions and their reciprocals whenever used;
$F_{\mu_j}$ is uniformly $C^2$.

\subsubsection{Connection blocks and holomorphic endomorphisms}\label{sec:r2-rigidity-blocks}

Use $\theta,Q,n_\theta,\mathfrak r_a$ from \eqref{eq:w1-spinor}
and $D_\theta$ from \eqref{eq:r2-phase-parameters}. The rank calculation
following \eqref{eq:w1-spinor} gives $\ker Q^T=\mathbb C\theta$.
For the connection \eqref{eq:w1-quotient-connection}, introduce the block notation
\begin{equation}\label{eq:w4-connection}
 \Gamma_j=\begin{pmatrix}\mathsf L_j&\mathsf N_j\\\mathsf P_j&0\end{pmatrix},
 \qquad
 \mathsf L_j=n_\theta\mathfrak r_{g_{\mu,j}},\quad
 \mathsf N_j=2d_jn_\theta,\quad \mathsf P_j=-\theta^TF_{\mu_j}Q.
\end{equation}
These three blocks have sizes $2\times2$, $2\times1$, and $1\times2$,
respectively. The reduced $4\times4$ operator $\mathcal M_j$ and the leaf
connection $P_j$ are not renamed as any of these blocks.

Let $H_0(x,\cdot)$ be a holomorphic endomorphism of the bundle
\eqref{eq:w3-fixed-bundle}. By Lemma~\ref{lem:w3-holomorphic-form},
its block form is \eqref{eq:w3-nine-dimensional}. Abbreviate its lower row by
\begin{equation}\label{eq:w4-Hform}
 H_0=\begin{pmatrix}A_H&0\\ \mathsf C_H&d_H\end{pmatrix},\qquad
 \mathsf C_H=s_0\mathsf C_{H,0}+s_1\mathsf C_{H,1}.
\end{equation}
Here $A_H,d_H,\mathsf C_{H,0},\mathsf C_{H,1}$ depend only on $x$ and may be complex.
Their $C_x^3$ bounds follow from fixed chart evaluations of $H_0$.
Theorem~\ref{thm:w4-rigidity} does not require $H_0$ to be invertible.

The following algebraic identity resolves the null contractions; scalar
matrices form the undetected subspace.

\begin{lemma}\label{lem:w4-six-directions}
Set
\begin{equation}\label{eq:r3-six-direction-constant}
a=(1-\mathrm i)/2.
\end{equation}
Choose the following six normalized spinors:
\begin{equation}\label{eq:w4-six-spinors}
 \begin{gathered}
 (1,0),\quad (0,\mathrm i),\quad
 (2^{-1/2},\mathrm i2^{-1/2}),\\
 (2^{-1/2},-\mathrm i2^{-1/2}),\quad (a,-a),\quad (a,a).
 \end{gathered}
\end{equation}
Their vectors $\theta_\ell$ are, in order,
$e_1\pm\mathrm i e_2$, $e_1\pm\mathrm i e_3$, and
$e_2\pm\mathrm i e_3$. Let $Q_\ell$ be their matrices in
\eqref{eq:w1-spinor}. For every $K\in \mathbb C^{3\times3}$,
\begin{equation}\label{eq:w4-six-identity}
 \sum_{\ell=1}^6|Q_\ell^TK\theta_\ell|^2
 =\tfrac32|\operatorname{tf}(\operatorname{sym}K)|^2
       +\tfrac12|\operatorname{skew}K|^2.
\end{equation}
Here $\operatorname{tf}$ is defined in \eqref{eq:w4-Pbeta}, and
\begin{equation}\label{eq:r2-symmetric-skew-parts}
 \operatorname{sym}K=(K+K^T)/2,\qquad
 \operatorname{skew}K=(K-K^T)/2.
\end{equation}
In particular,
\begin{equation}\label{eq:w4-six-bound}
 |\operatorname{tf}K|\le
 \sqrt2\left(\sum_{\ell=1}^6|Q_\ell^TK\theta_\ell|^2\right)^{1/2}.
\end{equation}
The kernel of this linear map is exactly $\mathbb C I_3$.
\end{lemma}
\begin{proof}
For $i<j$ and the remaining index $k$, substitution of the two
spinors with $\theta=e_i\pm\mathrm i e_j$ gives
\begin{equation}\label{eq:r2-calculation-024}
\sum_{\pm}|Q_{ij,\pm}^TK(e_i\pm\mathrm i e_j)|^2
 =\tfrac12\big(|K_{ii}-K_{jj}|^2+|K_{ij}+K_{ji}|^2
                          +|K_{ki}|^2+|K_{kj}|^2\big).
\end{equation}
Sum over the three pairs. The diagonal contribution is
$\sum_{i<j}|K_{ii}-K_{jj}|^2
 =3\sum_i|K_{ii}-\operatorname{tr}K/3|^2$.
The other terms split into the symmetric and skew-symmetric
parts of the off-diagonal entries, which proves
\eqref{eq:w4-six-identity}. These two parts are orthogonal for the
Hermitian Frobenius norm, although their definition uses ordinary
transpose. \end{proof}

We use the four transport blocks to derive the material equations
and then apply the scalar coercive estimate to recover the coefficients.

\begin{theorem}\label{thm:w4-rigidity}
Assume \eqref{eq:w4-material-priors}, agreement of the two materials on
$E$, and the fixed compact direction charts used in \eqref{eq:r2-mixed-chart-norm}. Let $H_0$ be as in
\eqref{eq:w4-Hform}, with $\|H_0\|_{C_x^3C_\zeta^0}\le B_H$.
Use the transport operator $\mathcal T$ defined in
\eqref{eq:r2-transport-projection-definitions}.
There are fixed $C<\infty$ and $\varepsilon_{\rm rig,0}>0$,
depending only on $R,R_0,M_4,c_0,B_H$ and the chosen charts, such that
\begin{equation}\label{eq:w4-rigidity-input}
 \|\mathcal T(H_0)\|_{C_x^2C_\zeta^0}
 +\|H_0-I\|_{C^0(\partial B_R\times\mathbb{CP}^1)}
 \le\varepsilon_{\rm rig}<\varepsilon_{\rm rig,0}
\end{equation}
implies
\begin{equation}\label{eq:w4-rigidity-output}
 \begin{split}
 &\|\kappa_1-\kappa_2\|_{C^2(\overline B_R)}
  +\|\rho\|_{H^2(B_R)}
  +\|\lambda_1-\lambda_2\|_{L^\infty(B_R)}
  +\|\mu_1-\mu_2\|_{L^\infty(B_R)}\\
 &\hspace{45mm}\le C\varepsilon_{\rm rig}.
 \end{split}
\end{equation}
Here $\kappa_j$ and $\rho$ are defined in
\eqref{eq:w1-material-variables} and \eqref{eq:r2-material-log-ratio};
$\varepsilon_{\rm rig}$ is the error bound in \eqref{eq:w4-rigidity-input}.
\end{theorem}

\begin{proof}
Write the four blocks of $\mathcal T(H_0)$ as
$R_{\rm ul},R_{\rm ur},R_{\rm ll},R_{\rm lr}$. Ordinary matrix
multiplication gives
\begin{equation}\label{eq:w4-four-blocks}
 \begin{aligned}
 2D_\theta A_H+\mathsf L_1A_H+\mathsf N_1\mathsf C_H-A_H\mathsf L_2&=R_{\rm ul},\\
 \mathsf N_1d_H-A_H\mathsf N_2&=R_{\rm ur},\\
 2D_\theta\mathsf C_H+\mathsf P_1A_H-\mathsf C_H\mathsf L_2-d_H\mathsf P_2&=R_{\rm ll},\\
 2D_\theta d_H-\mathsf C_H\mathsf N_2&=R_{\rm lr}.
 \end{aligned}
\end{equation}
Write $O_q(\varepsilon_{\rm rig})$ only for a specified error
bounded in $C_x^qC_\zeta^0$ by a fixed multiple of
$\varepsilon_{\rm rig}$. Under spinor rescaling by $b\ne0$,
$Q,n_\theta,\mathsf C_H$ have degree one, $\theta$ degree two, and $A_H,d_H$ degree
zero. The four equations have degrees two, one, three, and two.
Thus fixed representative changes at the finitely many directions
below alter error constants only, not their spatial derivatives.

\emph{1. Upper right: no spatial derivative is lost.}
Set
\begin{equation}\label{eq:w4-a-EA}
 a_H=\frac{d_1d_H}{d_2},\qquad \mathcal E_A=A_H-a_HI_2.
\end{equation}
The upper-right equation is
$-2d_2\mathcal E_A n_\theta=R_{\rm ur}$.
The spinors $(1,0)$ and $(0,1)$ give $n_\theta=(0,1)^T$ and
$n_\theta=(-1,0)^T$, so they determine both columns of $\mathcal E_A$.
Let $e_1,e_2$ be the coordinate columns of $\mathbb C^2$. Evaluation at the
spinors $(1,0)$ and $(0,1)$ gives
\begin{equation}\label{eq:w4-column-extraction}
 \mathcal E_A e_1=\frac{R_{\rm ur}(0,1)}{2d_2},\qquad
 \mathcal E_A e_2=-\frac{R_{\rm ur}(1,0)}{2d_2}.
\end{equation}
The material bounds for $d_2^{-1}$ and its first two derivatives give
\begin{equation}\label{eq:w4-EA-estimate}
 \|\mathcal E_A\|_{C_x^2}\le C\varepsilon_{\rm rig}.
\end{equation}

\emph{2. Upper left: one derivative is spent, then the scalar is normalized.}
Put
\begin{equation}\label{eq:r2-upper-left-error-source}
\mathcal Q_{\rm ul}=R_{\rm ul}-2D_\theta\mathcal E_A
                 -\mathsf L_1\mathcal E_A+\mathcal E_A\mathsf L_2.
\end{equation}
Then $\mathcal Q_{\rm ul}=O_1(\varepsilon_{\rm rig})$, and
\begin{equation}\label{eq:w4-a-equation}
 2(D_\theta a_H)I_2
     +n_\theta\big[a_H\mathfrak r_{\nabla\rho}+2d_1\mathsf C_H\big]
 =\mathcal Q_{\rm ul}.
\end{equation}
For a fixed nonzero spinor define
\begin{equation}\label{eq:r2-spinor-row-and-column}
 \ell=(s_0,s_1),\qquad
 v=(\overline{s_0},\overline{s_1})^T/(|s_0|^2+|s_1|^2).
\end{equation}
Then $\ell n_\theta=0$ and $\ell v=1$, whence
$2D_\theta a_H=\ell\mathcal Q_{\rm ul}v$.
The six directions in Lemma~\ref{lem:w4-six-directions} span
$\mathbb C^3$: sums and differences of each plus/minus pair recover
the coordinate derivatives. Apply this observation also after one
spatial derivative. It gives
$\|\nabla a_H\|_{C^1}\le C\varepsilon_{\rm rig}$.
At $\partial B_R$, $d_1=d_2$ and $d_H=1+O_0(\varepsilon_{\rm rig})$,
so $a_H=1+O_0(\varepsilon_{\rm rig})$ there.
Integration along segments in the closed ball, of length at most $2R$,
therefore gives
\begin{equation}\label{eq:w4-a-bound}
 \|a_H-1\|_{C^2}\le C\varepsilon_{\rm rig}.
\end{equation}
Only the zero-order boundary value, not a boundary derivative, was used.
Define the real vector field
\begin{equation}\label{eq:w4-brho}
 b_\rho=\frac{\nabla\rho}{d_1}.
\end{equation}
Since $\mathfrak r_{\nabla\rho}=-2(\nabla\rho)^TQ$, extraction of the
remaining outer product in \eqref{eq:w4-a-equation} gives
\begin{equation}\label{eq:w4-EC}
 \mathsf C_H=a_H b_\rho^TQ+\mathcal E_C,\qquad
 \mathcal E_C=O_1(\varepsilon_{\rm rig}).
\end{equation}
For this extraction use $n_\theta^*/|n_\theta|^2$ and then divide by $2d_1$.
The star appears only in this bounded auxiliary left inverse; the
transport identities still use ordinary transpose. All spinor
normalizing denominators are bounded below on the fixed charts.

\emph{3. Lower right: control the integrating factor and recover $\kappa$.}
Use $Qn_\theta=-\theta/2$ and the exact scalar identity
$a_Hd_2/d_1=d_H$ in the lower-right equation to obtain
\begin{equation}\label{eq:w4-dH-equation}
 2D_\theta d_H+d_HD_\theta\rho
       =R_{\rm lr}+2d_2\mathcal E_C n_\theta
       =O_1(\varepsilon_{\rm rig}).
\end{equation}
Define
\begin{equation}\label{eq:r2-normalized-scalar-block}
 b_H=d_He^{\rho/2},
\end{equation}
where $\rho$ is given by \eqref{eq:r2-material-log-ratio}.
Multiplication by $e^{\rho/2}/2$ gives
$D_\theta b_H=O_1(\varepsilon_{\rm rig})$.
The material prior already bounds $\rho$ and its needed derivatives;
smallness of $\rho$ is not being used. Since the chosen null directions span $\mathbb C^3$, the bounds for $D_\theta
b_H$ control $\nabla b_H$ in $C^1$. Together with $b_H=1+O_0(\varepsilon_{\rm
rig})$ on the sphere, these bounds give a fixed constant $C_b\ge1$ such that
\begin{equation}\label{eq:w4-bH-bound}
 \|b_H-1\|_{C^2}\le C_b\varepsilon_{\rm rig},\qquad
 \|d_H-e^{-\rho/2}\|_{C^2}\le C\varepsilon_{\rm rig}.
\end{equation}
Choose, once and for all,
\begin{equation}\label{eq:w4-rigidity-threshold}
 \varepsilon_{\rm rig,0}:=\min\{1,(2C_b)^{-1}\}.
\end{equation}
Then $|b_H|\ge1/2$, including when $H_0$ has complex entries.
Only now is the new denominator $b_H$ introduced. Its first and
second reciprocal derivatives are bounded, as follows from
$\partial b_H^{-1}=-b_H^{-2}\partial b_H$ and
$\partial_i\partial_jb_H^{-1}
 =2b_H^{-3}\partial_i b_H\partial_j b_H
       -b_H^{-2}\partial_i\partial_j b_H$.
The exact material identities
\begin{equation}\label{eq:w4-kappa-recovery}
 a_H=\frac{\kappa_1}{\kappa_2}b_H,\qquad
 \kappa_1-\kappa_2=\frac{\kappa_2}{b_H}(a_H-b_H)
\end{equation}
now imply
$\|\kappa_1-\kappa_2\|_{C^2}\le C\varepsilon_{\rm rig}$.

\emph{4. Lower left: a second derivative is spent, producing only a value error.}
Equations~\eqref{eq:w4-EA-estimate}, \eqref{eq:w4-EC},
\eqref{eq:w4-a-bound}, and \eqref{eq:w4-bH-bound} yield
\begin{equation}\label{eq:r2-normalized-block-errors}
A_H=I_2+\mathcal E_A^0,\qquad
 \mathsf C_H=b_\rho^TQ+\mathcal E_C^0,\qquad
 d_H=e^{-\rho/2}+\mathcal E_d,
\end{equation}
The fields $b_\rho$ and $\rho$ are defined in
\eqref{eq:w4-brho} and \eqref{eq:r2-material-log-ratio}.
Here $\mathcal E_A^0,\mathcal E_d=O_2(\varepsilon_{\rm rig})$ and
$\mathcal E_C^0=O_1(\varepsilon_{\rm rig})$.
Indeed $\mathcal E_C^0=\mathcal E_C+(a_H-1)b_\rho^TQ$ and the
$C^1$ norm of $b_\rho$ was already bounded by the material prior.
Define the real, direction-independent matrix
\begin{equation}\label{eq:w4-Krho}
 K_\rho=2\nabla b_\rho-F_{\mu_1}
               +e^{-\rho/2}F_{\mu_2}-g_{\mu,2}b_\rho^T.
\end{equation}
The lower-left equation gives exactly
\begin{equation}\label{eq:w4-LL}
 \begin{aligned}
 (Q^TK_\rho\theta)^T&=R_{\rm ll}-\mathcal E_{\rm ll},\\
 \mathcal E_{\rm ll}
 &=2D_\theta\mathcal E_C^0+\mathsf P_1\mathcal E_A^0
           -\mathcal E_C^0\mathsf L_2-\mathcal E_d\mathsf P_2
       =O_0(\varepsilon_{\rm rig}).
 \end{aligned}
\end{equation}
The transpose convention gives
\begin{equation}\label{eq:r2-lower-left-transpose-identities}
D_\theta(b_\rho^TQ)=\theta^T(\nabla b_\rho)^TQ,\qquad
 -(b_\rho^TQ)n_\theta\mathfrak r_{g_{\mu,2}}
       =-(b_\rho\cdot\theta)g_{\mu,2}^TQ.
\end{equation}
After transposing the second identity in
\eqref{eq:r2-lower-left-transpose-identities}, we obtain the order
$g_{\mu,2}b_\rho^T$ in \eqref{eq:w4-Krho}. Symmetry of both $F_{\mu_j}$
handles the other two terms. Lemma~\ref{lem:w4-six-directions}
therefore implies
\begin{equation}\label{eq:w4-tfK-bound}
 \|\operatorname{tf}K_\rho\|_{C^0}
       \le C\varepsilon_{\rm rig}.
\end{equation}

For the material expansion define
\begin{equation}\label{eq:r2-material-gradient-combinations}
 r=\nabla\rho,\qquad k_1=\nabla\log\kappa_1,\qquad
 g_\Sigma=g_{\mu,1}+g_{\mu,2}.
\end{equation}
Here $\rho$, $\kappa_j$, and $g_{\mu,j}$ are defined in
\eqref{eq:r2-material-log-ratio} and \eqref{eq:w1-material-variables}.
Since $b_\rho=(c_1/\kappa_1)r$, the gradient convention gives
\begin{equation}\label{eq:w4-grad-brho}
 2\nabla b_\rho=\frac{2c_1}{\kappa_1}\nabla^2\rho
       -\frac{c_1}{\kappa_1}r\otimes g_{\mu,1}
       -\frac{2c_1}{\kappa_1}r\otimes k_1.
\end{equation}
Next $e^{-\rho/2}c_2=c_1$ and
$g_{\mu,1}-g_{\mu,2}=r$ imply
\begin{equation}\label{eq:w4-F-difference}
 \begin{split}
 -F_{\mu_1}+e^{-\rho/2}F_{\mu_2}
 &=2c_1\big[-\nabla^2\rho
       +g_{\mu,1}\otimes g_{\mu,1}-g_{\mu,2}\otimes g_{\mu,2}
       +\Xi_\rho I_3\big],\\
 \Xi_\rho&=\Delta\rho-g_\Sigma\cdot r.
 \end{split}
\end{equation}
Both identities
\begin{equation}\label{eq:r2-calculation-028}
\begin{aligned}
 g_{\mu,1}\otimes g_{\mu,1}-g_{\mu,2}\otimes g_{\mu,2}
   &=\tfrac12(g_\Sigma\otimes r+r\otimes g_\Sigma),\\
 r\otimes g_{\mu,1}+g_{\mu,2}\otimes r
   &=\tfrac12(g_\Sigma\otimes r+r\otimes g_\Sigma)
 \end{aligned}
\end{equation}
follow by writing $g_{\mu,1}=g_{\mu,2}+r$.
Together with $-g_{\mu,2}b_\rho^T
 =-(c_1/\kappa_1)g_{\mu,2}\otimes r$, they give the exact expansion
\begin{equation}\label{eq:w4-K-expansion}
 \begin{split}
 K_\rho/c_1={}&2(\kappa_1^{-1}-1)\nabla^2\rho
   +\left(1-\frac1{2\kappa_1}\right)
          (g_\Sigma\otimes r+r\otimes g_\Sigma)\\
 &-\frac2{\kappa_1}r\otimes k_1+2\Xi_\rho I_3.
 \end{split}
\end{equation}
No $\kappa_2$ appears in this identity. Define
\begin{equation}\label{eq:w4-beta-source}
 \beta=\frac{2k_1-(2\kappa_1-1)g_\Sigma}{4(1-\kappa_1)},\qquad
 E_\rho=\frac{\operatorname{tf}(\operatorname{sym}K_\rho)}
                 {2c_1(\kappa_1^{-1}-1)}.
\end{equation}
The tensors $\operatorname{tf}$ and $\operatorname{sym}$ are defined in
\eqref{eq:w4-Pbeta} and \eqref{eq:r2-symmetric-skew-parts};
$k_1,g_\Sigma$ are defined in \eqref{eq:r2-material-gradient-combinations}.
We apply the symmetric trace-free projection to \eqref{eq:w4-K-expansion}.
With
\begin{equation}\label{eq:r2-symmetric-gradient-tensors}
 S_g=g_\Sigma\otimes r+r\otimes g_\Sigma,\qquad
 S_k=k_1\otimes r+r\otimes k_1,
\end{equation}
the first-order terms after division by $2(\kappa_1^{-1}-1)$ are
\begin{equation}\label{eq:r2-calculation-029}
\frac{2\kappa_1-1}{4(1-\kappa_1)}S_g
       -\frac{1}{2(1-\kappa_1)}S_k
   =-(\beta\otimes r+r\otimes\beta).
\end{equation}
The scalar term $2\Xi_\rho I_3$ vanishes under the trace-free
projection. This proves, without any discarded error,
\begin{equation}\label{eq:w4-Pbeta-rho}
 \mathcal P_\beta\rho=E_\rho,\qquad
 \|E_\rho\|_{C^0}\le C\varepsilon_{\rm rig}.
\end{equation}
The source is real, because it is defined from the real material tensor,
even if the auxiliary residual is complex.

Here is an explicit way to fix the drift class before using coercivity.
Let $K_a,G_a$ ($a=0,1$) bound
$\|\nabla^a k_1\|_\infty$, $\|\nabla^a g_\Sigma\|_\infty$, respectively,
and let $L_1$ bound $\|\nabla\kappa_1\|_\infty$, all over the entire
material class \eqref{eq:w4-material-priors}. They are finite functions
of $M_4,c_0$ by the chain rule. With the convention
$\|\beta\|_{W^{1,\infty}}=\|\beta\|_\infty+\|\nabla\beta\|_\infty$,
a permissible fixed bound is
\begin{equation}\label{eq:w4-beta-bound}
 B_\beta=
 \frac{2K_0+G_0+2K_1+2L_1G_0+G_1}{4\nu_-}
 +\frac{L_1(2K_0+G_0)}{4\nu_-^2}.
\end{equation}
Indeed $|2\kappa_1-1|\le1$ and differentiating the numerator and
$(1-\kappa_1)^{-1}$ proves this bound. In particular $\beta$ is real
and uniformly $W^{1,\infty}$ before any small estimate for $\rho$ is known.

\emph{5. Apply the scalar estimate to the undifferentiated source.}
Since $\rho\in\mathcal V_{R_0}$, Theorem~\ref{thm:w4-hessian} gives
\begin{equation}\label{eq:w4-final-coercive-step}
 \|\rho\|_{H^2}+\|\rho\|_\infty
 \le C\|E_\rho\|_2
 \le C|B_R|^{1/2}\|E_\rho\|_{C^0}
 \le C\varepsilon_{\rm rig}.
\end{equation}
No derivative of $E_\rho$ occurs.
The a priori bound $|\rho|\le\log(M_4/c_0)$ and
$\mu_1-\mu_2=\mu_2(e^\rho-1)$ give
$\|\mu_1-\mu_2\|_\infty\le C\varepsilon_{\rm rig}$.
Finally, the recovery formula for $\lambda$ in \eqref{eq:w1-denominators}
is Lipschitz on
$[c_0,M_4]\times[\kappa_-,1-\nu_-]$: its partial derivatives have
absolute values at most $\nu_-^{-1}$ and $M_4\nu_-^{-2}$.
Combine this fact with \eqref{eq:w4-kappa-recovery} to conclude
\eqref{eq:w4-rigidity-output}.
\end{proof}

\begin{remark}
The estimate is linear in the transport residual and boundary
normalization error, so this step causes no further loss of
their power. Smallness is required only for these errors,
not for the material gradients or $\rho$.
Neither invertibility of $H_0$ nor equality of
$\kappa_1$ and $\kappa_2$ is assumed.
\end{remark}

The bounds \eqref{eq:w4-EA-estimate}, \eqref{eq:w4-EC},
\eqref{eq:w4-a-bound}, and \eqref{eq:w4-tfK-bound} use at most
two spatial derivatives of the transport residual. The $C^4$ material
prior \eqref{eq:w4-material-priors} supplies the required $C_x^2$
connection bounds, while \eqref{eq:w3-residual-Leibniz} uses three
spatial derivatives of the projection error. Finally,
Theorem~\ref{thm:w4-hessian} controls $\rho$ without differentiating
its $L^2$ source.

Applied to the projected transport, the rigidity estimate bounds the
coefficient difference in terms of the Fourier boundary error.

\begin{corollary}
\label{cor:transport-material}
Assume the material hypotheses of Theorem~\ref{thm:w4-rigidity}. Let $H$ be
the global invertible endomorphism constructed in
Proposition~\ref{prop:w3-gluing}; it satisfies the second equation in
\eqref{eq:w3-global-output}.
For $m\ge4$, use $\beta_m$ from \eqref{eq:r2-common-error-definitions}
and $H_0$ from \eqref{eq:r2-transport-projection-definitions}.
With the fixed constants and thresholds of
\eqref{eq:w3-link-constants} and \eqref{eq:w3-joint-threshold},
\begin{equation}\label{eq:w4-I5-inputs}
 \|\lambda_1-\lambda_2\|_{L^\infty(B_R)}+
 \|\mu_1-\mu_2\|_{L^\infty(B_R)}
 \le C\omega_{\rm F}^{\beta_m}.
\end{equation}
\end{corollary}
\begin{proof}
By Lemma~\ref{lem:w3-holomorphic-form}, the field $H_0=\Pi H$
has the block form \eqref{eq:w4-Hform}. The projection bound in
\eqref{eq:w3-defect-output}, with $m\ge4$, gives
\[
 \|H_0\|_{C_x^3C_\zeta^0}\le C_\Pi M_H=B_H.
\]
Thus the projected field satisfies the algebraic and a priori
hypotheses of Theorem~\ref{thm:w4-rigidity}.

Choose its error bound to be
$\varepsilon_{\rm rig}=C_{\rm link}\omega_{\rm F}^{\beta_m}$.
Equations~\eqref{eq:w3-rigidity-output} and
\eqref{eq:w3-joint-threshold} give
\[
 \begin{aligned}
 &\|\mathcal T(H_0)\|_{C_x^2C_\zeta^0}
 +\|H_0-I\|_{C^0(\partial B_R\times\mathbb{CP}^1)}
 \le\varepsilon_{\rm rig},\\
 &\varepsilon_{\rm rig}<\varepsilon_{\rm rig,0}/2.
 \end{aligned}
\]
Apply Theorem~\ref{thm:w4-rigidity} and retain the two material terms
in \eqref{eq:w4-rigidity-output}. Their sum is at most
$CC_{\rm link}\omega_{\rm F}^{\beta_m}$, which is
\eqref{eq:w4-I5-inputs} after adjusting the fixed constant.
If $\omega_{\rm F}=0$, both errors vanish and the same theorem
applies with $\varepsilon_{\rm rig}=0$.
\end{proof}

\subsection{Common-exterior geometry and quantitative estimates}\label{sec:r2-common-exterior-setup}

On a fixed common-exterior class, the moment, transport, projection, and
rigidity bounds yield a single coefficient estimate. The positive-error and
exact-data cases are treated separately.

Fix $m\ge4$, $s=2m+14$, and a bounded smooth measurement domain $D$.
Choose fixed concentric balls with
\begin{equation}\label{eq:w0-common-geometry}
 \overline D\Subset B_{R_{\rm sup}},\qquad
 R_{\rm sup}<R_0<R_s<R,\qquad E=\{x:R_0<|x|<R\}.
\end{equation}
Fix the background $(\lambda_{\rm bg},\mu_{\rm bg})$ of
\eqref{eq:m4-common-exterior}, constant beyond a fixed ball.
Each admissible material is a real $C^s$ pair on $\R^3$ satisfying
\eqref{eq:m4-common-exterior}, with fixed $C^s$ bounds and fixed positive
lower bounds for $\mu$ and
$3\lambda+2\mu$. All working balls, leaf cylinders, support margins,
and background bounds are fixed before the noise is chosen. In
particular the two materials agree on $E$ and in a neighborhood of the
working sphere $\partial B_R$. These are assumptions of the auxiliary
common-exterior problem; the original-domain problem supplies them only
after the extension step in Section~\ref{sec:w5-extensions}. The high-order physical frame is chosen
for each material pair before varying degree, width, or $h$.

Use the coefficient distance $\mathfrak d_D$ and exponent $\beta_m$
from \eqref{eq:r2-common-error-definitions}. Define the remaining exponents by
\begin{equation}\label{eq:w0-chain-powers}
 b_{\rm F}=\frac{\gamma'}{14},
 \qquad a_{\rm chain}=b_{\rm F}\beta_m,
 \qquad 0<\gamma'<\gamma_*<1.
\end{equation}
The number $\gamma_*$ is the fixed exponent in
\eqref{eq:w3-gamma}. For each fixed $\gamma'$, Theorem~\ref{thm:w1-all-moments},
Lemma~\ref{lem:w3-negative}, and Corollary~\ref{cor:transport-material} give
\begin{equation}\label{eq:w0-CE1}
 \begin{gathered}
 \varsigma=\eta_D^{1/14},\qquad
 \omega_{\rm F}\le C_{\rm F}\eta_D^{b_{\rm F}},\\
 \mathfrak d_D\le C_6\varepsilon_{\rm rig}
                  \le C_6 C_{\rm R}\eta_D^{a_{\rm chain}},\qquad
 C_{\rm R}=C_{\rm link}C_{\rm F}^{\beta_m}.
 \end{gathered}
\end{equation}
Here $\eta_D$ is defined in \eqref{eq:r2-common-error-definitions},
$\varepsilon_{\rm rig}$ satisfies \eqref{eq:w4-rigidity-input}, and
\eqref{eq:w3-rigidity-output} supplies its bound. All constants are fixed by the class,
geometry and $\gamma'$; enlarge $C_0$ so that $C_0\ge1$.
Choose $\eta_\dagger>0$ smaller than the joint threshold in
\eqref{eq:w3-joint-threshold} and the width threshold
$\varsigma_*^{14}$:
\begin{equation}\label{eq:w0-eta-threshold}
 \eta_\dagger=\tfrac12\min\{1,\eta_{\rm all},\varsigma_*^{14}\}.
\end{equation}
Then \eqref{eq:w0-CE1} holds for $0<h<h_*$ and
$0<\eta_D<\eta_\dagger$. The common $h_*$ is also smaller than the thresholds in
Theorems~\ref{thm:w1-all-moments} and \ref{thm:finite-right-inverse}.

Under the fixed thresholds above, the common-exterior estimates yield the
following logarithmic bound and its exact-data conclusion.

\begin{proposition}
\label{prop:w0-common-assembly}
Fix the common-exterior class and geometry of
Subsection~\ref{sec:r2-common-exterior-setup}.
For every $0<\alpha<\gamma_*\beta_m/14$, there exist
$C_{\rm com},\delta_{{\rm com},0}>0$, depending only on
this fixed class, its geometry, and $\alpha$, such that
\begin{equation}\label{eq:w0-common-conclusion}
 \mathfrak d_D\le C_{\rm com}|\log\delta_D|^{-\alpha},
 \qquad 0<\delta_D<\delta_{{\rm com},0}.
\end{equation}
Here $\mathfrak d_D$ and $\delta_D$ are defined in
\eqref{eq:r2-common-error-definitions} and \eqref{eq:w0-domain-error},
and $\gamma_*$ in \eqref{eq:w3-gamma}.
At $\delta_D=0$, the two pairs coincide on $D$.
\end{proposition}

\begin{proof}
First choose
\begin{equation}\label{eq:w0-gamma-prime-choice}
 \gamma'=\tfrac12\left(\gamma_*+\frac{14\alpha}{\beta_m}\right).
\end{equation}
Then $0<\gamma'<\gamma_*$ and $a_{\rm chain}>\alpha$. All constants
in \eqref{eq:w0-CE1}, including those depending on this fixed $\gamma'$, are
fixed before choosing the noise.

For positive noise put $L=|\log\delta_D|$ and
\begin{equation}\label{eq:w0-positive-parameters}
 h=\frac{2C_0}{L},\qquad
 \eta_D=\frac{2C_0}{L}+\frac{L^2e^{-L/2}}{4C_0^2},\qquad
 \varsigma=\eta_D^{1/14}.
\end{equation}
The elementary bound $L^3e^{-L/2}\le216$ for $L>0$ gives
\begin{equation}\label{eq:w0-eta-log-bound}
 \eta_D\le\frac{C_\eta}{L},\qquad
 C_\eta=2C_0+\frac{54}{C_0^2}.
\end{equation}
One explicit sufficient choice is
\begin{equation}\label{eq:w0-common-threshold}
 L_0=\max\left\{1,\frac{4C_0}{h_*},
                  \frac{2C_\eta}{\eta_\dagger}\right\},
 \qquad\delta_{{\rm com},0}=e^{-L_0}.
\end{equation}
For $0<\delta_D<\delta_{{\rm com},0}$, this gives $h<h_*/2$ and
$\eta_D<\eta_\dagger/2$. Using Theorem~\ref{thm:w4-rigidity} and
\eqref{eq:w0-CE1}, we obtain
\begin{equation}\label{eq:r2-calculation-030}
\mathfrak d_D\le C_6C_{\rm R}C_\eta^{a_{\rm chain}}L^{-a_{\rm chain}}
 \le C_{\rm com}L^{-\alpha}.
\end{equation}
The threshold contains neither the polynomial degree nor the linear
amplitude size.

For $\delta_D=0$, use every
$0<h<\min\{h_*,\eta_\dagger\}$ and set
$\eta_D=h$, $\varsigma=h^{1/14}$. The same estimates give
\begin{equation}\label{eq:w0-exact-common-bound}
 \mathfrak d_D\le C_6C_{\rm R}h^{a_{\rm chain}}.
\end{equation}
The material pair and the left side are fixed, while the constants are
uniform and $a_{\rm chain}>0$. Letting $h\downarrow0$ in \eqref{eq:w0-exact-common-bound}, we obtain
$\mathfrak d_D=0$.
This argument uses no limit of the transports and never evaluates
$\log0$.
\end{proof}

\begin{remark}
At zero original-domain error, the paired extension construction
gives a family within one fixed class. This proposition applies
separately at each positive tolerance, with constants independent
of the tolerance.
\end{remark}

\section{Sharp stability of boundary derivatives}\label{sec:v63-boundary-jets}

Boundary stability follows from localized elastic Dirichlet energies through
a finite Taylor expansion, coefficient extraction, and recovery of
successive normal derivatives. Boundary-localized perturbations establish
sharpness and the top-order obstruction under strict prior slack.

\subsection{Local hypotheses and the measurement norm}

Fix an integer $s\ge2$ and the material class $\mathscr C^s(\Omega;M,c_0)$
of \eqref{eq:w0-interior-class}. All boundary estimates use only this class,
not prescribed exterior coefficients. Theorem~\ref{thm:w0-target-A} uses
$s=22$, while Proposition~\ref{prop:w0-interior-assembly} retains the general
orders of \eqref{eq:r2-regularity-orders}; the present section allows every
integer $s\ge2$. Constants depend on the
fixed class, the domain, the boundary atlas \eqref{eq:w5-explicit-normal-chart},
and the tests \eqref{eq:v63-two-tests}.
There is no uniform modulus of continuity assumption for the derivatives
of order $s$.

We use the physical variational forms and the uniform Dirichlet estimate
\eqref{eq:w5-physical-form}--\eqref{eq:w5-dirichlet-bound} from
Section~\ref{sec:physical-preliminaries}. The abbreviation $\mathfrak a$ is defined in \eqref{eq:r2-energy-abbreviation}.

Use $J_q$ from \eqref{eq:w0-boundary-jets}; all its derivatives are
Cartesian traces from the interior.
For the outward normal $\nu$ fix the tubular map and its normal derivative convention
\begin{equation}\label{eq:w5-normal-rays}
 \Phi(p,t)=p+t\nu(p),\qquad
 \partial_\nu^a f(p)=\left.\partial_t^a(f\circ\Phi)(p,t)\right|_{t=0^-}.
\end{equation}
The derivative follows a straight normal segment; no unspecified extension
of the normal vector field is used.
\subsection{Fixed boundary coordinates and annular tests}\label{sec:r2-boundary-tests}

The boundary tests use fixed normal coordinates and annular tangential
frequencies, avoiding the zero-frequency singularity while retaining uniform
trace norms.

We use only \eqref{eq:w0-interior-class}. The atlas and its
cutoffs are fixed before either candidate and before the scale is chosen.
The geometric $C^{s+1}$ bounds needed for the coordinate tensor are fixed
bounds on the domain, not additional material derivatives.

For each $p\in\partial\Omega$ choose the boundary normal chart
$\Psi_p$ in \eqref{eq:w5-explicit-normal-chart}, with $z_3\ge0$ inward. The sizes, inverse bounds and every required
finite geometric derivative are uniform after subdividing a finite atlas.
More precisely, choose a surface parametrization $P_p$ and set
\begin{equation}\label{eq:w5-explicit-normal-chart}
 \Psi_p(z',z_3)=P_p(z')-z_3\nu(P_p(z')),\qquad
 P_p(0)=p,\qquad R_p=D\Psi_p(0)\in O(3).
\end{equation}
Tangential coordinates are linearly adjusted at their center to obtain
this orthogonality. Throughout the chart, not merely at its center,
\begin{equation}\label{eq:w5-mixed-normal-trace}
 \left.\partial_{z'}^\gamma\partial_{z_3}^a(f\circ\Psi_p)\right|_{z_3=0}
 =(-1)^a\partial_{z'}^\gamma
       [\partial_\nu^a f\circ P_p],\qquad |\gamma|+a\le s.
\end{equation}
The equality follows by holding $P_p(z')$ and its normal fixed while
performing the $a$ normal derivatives, and then differentiating in $z'$. Rotate the physical displacement by the \emph{constant}
matrix $R_p$, not by a moving frame. The physical Hermitian energy in the
chart then has the form
\begin{equation}\label{eq:v63-pulled-energy}
 \begin{aligned}
 \mathfrak a(v,w)&=\int \mathbb A_p(z)\nabla_z v:
                               \overline{\nabla_z w}\,\mathrm{d} z,\\
 \mathbb A_p(0)\nabla v:\overline{\nabla w}
 &=\lambda(p)\operatorname{div} v\,\overline{\operatorname{div} w}+2\mu(p)\varepsilon(v):\overline{\varepsilon(w)}.
 \end{aligned}
\end{equation}
Here $\mathbb A_p$ includes the Jacobian and the inverse coordinate matrices. The transformed tensor $\mathbb A_p$ is uniformly bounded in $C^s$. Away from $z=0$ it need not be
an isotropic tensor in these coordinates. There are no lower-order terms,
because the component rotation is constant.

For boundary profiles use the unitary two-dimensional Fourier transform
\begin{equation}\label{eq:r2-boundary-Fourier-transform}
 \widehat\phi(\xi)=(2\pi)^{-1}\int_{\mathbb R^2}
                    e^{-iy'\cdot\xi}\phi(y')\,dy'.
\end{equation} Choose once
and for all a nonzero scalar $\psi_\partial\in C_c^\infty(\{1<|\xi|<2\})$, normalized
by the following conventions:
\begin{equation}\label{eq:r2-boundary-fourier-directions}
 \begin{gathered}
 \int_{\mathbb R^2}|\psi_\partial(\xi)|^2\,d\xi=1,\qquad
 \varpi=|\xi|,\quad \xi\in\mathbb R^2\setminus\{0\},\\
 \widehat\xi=(\xi_1,\xi_2,0)/\varpi,\qquad
 \widehat\xi^\perp=(-\xi_2,\xi_1,0)/\varpi.
 \end{gathered}
\end{equation}
The two fixed Schwartz boundary tests are
\begin{equation}\label{eq:v63-two-tests}
 \widehat\phi^S(\xi)=\psi_\partial(\xi)\widehat{\xi}^\perp,
 \qquad \widehat\phi^N(\xi)=\psi_\partial(\xi)e_3.
\end{equation}
Their Fourier supports avoid zero. Superscripts $S,N$ refer to these shear and normal boundary tests. Let $\chi_p$ be a fixed real cutoff in the
chart, equal to one on a smaller half-ball. For either test $\phi$ prescribe
\begin{equation}\label{eq:v63-physical-tests}
 f_{p,\epsilon}^{\phi}(\Psi_p(z',0))
   =\epsilon^{-1/2}\chi_p(z',0)R_p\phi(z'/\epsilon),
 \qquad f_{p,\epsilon}^{\phi}=0\ \text{off the chart}.
\end{equation}
The scale $\epsilon>0$ in this section is a boundary dilation scale, not
the CGO parameter $h$. Although $\phi$ is not compactly supported, the
physical trace is. Multiplication by the fixed cutoff and chart norm
equivalence give
\begin{equation}\label{eq:v63-trace-norm}
 \sup_{p,\,0<\epsilon<\epsilon_*}
    \lVert f_{p,\epsilon}^{\phi}\rVert _{H^{1/2}(\partial\Omega)}\le C_\phi.
\end{equation}
Indeed the squared flat norm before cutoff is
$\int(\epsilon^2+|\xi|^2)^{1/2}|\widehat\phi(\xi)|^2\mathrm{d}\xi$.
The trace normalization in \eqref{eq:v63-trace-norm} is the one used
in the DN pairing \eqref{eq:w5-physical-dn}.

\subsection{Frozen physical impedance and its inverse}\label{sec:r2-boundary-impedance}
The frozen elastic impedance gives the order-zero boundary response, whose
dependence on the two Lam\'e parameters can be inverted. For classical
boundary determination, see \cite{LinNakamura2017}. Let $(\lambda_0,\mu_0)$
be frozen boundary values. In the ordered basis
$(\widehat\xi,\widehat\xi^\perp,e_3)$, with $e_3$ inward, the physical
traction uses $-e_3$. The corresponding $3\times3$ boundary impedance
multiplier, distinct from any interior amplitude, is
\begin{equation}\label{eq:v62-impedance}
 \mathsf Z_{\partial}(\xi)=\varpi
 \begin{pmatrix}a_Z&0&-i b_Z\\0&\mu_0&0\\i b_Z&0&a_Z\end{pmatrix},
 \quad a_Z=\frac{2\mu_0(\lambda_0+2\mu_0)}{\lambda_0+3\mu_0},
 \quad b_Z=\frac{2\mu_0^2}{\lambda_0+3\mu_0}.
\end{equation}
Indeed the profile in \eqref{eq:v63-Poisson} has boundary derivative
$\varpi(b-a)$ and boundary divergence
$\varpi(1-g_0)(i a_\parallel-a_3)$.
The three traction components are respectively
$-\mu_0(\partial_tu_\parallel+i\varpi u_3)$,
$-\mu_0\partial_tu_\perp$, and
$-\lambda_0\operatorname{div}u-2\mu_0\partial_tu_3$.
Substituting these boundary derivatives into the three traction components, we
obtain \eqref{eq:v62-impedance}. In particular the sign of each off-diagonal
entry is tied to the outward traction, not the inward coordinate.
The algebraic inverse needed at order zero is
\begin{equation}\label{eq:v62-impedance-inversion}
 \lambda_0=\mu_0\frac{3a_Z-4\mu_0}{2\mu_0-a_Z},
 \qquad 2\mu_0-a_Z=\frac{2\mu_0^2}{\lambda_0+3\mu_0}>0.
\end{equation}
The class bounds keep this denominator uniformly away from zero, so this
inverse is uniformly Lipschitz on the admissible value range.

\subsection{Decaying half-space corrections without low frequencies}\label{sec:r2-halfspace-profiles}

Half-space corrections retain annular tangential frequency support and admit
weighted decay estimates uniform in the material class.

Freeze the energy form \eqref{eq:v63-pulled-energy} at the boundary point $p$:
\begin{equation}\label{eq:r2-frozen-halfspace-operator}
 \begin{gathered}
 (\lambda_0,\mu_0)=(\lambda(p),\mu(p)),\\
 \mathfrak a_0(V,W)=\int_{\mathbb R^3_+}\mathbb A_p(0)\nabla V:
                  \overline{\nabla W}\,dy,\\
 L_0^+=-\mu_0\Delta-(\lambda_0+\mu_0)\nabla\operatorname{div}.
 \end{gathered}
\end{equation}
Here $L_0^+$ is the positive $3\times3$ second-order elastic operator
associated with $\mathfrak a_0$. For tangential frequency $\xi$ and
real normal frequency $\xi_3$, its full-space inverse symbol is
\begin{equation}\label{eq:v63-full-space-inverse}
 \begin{gathered}
 \bigl[\mu_0(\varpi^2+\xi_3^2)I_3+
       (\lambda_0+\mu_0)\xi_{\mathrm{full}}\xi_{\mathrm{full}}^T\bigr]^{-1}\\
  =\frac{I_3}{\mu_0(\varpi^2+\xi_3^2)}
  -\frac{\lambda_0+\mu_0}{\mu_0(\lambda_0+2\mu_0)}
       \frac{\xi_{\mathrm{full}}\xi_{\mathrm{full}}^T}{(\varpi^2+\xi_3^2)^2},\\
 \xi_{\mathrm{full}}=(\xi_1,\xi_2,\xi_3)^T.
 \end{gathered}
\end{equation}
The material denominators in \eqref{eq:v63-full-space-inverse} and
\eqref{eq:v63-Poisson} are uniformly separated from zero by \eqref{eq:w0-interior-class}.

Restricting the tangential Fourier support to a fixed annulus avoids the
low-frequency singularity and preserves exponential normal decay through the
correction equations.

\begin{lemma}\label{lem:v63-profile}
Use the finite-energy spaces
\begin{equation}\label{eq:r2-halfspace-energy-spaces}
 \begin{aligned}
 \dot H^1(\mathbb R^3_+;\mathbb C^3)
  &=\{V\in L^6(\mathbb R^3_+;\mathbb C^3):\nabla V\in L^2\},\\
 \dot H^1_0(\mathbb R^3_+;\mathbb C^3)
  &=\{V\in\dot H^1(\mathbb R^3_+;\mathbb C^3):\operatorname{Tr}V=0\}.
 \end{aligned}
\end{equation}
Consider profiles with tangential Fourier transforms given by finite sums
\begin{equation}\label{eq:v63-profile-class}
 \widehat V(\xi,t)=e^{-\varpi t}\sum_{n=0}^{N_{\rm pol}}t^nv_n(\xi),\qquad t\ge0,
 \qquad v_n\in C_c^\infty(\{1<\varpi<2\};\mathbb{C}^3).
\end{equation}
The class \eqref{eq:v63-profile-class} is preserved by spatial
differentiation and multiplication by
polynomials in $(y',t)$. If $F$ belongs to this class, the zero-trace
finite-energy solution of $L_0^+V=F$ also belongs to it. The constant
half-space solution with any annular smooth boundary Fourier data belongs
to the same class. For every fixed pair of integers $B,N\ge0$, the
seminorms
\begin{equation}\label{eq:v63-weighted-profiles}
 \sum_{|\alpha|\le B}
    \lVert (1+|y|)^N\partial^\alpha V\rVert _{L^2(\mathbb{R}^3_+)}
\end{equation}
are bounded by finitely many seminorms of the input coefficients, uniformly
on the frozen material range. The solution maps for $L_0^+$ in \eqref{eq:r2-frozen-halfspace-operator}
depend smoothly and locally Lipschitz-continuously, in the seminorms
\eqref{eq:v63-weighted-profiles}, on the frozen material values
and the finitely many input coefficients.
\end{lemma}
\begin{proof}
Multiplication by $y'$ acts by differentiation in tangential Fourier
variables. Differentiating $e^{-\varpi t}$ only creates further powers of $t$;
it neither introduces small frequencies nor leaves the fixed compact
annulus. Normal multiplication and differentiation are immediate.

For the inverse, take the inverse Fourier transform of
\eqref{eq:v63-full-space-inverse} in $\xi_3$. The scalar kernels
$(\varpi^2+\xi_3^2)^{-1}$ and $(\varpi^2+\xi_3^2)^{-2}$ are constant multiples of
$e^{-\varpi|t|}/\varpi$ and $(1+\varpi|t|)e^{-\varpi|t|}/\varpi^3$. The factors $\xi_3$ and
$\xi_3^2$ produce their first and second derivatives. Thus the matrix
Green kernel on the full line is, on each side of zero, a polynomial of
degree at most one times $e^{-\varpi|t|}$, with smooth coefficients on the
annulus and the frozen material range. Extend $F$ by zero to $t<0$ and
convolve this kernel with $F$. For $t>0$, split the convolution at $t$.
The integrals over $[0,t]$ and $[t,\infty)$ are elementary integrals of
polynomials against exponentials; their sum is $e^{-\varpi t}$ times a
polynomial in $t$. Its coefficients remain smooth and compactly supported
in $\xi$. The full-line convolution solution may have nonzero boundary trace.

For an arbitrary boundary vector $a=a(\xi)$ the decaying homogeneous
solution is
\begin{equation}\label{eq:v63-Poisson}
 \begin{gathered}
 \widehat{\mathcal P_0 a}(\xi,t)
       =e^{-\varpi t}\bigl(a+\varpi t\,b(a)\bigr),\\
 b(a)=g_0\bigl[-(a_\parallel+i a_3)\widehat{\xi}
                       +(-i a_\parallel+a_3)e_3\bigr],\\
 a_\parallel=\widehat{\xi}\cdot a,\qquad
 g_0=\frac{\lambda_0+\mu_0}{\lambda_0+3\mu_0},\qquad
 d_0=1-g_0=\frac{2\mu_0}{\lambda_0+3\mu_0}.
 \end{gathered}
\end{equation}
Set
\begin{equation}\label{eq:r2-poisson-profile-identities}
 \begin{gathered}
 c_a=i a_\parallel-a_3,\qquad b=g_0c_a(i\widehat\xi-e_3),\qquad
 i b_\parallel-b_3=0,\\
 \widehat{\operatorname{div}\mathcal P_0a}
       =\varpi d_0c_a e^{-\varpi t},\qquad
 \widehat{\Delta\mathcal P_0a}=-2\varpi^2b e^{-\varpi t}.
 \end{gathered}
\end{equation}
Substitution in \eqref{eq:r2-frozen-halfspace-operator} gives
\begin{equation}\label{eq:r2-poisson-equation-verification}
 \widehat{L_0^+\mathcal P_0a}
 =\varpi^2c_a e^{-\varpi t}
       [2\mu_0g_0-(\lambda_0+\mu_0)d_0](i\widehat\xi-e_3)=0.
\end{equation}
We subtract the Poisson solution \eqref{eq:v63-Poisson} with boundary data
equal to the trace of the convolution solution. The resulting field has zero
trace and remains in the class \eqref{eq:v63-profile-class}. Uniqueness in the zero-trace
finite-energy space follows from strong convexity and Korn's inequality;
one may extend a zero-trace field by zero and apply the full-space
Fourier Korn inequality.

Plancherel turns tangential weights into frequency derivatives. All such
derivatives are supported in the fixed annulus and produce only finitely
many additional powers of $t$. Normal weighted integrals of polynomials
times $e^{-\varpi t}$ are uniformly finite because $\varpi\ge1$. This proves
\eqref{eq:v63-weighted-profiles}. The convolution formulas have only
uniformly nonvanishing material denominators and positive powers of $\varpi$
in denominators. Differentiation in any finite collection of material or
source parameters is consequently legitimate under the integrals and has
the same bounds. Differentiation under the integral therefore establishes the asserted Lipschitz dependence of the half-space profile operator.
\end{proof}

Lemma~\ref{lem:v63-profile} uses arbitrary \emph{fixed} spatial weights of a profile built
from finitely many derivative values. It does not require arbitrarily many
spatial derivatives of the original material. The distinction between profile polynomial degree and coefficient regularity controls the remote cutoff tails without strengthening the $C^s$ prior.

\subsection{Uniform energy expansion for the exact Dirichlet solution}\label{sec:r2-boundary-energy}

The transplanted half-space profiles are compared with the exact Dirichlet
solution to obtain a boundary-energy expansion with a uniform order-$s$
remainder.

Define the homogeneous Taylor tensors and their forms by
\begin{equation}\label{eq:v63-Taylor-tensors}
 \mathbb A_{p,r}(y)
   =\sum_{|\alpha|=r}\frac{\partial^\alpha\mathbb A_p(0)}{\alpha!}y^\alpha,
 \qquad
 \mathfrak a_r(V,W)=\int_{\mathbb{R}^3_+}\mathbb A_{p,r}(y)\nabla V:
                                      \overline{\nabla W}\,\mathrm{d} y.
\end{equation}
For $r=0$, this is the frozen form \eqref{eq:r2-frozen-halfspace-operator}.
On the fixed chart, Taylor's integral remainder gives the uniform bound
\begin{equation}\label{eq:v63-Taylor-remainder}
 \left|\mathbb A_p(\epsilon y)
       -\sum_{r=0}^{s-1}\epsilon^r\mathbb A_{p,r}(y)\right|
       \le C\epsilon^s|y|^s.
\end{equation}
Only the bounded order-$s$ derivatives are used; their modulus of
continuity need not be uniform across the material class.

Let $\widehat V_0(\xi,t)$ be the right-hand side of
\eqref{eq:v63-Poisson} with $a=\widehat\phi(\xi)$. For $1\le r\le s-1$ define the zero-trace corrections
by
\begin{equation}\label{eq:v63-recursion}
 \mathfrak a_0(V_r,W)=-\sum_{\ell=1}^{r}\mathfrak a_\ell(V_{r-\ell},W),
 \qquad \operatorname{Tr} V_r=0,
 \qquad W\in\dot H^1_0(\mathbb{R}^3_+).
\end{equation}
Equivalently, the positive operator satisfies
$L_0^+V_r=\operatorname{div}\bigl(\sum_{\ell=1}^r
\mathbb A_{p,\ell}\nabla V_{r-\ell}\bigr)$ in the half-space; this fixes the
source sign. It is a divergence source with polynomial coefficients times
previous profiles. Lemma~\ref{lem:v63-profile} constructs $V_r$ successively.
Every $V_r$ has the weighted bounds \eqref{eq:v63-weighted-profiles},
uniformly in $p$ and the material, and depends smoothly on the tensor derivatives
through order $r$.

\begin{theorem}
\label{thm:v63-energy-expansion}
Let \eqref{eq:w0-interior-class} hold, and let $\phi$ be either of the fixed
tests in \eqref{eq:v63-two-tests}. If $u_{p,\epsilon}^{\phi}$ is the actual
elastic Dirichlet solution with trace \eqref{eq:v63-physical-tests}, put
\begin{equation}\label{eq:v63-measured-energy}
 \mathcal E_{p,\phi}(\epsilon)
   =\big\langle\Lambda_{\lambda,\mu}f_{p,\epsilon}^{\phi},
                              \overline{f_{p,\epsilon}^{\phi}}\big\rangle.
\end{equation}
There are uniform $C,\epsilon_*>0$ such that
\begin{equation}\label{eq:v63-energy-expansion}
 \left|\mathcal E_{p,\phi}(\epsilon)
       -\sum_{r=0}^{s-1}\epsilon^r E_{p,r}^{\phi}\right|
      \le C\epsilon^s,
 \qquad
 E_{p,r}^{\phi}=\sum_{\ell+b+c=r}\mathfrak a_\ell(V_b,V_c),
 \qquad 0<\epsilon<\epsilon_*.
\end{equation}
Here $\mathfrak a_\ell$ is defined in \eqref{eq:v63-Taylor-tensors},
and $V_b,V_c$ are the profiles defined by
\eqref{eq:v63-Poisson} and \eqref{eq:v63-recursion}.
The coefficients are smooth functions of the finite material and
geometric derivatives of total order at most $r$. Their first derivatives with
respect to those derivative values are uniformly bounded on the indicated
bounded, strongly convex parameter range.
\end{theorem}
\begin{proof}
Set
\begin{equation}\label{eq:r2-truncated-boundary-profile}
 V_\epsilon^{[s-1]}=\sum_{r=0}^{s-1}\epsilon^rV_r,
\end{equation}
where the profiles are defined by \eqref{eq:v63-Poisson} and
\eqref{eq:v63-recursion}. Define the physical approximation by
\begin{equation}\label{eq:v63-transplant}
 U_{p,\epsilon}(\Psi_p(z))
 =\epsilon^{-1/2}R_p\chi_p(z)V^{[s-1]}_\epsilon(z/\epsilon),
 \qquad U_{p,\epsilon}=0\text{ outside the chart}.
\end{equation}
All positive-order correctors have zero trace, so this is an exact
lifting of the prescribed Dirichlet data \eqref{eq:v63-physical-tests}. Lemma~\ref{lem:v64-profile-budget}
gives uniform weighted bounds for all the profiles. The four-term
calculation in Proposition~\ref{prop:boundary-residual}, which uses only
the recursion and these weighted bounds, proves
\begin{equation}\label{eq:v63-residual-bound}
 |\mathfrak a(U_{p,\epsilon},w)|\le C\epsilon^s\|w\|_{H^1(\Omega)},
       \qquad w\in H^1_0(\Omega;\mathbb C^3).
\end{equation}
By \eqref{eq:r2-energy-coercivity}, the zero-trace error
\begin{equation}\label{eq:r2-boundary-solution-error-field}
 z=u_{p,\epsilon}^{\phi}-U_{p,\epsilon}\in H^1_0(\Omega;\mathbb C^3)
\end{equation}
satisfies
\begin{equation}\label{eq:v63-actual-solution-error}
 \|z\|_{H^1(\Omega)}\le C\epsilon^s.
\end{equation}
Since $\mathfrak a(u_{p,\epsilon}^{\phi},z)=0$, Hermitian symmetry gives the exact identity
\begin{equation}\label{eq:v63-energy-orthogonality}
 \mathcal E_{p,\phi}(\epsilon)
 =\mathfrak a(U_{p,\epsilon},U_{p,\epsilon})-\mathfrak a(z,z).
\end{equation}
The last term is $O(\epsilon^{2s})$.
Applying \eqref{eq:v64-approximate-energy} to the first term
proves the expansion in \eqref{eq:v63-energy-expansion}.

For each fixed order, the profile recursion uses finitely many tensor
jets, sums and products, and the constant Green operator from
Lemma~\ref{lem:v63-profile}. Its frequency coefficients are rational smooth
functions of the frozen moduli and polynomial functions of the remaining
jets; all material denominators have fixed positive lower bounds.
Weighted integrals and their parameter derivatives are dominated by the
same weighted profile bounds. The coefficients and their first jet
parameter derivatives are therefore uniformly bounded. The argument
uses order-$s$ material derivatives only in the Taylor remainder.
\end{proof}

\begin{remark}
The expansion concerns energies of the exact physical Dirichlet
solutions, so its coefficients can be estimated using the measured DN
norm. Its uniform order-$s$ remainder is also what permits the scale
choice used to recover boundary derivatives of orders below $s$, without a common
modulus of continuity for the top material derivatives.
\end{remark}

\subsection{Vandermonde extraction at fixed scale ratios}\label{sec:r2-energy-extraction}
\begin{lemma}
\label{lem:v63-extraction}
For $0\le q\le s-1$ choose the real weights $w_{q\ell}$, $1\le\ell\le s$,
uniquely determined by
\begin{equation}\label{eq:v63-Vandermonde}
 \sum_{\ell=1}^s w_{q\ell}\ell^r=\begin{cases}1,&r=q,\\0,&r\ne q,
                       \end{cases}
 \qquad 0\le r\le s-1.
\end{equation}
For two materials in \eqref{eq:w0-interior-class}, use $\delta$ from
\eqref{eq:w0-measured-error}. For either fixed
test and all $0<\epsilon<\epsilon_*/s$,
\begin{equation}\label{eq:v63-coefficient-noise}
 |E_{1,p,q}^{\phi}-E_{2,p,q}^{\phi}|
      \le C_s\bigl(\delta\epsilon^{-q}+\epsilon^{s-q}\bigr).
\end{equation}
Here $E_{j,p,q}^{\phi}$ is the coefficient in
\eqref{eq:v63-energy-expansion} for material $j$, and $\epsilon_*$ is the
uniform scale threshold of Theorem~\ref{thm:v63-energy-expansion}.
In particular, for $\delta$ sufficiently small and $0\le q<s$,
\begin{equation}\label{eq:v63-coefficient-Holder}
 \sup_p|E_{1,p,q}^{\phi}-E_{2,p,q}^{\phi}|
                 \le C_s\delta^{(s-q)/s}.
\end{equation}
The weights depend only on $s$, and the constants are uniform
over the fixed admissible coefficient class.
\end{lemma}
\begin{proof}
The nodes $1,\ldots,s$ are distinct, so the Vandermonde matrix in
\eqref{eq:v63-Vandermonde} is invertible. Apply
\eqref{eq:v63-energy-expansion} at $\ell\epsilon$ and form
\begin{equation}\label{eq:r2-calculation-031}
\epsilon^{-q}\sum_{\ell=1}^s
                   w_{q\ell}\mathcal E_{p,\phi}(\ell\epsilon).
\end{equation}
The moment identities \eqref{eq:v63-Vandermonde} cancel every
polynomial coefficient except the one of order $q$. Thus, for
material $j=1,2$,
\[
 \left|\epsilon^{-q}\sum_{\ell=1}^s
  w_{q\ell}\mathcal E_{j,p,\phi}(\ell\epsilon)
       -E_{j,p,q}^\phi\right|\le C_s\epsilon^{s-q}.
\]

The trace bound \eqref{eq:v63-trace-norm} and the measured DN norm
\eqref{eq:w0-measured-error} give
\begin{equation}\label{eq:r2-calculation-032}
|\mathcal E_{1,p,\phi}(\ell\epsilon)
       -\mathcal E_{2,p,\phi}(\ell\epsilon)|\le C\delta.
\end{equation}
Subtracting the two extraction identities and using
\eqref{eq:r2-calculation-032} therefore yields
\[
 \begin{aligned}
 |E_{1,p,q}^\phi-E_{2,p,q}^\phi|
 &\le C\delta\epsilon^{-q}\sum_{\ell=1}^s|w_{q\ell}|
       +C_s\epsilon^{s-q}\\
 &\le C_s(\delta\epsilon^{-q}+\epsilon^{s-q}),
 \end{aligned}
\]
which proves \eqref{eq:v63-coefficient-noise}.

For $\delta>0$, take $\epsilon=\delta^{1/s}$ and choose a uniform
noise threshold ensuring $s\epsilon<\epsilon_*$. Both terms then
have order $\delta^{(s-q)/s}$, proving
\eqref{eq:v63-coefficient-Holder}. For $q=0$, letting
$\epsilon\downarrow0$ at fixed $\delta$ already gives the Lipschitz
bound $|E_{1,p,0}^\phi-E_{2,p,0}^\phi|\le C_s\delta$.
If $\delta=0$, keep the nodes and weights fixed and let their common
scale $\epsilon\downarrow0$. Since $s-q>0$,
\eqref{eq:v63-coefficient-noise} gives equality of the two coefficients.
Only the measured DN maps and the expansion have been used.
\end{proof}

\subsection{A triangular matrix for the new normal derivative}\label{sec:r2-normal-recovery}

At each step of the normal-derivative induction, two extracted energy
coefficients determine the new pair of Lam\'e derivatives through a
uniformly invertible $2\times2$ matrix.

Let
\begin{equation}\label{eq:r2-normal-jet-vector}
 j_q^+(p)=\binom{\lambda_q^+(p)}{\mu_q^+(p)},\qquad
 \lambda_q^+(p)=\partial_{z_3}^q(\lambda\circ\Psi_p)(0),\quad
 \mu_q^+(p)=\partial_{z_3}^q(\mu\circ\Psi_p)(0).
\end{equation}
The inward normal derivative differs from the outward derivative in
\eqref{eq:w5-normal-rays} by $(-1)^q$ without changing any
stability norm. For $q\ge1$, define the finite jet list
\begin{equation}\label{eq:r2-lower-normal-jets}
 \mathcal K_{p,q}^{<}
 =\bigl(\partial_z^\alpha(f\circ\Psi_p)(0):
       f\in\{\lambda,\mu\},\ |\alpha|\le q,\ \alpha_3<q\bigr).
\end{equation} The collection $\mathcal K_{p,q}^{<}$ includes boundary values and
all lower normal derivatives and their required tangential derivatives.
The derivatives of the coordinate geometry are fixed and implicit in this notation.

The extracted coefficients still have to be related to the material derivatives. At each normal order, the two fixed polarizations produce a triangular system with a leading matrix acting on the two new Lam\'e normal derivatives.

\begin{proposition}
\label{prop:v63-triangular}
For every $1\le q<s$ the two coefficients in
\eqref{eq:v63-energy-expansion} have the form
\begin{equation}\label{eq:v63-coefficient-triangular}
 \binom{E_{p,q}^S}{E_{p,q}^N}
       =\mathsf M_q(\lambda(p),\mu(p))j_q^+(p)
          +\mathcal F_{p,q}(\mathcal K_{p,q}^{<}),
\end{equation}
where $j_q^+$ is defined in \eqref{eq:r2-normal-jet-vector},
$\mathcal K_{p,q}^{<}$ in \eqref{eq:r2-lower-normal-jets}, and
\begin{equation}\label{eq:r2-polarized-energy-coefficients}
 E_{p,q}^S=E_{p,q}^{\phi^S},\qquad E_{p,q}^N=E_{p,q}^{\phi^N}.
\end{equation}
The coefficients on the right are from \eqref{eq:v63-energy-expansion},
with the tests \eqref{eq:v63-two-tests}.
The function $\mathcal F_{p,q}$ is uniformly Lipschitz on the finite-order derivative range.
Put
\begin{equation}\label{eq:v63-moment-definitions}
 \begin{gathered}
 I_q=\int_{\mathbb{R}^2}|\psi_\partial(\xi)|^2|\xi|^{1-q}\,\mathrm{d}\xi>0,\\
 B_q=3d_0^2+4d_0g_0(q+1)+2g_0^2(q+1)(q+2).
 \end{gathered}
\end{equation}
Here $g_0,d_0$ are defined in \eqref{eq:v63-Poisson}, and
$\psi_\partial$ is the fixed annular test profile in \eqref{eq:v63-two-tests}.
Then
\begin{equation}\label{eq:v63-moment-matrix}
 \mathsf M_q=I_q
 \begin{pmatrix}
   0&2^{-q}\\
   2^{-q-1}d_0^2&2^{-q-1}B_q
 \end{pmatrix},
 \qquad
 \det\mathsf M_q=-2^{-2q-1}I_q^2d_0^2\ne0.
\end{equation}
The matrices and their inverses are uniformly bounded for each of the
finitely many orders $q<s$ under \eqref{eq:w0-interior-class}.
\end{proposition}
\begin{proof}
In $E_{p,q}^{\phi}=\sum_{\ell+b+c=q}\mathfrak a_\ell(V_b,V_c)$, the only terms
containing $V_q$ are $\mathfrak a_0(V_q,V_0)$ and $\mathfrak a_0(V_0,V_q)$. Both vanish: $V_q$
has zero trace and $V_0$ solves the homogeneous constant problem. Hence
\begin{equation}\label{eq:v63-highest-linear}
 E_{p,q}^{\phi}=\mathfrak a_q(V_0,V_0)
                    +\text{a smooth function of tensor derivatives of order }<q.
\end{equation}
The pure-normal part of the order-$q$ Taylor tensor in
$\mathfrak a_q(V_0,V_0)$ contributes exactly
\begin{equation}\label{eq:v63-weighted-energy-moments}
 \frac{\lambda_q^+}{q!}\int_{\mathbb{R}^3_+}t^q|\operatorname{div} V_0|^2\,\mathrm{d} y
 +\frac{2\mu_q^+}{q!}\int_{\mathbb{R}^3_+}t^q
                                      |\varepsilon(V_0)|^2\,\mathrm{d} y.
\end{equation}
Indeed the coordinate Jacobian is orthogonal at the base point. In the
$q$th tensor derivative, derivatives falling on coordinate factors
multiply lower-order material-parameter derivatives. All other order-$q$ material
monomials have smaller normal order and belong to $\mathcal K_{p,q}^{<}$.
For clarity, the tensor in Appendix~\ref{app:boundary-residual} is a sum
of fixed smooth coordinate factors times $\lambda\circ\Psi_p$ or
$\mu\circ\Psi_p$. After applying $\partial_z^\alpha$, $|\alpha|=q$, to the transformed tensor,
we separate the terms according to whether a derivative falls on a coordinate
factor. The terms with all $q$ derivatives on a material coefficient are
linear in its order-$q$ coordinate derivatives. Every term with a differentiated
coordinate factor uses a material derivative of order at most $q-1$.
Among the terms with all $q$ derivatives on a material component,
$\alpha=(0,0,q)$ gives exactly
\eqref{eq:v63-weighted-energy-moments}; every $\alpha$ with
$\alpha_3<q$ is in $\mathcal K_{p,q}^{<}$ by
\eqref{eq:w5-mixed-normal-trace}. Also $V_r$ for $r<q$ depends only
on tensor derivatives through order $r$. Thus the remainder in
\eqref{eq:v63-highest-linear} contains no order-$q$ pure normal
material derivative, even nonlinearly. This proves the structural
decomposition including all mixed and geometric terms.
Uniform Lipschitz dependence follows from
Theorem~\ref{thm:v63-energy-expansion} and
Lemma~\ref{lem:v63-profile} on the bounded finite-order derivative set.

For $\phi^S$, \eqref{eq:v63-Poisson} gives
$\widehat V_0^S=\psi_\partial e^{-\varpi t}\widehat{\xi}^\perp$. It is divergence free, and at each
frequency
$2|\varepsilon(\widehat V_0^S)|^2=2\varpi^2|\psi_\partial|^2e^{-2\varpi t}$, and the strain
includes the tangential Fourier derivatives. Since
$\int_0^\infty t^qe^{-2\varpi t}\mathrm{d} t/q!=(2\varpi)^{-q-1}$, its two weights are
$0$ and $2^{-q}I_q$.

For $\phi^N$ one has
\begin{equation}\label{eq:v63-normal-profile}
 \widehat V_0^N(\xi,t)
   =\psi_\partial e^{-\varpi t}\bigl[e_3+g_0\varpi t(-i\widehat{\xi}+e_3)\bigr],
 \qquad
 \widehat{\operatorname{div} V_0^N}=-d_0\varpi\psi_\partial e^{-\varpi t}.
\end{equation}
At $\varpi=1$, $\widehat{\xi}=e_1$, the three nonzero strain types are
\begin{equation}\label{eq:r2-calculation-033}
\varepsilon_{11}=g_0t e^{-t},\qquad
 \varepsilon_{33}=-(d_0+g_0t)e^{-t},\qquad
 \varepsilon_{13}=\varepsilon_{31}
                    =\tfrac{i}{2}(d_0+2g_0t)e^{-t}.
\end{equation}
Restoring the frequency $\varpi$ and the profile $\psi_\partial$ in
\eqref{eq:r2-calculation-033}, we obtain
\begin{equation}\label{eq:v63-normal-strain-density}
 2|\varepsilon(\widehat V_0^N)|^2
   =\varpi^2|\psi_\partial|^2e^{-2\varpi t}
          [3d_0^2+8d_0g_0\varpi t+8g_0^2(\varpi t)^2].
\end{equation}
Integrate as in the shear case. The resulting weights are
$2^{-q-1}d_0^2I_q$ and $2^{-q-1}B_qI_q$, proving
\eqref{eq:v63-moment-matrix}. Its determinant is the displayed nonzero
quantity; $d_0$ has a uniform positive lower bound and $I_q$ is a fixed
positive number.
\end{proof}

\begin{remark}
At each order the shear test isolates the new normal derivative
of $\mu$, and the normal test then determines that of $\lambda$.
The remainder contains only lower normal orders and their tangential
derivatives. This triangular structure makes an induction possible
without assuming that any boundary derivatives agree beforehand.
\end{remark}

\subsection{The full induction with tangential derivatives}\label{sec:r2-boundary-induction}

The triangular recovery system and tangential interpolation yield estimates
for all Cartesian boundary derivatives below order $s$, including the
exact-data case.

For two candidates put
\begin{equation}\label{eq:v63-normal-discrepancies}
 n_a=\lVert \partial_\nu^a(\lambda_1-\lambda_2)\rVert _{C^0(\partial\Omega)}
        +\lVert \partial_\nu^a(\mu_1-\mu_2)\rVert _{C^0(\partial\Omega)},
 \qquad 0\le a\le s-1.
\end{equation}
Normal derivatives mean derivatives along the fixed straight normal
segments, as in \eqref{eq:w5-normal-rays}. This definition avoids
unspecified extensions of the normal vector field. Cartesian boundary
derivative discrepancies $J_q$ retain their definition in
\eqref{eq:w0-boundary-jets}, truncated at order $q$.

Proposition~\ref{prop:v63-triangular} and the tangential interpolation
estimate \eqref{eq:v63-tangential-interpolation} yield the following
boundary result, independently of the global CGO construction.

\begin{theorem}
\label{thm:v63-boundary-jet-stability}
For the fixed class $\mathscr C^s(\Omega;M,c_0)$ in
\eqref{eq:w0-interior-class}, there
are $C,\delta_0>0$ such that every two members with $0<\delta<\delta_0$
satisfy, simultaneously for every integer $0\le q<s$,
\begin{equation}\label{eq:v63-jet-stability}
 J_q\le C\delta^{(s-q)/s}.
\end{equation}
Here $J_q$ and $\delta$ are defined in
\eqref{eq:w0-boundary-jets} and \eqref{eq:w0-measured-error}.
The constant depends only on the fixed class and geometry. If $\delta=0$,
all these boundary derivatives through order $s-1$ are equal.
\end{theorem}
\begin{proof}
\textbf{Step 1: the order-zero anchor.}
For the two annular tests the zeroth coefficients are
\begin{equation}\label{eq:r2-calculation-034}
E_{p,0}^S=\mu(p)I_0,\qquad
 E_{p,0}^N=a_Z(p)I_0.
\end{equation}
Here $I_0$ and $a_Z$ are defined in \eqref{eq:v63-moment-definitions}
and \eqref{eq:v62-impedance}. Use Lemma~\ref{lem:v63-extraction} at $q=0$ and the uniformly Lipschitz
inverse \eqref{eq:v62-impedance-inversion}. It follows that $n_0\le C\delta$.
This is a direct Lipschitz boundary-value estimate in the DN operator norm of \eqref{eq:w0-measured-error}.

\textbf{Step 2: estimate mixed derivatives of lower normal order.}
Suppose $n_a\le C\delta^{(s-a)/s}$ for $0\le a<q$. The $a$th normal
boundary derivative is uniformly $C^{s-a}$ in surface coordinates. Fixed
surface-chart interpolation, using cutoffs in slightly larger charts,
gives for $0\le b\le q-a$
\begin{equation}\label{eq:v63-tangential-interpolation}
 \lVert \partial_{y'}^b\partial_\nu^a(f_1-f_2)\rVert _{C^0}
  \le C n_a^{\,1-b/(s-a)}
  \le C\delta^{(s-q)/s},\qquad f=\lambda,\mu.
\end{equation}
The usual additional $Cn_a$ term is absorbed for $n_a\le1$. The bound
uses only $s-a$ tangential derivatives, already furnished by the $C^s$
material prior. In particular the exponents are not replaced at every
stage by one fixed worst-case exponent: the identity
\begin{equation}\label{eq:r2-calculation-035}
\frac{s-a}{s}\left(1-\frac{q-a}{s-a}\right)=\frac{s-q}{s}
\end{equation}
is retained. By \eqref{eq:w5-mixed-normal-trace}, a coordinate derivative of lower normal order
$\partial_{z'}^\gamma\partial_{z_3}^a(f\circ\Psi_p)$ has $a<q$ and
$|\gamma|+a\le q$. It is a tangential derivative of the already
recovered $a$th normal boundary function. Hence every difference in
$\mathcal K_{p,q}^{<}$ is $O(\delta^{(s-q)/s})$, uniformly in $p$.
This is the order of the induction: lower normal functions, their
tangential derivatives through total order $q$, then the new normal
function of order $q$.

Here is an elementary justification of the interpolation inequality used
in Step 2. On a larger fixed surface chart, multiply the relevant trace
by a cutoff equal to one on the smaller chart. Let $F$ be the resulting compactly supported function, and set
\begin{equation}\label{eq:r3-tangential-interpolation-parameters}
N=s-a,\qquad A=\|F\|_\infty.
\end{equation}
Its $C^N$ norm is uniformly bounded. Choose a smooth compactly supported, possibly
signed convolution kernel $K$ with integral one and moments zero through
order $N-1$. Such a kernel is obtained from a fixed bump by a finite
linear combination of its distinct dilates: a Vandermonde system sets
each moment order to zero. Define
\begin{equation}\label{eq:r3-tangential-kernel}
K_\eta=\eta^{-2}K(\cdot/\eta).
\end{equation}
Taylor's formula applied to $\partial^\beta F$, $|\beta|=b<N$, gives
\begin{equation}\label{eq:w5-interpolation-proof}
 \|\partial^\beta(F-K_\eta*F)\|_\infty\le C\eta^{N-b}\|F\|_{C^N},
 \qquad
 \|\partial^\beta K_\eta*F\|_\infty\le C\eta^{-b}A.
\end{equation}
For small $A>0$ take $\eta=A^{1/N}$; for larger $A$ enlarge the fixed
constant. If $A=0$, $F=0$. This proves
$\|\partial^\beta F\|_\infty\le C(A^{1-b/N}+A)$.
Only $N$ derivatives, not $N+1$, were used. In the induction $b\le q-a<N$.

\textbf{Step 3: recover the new normal derivative.}
Subtract \eqref{eq:v63-coefficient-triangular} for the two materials, using
$\mathsf M_q$ of the first material on the left. The difference of the
two matrices is $O(n_0)$ because it depends smoothly only on the boundary
values. The second material's $q$th normal derivative is uniformly bounded.
Consequently Lemma~\ref{lem:v63-extraction}, Step 2 and the uniform inverse
of \eqref{eq:v63-moment-matrix} give
\begin{equation}\label{eq:v63-normal-induction}
 \begin{split}
 n_q&\le C\left[\delta^{(s-q)/s}+n_0
      +\sum_{a=0}^{q-1}n_a^{(s-q)/(s-a)}\right]\\
     &\le C\delta^{(s-q)/s}.
 \end{split}
\end{equation}
Construct the finitely many induction constants before choosing the
noise level, and let $C_*$ bound them. Then one may shrink a single
threshold to ensure both $s\delta^{1/s}<\epsilon_*$ and
$C_*\delta^{1/s}\le1$. This makes every interpolation absorption
valid for all $q<s$ in the same class. No new class or threshold is
chosen after $q$ or a material pair is fixed.

\textbf{Step 4: return to full Cartesian boundary derivatives and handle zero noise.}
For each $a\le q$ apply the same interpolation up to tangential order
$q-a$. For the coordinate conversion, define
\begin{equation}\label{eq:r3-inverse-chart-jacobian}
B(z)=D\Psi_p(z)^{-1},
\end{equation}
where $\Psi_p$ is given by \eqref{eq:w5-explicit-normal-chart}.
For a Cartesian multi-index $\alpha$ with $1\le|\alpha|\le q$,
repeatedly applying $\partial_{x_i}=\sum_bB_{bi}(z)\partial_{z_b}$ gives
\begin{equation}\label{eq:w5-cartesian-conversion}
 (\partial_x^\alpha f)\circ\Psi_p
   =\sum_{1\le|\beta|\le|\alpha|}
        c_{\alpha\beta,p}(z)\partial_z^\beta(f\circ\Psi_p),
 \qquad \|c_{\alpha\beta,p}\|_\infty\le C_{q,\Omega}.
\end{equation}
The coefficients are fixed functions of the coordinate geometry.
At $z_3=0$ each term on the right is, by
\eqref{eq:w5-mixed-normal-trace}, a tangential derivative of a normal
boundary function with total order at most $q$. Every one has just been
bounded; the order-zero term is the anchor. Summing the estimates over the finite atlas and over all $|\alpha|\le q$, we
obtain \eqref{eq:v63-jet-stability}. If $\delta=0$, coefficient extraction follows
by letting $\epsilon\downarrow0$. We repeat the induction from order zero and obtain equality of all the normal
derivatives under consideration. Their surface
functions agree, hence so do their existing tangential derivatives; thus
$J_q=0$ for $q<s$.

\end{proof}

\begin{remark}
The order-dependent interpolation estimates preserve the
exponent $(s-q)/s$ through the induction.
The estimate for $J_{s-1}$ supplies the boundary control
used in Theorem~\ref{thm:v62-compensated-extension}.
Sharpness requires the additional strict-slack hypothesis
of Proposition~\ref{prop:w5-boundary-sharpness}.
\end{remark}

Boundary-localized perturbations with bounded $C^s$ norm give the following
obstruction to improving the exponents.

\begin{proposition}
\label{prop:w5-boundary-sharpness}
Fix the class $\mathscr C^s(\Omega;M,c_0)$ in
\eqref{eq:w0-interior-class}. Suppose there are real
constants $(\lambda_b,\mu_b)$ and a number $\sigma_b>0$ such that
\begin{equation}\label{eq:w5-explicit-prior-slack}
 \begin{gathered}
 \mu_b\ge c_0+\sigma_b,\qquad
 3\lambda_b+2\mu_b\ge c_0+\sigma_b,\\
 N_b:=\|\lambda_b\|_{C^s(\overline\Omega)}+
       \|\mu_b\|_{C^s(\overline\Omega)}\le M-\sigma_b.
 \end{gathered}
\end{equation}
Choose one boundary patch, one point $p$ in it and the profiles
\eqref{eq:r2-sharpness-normal-profile}. There are fixed $a_0,a_{\rm amp}>0$ such that for each integer
$q\in\{0,\ldots,s\}$ the entire family
$(\lambda_{a,q},\mu_b)$, $0<a<a_0$, belongs to this fixed class. Compared
with $(\lambda_b,\mu_b)$, its DN discrepancy and Cartesian
jet discrepancy satisfy
\begin{equation}\label{eq:w5-quantified-sharpness}
 0<\delta_{a,q}\le C_qa_{\rm amp}a^s,\qquad
 J_q\ge c_qa_{\rm amp}a^{s-q}.
\end{equation}
The constants and the amplitude are independent of $a$. Thus for $q<s$
no estimate $J_q\le C\delta^b$ with $b>(s-q)/s$ can hold uniformly on
this fixed class. For $q=s$ no uniform modulus tending to zero with
$\delta$ can bound $J_s$ on this class.
\end{proposition}
\begin{proof}
Use an outward tubular chart $(y',t)$ on a fixed patch, with $t\le0$
inside. Fix $0\le\chi\le1$ supported compactly in the surface patch
and equal to one near $p$. Choose $\chi_{\rm n}\in C_c^\infty((-1,1))$,
$0\le\chi_{\rm n}\le1$, equal to one near zero, and put
\begin{equation}\label{eq:r2-sharpness-normal-profile}
\vartheta_q(\tau)=\chi_{\rm n}(\tau)\frac{(-\tau)^q}{q!},
 \qquad q=0,\ldots,s.
\end{equation}
The cutoff $\chi_{\rm n}$ is fixed independently of $a$;
these profiles enter the perturbations \eqref{eq:v63-sharp-family}.
These are fixed profiles, nonnegative for $\tau\le0$, positive on a
nonempty interior interval, and equal to $(-\tau)^q/q!$ near zero.
Choose $a_0\le1$ so small that $-a_0<t\le0$ remains within the tubular
chart, and define
\begin{equation}\label{eq:v63-sharp-family}
 \lambda_{a,q}-\lambda_b
       =a_{\rm amp} a^s\chi(y')\vartheta_q(t/a),\qquad
 \mu_{a,q}=\mu_b,\qquad 0<a<a_0,
\end{equation}
with zero perturbation off this collar patch. The compact support of
$\chi$ and $\chi_{\rm n}$ makes this a smooth function on $\overline\Omega$.

We verify the a priori bounds uniformly for these families. In chart
coordinates a derivative with $b$ normal derivatives and tangential
multi-index $\gamma$ is
\begin{equation}\label{eq:r2-calculation-037}
a_{\rm amp} a^{s-b}(\partial_{y'}^\gamma\chi)
                         \vartheta_q^{(b)}(t/a),
 \qquad b+|\gamma|\le s.
\end{equation}
The fixed coordinate chain rule expresses every Cartesian derivative
through order $s$ as a finite sum of these terms with fixed bounded
coefficients. Since $a\le1$ and $b\le s$, there is a fixed $K_s\ge1$,
the maximum over the finitely many profiles $q=0,\ldots,s$, such that
\begin{equation}\label{eq:w5-sharpness-prior-budget}
 \|\lambda_{a,q}-\lambda_b\|_{C^s(\overline\Omega)}\le K_sa_{\rm amp}.
\end{equation}
Fix $0<a_{\rm amp}\le\sigma_b/(2K_s)$ before taking any $a\downarrow0$.
The triangle inequality and \eqref{eq:w5-explicit-prior-slack} give
$\|\lambda_{a,q}\|_{C^s}+\|\mu_b\|_{C^s}\le M-\sigma_b/2$.
Moreover $\lambda_{a,q}-\lambda_b\ge0$ and $\mu_{a,q}=\mu_b$, so both
strong-convexity inequalities retain at least the original lower bound
$c_0+\sigma_b$. The amplitude therefore works for the whole family in
the one prescribed class; it is not selected separately for each $a$.

At $p$, the straight-normal convention \eqref{eq:w5-normal-rays} gives
\begin{equation}\label{eq:r2-calculation-038}
\left|\partial_\nu^q(\lambda_{a,q}-\lambda_b)(p)\right|
      =a_{\rm amp}a^{s-q}.
\end{equation}
Expansion of $(\nu(p)\cdot\nabla)^q$ in Cartesian derivatives shows
$|\partial_\nu^q f(p)|\le3^{q/2}\sum_{|\alpha|=q}
|\partial^\alpha f(p)|$. Hence the asserted lower bound for $J_q$ holds
with $c_q=3^{-q/2}$. The same argument includes $q=0$.
The undifferentiated perturbation is at most
$a_{\rm amp}a^s\|\vartheta_q\|_\infty$. The physical forward estimate
\eqref{eq:w5-forward-lipschitz} therefore gives the upper bound for
$\delta_{a,q}$.

Strict positivity of that discrepancy does not require inverse
uniqueness. Prescribe the real affine trace $f(x)=x$ on every boundary
component. The constant-medium solution is $u_b(x)=x$. Write
$u_a=u_b+w$ with $w\in H_0^1(\Omega;\mathbb R^3)$, and denote the
baseline Hermitian form by $\mathfrak a_b$. Its weak equation gives
$\mathfrak a_b(u_b,w)=0$. Consequently the exact energy identity is
\begin{equation}\label{eq:w5-sharpness-positive-data}
 \langle(\Lambda_{a,q}-\Lambda_b)f,f\rangle
   =\mathfrak a_b(w,w)
       +\int_\Omega(\lambda_{a,q}-\lambda_b)
                           |\operatorname{div}u_a|^2\,\mathrm dx>0.
\end{equation}
Both terms are nonnegative. If their sum were zero, coercivity on
$H_0^1$ would force $w=0$; the second term would then equal
$9\int_\Omega(\lambda_{a,q}-\lambda_b)>0$, a contradiction. Since the
trace $f$ has finite $H^{1/2}$ norm, its strictly positive pairing implies
$\delta_{a,q}>0$.

For $q<s$, an improved power would give
$c_qa_{\rm amp}a^{s-q}\le C(C_qa_{\rm amp}a^s)^b$ for all small $a$.
This is impossible when $sb>s-q$. For $q=s$, the positive data errors
tend to zero while $J_s\ge c_sa_{\rm amp}>0$, excluding every modulus
that tends to zero. This proposition concerns boundary sharpness; it
makes no claim about optimality of the interior logarithmic exponent.
\end{proof}

A strict inequality in each of the three prior conditions for a
constant background is equivalent to the existence of a
$\sigma_b>0$ in \eqref{eq:w5-explicit-prior-slack}.

For smooth materials with equal data, the finite-order boundary
result can instead be applied separately at every derivative order.
This yields the following boundary consequence without a uniform
estimate over all orders.

\begin{corollary}\label{cor:w5-smooth-exact-jets}
If both materials are smooth up to the boundary, are strongly convex,
and have identical full DN maps, all their Cartesian boundary
derivatives agree, on every boundary component.
\end{corollary}
\begin{proof}
Fix $q\ge0$ and choose an integer $s>q$, $s\ge2$.
Smoothness on $\overline\Omega$ and strong convexity give finite
$M_s$ and $c_0>0$ such that both materials belong to
$\mathscr C^s(\Omega;M_s,c_0)$ in \eqref{eq:w0-interior-class}.
Equality of the DN maps means $\delta=0$. The exact-data conclusion
of Theorem~\ref{thm:v63-boundary-jet-stability} therefore gives $J_q=0$,
and hence
\[
 \partial^a(\lambda_1-\lambda_2)|_{\partial\Omega}=0,\qquad
 \partial^a(\mu_1-\mu_2)|_{\partial\Omega}=0,\quad |a|\le q.
\]
The boundary norm defining $J_q$ covers every boundary component.
Since $q$ was arbitrary, all finite-order boundary derivatives agree.
The bound $M_s$ may depend on $s$; no uniform estimate as
$q\to\infty$ is needed.
\end{proof}

At one fixed point and scale the coefficient extraction uses $s$ scales
and two polarizations. The point $p$ still ranges over the entire
measurement boundary and the scale is tied to the noise level. This is
an estimate from the full DN operator, not a uniqueness or stability
theorem for a fixed finite set of globally prescribed displacements.

\subsection{The top-order endpoint and extension}
\label{sec:v62-endpoint}
The case $q=s$ of Proposition~\ref{prop:w5-boundary-sharpness}
shows that a bounded $C^s$ class with strict prior slack has no uniform
vanishing modulus for the top-order boundary jet. The extension in
Section~\ref{sec:w5-extensions} therefore uses only $J_{s-1}$, together
with the uniform $C^s$ bound and a positive tolerance. Its construction
preserves the original coefficients without requiring their order-$s$
boundary derivatives to match.

\section{Uniform extensions and DN comparison on an enlarged domain}
\label{sec:w5-extensions}

Uniform paired extensions reduce the original-domain problem to a
common-exterior class; the enlargement estimate then controls the DN
discrepancy on the fixed larger domain.

\subsection{A fixed strongly convex initial extension}\label{sec:r2-initial-extension}
The initial extension to a fixed constant background depends only on the
original admissible class \eqref{eq:w0-interior-class}, independently of
boundary determination.

\begin{lemma}\label{lem:w5-fixed-initial-extension}
There are a fixed ball $B_{R_e}\supset\!\supset\overline\Omega$, a fixed
constant pair $(\lambda_*,\mu_*)$, a fixed number $\eta_e>0$ and $M_0$,
depending only on $s,\Omega,M,c_0$, such that every admissible material has
an extension $(\lambda^0,\mu^0)\in C^s(\mathbb R^3;\mathbb R)^2$ satisfying
\begin{equation}\label{eq:w5-initial-extension}
\begin{gathered}
 (\lambda^0,\mu^0)=(\lambda,\mu)\text{ on }\overline\Omega,\qquad
 \|\lambda^0\|_{C^s}+\|\mu^0\|_{C^s}\le M_0,\\
 \mu^0\ge c_0/2,\qquad 3\lambda^0+2\mu^0\ge c_0/2,\\
 (\lambda^0,\mu^0)=(\lambda_*,\mu_*)
       \text{ on }\mathbb R^3\setminus B_{R_e-\eta_e}.
\end{gathered}
\end{equation}
\end{lemma}
\begin{proof}
Use the tubular map \eqref{eq:w5-normal-rays} for $|t|<a$, with a
fixed $a>0$. Its inverse and derivatives through the orders used in
\eqref{eq:w5-reflection}--\eqref{eq:w5-initial-blend} have fixed bounds
on every boundary component.
The disjoint fixed tubular collars are chosen before a material is
selected; the remaining exterior, including bounded complementary
components, will be assigned the same background. For a scalar $f\in C^s(\overline\Omega)$
and $0<t<a/(s+1)$ define a preliminary reflection by
\begin{equation}\label{eq:w5-reflection}
 F_f(p,t)=\sum_{\ell=1}^{s+1}c_\ell f(\Phi(p,-\ell t)),
 \qquad
 c_\ell=(-1)^{\ell-1}\ell\binom{s+2}{\ell+1}.
\end{equation}
The coefficients are the Lagrange evaluation weights at $1$ for the
nodes $-1,\ldots,-(s+1)$. Consequently
\begin{equation}\label{eq:w5-reflection-moments}
 \sum_{\ell=1}^{s+1}c_\ell(-\ell)^b=1,\qquad 0\le b\le s.
\end{equation}
One can verify the weight formula directly by evaluating
$\prod_{\substack{1\le\ell'\le s+1\\\ell'\ne\ell}}(x+\ell')/(\ell'-\ell)$ at $x=1$.
For any surface multi-index $\gamma$ and normal order $b$ with
$|\gamma|+b\le s$, differentiation of \eqref{eq:w5-reflection} gives the
interior trace $\partial_p^\gamma\partial_t^b(f\circ\Phi)(p,0)$ by
\eqref{eq:w5-reflection-moments}. Each such derivative is continuous up
to $t=0$. The interior function $f$ and the exterior function $F_f$ have matching traces
of all derivatives through order $s$, by \eqref{eq:w5-reflection-moments}.
They therefore define a $C^s$ function across the boundary. The chain rule on fixed charts gives
its norm bound $C_{s,\Omega}\|f\|_{C^s}$.

The reflection weights need not be nonnegative, so positivity is not
inferred from their signs. Instead, their sum is one and the mean-value
formula along the interior normal segments gives
\begin{equation}\label{eq:w5-reflection-closeness}
 |F_f(p,t)-f(p)|
 \le t\|\nabla f\|_\infty
           \sum_{\ell=1}^{s+1}|c_\ell|\ell.
\end{equation}
Apply this inequality to $f=\mu$ and $f=3\lambda+2\mu$.
Choose a fixed $b_0<a/(s+1)$, also $b_0\le1$, so that the right-hand side
is at most $c_0/2$ for all the materials when $0<t<b_0$.
Both elastic inequalities then hold with lower bound $c_0/2$ for the
preliminary reflected pair on this collar.

Fix $(\lambda_*,\mu_*)=(c_0,c_0)$ and a smooth scalar function $\vartheta$
with values in $[0,1]$, equal to one for $t\le b_0/3$ and zero for
$t\ge2b_0/3$. On the exterior collar put
\begin{equation}\label{eq:w5-initial-blend}
 (\lambda^0,\mu^0)(\Phi(p,t))
 =\vartheta(t)(F_\lambda,F_\mu)(p,t)
    +(1-\vartheta(t))(\lambda_*,\mu_*).
\end{equation}
Use the original material inside and the constant pair everywhere else.
Near $t=0$ this is the reflected extension, so all derivatives through
order $s$ still match. Near the far edge of the collar it is constant.
Convexity of the two linear elastic inequalities preserves the lower
bound $c_0/2$. All cutoff derivatives have fixed bounds, so the $C^s$
norm is bounded by a fixed $M_0$. Finally choose $R_e,\eta_e$ so that the
closure of the fixed collar lies strictly inside $B_{R_e-\eta_e}$.
These choices were made before considering either a discrepancy or a
noise level. This proves the lemma.
\end{proof}

Applied separately to two materials, this construction gives
uniform extensions but does not control their exterior
difference by the boundary discrepancy. We obtain that
control in the following paired construction.

\subsection{All derivative orders in a compensated collar}\label{sec:r2-paired-extension}

A compensated collar preserves the full $C^s$ bounds of the extensions while
reducing their exterior discrepancy using boundary derivatives below order
$s$.

Before selecting a pair or a noise level, fix the extension in
Lemma~\ref{lem:w5-fixed-initial-extension}, its background, $B_{R_e}$,
$\eta_e$, $M_0$, and the cutoff used in \eqref{eq:r2-extension-cutoff}. Fix also one number $M_e$ large enough for
the finite bounds in \eqref{eq:v62-compensated-derivatives} after
replacing $e_\partial r_{\rm ext}^{-s}$ by $1$ and $r_{\rm ext}^{s-n},r_{\rm ext}^{s-b}$ by $1$.
This defines $M_e$ from $s,\Omega,M,c_0$ alone;
Theorem~\ref{thm:v62-compensated-extension} verifies the uniform bound.
In particular $M_e$ is fixed before any later common-exterior exponent
$\gamma_{*,\mathrm{ext}}$ is chosen.
For two materials set
\begin{equation}\label{eq:w5-extension-input}
 e_\partial=J_{s-1},\qquad \epsilon_{\rm tol}>0,\qquad
 r_{\rm ext}=(e_\partial+\epsilon_{\rm tol})^{1/s}.
\end{equation}
Here $J_{s-1}$ is defined in \eqref{eq:w0-boundary-jets}.
The symbols $\epsilon_{\rm tol}$ and $r_{\rm ext}$ denote a free positive tolerance
and an exterior collar width; neither is a CGO parameter or the boundary
test scale $\epsilon$.

The positive collar width permits exact boundary matching while suppressing
the exterior difference, with derivative bounds uniform as the collar
shrinks.

\begin{theorem}\label{thm:v62-compensated-extension}
There exist fixed constants $M_e,C,t_{\mathrm{ext},0}>0$ and the ball and background
from Lemma~\ref{lem:w5-fixed-initial-extension} such that, if
\begin{equation}\label{eq:w5-extension-smallness}
 0<e_\partial+\epsilon_{\rm tol}<t_{\mathrm{ext},0},
\end{equation}
the materials have extensions
$(\widetilde\lambda_j,\widetilde\mu_j)\in C^s(\mathbb R^3;\mathbb R)^2$,
$j=1,2$, with exact original values on $\overline\Omega$, common fixed
background outside $B_{R_e-\eta_e}$, $C^s$ norm at most $M_e$, and
strong-convexity lower bounds $c_0/2$. They satisfy
\begin{equation}\label{eq:v62-compensated-error}
 \epsilon_{\rm ext}:=
 \|\widetilde\lambda_1-\widetilde\lambda_2\|_{L^\infty(\mathbb R^3\setminus\overline\Omega)}
 +\|\widetilde\mu_1-\widetilde\mu_2\|_{L^\infty(\mathbb R^3\setminus\overline\Omega)}
 \le C(e_\partial+\epsilon_{\rm tol}).
\end{equation}
The input discrepancy $e_\partial$, tolerance $\epsilon_{\rm tol}$,
and width $r_{\rm ext}$ are defined in \eqref{eq:w5-extension-input}.
All constants and the outer support geometry are independent of the two
materials, $e_\partial$ and $\epsilon_{\rm tol}$ within the stated class.
\end{theorem}
\begin{proof}
Use Lemma~\ref{lem:w5-fixed-initial-extension} for each member, with the
same fixed construction. Define, for either scalar component,
\begin{equation}\label{eq:r2-initial-extension-difference}
 D_f=f_2^0-f_1^0,\qquad f\in\{\lambda,\mu\}.
\end{equation}
This is a $C^s$ scalar function on $\mathbb R^3$ with
norm at most $2M_0$, and its Cartesian boundary derivatives through
order $s-1$ have norm bounded by $e_\partial$.

For a multi-index $\alpha$, $n=|\alpha|\le s-1$, apply the one-dimensional
Taylor formula to $t\mapsto\partial^\alpha D_f(p+t\nu(p))$ through
order $s-1-n$. The vector $\nu(p)$ is held fixed in this differentiation.
Every Taylor coefficient is a linear combination of Cartesian boundary
derivatives of total order at most $s-1$. The integral remainder uses
only derivatives of total order $s$. For $0<t<b_0\le1$ this proves
\begin{equation}\label{eq:v62-low-jet-Taylor}
\begin{aligned}
 |\partial^\alpha D_f(p+t\nu(p))|
 &\le C e_\partial+C M_0t^{s-|\alpha|},&& |\alpha|\le s-1,\\
 |\partial^\alpha D_f|&\le2M_0,&& |\alpha|=s.
\end{aligned}
\end{equation}
No derivative of order $s+1$ occurs in this argument.

Fix a smooth function $\chi:\mathbb R\to[0,1]$ equal to one on
$(-\infty,1/3]$ and zero on $[2/3,\infty)$. Set
\begin{equation}\label{eq:r2-extension-cutoff}
 d_\partial(x)=\operatorname{dist}(x,\partial\Omega),\qquad
 \chi_{r_{\rm ext}}(x)=\chi(d_\partial(x)/r_{\rm ext})
       \quad\text{in the exterior collar},
\end{equation}
and extend $\chi_{r_{\rm ext}}$ by zero beyond that collar. Choose
$t_{\mathrm{ext},0}\le\min\{1,(b_0/4)^s\}$ so that $0<r_{\rm ext}<b_0/4$.
The distance function is smooth here with fixed derivatives. For $b=|\beta|\ge1$, the chain rule is a finite sum of terms
\begin{equation}\label{eq:r2-calculation-039}
r_{\rm ext}^{-j}\chi^{(j)}(d_\partial/r_{\rm ext})
       \prod_{i=1}^{j}\partial^{\gamma_i}d_\partial,
 \qquad 1\le j\le b,\quad |\gamma_i|\ge1,\quad
       \sum_i\gamma_i=\beta.
\end{equation}
All derivatives of $d_\partial$ are fixed and bounded. Since $r_{\rm ext}\le1$, each
$r_{\rm ext}^{-j}\le r_{\rm ext}^{-b}$. For $b=0$ use $0\le\chi_{r_{\rm ext}}\le1$. Thus every
multi-index satisfies
\begin{equation}\label{eq:w5-cutoff-orders}
 |\partial^\beta\chi_{r_{\rm ext}}|\le C_\beta r_{\rm ext}^{-|\beta|},
 \qquad 0\le|\beta|\le s,
 \qquad\operatorname{supp}\chi_{r_{\rm ext}}\subset\{d_\partial<2r_{\rm ext}/3\}.
\end{equation}
For the order-zero bound use $0\le\chi_{r_{\rm ext}}\le1$.
Keep the first extension and, outside $\Omega$, define
\begin{equation}\label{eq:v62-compensated-definition}
 \widetilde f_1=f_1^0,\qquad
 \widetilde f_2=f_1^0+\chi_{r_{\rm ext}}D_f,
 \qquad f=\lambda,\mu.
\end{equation}
Inside $\Omega$ keep each original member unchanged. Since $r_{\rm ext}>0$ and
$\chi_{r_{\rm ext}}=1$ on a neighborhood of the boundary, the second extension
there is exactly its initial extension. This proves $C^s$ matching,
including the top derivative, without assuming that the two top jets
agree. The second pair is a convex combination of the two initial pairs,
so it has the same lower strong-convexity bounds.

For the derivative estimates, let $|\alpha|=n\le s$ and
$\beta\le\alpha$ componentwise, and write $b=|\beta|$.
For $b\ge1$, $n-b\le s-1$ and the support is in $d\le r_{\rm ext}$, so
\begin{equation}\label{eq:v62-compensated-derivatives}
\begin{aligned}
 |\partial^\beta\chi_{r_{\rm ext}}\,\partial^{\alpha-\beta}D_f|
 &\le C r_{\rm ext}^{-b}(e_\partial+M_0r_{\rm ext}^{s-n+b})\\
 &=C\{(e_\partial r_{\rm ext}^{-s})r_{\rm ext}^{s-b}+M_0r_{\rm ext}^{s-n}\}
 \le C(1+M_0),\\
 e_\partial r_{\rm ext}^{-s}&=\frac{e_\partial}{e_\partial+\epsilon_{\rm tol}}\le1.
\end{aligned}
\end{equation}

Both exponents $s-b$ and $s-n$ are nonnegative. The exact rule is
$\partial^\alpha(\chi_{r_{\rm ext}}D_f)=\sum_{\beta\le\alpha}
\binom\alpha\beta\partial^\beta\chi_{r_{\rm ext}}\partial^{\alpha-\beta}D_f$.
Its binomial factors and number of terms depend only on $s$.
The sum of these estimates gives the preselected $C^s$ bound $M_e$, uniformly
in $r_{\rm ext}$. The undifferentiated-cutoff term uses the initial $C^s$ norm and
does not require a small top jet.

Finally, the order-zero line in \eqref{eq:v62-low-jet-Taylor} gives
$|\chi_{r_{\rm ext}}D_f|\le C(e_\partial+M_0r_{\rm ext}^s)
\le C'(e_\partial+\epsilon_{\rm tol})$.
This proves \eqref{eq:v62-compensated-error}. The background and the
fixed outer collar from the initial extension have not moved.
\end{proof}

\begin{remark}
The extensions belong to one fixed $C^s$ class, independently
of the boundary discrepancy and the positive tolerance.
This uniformity permits the subsequent common-exterior
estimates to be applied with fixed constants.
When $e_\partial=0$, the construction remains valid for every
sufficiently small positive $\epsilon_{\rm tol}$; it does not
require a zero-width collar or convergence of the extensions
as the tolerance tends to zero.
\end{remark}

\subsection{Noise transfer on one fixed enlarged class}\label{sec:r2-extension-noise}
Theorem~\ref{thm:v63-boundary-jet-stability} and
Proposition~\ref{prop:v61-enlargement-error} give a bound
for the enlarged-domain error from the measured original-domain error.
The order of choices is important: first fix the original class, the
boundary atlas, the initial extension and the compensated bound $M_e$;
then fix any constants of a common-exterior theorem on that entire
enlarged class; only afterwards choose the noise and positive tolerance.

\begin{corollary}\label{cor:w5-noise-transfer}
For two members of $\mathscr C^s(\Omega;M,c_0)$ in
\eqref{eq:w0-interior-class}, let $C_J$ be the
class-uniform constant of Theorem~\ref{thm:v63-boundary-jet-stability}
at order $s-1$. For $0<\delta<\min\{1,\delta_0\}$, that theorem gives
\begin{equation}\label{eq:w5-I7-input}
 e_\partial=J_{s-1}\le C_J\delta^{1/s}
\end{equation}
with the original-domain DN discrepancy $\delta$. There is a class-uniform threshold
$\delta_{\mathrm{tr},0}\le\min\{1,\delta_0\}$ such that, for
$0<\delta<\delta_{\mathrm{tr},0}$, one may choose
\begin{equation}\label{eq:w5-tolerance-choice}
 \epsilon_{\rm tol}=\delta+e_\partial>0,\qquad
 r_{\rm ext}^s=e_\partial+\epsilon_{\rm tol}=\delta+2e_\partial.
\end{equation}
Then the extensions of Theorem~\ref{thm:v62-compensated-extension}
belong to one fixed enlarged-domain class
and satisfy
\begin{equation}\label{eq:v63-outer-noise}
 \Delta_e\le C(\delta+e_\partial)\le C'\delta^{1/s}.
\end{equation}
Here $\delta$, $e_\partial$, and $\Delta_e$ are defined in
\eqref{eq:w0-measured-error}, \eqref{eq:w5-extension-input}, and
\eqref{eq:w0-domain-error}, respectively.
If $\delta=0$, then $e_\partial=0$, and every
$0<\epsilon_{\rm tol}<t_{\mathrm{ext},0}$ gives extensions
in the same fixed class satisfying
$\Delta_e\le C\epsilon_{\rm tol}$.
\end{corollary}
\begin{proof}
For positive noise, \eqref{eq:w5-I7-input} and $0<\delta<1$ give
\[
 0<e_\partial+\epsilon_{\rm tol}
   =\delta+2e_\partial
   \le(1+2C_J)\delta^{1/s}.
\]
Choose
\[
 0<\delta_{\mathrm{tr},0}<\min\left\{1,\delta_0,
       \left(\frac{t_{\mathrm{ext},0}}{2(1+2C_J)}\right)^s\right\}.
\]
Then \eqref{eq:w5-extension-smallness} holds for
$0<\delta<\delta_{\mathrm{tr},0}$. Apply
Theorem~\ref{thm:v62-compensated-extension}, followed by
Proposition~\ref{prop:v61-enlargement-error}, to obtain
\[
 \begin{aligned}
 \epsilon_{\rm ext}&\le C_E(\delta+2e_\partial),\\
 \Delta_e&\le C_A(\delta+\epsilon_{\rm ext})
       \le C_A\{\delta+C_E(\delta+2e_\partial)\}
       \le C'\delta^{1/s}.
 \end{aligned}
\]
Here $C_E$ is the extension constant and $C_A$ is the enlargement
constant in the cited results. This also gives
$\Delta_e\le C(\delta+e_\partial)$ and proves
\eqref{eq:v63-outer-noise}.

If $\delta=0$, the exact-data conclusion of
Theorem~\ref{thm:v63-boundary-jet-stability} gives $e_\partial=0$.
For every $0<\epsilon_{\rm tol}<t_{\mathrm{ext},0}$,
Theorem~\ref{thm:v62-compensated-extension} and
Proposition~\ref{prop:v61-enlargement-error} give
\[
 \epsilon_{\rm ext}\le C_E\epsilon_{\rm tol},\qquad
 \Delta_e\le C_AC_E\epsilon_{\rm tol}.
\]
The ball, background, $C^s$ bound $M_e$, and strong-convexity bound
are those fixed in Theorem~\ref{thm:v62-compensated-extension}.
They are independent of the positive tolerance, as required.
\end{proof}

The enlarged DN maps are proof quantities, not additional measurements.
If a larger working ball $B_R$ is required by a common-exterior
CGO/Betti argument, fix it with $B_{R_e}\Subset B_R$ before choosing
noise. Both extensions already equal the same background near and
outside $\partial B_{R_e}$; extending them by that background changes
no class bound. Proposition~\ref{prop:v61-enlargement-error} on
$B_{R_e}\subset B_R$ gives
$\|\Lambda_1^{B_R}-\Lambda_2^{B_R}\|\le C\Delta_e$.
Thus $B_{R_e}$ remains the DN comparison domain and $B_R$ the larger
working domain. Choose any later exponent $\gamma_{*,\mathrm{ext}}$
for this fixed class before selecting the noise level.

\section{Original-domain stability and exact-data uniqueness}
\label{sec:w0-assembly}
Boundary recovery and paired extensions transfer
Proposition~\ref{prop:w0-common-assembly} to the original domain. The
enlarged prior class is fixed before the measurement error or tolerance is
chosen.

\subsection{Boundary control and the fixed enlarged class}\label{sec:r2-original-assembly}
Fix an integer $m\ge4$, put $s=2m+14$, and fix
$\mathscr C^s(\Omega;M,c_0)$ as in \eqref{eq:w0-interior-class}.
This section proves the general regularity-dependent estimate;
Theorem~\ref{thm:w0-target-A} is its specialization to $m=4$.
Theorem~\ref{thm:v63-boundary-jet-stability},
Theorem~\ref{thm:v62-compensated-extension}, and
Proposition~\ref{prop:v61-enlargement-error} give the following
original-domain estimates independently of the common-exterior theorem.

By Theorem~\ref{thm:v63-boundary-jet-stability} with $q=s-1$,
\eqref{eq:w5-I7-input} holds for fixed $C_J,\delta_{B,0}>0$, and
$e_\partial=0$ for $\delta=0$. Here $e_\partial$ is defined in
\eqref{eq:w5-extension-input}.

Theorem~\ref{thm:v62-compensated-extension} supplies fixed $B_{R_e}$, $M_e$, strong-convexity bounds, a common
constant exterior, and $C_E,t_{{\rm ext},0}>0$ such that, whenever
\begin{equation}\label{eq:w0-extension-range}
 \epsilon_{\rm tol}>0,\qquad
 0<e_\partial+\epsilon_{\rm tol}<t_{{\rm ext},0},
\end{equation}
the original coefficients have exact interior $C^s$ extensions in this
fixed enlarged class, with
\begin{equation}\label{eq:w0-extension-error}
 \epsilon_{\rm ext}\le C_E(e_\partial+\epsilon_{\rm tol}).
\end{equation}
Here $\epsilon_{\rm ext}$ is defined in \eqref{eq:v62-compensated-error}
and $r_{\rm ext}$ in \eqref{eq:w5-extension-input}.
Only the collar width changes. Theorem~\ref{thm:v62-compensated-extension}
fixes the enlarged norm and convexity bounds and the common background.
The moment and annulus constants for that class, hence
$\gamma_{*,\mathrm{ext}}$, are fixed before choosing the tolerance.

Proposition~\ref{prop:v61-enlargement-error} gives, without exterior agreement
on $\partial\Omega$,
\begin{equation}\label{eq:w0-enlargement}
 \Delta_e\le C_A(\delta+\epsilon_{\rm ext}).
\end{equation}
Here $\Delta_e$ is defined in \eqref{eq:w0-domain-error}.
This comparison follows from the elastic Alessandrini identity
\eqref{eq:w5-physical-alessandrini}. The reduced
Betti identity with a common sphere-neighborhood condition is used in
the separate common-exterior construction, on its larger working sphere.

The extension bound \eqref{eq:w0-extension-error} and DN comparison
\eqref{eq:w0-enlargement} allow application of
Proposition~\ref{prop:w0-common-assembly} with uniform constants.

\begin{proposition}
\label{prop:w0-interior-assembly}
Fix an integer $m\ge4$, put $s=2m+14$, and let
$\mathscr C^s(\Omega;M,c_0)$ be defined by \eqref{eq:w0-interior-class}.
There exists $\gamma_{*,\mathrm{ext}}\in(0,1)$, depending only on
$\Omega,s,M,c_0$, such that, for every
\begin{equation}\label{eq:m4-general-exponent}
 0<\alpha<\frac{\gamma_{*,\mathrm{ext}}(m-3)}{28m},
\end{equation}
there are constants $C_\alpha,\delta_{A,0}>0$, depending only on
$\Omega,s,M,c_0,\alpha$, for which \eqref{eq:w0-target-stability} holds
for every two pairs $(\lambda_j,\mu_j)\in\mathscr C^s(\Omega;M,c_0)$,
$j=1,2$, with $0<\delta<\delta_{A,0}$.
The discrepancy $\delta$ is defined in
\eqref{eq:w0-measured-error}. If $\delta=0$, the two pairs coincide in
$\Omega$.
\end{proposition}

\begin{proof}
Fix $\alpha$ satisfying \eqref{eq:m4-general-exponent} and use the constants
of Proposition~\ref{prop:w0-common-assembly} for the fixed enlarged class
constructed in Subsection~\ref{sec:r2-original-assembly}. Its exterior
condition is \eqref{eq:m4-common-exterior} with $D=B_{R_e}$ and with the
constant background from Lemma~\ref{lem:w5-fixed-initial-extension}.
For positive noise choose the tolerance and width in
\eqref{eq:w5-tolerance-choice}.
For $0<\delta<1$, the boundary estimate gives
\begin{equation}\label{eq:r2-calculation-040}
\delta+2e_\partial\le(1+2C_J)\delta^{1/s}.
\end{equation}
The extension and physical enlargement estimates therefore imply
\begin{equation}\label{eq:w0-outer-noise}
 \Delta_e\le K_e\delta^{1/s},\qquad
 K_e=\max\{1,C_A(1+C_E+2C_EC_J)\}.
\end{equation}
An explicit sufficient original-domain threshold is
\begin{equation}\label{eq:w0-original-threshold}
\begin{aligned}
 \delta_{A,0}=\min\Bigl\{&
 1,\delta_{B,0},
 \left(\frac{t_{{\rm ext},0}}{2(1+2C_J)}\right)^s,\\
 &\left(\frac{\delta_{{\rm com},0}}{2K_e}\right)^s,
 (2K_e)^{-2s}\Bigr\}>0.
\end{aligned}
\end{equation}
For $0<\delta<\delta_{A,0}$, the extension range is admissible and
$\Delta_e<\delta_{{\rm com},0}$. If $\Delta_e=0$, the exact-data
conclusion of Proposition~\ref{prop:w0-common-assembly} gives equality
on $B_{R_e}$, hence on $\Omega$.

If $\Delta_e>0$, let $L=|\log\delta|$. From
\eqref{eq:w0-outer-noise} and \eqref{eq:w0-original-threshold},
\begin{equation}\label{eq:w0-log-comparison}
 |\log\Delta_e|\ge\frac{L}{s}-\log K_e\ge\frac{L}{2s}.
\end{equation}
The common-exterior estimate, with $D=B_{R_e}$ and
$\delta_D=\Delta_e$, then gives
\begin{equation}\label{eq:r2-calculation-041}
\norm{\lambda_1-\lambda_2}_{L^\infty(\Omega)}+
 \norm{\mu_1-\mu_2}_{L^\infty(\Omega)}
 \le C_{\rm com}(2s)^\alpha|\log\delta|^{-\alpha}.
\end{equation}
The factor $1/s$ in the outer-noise power changes the multiplicative
constant in this comparison, not the logarithmic exponent.

If the original $\delta=0$, the exact-data boundary estimate gives
$e_\partial=0$. For every sufficiently small $t>0$, use
$\epsilon_{\rm tol}=t$ in Theorem~\ref{thm:v62-compensated-extension}. The enlarged discrepancy satisfies
\begin{equation}\label{eq:w0-exact-original-noise}
 \Delta_e(t)\le C_AC_Et.
\end{equation}
If it vanishes for any such $t$, equality follows from the common-exterior
zero-error assertion. Otherwise put $K_t=\max\{1,C_AC_E\}$ and take
$t$ small enough that $K_tt<\min\{1,\delta_{{\rm com},0}\}$. For the
fixed original pair, the same uniform common-exterior estimate yields
\begin{equation}\label{eq:r2-calculation-042}
\norm{\lambda_1-\lambda_2}_{L^\infty(\Omega)}+
 \norm{\mu_1-\mu_2}_{L^\infty(\Omega)}
 \le C_{\rm com}|\log(K_tt)|^{-\alpha}\longrightarrow0.
\end{equation}
Only this numerical upper bound is sent to zero. Each extension is used
at a positive tolerance; neither a zero-width extension nor convergence
of the extensions in $C^s$ is assumed.
\end{proof}

Taking $m=4$ in Proposition~\ref{prop:w0-interior-assembly}
yields Theorem~\ref{thm:w0-target-A}.

\subsection{Boundary stability and finite-regularity uniqueness}\label{sec:r2-main-boundary-proofs}

The boundary upper bounds and finite-regularity uniqueness require no strict
prior slack; that additional hypothesis enters only the sharpness assertion.

\begin{proof}[Proof of Theorem~\ref{thm:w0-target-B}]
For each $q=0,\ldots,s-1$, Theorem~\ref{thm:v63-boundary-jet-stability}
on $\Omega$ gives fixed constants $C_q,\delta_q>0$ such that
\[
 J_q\le C_q\delta^{(s-q)/s},\qquad 0<\delta<\delta_q.
\]
Set $C=\max_{0\le q<s}C_q$ and
$\delta_{B,0}=\min_{0\le q<s}\delta_q>0$.
These choices prove \eqref{eq:w0-target-boundary} simultaneously
for all the stated orders. At $\delta=0$, the exact-data conclusion
of the cited theorem gives $J_q=0$ for every $q<s$.
\end{proof}

\begin{proof}[Proof of Proposition~\ref{prop:w0-boundary-sharpness}]
Proposition~\ref{prop:w5-boundary-sharpness} constructs the smooth
families in the prescribed class using the strict slack
\eqref{eq:w0-sharpness-slack}. Absorb its fixed positive amplitude
into $c_q,C_q$. Its estimates \eqref{eq:w5-quantified-sharpness} give
\[
 0<\delta_{a,q}\le C_qa^s,\qquad J_q\ge c_qa^{s-q}.
\]
For $q<s$, set $\vartheta_q=(s-q)/s$.
The inequalities above imply
\[
\frac{J_q}{\delta_{a,q}^{\vartheta_q}}
\ge \frac{c_q}{C_q^{\vartheta_q}}>0.
\]
Since $0<\delta_{a,q}\le C_q a^s\to0$, a uniform estimate
$J_q\le\omega(\delta)$ with
$\omega(t)=o(t^{\vartheta_q})$ would imply
\[
0<\frac{c_q}{C_q^{\vartheta_q}}
\le \frac{\omega(\delta_{a,q})}{\delta_{a,q}^{\vartheta_q}}
\longrightarrow0,
\]
a contradiction. For $q=s$, the same construction gives
$J_s\ge c_s>0$ while $\delta_{a,s}\to0$ through positive values.
An estimate $J_s\le\omega(\delta_{a,s})$ with
$\omega(t)\to0$ as $t\downarrow0$ is therefore impossible.
\end{proof}

\begin{proof}[Proof of Corollary~\ref{cor:w0-target-C}]
For the fixed $C^{22}$ strongly convex materials, choose
\[
 \begin{aligned}
 c_0&=\tfrac12\min_{\substack{j=1,2\\x\in\overline\Omega}}
      \{\mu_j(x),3\lambda_j(x)+2\mu_j(x)\}>0,\\
 M&=1+\max_{j=1,2}
       \bigl(\|\lambda_j\|_{C^{22}(\overline\Omega)}
                 +\|\mu_j\|_{C^{22}(\overline\Omega)}\bigr)<\infty.
 \end{aligned}
\]
The strict strong-convexity inequalities on the compact set
$\overline\Omega$ give $c_0>0$, and the $C^{22}$ regularity makes $M$ finite.
Both pairs then belong to $\mathscr C^{22}(\Omega;M,c_0)$ in
\eqref{eq:w0-interior-class}. Equality of their DN maps gives
$\delta=0$, so the exact-data conclusion of
Proposition~\ref{prop:w0-interior-assembly} with $m=4$ yields
\[
 \|\lambda_1-\lambda_2\|_{L^\infty(\Omega)}
 +\|\mu_1-\mu_2\|_{L^\infty(\Omega)}=0.
\]
The coefficients are continuous, hence agree on $\overline\Omega$.
The constants $M,c_0$ are chosen for this fixed pair and need not be
prescribed in advance.
\end{proof}

\appendix

\section{Algebraic identities and local regularity}
\label{app:w1-algebra}

The computations below verify the exact reduction, its formal transpose, and
the supplementary physical polarization. Interior regularity for each fixed
$h$ then justifies the Betti identity for the exact reduced solutions.

\subsection{Verification of the exact reduction}\label{sec:r2-factorization-proof}

All indices below range from one to three, and repeated indices are
summed. For smooth fields, set
\begin{equation}\label{eq:r3-factorization-auxiliary-fields}
\begin{gathered}
 a=\mu^{-1}=c^2,\qquad \Sigma=\lambda+2\mu,\qquad
 p_f=\Delta f-(\Delta a/a)f,\\
 u^{(v)}=cv,\qquad u^{(f)}=a\nabla f-f\nabla a,\qquad
 T(v,f)=u^{(v)}+u^{(f)}.
\end{gathered}
\end{equation}
Here $c$ and $T$ are defined in \eqref{eq:w1-material-variables}
and \eqref{eq:w1-T-S-M}.

\paragraph{Scalar contribution.}
By the product rule,
\begin{equation}\label{eq:r3-scalar-component-kinematics}
\operatorname{div}u^{(f)}=a\Delta f-f\Delta a=ap_f,\qquad
\varepsilon(u^{(f)})=a\nabla^2f-f\nabla^2a.
\end{equation}
Consequently
\begin{equation}\label{eq:r2-calculation-043}
\sigma_{\lambda,\mu}(u^{(f)})_{ij}=a\lambda p_f\delta_{ij}+2f_{ij}-2\mu f a_{ij}.
\end{equation}
Taking the divergence of the stress tensor in \eqref{eq:r2-calculation-043},
we obtain the component identities
\begin{equation}\label{eq:r2-calculation-044}
(Lu^{(f)})_i=\partial_i(a\lambda p_f)+2\partial_i\Delta f
              -2\mu a_{ij}f_j-2f\mu_j a_{ij}-2\mu f\partial_i\Delta a.
\end{equation}
Since $a\Sigma p_f=a\lambda p_f+2\Delta f-2\mu(\Delta a)f$,
subtract its gradient from this equality. The derivatives
$\partial_i\Delta a$ cancel, leaving
\begin{equation}\label{eq:w1-app-scalar}
 Lu^{(f)}=\nabla(a\Sigma p_f)
  +2\mu(\Delta a I_3-\nabla^2a)\nabla f
  +2\mu^2(\nabla^2a-\Delta a I_3)\nabla a\,f.
\end{equation}
Now
$\nabla a=-ag_\mu$,
$\nabla^2a=a(g_\mu\otimes g_\mu-\nabla^2\ell_\mu)$,
$-\Delta a/a=q_\mu$, and
$\Delta a I_3-\nabla^2a=(c/2)F_\mu$.
Thus the nongradient terms of \eqref{eq:w1-app-scalar} are
$c^{-1}F_\mu\nabla f+c^{-1}F_\mu g_\mu f$, and its gradient is
$\nabla[(\Delta f+q_\mu f)/(1-\kappa)]$.

\paragraph{Vector contribution.}
The expanded Lam\'e operator is
\begin{equation}\label{eq:r2-calculation-045}
(Lu)_i=\mu\Delta u_i+(\lambda+\mu)\partial_i\operatorname{div}u
       +\lambda_i\operatorname{div}u+\mu_j(\partial_j u_i+\partial_i u_j).
\end{equation}
Define
\begin{equation}\label{eq:r3-vector-gradient-potential}
\Psi_v=c(\lambda+\mu)\operatorname{div}v+(\lambda-\mu)c_jv_j.
\end{equation}
To make the first-order cancellation explicit, before subtracting
$\partial_i\Psi_v$, the coefficients of the indicated terms are
\begin{equation}\label{eq:r2-calculation-046}
\begin{aligned}
 [\partial_jv_i]&:\quad 2\mu c_j+c\mu_j=0,\\
 [\partial_iv_j]&:\quad (\lambda+\mu)c_j+c\mu_j,\\
 [\operatorname{div}v]&:\quad (\lambda+\mu)c_i+c\lambda_i.
\end{aligned}
\end{equation}
We subtract the displayed gradient contribution. The term
$(\lambda-\mu)c_j\partial_iv_j$ cancels, and $2\mu c_j+c\mu_j=0$. We also
subtract $\partial_i[c(\lambda+\mu)]\operatorname{div}v$, which leaves
$-c\mu_i\operatorname{div}v=-c^{-1}(g_\mu)_i\operatorname{div}v$.
The mixed second derivatives cancel as well. The remaining zeroth-order
matrix is
\begin{equation}\label{eq:r2-vector-zero-order-coefficient}
Z_{ij}=2\mu c_{ij}+\mu_i c_j+\mu_j c_i
                  +(\mu\Delta c+\mu_kc_k)\delta_{ij}.
\end{equation}
Since $\mu=c^{-2}$, we obtain
\begin{equation}\label{eq:r2-calculation-048}
cZ=2c^{-1}\nabla^2c-4c^{-2}\nabla c\otimes\nabla c
        +(c^{-1}\Delta c-2c^{-2}|\nabla c|^2)I_3=V_\mu.
\end{equation}
The final equality follows by substituting
$\nabla c=-cg_\mu/2$ and
$\nabla^2c=c(-\nabla^2\ell_\mu/2+g_\mu\otimes g_\mu/4)$.
We have therefore proved
\begin{equation}\label{eq:w1-app-vector}
 L(cv)=c^{-1}(\Delta v-g_\mu\operatorname{div}v+V_\mu v)
          +\nabla\big[c(\lambda+\mu)\operatorname{div}v
                          +(\lambda-\mu)\nabla c\cdot v\big].
\end{equation}
The identities
$c(\lambda+\mu)=d/(1-\kappa)$ and
$(\lambda-\mu)\nabla c=-e g_\mu/(1-\kappa)$ convert the last gradient
into the scalar block of $S\mathcal M$. The scalar and vector product-rule identities \eqref{eq:w1-app-scalar} and
\eqref{eq:w1-app-vector} together imply $LT=S\mathcal M$, with all
coefficients given in \eqref{eq:w1-factorization}.

\paragraph{Formal transpose and constant background.}
The formal transpose is taken with respect to the complex-bilinear pairing.

For compactly supported tests,
$\int u^T\nabla(q/(1-\kappa))=-\int q\operatorname{div}u/(1-\kappa)$.
There is no term $-\nabla(1-\kappa)^{-1}\cdot u$ in this transpose.
The divergence of the physical transform is
\begin{equation}\label{eq:r2-calculation-049}
\operatorname{div}T(v,f)=c\operatorname{div}v-\tfrac12c g_\mu\cdot v
                         +\mu^{-1}(\Delta f+q_\mu f).
\end{equation}
These two identities prove \eqref{eq:w1-ST-K} directly.
For constant coefficients $g_\mu=F_\mu=V_\mu=q_\mu=0$ and
$\mathcal M(v,f)=(\Delta v,\Delta f+d\operatorname{div}v)^T$.
Applying the constant-background formula to $(v,f-d_*x\cdot v/2)$ with
$d=d_*$, we obtain $(\Delta v,\Delta f-d_*x\cdot\Delta v/2)^T=\mathsf
R_*\Delta(v,f)$, as used in \eqref{eq:w1-background}.

\paragraph{Local regularity of reduced solutions.}

We justify the local $H^4$ regularity required by the Betti identity
of Lemma~\ref{lem:w1-betti} using interior Laplacian estimates at fixed $h$. This regularity
argument is separate from the uniform two-norm remainder estimates.

If an exact reduced field $\boldsymbol z_h\in H^2$ solves
$\mathcal M\boldsymbol z_h=0$, then
$\Delta \boldsymbol z_h=-\sum A_k\partial_k\boldsymbol z_h-B\boldsymbol z_h\in H^1$ locally.
An interior Laplacian estimate on nested fixed neighborhoods gives
$\boldsymbol z_h\in H^3$; with $C^2$ bounds for $A_k,B$, the right side is then
in $H^2$ and gives $\boldsymbol z_h\in H^4$ on the smaller neighborhood.
The interior estimate follows by cutting off and using the Fourier
estimate for the scalar Laplacian componentwise, with the cutoff
commutators controlled by the $H^2$ estimate for the $H^3$ step,
and by the $H^3$ estimate for the $H^4$ step. Thus no system-specific
regularity theorem or higher-order uniform CGO inverse is needed.
For fixed $h$ this proves the regularity used by Lemma~\ref{lem:w1-betti};
it makes no $h$-uniform $H^4$ or $H^5$ claim.

\subsection{The supplementary physical polarization}\label{sec:r2-supplementary-proof}

The supplementary lift is checked in both transport blocks, including the
transverse cutoff commutators on the entire working domain.

Substitution of $g_\mu=-2\nabla c/c$ in the coefficients gives
\begin{equation}\label{eq:w1-c-coefficients}
 \begin{aligned}
 F_\mu&=-4\nabla^2c-4c^{-1}\nabla c\otimes\nabla c
                 +(4\Delta c+4c^{-1}|\nabla c|^2)I_3,\\
 cV_\mu&=2\nabla^2c-4c^{-1}\nabla c\otimes\nabla c
                 +(\Delta c-2c^{-1}|\nabla c|^2)I_3,\\
 F_\mu-cV_\mu-3(\nabla(cg_\mu))^T
       &=(3\Delta c+6c^{-1}|\nabla c|^2)I_3.
 \end{aligned}
\end{equation}
Indeed $cg_\mu=-2\nabla c$, so the third matrix is symmetric and all
Hessian terms cancel. Contracting the last identity in \eqref{eq:w1-c-coefficients} with $\theta^T$
and $q_\theta$, and using $\theta\cdot q_\theta=0$, we obtain
\begin{equation}\label{eq:w1-slow-contracted-coefficient}
 \theta^TF_\mu q_\theta=c\theta^TV_\mu q_\theta
                         +3D_\theta(cg_\mu\cdot q_\theta).
\end{equation}
Contracting the first identity in \eqref{eq:w1-c-coefficients} with $\theta$
in both slots, we obtain $-c\theta^TF_\mu\theta/2=D_\theta^2(c^2)=\alpha_\mu$.
The contraction uses nullness to remove the scalar multiple of $I_3$.

For $\boldsymbol z_0^{\rm sup}=(q_\theta,0)^T$, $\boldsymbol z_1^{\rm sup}=(v_{\rm sup},f_{\rm sup})^T$ and
$R_{\rm col}=c\theta\cdot v_{\rm sup}$, the scalar component of
$\mathcal K\boldsymbol z_0^{\rm sup}-J_\theta \boldsymbol z_1^{\rm sup}$ is
\begin{equation}\label{eq:r2-calculation-050}
\frac{\tfrac12c g_\mu\cdot q_\theta+R_{\rm col}+2c^2D_\theta f_{\rm sup}}{1-\kappa}.
\end{equation}
The scalar equation in \eqref{eq:w1-slow-two-equations} replaces
$2c^2D_\theta f_{\rm sup}$ by
$-\kappa R_{\rm col}-c^2e\,g_\mu\cdot q_\theta$.
Since $c/2-c^2e=3c(1-\kappa)/2$, this is precisely
$R_{\rm col}+3cg_\mu\cdot q_\theta/2$.
For a general $(p,\eta)^T$ the ambient dual equation has blocks
$2D_\theta p+\theta(g_\mu\cdot p-d\eta)$ and
$2D_\theta\eta-\theta^TF_\mu p$.
Insert $p=q_\theta-cf_{\rm sup}\theta$ and the scalar just computed.
The $D_\theta c=-c g_\mu\cdot\theta/2$ terms cancel in the first block,
leaving exactly the first line of \eqref{eq:w1-slow-defects}.
For the second block, use \eqref{eq:w1-slow-contracted-coefficient}
and $c\theta^TF_\mu\theta=-2\alpha_\mu$ to obtain exactly its second
line. This proves the supplementary lift and both transport equations.

In the basis $(n_\theta,\xi_q)$ of the upper two coordinates, $Y$ has
coordinate matrix
\begin{equation}\label{eq:r2-calculation-051}
\begin{pmatrix}-2cf_1&-2cf_2&2cf_{\rm sup}\\0&0&1\\-R_1&-R_2&\eta_{\rm sup}\end{pmatrix}.
\end{equation}
Its determinant is $-2c(f_1R_2-f_2R_1)$.
Multiplication by $\det(n_\theta,\xi_q)$ proves
\eqref{eq:w1-physical-Y} for every value of $f_{\rm sup},\eta_{\rm sup}$.
The geometric minor bound used for the paired lift is also explicit:
\begin{equation}\label{eq:r2-spinor-two-row-minors}
q_{12}=i s_0s_1/2,\qquad
 q_{13}=(s_0^2+s_1^2)/4,\qquad
 q_{23}=i(s_0^2-s_1^2)/4.
\end{equation}
After taking squared moduli and adding the expressions, we obtain
\eqref{eq:w1-minor-bound}; the mixed phase terms cancel. The formula also
holds on both chart axes.

For the localized supplementary column one must also keep the
cutoff term away from the test support. Since
$D_\theta\chi=0$, $\operatorname{div}(\chi q_\theta)
=\chi'(\nu_{\rm leaf}\cdot q_\theta)$, and the scalar component of
$\boldsymbol z_0^{\rm sup}$ is zero, \eqref{eq:w1-ST-K} gives
\begin{equation}\label{eq:w1-app-localized-lift}
 \begin{aligned}
 [\mathcal K,\chi]\boldsymbol z_0^{\rm sup}
   &=-\binom{0_{\mathbb C^3}}
                {c\chi'(\nu_{\rm leaf}\cdot q_\theta)/(1-\kappa)},\\
 [J_\theta,\chi]\boldsymbol z_1^{\rm sup}&=0,\\
 \mathcal K a_0^{\rm sup}-J_\theta a_1^{\rm sup}
   &=\chi\mathcal Y_{\rm sup}-J_\theta w_\chi
          +[\mathcal K,\chi]\boldsymbol z_0^{\rm sup}.
 \end{aligned}
\end{equation}
This proves \eqref{eq:w1-localized-leading-matrix} on the whole
working domain. The fixed leaf inverse defining $w_\chi$ preserves
transverse support, so $w_\chi$ and its available derivatives vanish
on every open interval where $\chi$ is constant. In particular the
localized leading column is exactly $\mathcal Y_{\rm sup}$ near the test
support, while remaining uniformly bounded off that support.

\section{Finite-order frames, free inversion, and the parametrix}
\label{app:w2-parametrix}
\label{sec:analytic-budgets}

This appendix proves the frame and local right-inverse theorems of
Section~\ref{sec:right-inverse}, with explicit derivative requirements. Both
Sobolev bounds are obtained from one Neumann inverse.

Throughout this appendix, norms of finite matrices can be any fixed
submultiplicative norm. Constants absorb the fixed equivalence constants.
All products involving matrix symbols are kept in their displayed order.
The coefficient label $\mathfrak a$ is a parameter in a compact metric
space; no coefficient-label derivatives will be taken.

\subsection{The finite-order planar input}\label{sec:r2-planar-estimates}

\begin{lemma}
\label{lem:w2-cauchy}
For compactly supported functions in a fixed disk of the complex plane,
set
\begin{equation}\label{eq:local-cauchy-contraction}
 (\mathsf C F)(z)=\frac1\pi\int_{\mathbb C}
                          \frac{F(w)}{z-w}\,dA(w).
\end{equation}
It acts entrywise on matrices. If $F\in C^{r,\alpha_{\rm H}}$ with that
support, then on any fixed bounded disk
\begin{equation}\label{eq:w2-cauchy-schauder}
 \partial_{\bar z}\mathsf C F=F,
 \qquad
 \lVert\mathsf C F\rVert_{C^{r+1,\alpha_{\rm H}}}
                  \leq C\lVert F\rVert_{C^{r,\alpha_{\rm H}}}.
\end{equation}
There is also a bound from compactly supported $L^\infty$ into
$C^{0,\alpha_{\rm H}}$ on fixed disks, for every $\alpha_{\rm H}<1$.
Derivatives in additional real parameters commute with this integral;
consequently the estimate applies separately to every mixed derivative
whose indicated input norm is available. It gains a leaf derivative,
not a transverse derivative.
\end{lemma}
\begin{proof}
The distributional identity
$\partial_{\bar z}(1/(\pi z))=\delta_0$ follows by integration by parts
outside a disk of radius $\varepsilon$ and passage to the limit in its
boundary integral. The normalization is thus fixed.
The integral of $|z-w|^{-1}$ over a disk of radius $a$ is at most $Ca$,
uniformly in $z$. For two points a distance $t$ apart, split the difference
of the kernels into the union of their $2t$-disks and its complement.
The first part is bounded by $C\lVert F\rVert_\infty t$; on the second,
the kernel difference is bounded by $Ct|z-w|^{-2}$. Radial integration
bounds this part by
$C\lVert F\rVert_\infty t(1+|\log t|)$ on fixed disks. Since
$t(1+|\log t|)\leq C_{\alpha_{\rm H}}t^{\alpha_{\rm H}}$, this proves
the $L^\infty$ assertion.

For the first derivative, $\partial_z\mathsf C F$ is the principal-value
integral with kernel $-1/(\pi z^2)$. Since the kernel has zero integral on centered annuli, we may replace the
local integrand $F(w)$ by $F(w)-F(z)$. The resulting integral is bounded by
$C[F]_{\alpha_{\rm H}}\int_0^R s^{\alpha_{\rm H}-1}\,ds$.
For its difference at two points separated by $t$, split at radius $4t$.
The near part is bounded by $C[F]_{\alpha_{\rm H}}t^{\alpha_{\rm H}}$.
In the common far region, subtract the value at the first point and
shift the second integration variable to put the two kernels on the
same domain. The kernel difference is bounded by $Ct s^{-3}$ and the
subtracted function by $[F]_{\alpha_{\rm H}}s^{\alpha_{\rm H}}$.
The resulting bound is
\begin{equation}\label{eq:r2-calculation-053}
C[F]_{\alpha_{\rm H}}t\int_{4t}^{R}
                         s^{\alpha_{\rm H}-2}\,ds
 \leq C_{\alpha_{\rm H}}[F]_{\alpha_{\rm H}}t^{\alpha_{\rm H}}.
\end{equation}
The annular pieces introduced by the shift have area $O(ts)$ at radius
$s$; the same subtraction and annular cancellation give the identical
bound at the inner boundary and $C\lVert F\rVert_{C^{0,\alpha_{\rm H}}}t$
at the fixed outer boundary. The nonsingular part outside a fixed disk
is bounded and Lipschitz on the observation disk. This proves the
$C^{0,\alpha_{\rm H}}$ estimate for $\partial_z\mathsf C F$.
Together with $\partial_{\bar z}\mathsf C F=F$, it proves
\eqref{eq:w2-cauchy-schauder} for $r=0$.
For higher $r$, real distributional derivatives commute with convolution;
apply the $r=0$ estimate to $\partial^\gamma F$, $|\gamma|\leq r$.
A fixed cutoff reduces interior variants to this compact-support version.
Difference quotients in the extra parameters and the same estimates
justify differentiation under the integral and then its iteration.
\end{proof}

\subsection{Frames at closure points and a uniform finite atlas}\label{sec:r2-proof-frames}

\begin{proof}[Proof of Theorem~\ref{thm:finite-frames}]
\textit{Local existence without a smoothness assumption.}
On a disk of radius $2a$ choose a cutoff equal to one on a smaller disk
and solve
\begin{equation}\label{eq:w2-small-disk-frame}
 \mathsf F_i=I+\mathsf C(\chi_i\mathsf B \mathsf F_i).
\end{equation}
On bounded continuous matrices the linear part has norm at most
$Ca\lVert\mathsf B\rVert_\infty$. Choose $a$ so that this is less than
$1/4$. Its Neumann series gives
$\lVert \mathsf F_i-I\rVert_\infty\leq1/3$, hence invertibility. The
$L^\infty$ part of Lemma~\ref{lem:w2-cauchy} first gives
$\mathsf F_i\in C^{0,\alpha_{\rm H}}$; its Schauder part then gives
$\mathsf F_i\in C^{1,\alpha_{\rm H}}$. Iterating the Schauder estimate of Lemma~\ref{lem:w2-cauchy} in
\eqref{eq:w2-small-disk-frame}, we obtain $\mathsf F_i\in C^{r+1,\alpha_{\rm
H}}$ on smaller disks. The construction applies
to finite-H\"older-regularity coefficients themselves, not only to their
smooth approximations.

\textit{Existence on a full disk.}
Choose an open disk $X$ with $\overline D\Subset X\Subset U$, and
cover $X$ by finitely many of the smaller disks just constructed; this
can be done by first covering $\overline X$ with disks lying in $U$.
On overlaps define
\begin{equation}\label{eq:r3-local-holomorphic-transitions}
g_{ij}=\mathsf F_i^{-1}\mathsf F_j.
\end{equation} Differentiation of the inverse gives
\begin{equation}\label{eq:r2-holomorphic-transition-equation}
\partial_{\bar z}g_{ij}
 =-\mathsf F_i^{-1}\mathsf B \mathsf F_i\mathsf F_i^{-1}\mathsf F_j
      +\mathsf F_i^{-1}\mathsf B \mathsf F_j=0.
\end{equation}
These $C^1$ functions are holomorphic, and
$g_{ij}g_{jk}=g_{ik}$. Glue the trivial complex manifolds
$U_i\times GL(d_{\rm sys},\mathbb C)$ by the holomorphic maps
$(z,h)_j\mapsto(z,g_{ij}(z)h)_i$. This is a locally trivial holomorphic
fiber bundle over the open disk $X$; the transition maps are
biholomorphic left multiplications on its fiber. The local continuous maps
\begin{equation}\label{eq:r3-frame-bundle-continuous-section}
h_i=\mathsf F_i^{-1},\qquad h_i=g_{ij}h_j,
\end{equation}
define a continuous section.

The precise external input is the existence consequence of
Forstneri\v c--Prezelj~\cite[Definitions 1.1--1.2 and Theorem 1.3]{FP}:
a continuous section of a locally trivial holomorphic fiber bundle over
a Stein manifold is homotopic to a holomorphic section when its fiber
admits a dominating holomorphic spray. Here the open disk $X$ is Stein.
The trivial vector bundle over the fiber with fiber
$\mathbb C^{d_{\rm sys}\times d_{\rm sys}}$ has the spray
$(g,Z)\mapsto g\exp Z$; its vertical differential at $Z=0$ is the
surjective map $Z\mapsto gZ$. All the hypotheses therefore hold.
Write the resulting holomorphic section in the local coordinates as
$\widetilde h_i$. Then $\mathsf F_i\widetilde h_i$ agree on overlaps, are
invertible, and solve $\partial_{\bar z}\mathsf F=\mathsf B\mathsf F$ on $X$.
On a neighborhood of the compact disk $\overline D$ they have the stated
finite regularity. This external input is solely qualitative. Uniform
bounds and parameter derivatives will follow from the local
perturbation construction, without a parametric Oka theorem.

\textit{Perturbation around every finite-regularity center.}
For clarity choose fixed disks
$D\Subset D_1\Subset D_2\Subset U$ within the original margin. A bounded
joint $C^{r+2}$ family, restricted to fixed slightly smaller neighborhoods,
has compact closure in $C^{r+1,\alpha_{\rm H}}$. Indeed, in each fixed
coordinate chart Arzel\`a--Ascoli gives a subsequence converging
uniformly through order $r+1$. For the differences $v$ of derivatives
of order $r+1$, their common Lipschitz bound and the interpolation
inequality
\begin{equation}\label{eq:r2-calculation-055}
[v]_{\alpha_{\rm H}}
 \leq C\|v\|_\infty^{1-\alpha_{\rm H}}
               \operatorname{Lip}(v)^{\alpha_{\rm H}}
\end{equation}
give convergence in that H\"older topology. Finitely many charts and
the fixed neighborhood margins give the joint assertion, including
parameter variables. Work on that entire
closure and the compact parameter set, not just on smooth members.
More explicitly, the coefficient labels are the joint functions on the
fixed spatial and parameter neighborhoods, with the
$C^{r+1,\alpha_{\rm H}}$ topology. Evaluation at $p$ and all parameter
derivatives through order $r$ are continuous into the corresponding
leaf H\"older spaces. An original material label enters through its
continuous coefficient map.
At a center $(\mathfrak a_0,p_0)$, the small-disk construction
\eqref{eq:w2-small-disk-frame} and the holomorphic transitions
\eqref{eq:r2-holomorphic-transition-equation} give an
invertible frame $\mathsf F_0$ on a neighborhood of $\overline D_2$.
Take one fixed cutoff $\chi=1$ near $\overline D_1$ and supported in $D_2$.
Put
\begin{equation}\label{eq:finite-frame-neumann}
 \begin{split}
 V_{\mathfrak a,p}
  &=\mathsf F_0^{-1}\bigl(\mathsf B^{\mathfrak a}(\cdot,p)
                  -\mathsf B^{\mathfrak a_0}(\cdot,p_0)\bigr)\mathsf F_0,\\
 \mathsf K_{\mathfrak a,p}u&=\mathsf C(\chi V_{\mathfrak a,p}u),
 \qquad \Psi=(I-\mathsf K_{\mathfrak a,p})^{-1}I,
 \qquad \mathsf F=\mathsf F_0\Psi .
 \end{split}
\end{equation}
Solve on $C^{0,\alpha_{\rm H}}(\overline D_2)$, with the integral extended
outside by its formula. Continuity of $\mathsf B$ into that space allows
a coefficient--parameter neighborhood on which
$\lVert\mathsf K_{\mathfrak a,p}\rVert<1/4$ and
$\lVert\Psi-I\rVert_\infty<1/2$. Then $\mathsf F$ is invertible and
$\partial_{\bar z}\mathsf F=\mathsf B^{\mathfrak a}\mathsf F$ on $D_1$.

For a nonzero parameter multi-index $\beta$, difference quotients in
\eqref{eq:finite-frame-neumann}, followed by induction, give
\begin{equation}\label{eq:parameter-differentiation}
 (I-\mathsf K)\partial_p^\beta\Psi
 =\sum_{\gamma<\beta}\binom\beta\gamma
       (\partial_p^{\beta-\gamma}\mathsf K)
                       \partial_p^\gamma\Psi.
\end{equation}
The inequality $\gamma<\beta$ means componentwise $\gamma\leq\beta$ and
$\gamma\ne\beta$. The inverse on the left is the same bounded
$C^{0,\alpha_{\rm H}}$ inverse for every $\beta$.
Thus all parameter derivatives through order $r$ are bounded in that
space. Differentiate the equation
$\partial_{\bar z}\Psi=V\Psi$ in $p$, and use
Lemma~\ref{lem:w2-cauchy} on fixed nested disks. Induction on the leaf
order then bounds $\partial_p^\beta\Psi$ in
$C^{r-|\beta|+1,\alpha_{\rm H}}$ near $\overline D$.
At each step the product on the right has total input order at most
$r+1$; the joint $C^{r+1,\alpha_{\rm H}}$ closure supplies it. No
rectangular requirement of $r$ derivatives in each variable is used.
Only finitely many nested disks are needed for these inductions:
for a fixed requested $r$, choose them between $D$ and $D_1$ before
any parameter is varied. Their inverse separation constants are
therefore fixed. This prevents the local Schauder estimates from
introducing a parameter-dependent boundary margin.
The identity $\mathsf F\mathsf F^{-1}=I$ and its differentiated versions give the inverse
bounds; for example
$\partial_j \mathsf F^{-1}=-\mathsf F^{-1}(\partial_j\mathsf F)\mathsf F^{-1}$, with ordered products
at every higher order.

The Neumann construction \eqref{eq:finite-frame-neumann} also proves the coefficient-label continuity used
in the physical construction. If $\mathfrak a_n\to\mathfrak a$ within this neighborhood,
then $\mathsf K_n\to\mathsf K$ in operator norm on
$C^{0,\alpha_{\rm H}}$, uniformly on a compact smaller parameter chart.
The resolvent identity reads
\begin{equation}\label{eq:r2-frame-resolvent-identity}
(I-\mathsf K_n)^{-1}-(I-\mathsf K)^{-1}
 =(I-\mathsf K_n)^{-1}(\mathsf K_n-\mathsf K)(I-\mathsf K)^{-1}.
\end{equation}
Its two inverse factors have norm at most $4/3$. Thus $\Psi_n\to\Psi$
in that space. Apply \eqref{eq:parameter-differentiation} successively
to obtain the same convergence for every parameter derivative through
order $r$. Applying the fixed-disk Cauchy estimate of Lemma~\ref{lem:w2-cauchy} to the
difference of the two frame equations, we obtain convergence for all total
mixed orders through $r$ on the smaller cylinders. The inverse identities give the
same continuity for $\mathsf F^{-1}$. No derivative of a coefficient label has
been used, and no neighborhood depends on $h$.

\textit{One atlas for the full prior.}
Apply the coefficient--parameter neighborhoods from
\eqref{eq:finite-frame-neumann} to the compact closure of the
\emph{whole} prior-bounded class, times $\mathfrak P$, and take a finite
subcover. The maximum of the finitely many bounds is $C_r$.
Refine to nested product neighborhoods. Choose continuous coefficient
bumps $\eta_\ell$ supported in the larger coefficient neighborhoods and
smooth parameter bumps $\varphi_\ell$ supported compactly in the larger
parameter charts, positive on their smaller covering sets. Set
\begin{equation}\label{eq:uniform-parameter-weights}
 w_\ell(\mathfrak a,p)=
 \frac{\eta_\ell(\mathfrak a)\varphi_\ell(p)}
 {\left(\sum_j\eta_j(\mathfrak a)^2\varphi_j(p)^2\right)^{1/2}}.
\end{equation}
The denominator is bounded below on the compact set. Its parameter
derivatives, and therefore those of every $w_\ell$, are uniformly
bounded. The products $w_\ell \mathsf F_\ell$ and $w_\ell \mathsf F_\ell^{-1}$ extend by
zero outside their charts because their supports lie strictly inside
the larger chart domains. A weight depending only on direction and
transverse position has zero derivative along the corresponding leaf.
For finitely many transport systems, take the product of their compact
coefficient closures with the same compact parameter set. Pull back
each system's cover and take all finite intersections, then choose nested
product neighborhoods and the coefficient and parameter cutoffs in
\eqref{eq:uniform-parameter-weights} on this common refinement.
The refinement is finite, and its largest seminorm and smallest fixed
margin are still prior-dependent constants. In particular it is fixed
before selecting any polynomial test, width, or $h$.
This proves all assertions. These oriented-pair weights are not asserted
to descend to the projectivized direction space.
\end{proof}

\subsection{Free inverse and weighted estimates}\label{sec:r2-free-inverse}

The free inverse satisfies a small same-order bound and a bound gaining one
derivative in weighted Sobolev spaces. Its full spatial output is retained
for the variable-coefficient parametrix.

For the rest of the appendix use the local notation
\begin{equation}\label{eq:frequency-convention}
 \begin{gathered}
 D_x=-i\nabla,\qquad \mathsf A_j=\widetilde A_j/(2i),\qquad
 \mathsf V=-\widetilde B,\qquad \tau=h^{-1},\\
 \xi_{\rm c}=i\tau\theta=\tau(\omega_0+i\nu),\qquad
 \omega_0=-e_2,\quad \nu=e_1,\\
 P_{\xi_{\rm c}}=(D_x+\xi_{\rm c})^2 I_{d_{\rm sys}}
       +2\sum_j \mathsf A_j(D_{x,j}+\xi_{{\rm c},j})+\mathsf V
       =-L_{h,\theta}.
 \end{gathered}
\end{equation}
Here $\mathsf A=(\mathsf A_1,\mathsf A_2,\mathsf A_3)$ is a triple of
$d_{\rm sys}\times d_{\rm sys}$ matrices, and $\mathsf V$ is a matrix potential.

For a Euclidean vector $x$, set
\begin{equation}\label{eq:r2-japanese-bracket}
 \langle x\rangle=(1+|x|^2)^{1/2}.
\end{equation}
For an integer $j\geq0$ and a real number $\sigma$, define
\begin{equation}\label{eq:w2-weighted-spaces}
 \lVert u\rVert_{H^j_\sigma}^2
 =\sum_{|\alpha|\leq j}\int_{\mathbb R^3}
          \langle x\rangle^{2\sigma}|\partial^\alpha u(x)|^2\,dx.
\end{equation}
Here $\sigma$ is only a weight exponent. The free operator and its inverse
are
\begin{equation}\label{eq:w2-free-symbol}
 P^0_{\xi_{\rm c}}=(D_x+\xi_{\rm c})^2,
 \quad p_{\xi_{\rm c}}(\xi)=|\xi|^2+2\xi_{\rm c}\cdot\xi,
 \quad E_{\xi_{\rm c}}=p_{\xi_{\rm c}}(D_x)^{-1}.
\end{equation}
The reciprocal is its locally integrable distribution, not a symbol
obtained by differentiating across its characteristic set.

\begin{lemma}
\label{lem:w2-free}
For $\tau\geq1$ and every integer $j\geq0$, the multiplier in
\eqref{eq:w2-free-symbol} has the two bounds
\begin{equation}\label{eq:free-two-estimates}
 \lVert E_{\xi_{\rm c}}f\rVert_{H^j_{-1}}
       \leq C_j\tau^{-1}\lVert f\rVert_{H^j_1},
 \qquad
 \lVert E_{\xi_{\rm c}}f\rVert_{H^{j+1}_{-1}}
       \leq C_j\lVert f\rVert_{H^j_1}.
\end{equation}
It satisfies $P^0_{\xi_{\rm c}}E_{\xi_{\rm c}}f=f$ in distributions.
Constants are uniform over all oriented orthonormal pairs
$(\omega_0,\nu)$. Its output is not assumed compactly supported.
\end{lemma}
\begin{proof}
Rotations reduce the proof to $\omega_0=(1,0,0)$ and $\nu=(0,1,0)$.
Write $u=e^{-i\tau x_1}v$, let $t=x_2$, and take the Fourier transform
in $y=(x_1,x_3)$. If $\eta\in\mathbb R^2$ and $b=|\eta|$, then
\begin{equation}\label{eq:w2-factor-free}
 e^{i\tau x_1}P^0_{\xi_{\rm c}}e^{-i\tau x_1}
  \quad\longleftrightarrow\quad
  -\partial_t^2+2\tau\partial_t+b^2-\tau^2
  =-(\partial_t-\tau-b)(\partial_t-\tau+b).
\end{equation}
For real $a\ne0$, put
\begin{equation}\label{eq:w2-one-dimensional-kernels}
 G_a(t)=
 \begin{cases}-\boldsymbol 1_{t<0}e^{at},&a>0,\\
               \boldsymbol 1_{t>0}e^{at},&a<0,
 \end{cases}
 \qquad K_b=-G_{\tau+b}*G_{\tau-b}.
\end{equation}
The jump at zero proves $(\partial_t-a)G_a=\delta_0$.
Since $\lVert G_{\tau+b}\rVert_{L^1}=(\tau+b)^{-1}$ and
$\lVert G_{\tau-b}\rVert_\infty=1$, convolution gives
\begin{equation}\label{eq:w2-kernel-bounds}
 |K_b(t)|\leq(\tau+b)^{-1},\qquad
 b|K_b(t)|\leq1,\qquad |K_b'(t)|\leq2.
\end{equation}
For the last inequality, differentiate distributionally:
$K_b'=-G_{\tau+b}+(\tau-b)K_b$ and use
$|\tau-b|\leq\tau+b$. The convolution is continuous, so no additional
Dirac mass occurs in this first derivative. The exceptional set $b=\tau$
has measure zero in the partial Fourier variables.

If a convolution kernel $H(t-s)$ is bounded by $B$, multiplication on
both sides by $\langle t\rangle^{-1}$ gives the Hilbert--Schmidt bound
\begin{equation}\label{eq:w2-hilbert-schmidt}
 \left\|\langle t\rangle^{-1}H*
                       \langle t\rangle^{-1}\right\|_{L^2\to L^2}
 \leq B\left(\int_{\mathbb R}\langle t\rangle^{-2}\,dt\right)=\pi B.
\end{equation}
Apply it to $K_b$, $\eta_1K_b$, $\eta_2K_b$, and $K_b'$, and integrate in
$\eta$ by Plancherel. When we undo the modulation, the $x_1$ derivative acquires the term $\tau
K_b$, which satisfies the same unit bound. Hence
\begin{equation}\label{eq:r2-calculation-057}
\lVert\langle t\rangle^{-1}u\rVert_2
       \leq C\tau^{-1}\lVert\langle t\rangle f\rVert_2,
 \qquad
 \sum_{|\alpha|\leq1}\lVert\langle t\rangle^{-1}
                         \partial^\alpha u\rVert_2
       \leq C\lVert\langle t\rangle f\rVert_2.
\end{equation}
The inequalities $\langle x\rangle^{-1}\leq\langle t\rangle^{-1}$ and
$\langle t\rangle\leq\langle x\rangle$ give the desired radial-weight
bounds for $j=0$.

For completeness the inverse just constructed is the specified
multiplier. Fourier transformation of $K_b$ in $t$ gives, for $b\ne\tau$,
\begin{equation}\label{eq:r2-one-dimensional-green-multiplier}
\widehat K_b(s)=
  -\frac1{(is-\tau-b)(is-\tau+b)}
  =\frac1{s^2+2i\tau s+b^2-\tau^2}.
\end{equation}
Near its zero set $b=\tau$, $s=0$, the real variables
$(b,s)\mapsto(b^2-\tau^2+s^2,2\tau s)$ have nonzero Jacobian.
The reciprocal is therefore locally integrable, like
$(u+iv)^{-1}$ in its two transverse variables. Fubini against Schwartz
test functions identifies the partial Fourier inverse with this
locally integrable reciprocal; the uniform kernel bound permits passage
through $b=\tau$. No distribution supported on the characteristic set
is added. Multiplication by $p_{\xi_{\rm c}}$ then proves the inverse
identity.

The original multiplier commutes with every constant-coefficient
spatial derivative. Apply the $j=0$ estimates to $\partial^\alpha f$,
$|\alpha|\leq j$, and sum. This proves both estimates at order $j$
without differentiating the modulation of the input. Density in the
weighted spaces completes the extension. Rotations preserve radial
weights and give uniformly equivalent integer Sobolev norms, proving
the directional uniformity.
\end{proof}

\subsection{Finite seminorm operator estimates}\label{sec:r2-operator-estimates}

We use left quantization
\begin{equation}\label{eq:w2-left-quantization}
 \operatorname{Op}(a)u(x)=(2\pi)^{-3}\int_{
 \mathbb R^3}e^{ix\cdot\xi}a(x,\xi)\widehat u(\xi)\,d\xi,
 \qquad \widehat u(\xi)=\int e^{-ix\cdot\xi}u(x)\,dx.
\end{equation}
We use the derivative order $N$ in \eqref{eq:uniform-neumann-threshold}
for Lemma~\ref{lem:w2-finite-pdo}. A common compact spatial support
for a symbol is denoted by $K_a$.

The following estimates apply to the left-quantized symbols and their
compositions, including weighted inputs.

\begin{lemma}
\label{lem:w2-finite-pdo}
Suppose $a(x,\xi)$ has fixed compact $x$-support and
$|\partial_x^\alpha a|\leq A$ for $|\alpha|\leq N$, uniformly in $\xi$.
Then, for $0\leq j\leq k+1$,
\begin{equation}\label{eq:finite-pdo-estimate}
 \lVert\operatorname{Op}(a)\rVert_{H^j\to H^j}\leq C A.
\end{equation}
If $c$ is likewise compactly supported in $x$ with these derivatives
bounded by $B$, and
$|\partial_x^\alpha\partial_{\xi_l}a|\leq A_1$ for $|\alpha|\leq N$,
then
\begin{equation}\label{eq:finite-composition-estimate}
 \lVert\operatorname{Op}(a)\operatorname{Op}(c)
                        -\operatorname{Op}(ac)\rVert_{H^k\to H^k}
                      \leq C A_1B.
\end{equation}
This also holds when $a$ is a pure Fourier multiplier with
$|\nabla_\xi a|\leq A_1$; its unweighted operator norm is bounded by
$\sup|a|$. Finally, if
$|\partial_x^\alpha\partial_\xi^\beta a|\leq A_2$ for
$|\alpha|\leq N$, $|\beta|\leq2$, then
\begin{equation}\label{eq:weighted-pdo-estimate}
 \lVert\operatorname{Op}(a)u\rVert_{H^j}
               \leq C A_2\lVert u\rVert_{H^j_{-1}}
                     \quad(0\leq j\leq k+1).
\end{equation}
The constants depend only on the orders, dimension and fixed spatial
supports. In \eqref{eq:weighted-pdo-estimate} no spatial truncation of
$u$ is performed.
\end{lemma}
\begin{proof}
Write $\widehat a(\zeta,\xi)$ for the partial Fourier transform in $x$.
For $|\zeta|\geq1$, integrate by parts $N$ times in the unit direction
$\zeta/|\zeta|$. Compact support and the assumed derivatives give
\begin{equation}\label{eq:r2-calculation-059}
|\widehat a(\zeta,\xi)|\leq C A\langle\zeta\rangle^{-N}.
\end{equation}
This uses $N$ derivatives even when $N$ is odd; no even-order rounding is
needed. The weighted Fourier kernel for the $H^j$ estimate is bounded
by $CA\langle\eta-\xi\rangle^{-N+j}$, because
$\langle\eta\rangle^j/\langle\xi\rangle^j
 \leq C_j\langle\eta-\xi\rangle^j$. Its integrals in either variable
are finite when $N>j+3$. Schur's test proves
\eqref{eq:finite-pdo-estimate}.

Up to fixed Fourier normalization factors, the Fourier kernel of the
composition difference is exactly
\begin{equation}\label{eq:w2-composition-kernel}
 \int\!\bigl[\widehat a(\eta-\zeta,\zeta)
                    -\widehat a(\eta-\zeta,\xi)\bigr]
                          \widehat c(\zeta-\xi,\xi)\,d\zeta.
\end{equation}
The mean value theorem in the second slot, applied after $x$ Fourier
transformation, bounds it by
\begin{equation}\label{eq:r2-calculation-060}
C A_1 B\int |\zeta-\xi|\langle\eta-\zeta\rangle^{-N}
                         \langle\zeta-\xi\rangle^{-N}\,d\zeta.
\end{equation}
Insert the Sobolev weight using
$\langle\eta\rangle^k/\langle\xi\rangle^k
 \leq C\langle\eta-\zeta\rangle^k
           \langle\zeta-\xi\rangle^k$.
Both kernels are integrable since $N>k+4$. The two Schur integrals are
bounded by the product of their $L^1$ norms. For a pure left multiplier,
the kernel is $(a(\eta)-a(\xi))\widehat c(\eta-\xi,\xi)$; the same
estimate applies directly. This proves the composition assertion
without commuting any matrix factors.

To prove the weighted assertion, direct integration by parts in $\xi$
gives the exact identity
\begin{equation}\label{eq:weighted-pdo-identity}
 \operatorname{Op}(a)\langle x\rangle^2
   =\langle x\rangle^2\operatorname{Op}(a)
    -2i\sum_{j=1}^3x_j\operatorname{Op}(\partial_{\xi_j}a)
    -\operatorname{Op}(\Delta_\xi a).
\end{equation}
Define
\begin{equation}\label{eq:r3-weighted-conjugation-field}
v=\langle x\rangle^{-2}u,
\end{equation}
using the weight in \eqref{eq:r2-japanese-bracket}. The product rule gives
$\lVert v\rVert_{H^j}\leq C_j\lVert u\rVert_{H^j_{-1}}$, since every
derivative of $\langle x\rangle^{-2}$ is bounded by
$C\langle x\rangle^{-1}$. Apply the unweighted estimate to each term on
the right of \eqref{eq:weighted-pdo-identity}. All polynomial $x$ factors
are bounded, with all needed derivatives, on $K_a$; each output there
is compactly supported. This proves \eqref{eq:weighted-pdo-estimate}
first for compactly supported smooth $u$ and then by weighted density.
In particular the bound applies directly to $u=E_{\xi_{\rm c}}F$:
\begin{equation}\label{eq:r2-calculation-061}
\|\operatorname{Op}(a)E_{\xi_{\rm c}}F\|_{H^j}
 \leq CA_2\|E_{\xi_{\rm c}}F\|_{H^j_{-1}}.
\end{equation}
The weighted identity applies to the full input $u$, including its tail
outside $K_a$.
\end{proof}

\subsection{Fixed frame geometry and near/far symbols}\label{sec:r2-parametrix-symbols}

The symbol construction separates a neighborhood of the characteristic set,
where transport frames are used, from its complement, where division by the
free symbol is permitted.

Choose once and for all smooth real spatial cutoffs satisfying
\begin{equation}\label{eq:w2-spatial-cutoffs}
 \begin{gathered}
 \psi_2,\psi_3,\chi_{\rm sp}\in C_c^\infty(\Omega_0),\qquad
 0\leq\psi_2,\psi_3,\chi_{\rm sp}\leq1,\\
 \psi_2=1\ \hbox{near }\overline\Omega_1,\qquad
 \psi_3=1\ \hbox{near }\operatorname{supp}\psi_2,\qquad
 \chi_{\rm sp}=1\ \hbox{near }\operatorname{supp}\psi_3.
 \end{gathered}
\end{equation}
They are chosen before $h$, coefficients, amplitudes or test parameters.

\subsubsection{Compact coefficient families}\label{sec:r2-parametrix-compactness}

The entire coefficient ball in \eqref{eq:finite-system-class}, with fixed
support, has compact closure in $C^{k+12,1/2}$ on a fixed large ball.
Indeed derivatives through order $k+13$ are uniformly Lipschitz, and
Arzel\`a--Ascoli followed by H\"older interpolation gives the asserted
compactness at the lower order. Limits remain supported in $K_{\rm coef}$.
They need not be $C^{k+14}$, so existence is used at their actual finite
regularity.

Use $r_{\rm par}$ from \eqref{eq:uniform-neumann-threshold}.
For an oriented pair $(\omega,\nu)$ set
\begin{equation}\label{eq:r3-parametrix-leaf-normal}
n=\omega\times\nu.
\end{equation}
Use the leaf coordinates
\begin{equation}\label{eq:w2-leaf-coordinates}
 x=(\Re z)\omega+(\Im z)\nu+t n,
 \qquad z=x\cdot\omega+i x\cdot\nu,\quad t=x\cdot n.
\end{equation}
The parameter $p=(\omega,\nu,t)$ belongs to the compact oriented-pair
manifold times a fixed closed interval containing all transverse
coordinates of $\operatorname{supp}\chi_{\rm sp}$, with a fixed margin.
Take fixed large leaf disks containing all its leaf projections, also
with fixed margins. Pulled-back coefficients have joint $C^{r_{\rm par}+2}$
bounds there, including at the $C^{k+12,1/2}$ closure points. Apply
Theorem~\ref{thm:finite-frames} to
the matrix coefficient
\begin{equation}\label{eq:r3-directional-matrix-coefficient}
\mathsf A_{\omega+i\nu}=\sum_j(\omega_j+i\nu_j)\mathsf A_j,\qquad\mathsf B=\mathsf A_{\omega+i\nu}/(2i).
\end{equation} This produces a finite
atlas $\mathsf F_\ell$ and weights $w_\ell$, $1\leq\ell\leq J$, such that
\begin{equation}\label{eq:w2-transport-equation}
 i(\omega+i\nu)\cdot\nabla \mathsf F_\ell=\mathsf A_{\omega+i\nu}\mathsf F_\ell,
 \qquad (\omega+i\nu)\cdot\nabla w_\ell=0,
 \qquad \sum_{\ell=1}^Jw_\ell^2=1.
\end{equation}
Every needed derivative of the frames, inverses, weights and coordinates
through total order $r_{\rm par}$ is uniformly bounded. For a particular
coefficient, charts with zero weight contribute zero. Products are
extended by zero through inactive chart boundaries. The same finite
atlas is fixed over the entire prior; it is not reselected as $h$ changes.

\subsubsection{Frequency geometry}\label{sec:r2-frequency-geometry}

Fix a smooth $\vartheta:\mathbb R\to[0,1]$ equal to one on $[-1,1]$ and
zero outside $[-2,2]$. For $\xi\in\mathbb R^3$ let
\begin{equation}\label{eq:near-far-cutoffs}
 \begin{gathered}
 \xi_\nu=\xi\cdot\nu,\quad \xi'=\xi-\xi_\nu\nu,
 \quad r=|\xi'+\tau\omega_0|,\quad
 \omega=(\xi'+\tau\omega_0)/r,\\
 t_\tau=\vartheta(\xi_\nu/\sqrt\tau)
              \vartheta((r^2-\tau^2)/\tau^{3/2}),\qquad
 \chi_N=\sin(\pi t_\tau/2),\quad \chi_F=\cos(\pi t_\tau/2).
 \end{gathered}
\end{equation}
The scalar $r$ here is a frequency radius, not the order in the frame
theorem; $t_\tau$ is a frequency cutoff, not the transverse coordinate
$t$ of \eqref{eq:w2-leaf-coordinates}. The products in
\eqref{eq:paired-symbols} vanish near
$r=0$, so $\omega$ is only used where it is defined. Set
\begin{equation}\label{eq:r3-frequency-cutoff-supports}
\Xi_N=\operatorname{supp}\chi_N,\qquad\Xi_F=\operatorname{supp}\chi_F.
\end{equation}
For $\tau\geq16$, on $\Xi_N$,
\begin{equation}\label{eq:w2-near-geometry}
 |\xi_\nu|\leq2\sqrt\tau,\qquad
 |r^2-\tau^2|\leq2\tau^{3/2},\qquad
 \tau/\sqrt2\leq r\leq\sqrt{3/2}\tau.
\end{equation}
The characteristic circle is $\{\xi_\nu=0,r=\tau\}$; a full neighborhood
of it has $\chi_F=0$. Since $\chi_N^2+\chi_F^2=1$, define the paired
matrix symbols
\begin{equation}\label{eq:paired-symbols}
 \begin{split}
 b_\ell(x,\xi)&=\chi_{\rm sp}(x)\chi_N(\xi)
                     w_\ell(\mathfrak a,p)\mathsf F_\ell(z,p),\\
 d_\ell(x,\xi)&=\chi_{\rm sp}(x)\chi_N(\xi)
                     w_\ell(\mathfrak a,p)\mathsf F_\ell(z,p)^{-1},
                   \qquad 1\leq\ell\leq J,\\
 b_0(\xi)&=d_0(\xi)=\chi_F(\xi)I_{d_{\rm sys}}.
 \end{split}
\end{equation}
Here $d_\ell$ denotes the paired matrix symbol. The leading identity,
with matrix order preserved, is
\begin{equation}\label{eq:w2-symbol-resolution}
 \sum_{\ell=0}^J b_\ell\psi_3 d_\ell=\psi_3 I_{d_{\rm sys}}
                              \quad\hbox{on all }\mathbb R^3_x\times\mathbb R^3_\xi.
\end{equation}

\begin{lemma}
\label{lem:u1-uniform-symbol-budget}
For $\tau\geq16$ the following bounds have constants fixed by the prior
and cutoffs. First, for $|\alpha|\leq N+2$ and $|\beta|\leq2$,
\begin{equation}\label{eq:symbol-estimates}
 |\partial_x^\alpha\partial_\xi^\beta b_\ell|
   +|\partial_x^\alpha\partial_\xi^\beta d_\ell|
     \leq C\tau^{-|\beta|/2},\qquad \ell\geq1.
\end{equation}
Their supports in $(x,\xi)$ lie in
$\operatorname{supp}\chi_{\rm sp}\times\Xi_N$ and the corresponding
parameter charts. The pure multipliers $b_0,d_0$ obey the frequency
bounds in \eqref{eq:symbol-estimates}, but have no compact spatial
support.

Define the differential remainder symbols by
\begin{equation}\label{eq:differential-errors}
 \begin{split}
 r_{0,\ell}&=-\Delta_xb_\ell-2i\sum_j \mathsf A_j\partial_{x_j}b_\ell
                                  +\mathsf V b_\ell,\\
 r_{1,\ell}&=2\sum_j(\xi_j+\xi_{{\rm c},j})
                         (\mathsf A_jb_\ell-i\partial_{x_j}b_\ell).
 \end{split}
\end{equation}
For $|\alpha|\leq N$, $|\beta|\leq2$, they satisfy
\begin{equation}\label{eq:w2-five-symbol-bounds}
 \begin{aligned}
 |\partial_x^\alpha\partial_\xi^\beta(\psi_3 r_{0,\ell})|
       &\leq C\tau^{-|\beta|/2}, &&\ell\geq1,\\
 |\partial_x^\alpha\partial_\xi^\beta(\psi_3 r_{1,\ell})|
       &\leq C\tau^{(1-|\beta|)/2}, &&\ell\geq1,\\
 |\partial_x^\alpha\partial_\xi^\beta(\psi_3\mathsf V\chi_F)|
       &\leq C\tau^{-|\beta|/2},\\
 |\partial_x^\alpha\partial_\xi^\beta a_F|
       &\leq C\tau^{-(1+|\beta|)/2},
 \end{aligned}
\end{equation}
where
\begin{equation}\label{eq:w2-far-symbol}
 a_F(x,\xi)=\frac{2\sum_j \mathsf A_j(x)(\xi_j+\xi_{{\rm c},j})\chi_F(\xi)}
                         {p_{\xi_{\rm c}}(\xi)}
\end{equation}
is zero in a neighborhood of the characteristic set. The two near
symbols in \eqref{eq:w2-five-symbol-bounds} are supported in
$\operatorname{supp}\psi_3\times\Xi_N$; the two far symbols are supported
in $K_{\rm coef}\times\Xi_F$.
The maximum frame order used is $N+4$, and the maximum coefficient
$x$ order in the remainder products is $N$.
\end{lemma}
\begin{proof}
On $\Xi_N$, direct differentiation gives
$\partial_\xi r=O(1)$, $\partial_\xi^2r=O(\tau^{-1})$,
$\partial_\xi\omega=O(\tau^{-1})$, and
$\partial_\xi^2\omega=O(\tau^{-2})$.
Where a derivative of either cutoff factor is nonzero, the first
derivatives of its scaled argument are $O(\tau^{-1/2})$; its second
derivatives are $O(\tau^{-3/2})$ or zero. Outside these transition
regions the corresponding cutoff derivatives vanish.
The chain rule therefore bounds the first and second derivatives of
$\chi_N,\chi_F$ by $C\tau^{-1/2}$ and $C\tau^{-1}$.
On the compact spatial support, \eqref{eq:w2-leaf-coordinates} shows that
one $\xi$ derivative of a pulled-back frame or weight costs at most one
additional leaf/parameter derivative and a factor $O(\tau^{-1})$;
two cost at most two and a factor $O(\tau^{-2})$.
Thus a mixed derivative with $|\alpha|\leq N+2$, $|\beta|\leq2$ costs
at most $N+4$ total frame derivatives. Leibniz' rule, allowing derivatives
to fall on the wider frequency cutoffs, gives \eqref{eq:symbol-estimates}.
All assertions extend through $r=0$ by the zero products.

The $r_{0,\ell}$ formula costs two additional $x$ derivatives of the base
frame product, so its maximal mixed input order is $N+2+2=N+4$.
Derivatives falling on $\mathsf A_j$ or $\mathsf V$ have order at most $N$.
This proves the first bound in \eqref{eq:w2-five-symbol-bounds}.
On a neighborhood of $\operatorname{supp}\psi_3$ the spatial cutoff is
constant. Equation~\eqref{eq:w2-transport-equation} then gives the exact
ordered cancellation
\begin{equation}\label{eq:appendix-u1-transport-remainder}
 \psi_3 r_{1,\ell}=2\psi_3\bigl((r-\tau)\omega+\xi_\nu\nu\bigr)
          \cdot(\mathsf A b_\ell-i\nabla_x b_\ell),\qquad \ell\geq1.
\end{equation}
The vector in parentheses and its first and second frequency
derivatives have orders $O(\sqrt\tau)$, $O(1)$ and $O(\tau^{-1})$,
respectively. Combining these frequency-derivative estimates with
\eqref{eq:symbol-estimates}, we obtain the second estimate in
\eqref{eq:w2-five-symbol-bounds}. Its maximum frame order is only $N+1+2=N+3$.
There are no cutoff-gradient terms from $\chi_{\rm sp}$ on this support.

For $\ell=0$ the spatial derivatives of $b_0$ vanish and
$r_{0,0}=\mathsf V\chi_F$. This proves the third bound. To treat
$r_{1,0}$, write
\begin{equation}\label{eq:r2-far-characteristic-symbol}
p_{\xi_{\rm c}}=r^2+\xi_\nu^2-\tau^2+2i\tau\xi_\nu.
\end{equation}
On $\Xi_F$, either $|\xi_\nu|\geq\sqrt\tau$ and the imaginary part
has modulus at least $2\tau^{3/2}$, or
$|r^2-\tau^2|\geq\tau^{3/2}$ with $|\xi_\nu|\leq\sqrt\tau$.
In the second case the real part has modulus at least
$\tau^{3/2}-\tau\geq(3/4)\tau^{3/2}$.
For $|\xi|\leq4\tau$ we also have
$|\partial_\xi p|\leq C\tau$ and $\partial_\xi^2p=2I$.
The formulae
\begin{equation}\label{eq:r2-reciprocal-symbol-derivatives}
\partial_l(p^{-1})=-p^{-2}\partial_lp,\qquad
 \partial_l\partial_j(p^{-1})
   =2p^{-3}(\partial_lp)(\partial_jp)-p^{-2}\partial_l\partial_jp
\end{equation}
give the first two derivatives of the reciprocal on this far set.
The uncut numerator of \eqref{eq:w2-far-symbol} has frequency orders
$O(\tau)$, $O(1)$, and zero. Applying the product rule also to $\chi_F$, we obtain the factor
$\tau^{-(1+|\beta|)/2}$ for $|\beta|\le2$.

For $|\xi|\geq4\tau$, we have $|p|\geq |\xi|^2/2$ and
$|\partial_\xi p|\leq C|\xi|$. The near support in
\eqref{eq:w2-near-geometry} lies inside $|\xi|<4\tau$, so $\chi_F=1$
in this exterior region. Differentiation gives
$C|\xi|^{-1-|\beta|}$, which is stronger than the displayed bounds.
The cutoff vanishes on an open neighborhood of the poles, so these
reciprocal derivatives never cross the characteristic circle. All
spatial derivatives cost only the indicated coefficient derivatives,
and support in $K_{\rm coef}$ follows from the factor $\mathsf A$.
\end{proof}

\subsection{Ordered parametrix errors and two Sobolev bounds}\label{sec:r2-parametrix-errors}

The parametrix error is decomposed into five terms with the matrix factors
kept in their stated order. Both Sobolev bounds are proved for this
parametrix before the Neumann correction.

Define, initially on Schwartz functions and then by the bounds in
Lemma~\ref{lem:w2-five-errors},
\begin{equation}\label{eq:parametrix}
 Q_\tau=\sum_{\ell=0}^J\operatorname{Op}(b_\ell)E_{\xi_{\rm c}}
                          \psi_3\operatorname{Op}(d_\ell).
\end{equation}
The full output of $E_{\xi_{\rm c}}$ is used by
$\operatorname{Op}(b_\ell)$.
The inner source
$F_\ell=\psi_3\operatorname{Op}(d_\ell)g$ is supported in
$\operatorname{supp}\psi_3$, and
\begin{equation}\label{eq:w2-inner-source}
 \lVert F_\ell\rVert_{H^k_1}\leq C\lVert g\rVert_{H^k}.
\end{equation}
For $\ell=0$ this uses the multiplier bound before multiplication by
$\psi_3$. For $\ell\geq1$ it uses the compact-symbol bound. Compact
spatial support then bounds the weight. Notice that $F_\ell$ need not
be frequency-supported in $\Xi_N$ or $\Xi_F$, because multiplication
by $\psi_3$ mixes frequencies.

A direct product-rule differentiation in left quantization gives
\begin{equation}\label{eq:w2-exact-differential-composition}
 P_{\xi_{\rm c}}\operatorname{Op}(b_\ell)E_{\xi_{\rm c}}
 =\operatorname{Op}(b_\ell)
   +\operatorname{Op}(r_{0,\ell})E_{\xi_{\rm c}}
   +\operatorname{Op}(r_{1,\ell})E_{\xi_{\rm c}}.
\end{equation}
For instance $D_{x,j}(b_\ell e^{ix\cdot\xi})=
(\xi_jb_\ell-i\partial_{x_j}b_\ell)e^{ix\cdot\xi}$;
applying it twice produces precisely \eqref{eq:differential-errors}.
There is no derivative of $\mathsf A_j$ here because it multiplies on the left
in the nondivergence operator. Write
\begin{equation}\label{eq:w2-error-identity}
 \begin{split}
 S_\tau&=\sum_{\ell=0}^J\operatorname{Op}(b_\ell)
                         \psi_3\operatorname{Op}(d_\ell)-\psi_3 I,\\
 \mathcal E_\tau&=\sum_{\ell=0}^J\sum_{j=0}^1
         \operatorname{Op}(r_{j,\ell})E_{\xi_{\rm c}}
                         \psi_3\operatorname{Op}(d_\ell),\\
 P_{\xi_{\rm c}}Q_\tau&=\psi_3 I+S_\tau+\mathcal E_\tau.
 \end{split}
\end{equation}
The last identity is distributional. Its localization is bounded on
$H^k$ by the following five estimates.

\begin{lemma}\label{lem:w2-five-errors}
Let
\begin{equation}\label{eq:uniform-five-term-errors}
 \begin{split}
 T_{{\rm comp},\tau}&=\psi_3 S_\tau,\\
 T_{{\rm near},j,\tau}&=\sum_{\ell=1}^J
      \operatorname{Op}(\psi_3 r_{j,\ell})E_{\xi_{\rm c}}
                          \psi_3\operatorname{Op}(d_\ell),\quad j=0,1,\\
 T_{{\rm far},j,\tau}&=
      \operatorname{Op}(\psi_3 r_{j,0})E_{\xi_{\rm c}}
                          \psi_3\operatorname{Op}(d_0),\quad j=0,1,\\
 T_\tau&=T_{{\rm comp},\tau}+T_{{\rm near},0,\tau}
     +T_{{\rm near},1,\tau}+T_{{\rm far},0,\tau}+T_{{\rm far},1,\tau}
       =\psi_3(S_\tau+\mathcal E_\tau).
 \end{split}
\end{equation}
There are five common constants $C_j$, independent of the coefficient,
normalized direction and $\tau\geq16$, such that
\begin{equation}\label{eq:uniform-five-term-bounds}
 \begin{aligned}
 \lVert T_{{\rm comp},\tau}\rVert_{H^k\to H^k}&\leq C_1\tau^{-1/2},\\
 \lVert T_{{\rm near},0,\tau}\rVert_{H^k\to H^k}&\leq C_2\tau^{-1},\\
 \lVert T_{{\rm near},1,\tau}\rVert_{H^k\to H^k}&\leq C_3\tau^{-1/2},\\
 \lVert T_{{\rm far},0,\tau}\rVert_{H^k\to H^k}&\leq C_4\tau^{-1},\\
 \lVert T_{{\rm far},1,\tau}\rVert_{H^k\to H^k}&\leq C_5\tau^{-1/2}.
 \end{aligned}
\end{equation}
Every error has output supported in $\operatorname{supp}\psi_3$; the two
far errors in fact have output in $K_{\rm coef}$. The symbol frequency
supports are those of Lemma~\ref{lem:u1-uniform-symbol-budget}; no
preservation of the Fourier support of an input or output function is
asserted.
\end{lemma}
\begin{proof}
\textit{Composition.} For each $\ell$, use the compact right symbol
$c_\ell=\psi_3 d_\ell$, since
$\psi_3\operatorname{Op}(d_\ell)=\operatorname{Op}(c_\ell)$ exactly.
For $\ell\geq1$, the left-symbol first frequency derivative, with $N$
spatial derivatives, has size $C\tau^{-1/2}$. For $\ell=0$ apply the
pure-multiplier version of Lemma~\ref{lem:w2-finite-pdo}. The principal
products sum to $\psi_3 I$ by \eqref{eq:w2-symbol-resolution}. Summing the finitely many composition remainders and multiplying on the left
by $\psi_3$, we obtain the first estimate in
\eqref{eq:uniform-five-term-bounds}.

\textit{Near zeroth order.} The inner source satisfies
\eqref{eq:w2-inner-source}. The first free estimate maps it to
$H^k_{-1}$ with factor $C\tau^{-1}$. The symbol
$\psi_3r_{0,\ell}$ has all seminorms required by the weighted estimate
bounded by a constant, by \eqref{eq:w2-five-symbol-bounds}.
Apply that weighted estimate after the free inverse and sum in $\ell$.
This gives the second bound.

\textit{Near first order.} The same ordered composition applies.
The seminorms of $\psi_3r_{1,\ell}$ are now bounded by $C\sqrt\tau$;
the $\tau^{-1}$ factor of the free inverse gives $C\tau^{-1/2}$.
The transport cancellation used to obtain this factor was proved in
\eqref{eq:appendix-u1-transport-remainder}, including the spatial-cutoff
condition. This proves the third bound.

\textit{Far zeroth order.} Here $r_{0,0}=\mathsf V\chi_F$.
Its compactly supported left symbol has uniformly bounded weighted
seminorms. Apply \eqref{eq:w2-inner-source}, the first free bound and the
weighted PDO bound in that order. This proves the fourth bound.

\textit{Far first order.} Since the right-hand free inverse is a Fourier
multiplier, there is the exact, not asymptotic, identity
\begin{equation}\label{eq:r2-calculation-064}
\operatorname{Op}(r_{1,0})E_{\xi_{\rm c}}
                   =\operatorname{Op}(a_F).
\end{equation}
It is valid because the numerator vanishes near the singular set of
$p^{-1}$. The unweighted PDO bound and the last bound in
\eqref{eq:w2-five-symbol-bounds} give $C\tau^{-1/2}$.
The remaining factor $\psi_3\operatorname{Op}(d_0)$ is uniformly
bounded on $H^k$. This proves the fifth bound.

All these constants are products and finite sums of the constants in
Lemmas~\ref{lem:w2-free}, \ref{lem:w2-finite-pdo} and
\ref{lem:u1-uniform-symbol-budget}, with the fixed cutoff multiplier
norms. The frame atlas was taken over the full compact prior closure,
so its number of charts and its seminorms are already common constants.
Thus all five bounds are uniform over the fixed compact coefficient class
and frame atlas. The output support assertions follow from the left spatial
symbols. They do not restrict the nonlocal input of those symbols.
\end{proof}

The symbols $b_\ell,d_\ell$ in \eqref{eq:paired-symbols} require
at most $N+2$ spatial and two frequency derivatives; the frame order
is at most $N+4$. For the near symbols in
\eqref{eq:appendix-u1-transport-remainder}, the bounds in
\eqref{eq:w2-five-symbol-bounds} use at most $N$ spatial and two frequency
derivatives. Their largest frame orders are $N+4$ for the zeroth-order
symbol and $N+3$ for the first-order symbol. The far symbols contain
no frame. The orders fixed in \eqref{eq:uniform-neumann-threshold} satisfy
\begin{equation}\label{eq:derivative-budget}
 N>k+4,\qquad r_{\rm par}\ge N+4.
\end{equation}
The composition estimate actually needs only $N$ spatial derivatives and
one frequency derivative of its left symbol. The bounds for
$b_\ell,d_\ell$ through spatial order $N+2$ and frequency order two
also control the differential remainders.
For an $H^{k+1}$ bound, Schur integrability needs $N>k+4$; for the
composition kernel its extra frequency-difference factor gives the same
strict inequality. Both are met by $N=k+6$.

The free estimates \eqref{eq:free-two-estimates} and weighted bound
\eqref{eq:weighted-pdo-estimate} yield both Sobolev bounds for the same
parametrix, in the stated operator order.

\begin{lemma}\label{lem:w2-Q-two-level}
For $\tau\geq16$ and arbitrary $g\in H^k(\mathbb R^3)$,
\begin{equation}\label{eq:localized-error-bounds}
 \lVert\psi_2 Q_\tau g\rVert_{H^k}\leq C_{Q,0}\tau^{-1}\lVert g\rVert_{H^k},
 \qquad
 \lVert\psi_2 Q_\tau g\rVert_{H^{k+1}}\leq C_{Q,1}\lVert g\rVert_{H^k}.
\end{equation}
Both constants are uniform over the full class and all directions.
\end{lemma}
\begin{proof}
For each summand of \eqref{eq:parametrix}, put
$F_\ell=\psi_3\operatorname{Op}(d_\ell)g$. By
\eqref{eq:w2-inner-source},
$\|F_\ell\|_{H^k_1}\le C\|g\|_{H^k}$ uniformly in $\ell$.
Apply the two free estimates \eqref{eq:free-two-estimates} to obtain
\[
 \|E_{\xi_{\rm c}}F_\ell\|_{H^k_{-1}}
       \le C\tau^{-1}\|g\|_{H^k},\qquad
 \|E_{\xi_{\rm c}}F_\ell\|_{H^{k+1}_{-1}}
       \le C\|g\|_{H^k}.
\]
The symbol $\psi_2b_\ell$ has fixed compact spatial support and
the mixed seminorm bounds of Lemma~\ref{lem:u1-uniform-symbol-budget}.
This also holds for $\ell=0$ after multiplication of the Fourier
multiplier $b_0$ by $\psi_2$. The derivative condition
\eqref{eq:derivative-budget} permits use of
\eqref{eq:weighted-pdo-estimate} at both orders $k$ and $k+1$.
Consequently
\[
 \begin{aligned}
 \|\operatorname{Op}(\psi_2b_\ell)E_{\xi_{\rm c}}F_\ell\|_{H^k}
       &\le C\tau^{-1}\|g\|_{H^k},\\
 \|\operatorname{Op}(\psi_2b_\ell)E_{\xi_{\rm c}}F_\ell\|_{H^{k+1}}
       &\le C\|g\|_{H^k}.
 \end{aligned}
\]
Since $\operatorname{Op}(\psi_2b_\ell)=
\psi_2\operatorname{Op}(b_\ell)$, summing over the fixed finite
atlas proves \eqref{eq:localized-error-bounds}.
The weighted estimate acts on the full output of the free inverse;
no spatial cutoff has been inserted between the two operators.
\end{proof}

\subsection{One Neumann inverse and the local equation}\label{sec:r2-proof-inverse}

\begin{proof}[Proof of Theorem~\ref{thm:finite-right-inverse}]
The frame construction of Theorem~\ref{thm:finite-frames}, the free
inverse of Lemma~\ref{lem:w2-free}, and the operator estimates
of Lemma~\ref{lem:w2-finite-pdo} apply to every finite-regularity point of
the weaker compact coefficient closure. Lemma~\ref{lem:w2-five-errors}
gives five common constants. Use \eqref{eq:uniform-neumann-threshold}.
For $\tau\geq\tau_0$, since $\tau^{-1}\leq\tau^{-1/2}$,
\begin{equation}\label{eq:w2-error-smallness}
 \lVert T_\tau\rVert_{H^k\to H^k}
        \leq C_\Sigma\tau^{-1/2}\leq1/4.
\end{equation}
Construct exactly one inverse, on $H^k(\mathbb R^3;\mathbb C^{d_{\rm sys}})$:
\begin{equation}\label{eq:exact-neumann-inverse}
 J_\tau=(I+T_\tau)^{-1}=\sum_{n=0}^{\infty}(-T_\tau)^n,
 \qquad \lVert J_\tau\rVert_{H^k\to H^k}\leq4/3.
\end{equation}
For $f\in E_k$ extend by zero, put $g=J_\tau f$, and define
\begin{equation}\label{eq:w2-final-inverse}
 \mathcal R_{h,\theta}f=-\psi_2Q_\tau J_\tau f\big|_{\Omega_0}.
\end{equation}
The negative sign is forced by $P_{\xi_{\rm c}}=-L_{h,\theta}$.
The relation $g=f-T_\tau g$ shows only
\begin{equation}\label{eq:r2-calculation-065}
\operatorname{supp}g\subset K_{\rm src}\cup\operatorname{supp}\psi_3;
\end{equation}
it does not say $g\in E_k$. This larger fixed support is harmless:
the actual inputs to the free inverse are always
$\psi_3\operatorname{Op}(d_\ell)g$ as in \eqref{eq:w2-inner-source}.

On $\Omega_1$ both spatial cutoffs equal one on a neighborhood. Thus
no derivatives of $\psi_2$ occur there, and
\eqref{eq:w2-error-identity} gives
\begin{equation}\label{eq:w2-local-exactness}
 \begin{split}
 L_{h,\theta}\mathcal R_{h,\theta}f
 &=P_{\xi_{\rm c}}Q_\tau g
   =g+(S_\tau+\mathcal E_\tau)g\\
 &=g+\psi_3(S_\tau+\mathcal E_\tau)g
   =(I+T_\tau)g=f \quad\text{in }\mathcal D'(\Omega_1).
 \end{split}
\end{equation}
In the middle equality the comparison is restricted to $\Omega_1$;
it is not an identity on all of $\mathbb R^3$.
Lemma~\ref{lem:w2-Q-two-level} and \eqref{eq:exact-neumann-inverse}
now give
\begin{equation}\label{eq:u1-same-neumann-two-levels}
 \begin{split}
 \lVert\psi_2Q_\tau J_\tau f\rVert_{H^k}
       &\leq\tfrac43 C_{Q,0}\tau^{-1}\lVert f\rVert_{H^k},\\
 \lVert\psi_2Q_\tau J_\tau f\rVert_{H^{k+1}}
       &\leq\tfrac43 C_{Q,1}\lVert f\rVert_{H^k}.
 \end{split}
\end{equation}
Choose $C_k=\tfrac43(C_{Q,0}+C_{Q,1})$ multiplied by the fixed extension norm.
Then \eqref{eq:u1-same-neumann-two-levels} gives the required sum of norms in
\eqref{eq:finite-right-inverse}. No Neumann series on $H^{k+1}$ has been
constructed or used. The atlas includes all oriented pairs, hence both
phases have the same $h_*$. Their respective $Q_\tau,T_\tau,J_\tau$ are
allowed to differ. The geometric cutoff of the \emph{final output}
limits the inverse identity to $\Omega_1$, but never deletes a tail
inside the parametrix. There is no amplitude-size restriction.
\end{proof}

\section{Finite-order physical boundary residual}\label{app:boundary-residual}
This appendix proves the physical residual and energy-remainder estimates
from the fixed charts, the constant half-space inverse and the recursion of
Section~\ref{sec:v63-boundary-jets}.

\subsection{Transformed tensor and profile estimates}\label{sec:r2-weighted-profiles}

The full transformed elasticity tensor and weighted half-space profiles
control the Taylor remainder and the remote cutoff terms at order $s$,
without additional material derivatives.

For the chart \eqref{eq:w5-explicit-normal-chart} define
\begin{equation}\label{eq:r2-elasticity-coordinate-tensors}
 \begin{gathered}
 J_p(z)=D\Psi_p(z),\qquad
 C_{ijkl}=\lambda\delta_{ij}\delta_{kl}
       +\mu(\delta_{ik}\delta_{jl}+\delta_{il}\delta_{jk}).
 \end{gathered}
\end{equation}
Here $\delta_{ij}$ is the Kronecker symbol, $J_p$ is a $3\times3$
Jacobian matrix, and $R_p$ is the fixed component rotation in
\eqref{eq:w5-explicit-normal-chart}. The transformed tensor in
\eqref{eq:v63-pulled-energy} is
\begin{equation}\label{eq:v64-coordinate-tensor}
 \begin{aligned}
 \mathbb A_{abcd,p}(z)
 ={}&|\det J_p(z)|
 \sum_{i_1,i_2,i_3,i_4=1}^3 C_{i_1i_2i_3i_4}(\Psi_p(z))\\
 &\quad{}\times (R_p)_{i_1a}(R_p)_{i_3c}
                 (J_p^{-1}(z))_{bi_2}(J_p^{-1}(z))_{di_4}.
 \end{aligned}
\end{equation}
Indeed, writing $u(\Psi_p(z))=R_pv(z)$, the chain rule gives
\[
 (\partial_{x_{i_2}}u_{i_1})(\Psi_p(z))
 =\sum_{a,b=1}^3(R_p)_{i_1a}(\partial_{z_b}v_a)(z)
                          (J_p^{-1}(z))_{bi_2}.
\]
Substitution in the physical energy gives
\eqref{eq:v64-coordinate-tensor} without differentiating $R_p$.
The absolute determinant handles either orientation of the orthogonal
base matrix. At $z=0$, $J_p=R_p$ and isotropy gives the frozen tensor in
\eqref{eq:v63-pulled-energy}. Away from the base point one must use the
entire tensor \eqref{eq:v64-coordinate-tensor}, not its frozen isotropic
formula. The fixed smooth charts give a uniform $C^s$ tensor bound from
the stated $C^s$ material bound. Taylor expansion is performed only along
segments contained in a fixed coordinate half-ball.

For this subsection $s\ge2$ is the fixed integer of
\eqref{eq:w0-interior-class}. Let $V_r$, $0\le r<s$, be the profiles in
\eqref{eq:v63-recursion}. Define the following single finite collection of
seminorms:
\begin{equation}\label{eq:v64-profile-budget}
 \mathfrak B_s:=\max_{0\le r<s}\left(
 \lVert \langle y\rangle^s V_r\rVert _{L^2(\mathbb{R}^3_+)}
 +\lVert \langle y\rangle^s\nabla V_r\rVert _{L^2(\mathbb{R}^3_+)}\right).
\end{equation}
The weight \eqref{eq:r2-japanese-bracket} is equivalent to the weight in
\eqref{eq:v63-weighted-profiles}.

\begin{lemma}
\label{lem:v64-profile-budget}
The profiles in \eqref{eq:v63-recursion} can be written in
\eqref{eq:v63-profile-class} with normal polynomial degree at most $3r+1$.
The number $\mathfrak B_s$ is uniformly bounded by the fixed $C^s$ material
class, the fixed charts, and the fixed annular tests. Moreover, for any
$N>1$, $L\ge1$, and any scalar, vector or matrix field $F$ on $\mathbb{R}^3_+$,
\begin{equation}\label{eq:v64-tail-bounds}
 \begin{aligned}
 \lVert \mathbf1_{|y|>L}F\rVert _2
   &\le C L^{-N}\lVert \langle y\rangle^N F\rVert _2,\\
 \lVert \mathbf1_{|y|>L}F\rVert _{6/5}
   &\le C_N L^{1-N}\lVert \langle y\rangle^N F\rVert _2.
 \end{aligned}
\end{equation}
\end{lemma}
\begin{proof}
The homogeneous Poisson profile has degree at most one. If a source has
normal polynomial degree $D$, convolution with the Green kernel in
Lemma~\ref{lem:v63-profile} has degree at most $D+2$: on $0<u<t$ one
integrates a polynomial of degree at most $D$ against a polynomial of
degree at most one in $t-u$, and on $u>t$ the exponentially decaying
integral has no larger degree. We subtract the homogeneous Poisson solution with boundary data equal to the
trace of the particular solution. This correction increases the normal
polynomial degree by at most one. The interior
source is first calculated for $t>0$ and only then extended by zero for
this convolution. Extending a divergence-form flux before differentiating
would create a boundary distribution and is not the construction used
here.

Multiplication by a physical polynomial of degree $\ell$ increases the
normal polynomial degree by at most $\ell$: each tangential frequency
derivative can create at most one additional power of $t$ in $e^{-\varpi t}$.
Spatial differentiation does not increase this degree. If $D_r$ denotes
a degree bound for $V_r$, the recursion therefore gives
\begin{equation}\label{eq:r2-calculation-066}
D_0\le1,\qquad
 D_r\le\max_{1\le\ell\le r}(D_{r-\ell}+\ell+2)\le3r+1.
\end{equation}
All frequency coefficients have support in the same compact subannulus.
Their finitely many seminorms depend smoothly on the finite-order tensor derivatives
and on the frozen moduli, with uniformly nonzero denominators.
The integral formulas in Lemma~\ref{lem:v63-profile} then bound
\eqref{eq:v64-profile-budget}. The profile bounds use derivatives of the Lam\'e parameters only
through the finite Taylor orders, not derivatives of order $3r+1$ of
the original material. Polynomial degree of a constructed profile and
material regularity are different quantities.

One may take as a concrete input seminorm for the fixed annular
test $\max_{|\gamma|\le2s+2}\|\partial_\xi^\gamma\psi_\partial\|_\infty$.
Indeed, to bound $N$ frequency derivatives of a coefficient of $V_r$,
multiplication by a Taylor polynomial of degree $\ell$ uses at most
$\ell$ additional frequency derivatives of $V_{r-\ell}$; the Green and
Poisson operations introduce none. Induction needs at most $N+r$
derivatives of the initial test. The weights in $\mathfrak B_s$ use
$N\le s$, and $r\le s-1$. Thus derivatives of the fixed annular
test through order $2s+2$ suffice. More explicitly, if $T_r(N)$ is the largest derivative order of the
initial annular test needed for $N$ tangential frequency derivatives
of the coefficients of $V_r$, then
\begin{equation}\label{eq:w5-finite-profile-recurrence}
 T_0(N)\le N,\qquad
 T_r(N)\le\max_{1\le\ell\le r}T_{r-\ell}(N+\ell)
                  \le N+r.
\end{equation}
Spatial differentiation of a profile only multiplies its frequency
coefficients or differentiates the finite normal polynomial; it creates
no derivatives of the test. The Green and trace-subtraction operations
are finite smooth multipliers and integrals on the fixed annulus.
Their own frequency derivatives cost fixed multiplier constants.
The weight $\langle y\rangle^s$ requires at most $s$ tangential
frequency derivatives and at most $s$ normal powers. Thus $r\le s-1$
uses at most $2s-1$ derivatives of the fixed test. In particular,
the seminorm involving test derivatives through order $2s+2$ controls
both polarizations and their gradients. Tensor jets are
still used only through order $r$, and the tensor Taylor remainder
only through order $s$.

The first tail inequality is immediate. For the second, apply H\"older
with $1/(6/5)=1/2+1/3$ to
$\langle y\rangle^NF\,\mathbf1_{|y|>L}\langle y\rangle^{-N}$.
The $L^3$ norm of the second factor is at most $C_NL^{1-N}$ in three
dimensions, by radial integration. The weighted H\"older argument proves \eqref{eq:v64-tail-bounds}. Choosing $N=s$ is legitimate since $s\ge2$.
\end{proof}

\subsection{The physical residual and energy expansion}\label{sec:r2-physical-residual}

The weak residual splits into Taylor interactions, a tensor remainder, and
two cutoff commutators. Each term is estimated against arbitrary zero-trace
physical tests.

For this proof, abbreviate
\begin{equation}\label{eq:boundary-residual-notation}
 \begin{gathered}
 V_\epsilon=\sum_{r=0}^{s-1}\epsilon^r V_r,\qquad
 \chi_\epsilon(y)=\chi_p(\epsilon y),\qquad
 \mathbb A_\epsilon(y)=\mathbb A_p(\epsilon y),\\
 \mathbb T_\epsilon(y)=\sum_{\ell=0}^{s-1}\epsilon^\ell\mathbb A_{p,\ell}(y),
 \qquad \mathbb R_\epsilon=\mathbb A_\epsilon-\mathbb T_\epsilon.
 \end{gathered}
\end{equation}
$\mathbb A_\epsilon$ is used only on the scaled chart. The polynomial
$\mathbb T_\epsilon$ is defined throughout the half-space. Here $\mathbb R_\epsilon$ denotes the Taylor remainder defined in
\eqref{eq:boundary-residual-notation}.
Let $w_\epsilon$ be the scaled physical test defined on the scaled chart by
\begin{equation}\label{eq:w5-scaled-test-definition}
 w_\epsilon(y)=\epsilon^{1/2}R_p^T w(\Psi_p(\epsilon y)).
\end{equation}
All terms containing it in \eqref{eq:v64-four-term-identity} have support inside that chart.

\begin{proposition}
\label{prop:boundary-residual}
For $U_{p,\epsilon}$ in \eqref{eq:v63-transplant} and every
$w\in H_0^1(\Omega;\mathbb{C}^3)$, one has the exact identity
\begin{equation}\label{eq:v64-four-term-identity}
 \mathfrak a(U_{p,\epsilon},w)
       =H_\epsilon(w)+D_\epsilon(w)+C_{1,\epsilon}(w)
                                      -C_{2,\epsilon}(w),
\end{equation}
with
\begin{equation}\label{eq:v64-four-terms}
 \begin{aligned}
 H_\epsilon(w)&=
   \sum_{\substack{0\le\ell,r<s\\\ell+r\ge s}}
       \epsilon^{\ell+r}\mathfrak a_\ell(V_r,\chi_\epsilon w_\epsilon),\\
 D_\epsilon(w)&=\int \chi_\epsilon \mathbb R_\epsilon\nabla V_\epsilon:
                                         \overline{\nabla w_\epsilon},\\
 C_{1,\epsilon}(w)&=\int \mathbb A_\epsilon
             (V_\epsilon\otimes\nabla\chi_\epsilon):
                                         \overline{\nabla w_\epsilon},\\
 C_{2,\epsilon}(w)&=\int \mathbb T_\epsilon\nabla V_\epsilon:
                 (\overline{w_\epsilon}\otimes\nabla\chi_\epsilon).
 \end{aligned}
\end{equation}
The integrals are over the scaled coordinate half-ball, with zero
extension on the region of vanishing cutoffs. Uniformly in the material, $p$ and
$0<\epsilon<\epsilon_*$,
\begin{equation}\label{eq:v64-four-bounds}
 \begin{aligned}
 |H_\epsilon(w)|+|D_\epsilon(w)|
       &\le C\epsilon^s\lVert w\rVert _{H^1},\\
 |C_{1,\epsilon}(w)|&\le C\epsilon^{s+1}\lVert w\rVert _{H^1},\\
 |C_{2,\epsilon}(w)|&\le C\epsilon^s\lVert w\rVert _{H^1}.
 \end{aligned}
\end{equation}
Moreover, for the fixed test $\phi$ used to construct
$U_{p,\epsilon}$,
\begin{equation}\label{eq:v64-approximate-energy}
\left|\mathfrak a(U_{p,\epsilon},U_{p,\epsilon})
-\sum_{r=0}^{s-1}\epsilon^rE_{p,r}^{\phi}\right|
\le C\epsilon^s,
\end{equation}
where $E_{p,r}^{\phi}=\sum_{\ell+b+c=r}\mathfrak a_\ell(V_b,V_c)$,
with $\mathfrak a_\ell$ defined in \eqref{eq:v63-Taylor-tensors}
and the profiles defined in \eqref{eq:v63-Poisson} and
\eqref{eq:v63-recursion}.
\end{proposition}
\begin{proof}
Physical scaling and \eqref{eq:v64-coordinate-tensor} give exactly
\begin{equation}\label{eq:r2-calculation-067}
\mathfrak a(U_{p,\epsilon},w)
 =\int \mathbb A_\epsilon\nabla(\chi_\epsilon V_\epsilon):
                                       \overline{\nabla w_\epsilon}.
\end{equation}
Apply the product rule, insert $\mathbb A_\epsilon=\mathbb T_\epsilon+\mathbb R_\epsilon$
only in the term with $\chi_\epsilon\nabla V_\epsilon$, and move
$\chi_\epsilon$ onto the test in the $\mathbb T_\epsilon$ term. The last operation contributes $-C_{2,\epsilon}$. The remaining Taylor-polynomial
integral is
\begin{equation}\label{eq:r2-calculation-068}
\sum_{\ell,r<s}\epsilon^{\ell+r}\mathfrak a_\ell(V_r,
\chi_\epsilon w_\epsilon).
\end{equation}
For every $n<s$, its coefficient of $\epsilon^n$ is
\begin{equation}\label{eq:r2-calculation-069}
\sum_{\ell+r=n}\mathfrak a_\ell(V_r,\chi_\epsilon w_\epsilon)=0.
\end{equation}
For $n=0$ this is the frozen homogeneous equation. For $n\ge1$ it is
\eqref{eq:v63-recursion}. The test is in $\dot H_0^1(\mathbb{R}^3_+)$ because
$w$ has zero physical boundary trace and the cutoff vanishes before the
edge of the chart. This gives the identity \eqref{eq:v64-four-term-identity}.

We record the scale factors. The volume scaling is $\epsilon^3$ and each
gradient scaling is $\epsilon^{-3/2}$, so there is no extra prefactor in
the identity. Also
\begin{equation}\label{eq:v64-scaled-norms}
 \lVert \nabla w_\epsilon\rVert _2+\lVert w_\epsilon\rVert _6
   +\lVert \nabla(\chi_\epsilon w_\epsilon)\rVert _2
       \le C\lVert w\rVert _{H^1(\Omega)},\qquad
 \lVert \nabla\chi_\epsilon\rVert _3\le C.
\end{equation}
These bounds are local on the scaled chart; its radius is of order
$\epsilon^{-1}$. No Poincare constant for that expanding chart is
used. The $L^6$ bound follows from the physical three-dimensional Sobolev
inequality; $\epsilon^{1/2}$ cancels the change of volume in that norm.
The last gradient bound follows by $L^3$--$L^6$ H\"older.

Each polynomial $\mathbb A_{p,\ell}$ grows at most as $\langle y\rangle^\ell$,
and $\ell\le s-1$. The uniform weighted half-space profile bounds therefore give
\begin{equation}\label{eq:r2-calculation-070}
|\mathfrak a_\ell(V_r,\chi_\epsilon w_\epsilon)|
       \le C\mathfrak B_s\lVert w\rVert _{H^1}.
\end{equation}
Since $\ell+r\ge s$ in $H_\epsilon$ and $\epsilon\le1$, the asserted
bound for $H_\epsilon$ follows. Similarly,
$|\mathbb R_\epsilon(y)|\le C\epsilon^s|y|^s$ and the first gradient bound in
\eqref{eq:v64-scaled-norms} control $D_\epsilon$. There is no derivative
of $\mathbb R_\epsilon$ in this calculation.

There are fixed $a,A>0$ such that
\begin{equation}\label{eq:r2-physical-cutoff-annulus}
\operatorname{supp}\nabla\chi_\epsilon\subset
       \{a/\epsilon\le |y|\le A/\epsilon\},\qquad
 |\nabla\chi_\epsilon|\le C\epsilon.
\end{equation}
On this set $\mathbb A_\epsilon$ is bounded by the physical coefficient bound,
and $\mathbb T_\epsilon$ is also bounded because
$\epsilon^\ell|\mathbb A_{p,\ell}(y)|\le C|\epsilon y|^\ell$.
The first tail inequality with $N=s$ yields
\begin{equation}\label{eq:r2-calculation-072}
\lVert \mathbb A_\epsilon(V_\epsilon\otimes\nabla\chi_\epsilon)\rVert _2
     \le C\epsilon\,\epsilon^s\mathfrak B_s.
\end{equation}
By Cauchy--Schwarz, \eqref{eq:r2-calculation-072}, and the gradient bound in
\eqref{eq:v64-scaled-norms}, we obtain the required estimate for
$C_{1,\epsilon}$.
For the other commutator, the second tail inequality gives
\begin{equation}\label{eq:r2-calculation-073}
\lVert \mathbf1_{\operatorname{supp}\nabla\chi_\epsilon}
          \mathbb T_\epsilon\nabla V_\epsilon\nabla\chi_\epsilon\rVert _{6/5}
     \le C\epsilon\,\epsilon^{s-1}\mathfrak B_s.
\end{equation}
By H\"older's inequality, \eqref{eq:r2-calculation-073}, and the $L^6$ bound
for $w_\epsilon$ in \eqref{eq:v64-scaled-norms}, we obtain the required
estimate for $C_{2,\epsilon}$. Equivalently, with $\mathcal R_\epsilon$ the
physical functional in \eqref{eq:v64-four-term-identity},
\begin{equation}\label{eq:w5-physical-Hminus1}
 \|\mathcal R_\epsilon\|_{H^{-1}(\Omega)}
 :=\sup_{w\in H_0^1(\Omega),\,\|w\|_{H^1}=1}
       |\mathfrak a(U_{p,\epsilon},w)|\le C\epsilon^s.
\end{equation}
For a scalar scaled density $F$, the corresponding physical density
$F_\epsilon^{\rm phys}(\Psi_p(\epsilon y))=\epsilon^{-5/2}F(y)$
(with the bounded tensor/Jacobian factors retained) pairs against
$w_\epsilon$ with no extra power. Its $L^{6/5}$ norm is invariant,
since $3/(6/5)-5/2=0$. A scaled flux becomes a physical flux of size
$\epsilon^{-3/2}$ and its $L^2$ norm is invariant, since $3/2-3/2=0$.
These are exactly the density pairing used for $C_{2,\epsilon}$ and
the flux pairings used for $D_\epsilon,C_{1,\epsilon}$.
The high-interaction term $H_\epsilon$ uses
$\nabla(\chi_\epsilon w_\epsilon)$, whose bound is already in
\eqref{eq:v64-scaled-norms}. Thus the dual estimate is on the fixed
physical domain, with no inhomogeneous $L^2$ bound for $w_\epsilon$
on the expanding half-ball.

For the approximate energy $\mathfrak a(U_{p,\epsilon},U_{p,\epsilon})$,
the Taylor error is controlled by
$\epsilon^s\int |y|^s|\nabla V_\epsilon|^2$. To make the energy-cutoff
terms explicit, set
\begin{equation}\label{eq:r3-energy-cutoff-flux}
Q_\epsilon=V_\epsilon\otimes\nabla\chi_\epsilon,
\end{equation}
where $V_\epsilon$ and $\chi_\epsilon$ are defined in
\eqref{eq:boundary-residual-notation}.
The already proved tail estimate gives
$\|Q_\epsilon\|_2\le C\epsilon^{s+1}\mathfrak B_s$, whereas
$\|\chi_\epsilon\nabla V_\epsilon\|_2\le C\mathfrak B_s$.
Boundedness of $\mathbb A_\epsilon$ therefore bounds the cross terms involving
one $Q_\epsilon$ by $C\epsilon^{s+1}\mathfrak B_s^2$ and the quadratic
term by $C\epsilon^{2s+2}\mathfrak B_s^2$. No differentiation of the
material remainder occurs in this calculation. The remaining cutoff
factor $\chi_\epsilon^2$ differs from one only on the remote tail.
Extending a Taylor-polynomial integral to the full half-space adds at most
\begin{equation}\label{eq:v64-energy-tail}
 \epsilon^\ell\int_{|y|>a/\epsilon}|y|^\ell
                     |\nabla V_b|\,|\nabla V_c|\,\mathrm{d} y
 \le C\epsilon^{2s}\mathfrak B_s^2,
 \qquad 0\le\ell<s.
\end{equation}
Indeed insert two weights $\langle y\rangle^s$ and bound the remaining
power on the tail. All polynomial products of total degree at least
$s$ are $O(\epsilon^s)$ by the finite weighted bounds. Thus the retained
coefficients are exactly
\begin{equation}\label{eq:r2-calculation-074}
\sum_{\ell+b+c=r}\mathfrak a_\ell(V_b,V_c),
\end{equation}
which proves \eqref{eq:v64-approximate-energy}.
\end{proof}

\section*{Acknowledgments}

This work was supported by the National Natural Science Foundation of China (Grant No. 12371422) and the Fundamental Research Funds for the Central Universities (Jilin University). The authors used ChatGPT (OpenAI) for exploratory
mathematical discussions and editorial assistance
during the preparation of this manuscript.
The authors take full responsibility for all
mathematical statements, arguments, and proofs.


\begin{thebibliography}{99}

\bibitem{Ahlfors1979}
L.~V.~Ahlfors,
\newblock \emph{Complex Analysis: An Introduction to the Theory
of Analytic Functions of One Complex Variable},
\newblock 3rd ed., McGraw--Hill, New York, 1979.

\bibitem{AkamatsuNakamuraSteinberg1991}
M.~Akamatsu, G.~Nakamura, and S.~Steinberg,
\newblock Identification of Lam\'e coefficients from boundary observations,
\newblock \emph{Inverse Problems} \textbf{7} (1991),
no.~3, 335--354.

\bibitem{Alessandrini1988}
G.~Alessandrini,
\newblock Stable determination of conductivity by boundary measurements,
\newblock \emph{Appl. Anal.} \textbf{27} (1988), 153--172.

\bibitem{AlessandriniDiCristoMorassiRosset2014}
G.~Alessandrini, M.~Di Cristo, A.~Morassi, and E.~Rosset,
\newblock Stable determination of an inclusion in an elastic body
by boundary measurements,
\newblock \emph{SIAM J. Math. Anal.} \textbf{46} (2014),
no.~4, 2692--2729.

\bibitem{AngIkehataTrongYamamoto1998}
D.~D.~Ang, M.~Ikehata, D.~D.~Trong, and M.~Yamamoto,
\newblock Unique continuation for a stationary isotropic Lam\'e
system with variable coefficients,
\newblock \emph{Comm. Partial Differential Equations}
\textbf{23} (1998), nos.~1--2, 371--385.

\bibitem{AstalaPaivarinta2006}
K.~Astala and L.~P\"aiv\"arinta,
\newblock Calder\'on's inverse conductivity problem in the plane,
\newblock \emph{Ann. of Math.} (2) \textbf{163} (2006),
no.~1, 265--299.

\bibitem{BerettaDeHoopFranciniVessellaZhai2017}
E.~Beretta, M.~V.~de Hoop, E.~Francini, S.~Vessella, and J.~Zhai,
\newblock Uniqueness and Lipschitz stability of an inverse
boundary value problem for time-harmonic elastic waves,
\newblock \emph{Inverse Problems} \textbf{33} (2017),
no.~3, 035013.

\bibitem{BerettaFranciniMorassiRossetVessella2014}
E.~Beretta, E.~Francini, A.~Morassi, E.~Rosset, and S.~Vessella,
\newblock Lipschitz continuous dependence of piecewise constant
Lam\'e coefficients from boundary data: the case of non-flat interfaces,
\newblock \emph{Inverse Problems} \textbf{30} (2014),
no.~12, 125005.

\bibitem{BerettaFranciniVessella2014}
E.~Beretta, E.~Francini, and S.~Vessella,
\newblock Uniqueness and Lipschitz stability for the identification
of Lam\'e parameters from boundary measurements,
\newblock \emph{Inverse Probl. Imaging} \textbf{8} (2014),
no.~3, 611--644.

\bibitem{Calderon1980}
A.-P.~Calder\'on,
\newblock On an inverse boundary value problem,
\newblock in \emph{Seminar on Numerical Analysis and its
Applications to Continuum Physics}
(Rio de Janeiro, 1980),
Soc. Brasil. Mat., Rio de Janeiro, 1980, 65--73.

\bibitem{CaroLaiLinZhou2021}
P.~Caro, R.-Y.~Lai, Y.-H.~Lin, and T.~Zhou,
\newblock Boundary determination of electromagnetic and
Lam\'e parameters with corrupted data,
\newblock \emph{Inverse Probl. Imaging} \textbf{15} (2021),
no.~5, 1171--1198.

\bibitem{Cekic2025}
M.~Ceki\'c,
\newblock Calder\'on problem for systems via complex parallel transport,
\newblock \emph{SIAM J. Math. Anal.} \textbf{57} (2025),
no.~4, 3524--3560.

\bibitem{ChenJiangLiuTao2026}
L.~Chen, Y.~Jiang, H.~Liu, and L.~Tao,
\newblock On an inverse boundary problem in elastodynamics,
\newblock preprint, arXiv:2609.20134v1 (2026).

\bibitem{Ciarlet1988}
P.~G.~Ciarlet,
\newblock \emph{Mathematical Elasticity. Vol.~I:
Three-Dimensional Elasticity},
\newblock Studies in Mathematics and Its Applications,
vol.~20, North-Holland, Amsterdam, 1988.

\bibitem{DDsource}
H.~Diao and U.~Durdiev,
\newblock Unique determination of the Lam\'e parameters
in three dimensions,
\newblock preprint, arXiv:2610.14741v1 (2026).

\bibitem{EberleHarrachMeftahiRezgui2021}
S.~Eberle, B.~Harrach, H.~Meftahi, and T.~Rezgui,
\newblock Lipschitz stability estimate and reconstruction
of Lam\'e parameters in linear elasticity,
\newblock \emph{Inverse Probl. Sci. Eng.} \textbf{29} (2021),
no.~3, 396--417.

\bibitem{EskinRalston2002}
G.~Eskin and J.~Ralston,
\newblock On the inverse boundary value problem for linear
isotropic elasticity,
\newblock \emph{Inverse Problems} \textbf{18} (2002),
no.~3, 907--921.

\bibitem{ER}
G.~Eskin and J.~Ralston,
\newblock Inverse boundary value problems for systems of
partial differential equations,
\newblock in \emph{Recent Development in Theories and Numerics},
World Scientific Publishing, River Edge, NJ, 2003, 105--113.

\bibitem{EskinRalston2004}
G.~Eskin and J.~Ralston,
\newblock On the inverse boundary value problem for linear
isotropic elasticity and Cauchy--Riemann systems,
\newblock in \emph{Inverse Problems and Spectral Theory},
Contemp. Math. \textbf{348},
Amer. Math. Soc., Providence, RI, 2004, 53--69.

\bibitem{FP}
F.~Forstneri\v{c} and J.~Prezelj,
\newblock Oka's principle for holomorphic fiber bundles with sprays,
\newblock \emph{Math. Ann.} \textbf{317} (2000), 117--154.

\bibitem{GarciaZhang2013}
A.~Garc\'ia and G.~Zhang,
\newblock Reconstruction from boundary measurements for less
regular conductivities,
\newblock \emph{Inverse Problems} \textbf{32} (2016),
no.~11, 115015.

\bibitem{ImanuvilovUhlmannYamamoto2012}
O.~Yu.~Imanuvilov, G.~Uhlmann, and M.~Yamamoto,
\newblock On uniqueness of Lam\'e coefficients from partial
Cauchy data in three dimensions,
\newblock \emph{Inverse Problems} \textbf{28} (2012),
no.~12, 125002.

\bibitem{ImanuvilovYamamoto2015}
O.~Yu.~Imanuvilov and M.~Yamamoto,
\newblock Global uniqueness in inverse boundary value problems
for the Navier--Stokes equations and Lam\'e system in two dimensions,
\newblock \emph{Inverse Problems} \textbf{31} (2015),
no.~3, 035004.

\bibitem{LinNakamura2017}
Y.-H.~Lin and G.~Nakamura,
\newblock Boundary determination of the Lam\'e moduli for the
isotropic elasticity system,
\newblock \emph{Inverse Problems} \textbf{33} (2017),
no.~12, 125004.

\bibitem{Mandache2001}
N.~Mandache,
\newblock Exponential instability in an inverse problem
for the Schr\"odinger equation,
\newblock \emph{Inverse Problems} \textbf{17} (2001),
no.~5, 1435--1444.

\bibitem{McDuffSalamon2004}
D.~McDuff and D.~Salamon,
\newblock \emph{$J$-Holomorphic Curves and Symplectic Topology},
\newblock American Mathematical Society Colloquium Publications,
vol.~52, Amer. Math. Soc., Providence, RI, 2004.

\bibitem{Nachman1996}
A.~I.~Nachman,
\newblock Global uniqueness for a two-dimensional inverse
boundary value problem,
\newblock \emph{Ann. of Math.} (2) \textbf{143} (1996),
no.~1, 71--96.

\bibitem{NakamuraUhlmann1994}
G.~Nakamura and G.~Uhlmann,
\newblock Global uniqueness for an inverse boundary problem
arising in elasticity,
\newblock \emph{Invent. Math.} \textbf{118} (1994), 457--474.

\bibitem{NakamuraUhlmann1995}
G.~Nakamura and G.~Uhlmann,
\newblock Inverse problems at the boundary for an elastic medium,
\newblock \emph{SIAM J. Math. Anal.} \textbf{26} (1995),
no.~2, 263--279.

\bibitem{NakamuraUhlmannErratum2003}
G.~Nakamura and G.~Uhlmann,
\newblock Erratum: Global uniqueness for an inverse boundary
value problem arising in elasticity,
\newblock \emph{Invent. Math.} \textbf{152} (2003), 205--207.

\bibitem{OpenAIElasticity2026}
OpenAI,
\newblock Global uniqueness for the smooth isotropic elasticity
inverse problem,
\newblock \href{https://github.com/openai/math/blob/main/preprints/Global-Uniqueness-for-the-Smooth-Isotropic-Elasticity-Inverse-Problem-September-24-2026/article.pdf}{OpenAI Math Release preprint} (2026).

\bibitem{SylvesterUhlmann1987}
J.~Sylvester and G.~Uhlmann,
\newblock A global uniqueness theorem for an inverse
boundary value problem,
\newblock \emph{Ann. of Math.} (2) \textbf{125} (1987),
no.~1, 153--169.

\bibitem{TanLiu2023}
X.~Tan and G.~Liu,
\newblock Determining Lam\'e coefficients by the elastic
Dirichlet-to-Neumann map on a Riemannian manifold,
\newblock \emph{Inverse Problems} \textbf{39} (2023),
no.~9, 095004.

\end{thebibliography}
\end{document}